\documentclass[letterpaper,11pt]{amsart}
\usepackage{amssymb}
\usepackage{amsthm}
\usepackage[utf8]{inputenc}
\usepackage[margin=1.0in]{geometry}
\usepackage{pst-node}
\usepackage{multirow}
\usepackage{tikz-cd} 
\usepackage{hyperref}
\newtheorem{theorem}{Theorem}[section]
\newtheorem{lemma}[theorem]{Lemma}
\newtheorem{proposition}[theorem]{Proposition}
\newtheorem{corollary}[theorem]{Corollary}
\newtheorem{conjecture}[theorem]{Conjecture}
 
\newtheorem{condition/definition}[theorem]{Condition/Definition} 

\newtheorem{assumption}[theorem]{Assumption}

\usepackage[T1]{fontenc}
\usepackage{appendix}
\usepackage[english]{babel}
\usepackage[utf8]{inputenc}
\usepackage[backend=biber,style=alphabetic]{biblatex}
\def\bb{\mathbb}
\theoremstyle{definition}
\newtheorem{definition}[theorem]{Definition}
\theoremstyle{remark}
\newtheorem{remark}[theorem]{Remark}
\theoremstyle{definition}
\newtheorem{Example}[theorem]{Example}
\newtheorem{setup}[theorem]{Setup}

\usepackage{amsmath}
\usepackage{amsfonts}
\usepackage{tikz-cd}
\def\ul{\underline}
\def\ol{\overline}
\def\Detale{\mathrm{D}_{\text{\'et}}}
\def\ULA{\mathrm{ULA}}
\def\wt{\widetilde}

\def\SL{\mathrm{SL}}
\def\U{\mathrm{U}}
\def\mf{\mathfrak}
\def\ra{\rightarrow}
\def\la{\leftarrow}
\def\lim{\mathop{\rm lim}\nolimits}
\def\colim{\mathop{\rm colim}\nolimits}
\def\Spec{\mathop{\rm Spec}}
\def\Spa{\mathop{\rm Spa}}
\def\Spd{\mathop{\rm Spd}}

\def\proj{\mathrm{proj}}

\def\Ker{\text{Ker}}
\def\Im{\text{Im}}
\def\Coker{\text{Coker}}

\def\Bun{\mathrm{Bun}}
\def\nt{\mathrm{nt}}
\def\Div{\mathrm{Div}}
\def\Eis{\mathrm{Eis}}
\def\nmEis{\mathrm{nEis}}
\def\IC{\mathrm{IC}}
\def\Hck{\mathrm{Hck}}
\def\Hckloc{\mathcal{H}\mathrm{ck}}
\def\Gr{\mathrm{Gr}}
\def\Ind{\mathrm{Ind}}
\def\Perf{\mathrm{Perf}}

\def\bb{\mathbb}
\def\PerfAdm{\mathrm{PfAdm}}
\def\Ad{\mathrm{Ad}}
\def\Aut{\mathrm{Aut}}
\def\Lie{\mathrm{Lie}}
\def\dim{\mathrm{dim}}

\def\Red{\mathrm{Red}}
\def\Irr{\mathrm{Irr}}
\def\Sht{\mathrm{Sht}}
\def\Rep{\mathrm{Rep}}
\def\D{\mathrm{D}}
\def\Sat{\mathrm{Sat}}
\def\RHom{\mathrm{RHom}}
\def\tr{\mathrm{tr}}
\def\GL{\mathrm{GL}}
\def\CT{\mathrm{CT}}
\def\nmCT{\mathrm{nCT}}

\def\Tilt{\mathrm{Tilt}}

\def\ch{\mathrm{ch}}
\def\LLC{\mathrm{LLC}}
\def\GU{\mathrm{GU}}

\def\Gal{\mathrm{Gal}}
\def\CT{\mathrm{CT}}
\newcommand{\Dlis}{\mathrm{D}_{\mathrm{lis}}}
\newcommand{\gamorb}{\mathbb{X}_{*}(T_{\ol{\mathbb{Q}}_{p}})/\Gamma}
\newcommand{\domgamorb}{\mathbb{X}_{*}(T_{\ol{\mathbb{Q}}_{p}})^{+}/\Gamma}
\newcommand{\coinv}{\mathbb{X}_{*}(T_{\ol{\mathbb{Q}}_{p}})_{\Gamma}}
\newcommand{\coinvdom}{\mathbb{X}_{*}(T_{\ol{\mathbb{Q}}_{p}})_{\Gamma}^{+}}
\newcommand{\Q}{\mathbb{Q}}
\newcommand{\C}{\mathbb{C}}
\newcommand{\mc}{\mathcal}
\newcommand{\Trans}{\mathrm{Trans}}
\newcommand{\triv}{\mathrm{triv}}
\newcommand{\id}{\mathrm{id}}
\newcommand{\ov}{\overline}
\newcommand{\Shim}{\mathrm{Sh}}
\newcommand{\A}{\mathbb{A}}
\newcommand{\Ig}{\mathrm{Ig}}
\newcommand{\cochar}{\mathbb{X}_{*}(T_{\ol{\mathbb{Q}}_{p}})}
\newcommand{\domcochar}{\mathbb{X}_{*}(T_{\ol{\mathbb{Q}}_{p}})^{+}}

\title{Geometric Eisenstein Series, Intertwining Operators, and Shin's Averaging Formula}
\author{Linus Hamann, with an Appendix by Alexander Bertoloni-Meli}

\begin{document}

\maketitle
\begin{abstract} In the geometric Langlands program over function fields, Braverman-Gaitsgory \cite{BG} constructed geometric Eisenstein functors which enhance the classical construction of Eisenstein series over function fields by replacing it with an automorphic sheaf on the moduli stack $\Bun_{G}$ of $G$-bundles of a smooth projective curve. Fargues and Scholze \cite{FS} very recently constructed a general candidate for the local Langlands correspondence, via a version of the geometric Langlands correspondence occurring over the Fargues-Fontaine curve. In this note, we carry some theory of geometric Eisenstein functors over to the Fargues-Fontaine setting, and use them to construct  Hecke eigensheaves attached to a sufficiently generic toral parameter $\phi_{T}$ corresponding to a smooth character $\chi$. These eigensheaves have several surprising consequences. They imply special cases of an averaging formula of Shin \cite{Shin,BM} for the cohomology of local Shimura varieties with rational coefficients and generalize them to the non-minuscule case. Moreover, they imply refinements of the averaging formula in the cases where the parameter $\phi$ is sufficiently nice, giving an explicit formula for the degrees of cohomology that certain parabolic inductions sit in, and this refined formula holds even with torsion coefficients. The sufficiently nice condition on $\phi_T$ is related to the zeros and poles of the intertwining operators attached to normalized parabolic induction. Moreover, the eigensheaf provides a geometric realization of these intertwining operators via the relationship of Hecke operators with the cohomology of local shtuka spaces.
\end{abstract}
\tableofcontents
\section{Introduction} 
\subsection{Geometric Eisenstein Series over Function Fields}
Eisenstein series are a way of describing the non-cuspidal automorphic spectrum of a group in terms of the cuspidal automorphic spectrum of its proper Levi subgroups. Over function fields these objects have several geometric incarnations, as first studied extensively by Laumon \cite{Lau} in the case of $\GL_{n}$ and later refined by Braverman-Gaitsgory \cite{BG} for general reductive groups. In particular, if one is interested in the function field of a curve $Y$ over a finite field $k$ then one replaces the functions defining Eisenstein series by certain automorphic sheaves. Namely, let $G/k$ be a split connected reductive group; then one wishes to construct "Eisenstein sheaves" on $\Bun_{G}$ the moduli stack of $G$-bundles on $Y$. To do this, for $P \subset G$ a proper parabolic subgroup with Levi factor $M$ one considers the following diagram of moduli stacks of bundles
\begin{equation}{\label{eqn: EisensteinDiagram}}
\begin{tikzcd}
\Bun_{M} &  \arrow[l,"\mf{q}_{P}"] \Bun_{P} \arrow[r,"\mathfrak{p}_{P}"]  & \Bun_{G} 
\end{tikzcd}
\end{equation}
where $\Bun_{P}$ is the moduli stack of $P$-bundles $\mathcal{G}_{P}$ on $Y$ and the maps $\mathfrak{p}_{P}$ (resp. $\mathfrak{q}_{P}$) send $\mathcal{G}_{P}$ to the $G$-bundle (resp. $M$-bundle) $\mathcal{G}_{P} \times^{P} G$ (resp. $\mathcal{G}_{P} \times^{P} M$). Using this, one can define an Eisenstein functor which (up to shifts and twists) is given by
\[ \Eis_{P}(-) := \mathfrak{p}_{P!}\mathfrak{q}_{P}^{*}(-). \]
Under the function-sheaf dictionary, the values of this functor give rise to the automorphic functions defining Eisenstein series. 

In this geometric context, one can ask for even more. Namely, in the geometric Langlands correspondence one is interested in constructing Hecke eigensheaves on $\Bun_{G}$, and it is natural to ask whether one could upgrade $\Eis_{P}$ to a functor that is well-behaved with respect to the eigensheaf property. In particular, if $\hat{M}$ (resp. $\hat{G}$) denotes the Langlands dual group of $M$ (resp. $G$) one can consider a $\hat{M}$-local system $E_{\hat{M}}$ and a Hecke eigensheaf $\mathcal{S}_{E_{\hat{M}}}$ with eigenvalue $E_{\hat{M}}$. One then considers the induced $\hat{G}$ local system $E_{\hat{G}}$ given by the natural embedding $\hat{M} \hookrightarrow \hat{G}$, and one would like to construct a functor that produces an eigensheaf with eigenvalue $E_{\hat{G}}$ from the Hecke eigensheaf $\mathcal{S}_{E_{\hat{M}}}$. One might hope that $\Eis_{P}(\mathcal{S}_{E_{\hat{M}}})$ works; however, this is too naive. Namely, one expects eigensheaves to be well-behaved under Verdier duality, and one can easily check that $\Eis_{P}(-)$ will not commute with Verdier duality, since the map $\mathfrak{p}_{P}$ is not proper. To remedy this, one considers a compactification
\begin{equation}\label{eqn:GeometricEisensteinDiagram}
\begin{tikzcd}
& \arrow[dl,"\tilde{\mf{q}}_{P}"] \wt{\Bun}_{P} \arrow[dr,"\tilde{\mf{p}}_{P}"] & \\
 \Bun_{M} & \arrow[l,"\mf{q}_{P}"] \Bun_{P} \arrow[r,"\mf{p}_{P}"] \arrow[u,"j",hook]  & \Bun_{G}
\end{tikzcd}
\end{equation}
of the map $\mathfrak{p}_{P}$, where $\tilde{\mf{p}}_{P}: \wt{\Bun}_{P} \ra \Bun_{G}$ is the relative Drinfeld compactification of the map $\mf{p}_{P}$.  

One then considers the Verdier self-dual intersection cohomology sheaf $\IC_{\wt{\Bun}_{P}}$. The desired functor is then given by
\[ \widetilde{\Eis}_{P}(-) := \widetilde{\mathfrak{p}}_{P*}(\widetilde{\mathfrak{q}}_{P}^{*}(-) \otimes \IC_{\widetilde{\Bun}_{P}})\]
One of the main results of Braverman-Gaitsgory \cite{BG} is that this satisfies the desired Hecke eigensheaf property when applied to $\mathcal{S}_{E_{\hat{M}}}$. One may also wonder what this corresponds to at the level of functions. We recall that the classical Eisenstein series is known to satisfy a functional equation after multiplying by an appropriate ratio of $L$-values. The compactified Eisenstein series corresponds to this completed Eisenstein series under the function-sheaf dictionary. In fact, in certain cases one can see the usual functional equation at the sheaf-theoretic level. If $P = B$ is the Borel and we consider the appropriately normalized Hecke eigensheaf $\mathcal{S}_{E_{\hat{T}}}$ associated to $E_{\hat{T}}$ via geometric class field theory then $\widetilde{\Eis}(\mathcal{S}_{\phi_{T}})$ satisfies a functional equation under a regularity hypothesis on the local system $E_{\hat{T}}$. Now let $w \in W_{G}$ be an element of the Weyl group of $G$ with $\widetilde{w} \in N(T)$ a choice of representative. Then $\wt{w}$ acts on $\Bun_{T}$, and, if we write $\mathcal{S}_{E_{\hat{T}}}^{w}$ for the pullback of $\mathcal{S}_{E_{\hat{T}}}$ along this automorphism, we have the following result.
\begin{theorem}{\cite[Theorem~2.2.4]{BG}}{\label{classicalfunctequation}}
For $E_{\hat{T}}$ a regular $\hat{T}$-local system on $Y$ and a choice of representative $\widetilde{w} \in N(\hat{T})$ of $w \in W_{G}$, we have an isomorphism 
\[ \widetilde{\Eis}_{B}(\mathcal{S}_{E_{\hat{T}}}) \simeq \widetilde{\Eis}_{B}(\mathcal{S}^{w}_{E_{\hat{T}}}) \]
of $\ell$-adic sheaves on $\Bun_{G}$.
\end{theorem}
As alluded to above, one can see that, after passing to functions, this gives precisely the well-known functional equation for the Eisenstein series multiplied by the appropriate ratio of $L$-values (\cite[Section~2.2]{BG}). Moreover, by  \cite[Theorem~1.5]{BG2}, under the regularity assumption the sheaf $\widetilde{\Eis}_{B}(\mathcal{S}_{E_{\hat{T}}})$ is perverse.

The main goal of this note is to explore what this theory of geometric Eisenstein series has to tell us in the context of the recent geometric Langlands correspondence constructed by Fargues and Scholze.
\subsection{Geometric Eisenstein Series over the Fargues-Fontaine Curve}
We will assume that $G/\bb{Q}_{p}$ is a quasi-split group. Let $A \subset T \subset B \subset G$ be a choice of maximal split torus, maximal torus, and Borel, respectively. In this paper, we will be concerned with studying the geometric Eisenstein functor over the Fargues-Fontaine curve in the principal case (i.e where the parabolic $P \subset G$ is the Borel). Restricting to the principal case has several advantages. For one, one can unconditionally speak about the Hecke eigensheaf attached to a parameter $\phi_{T}: W_{\mathbb{Q}_{p}} \ra \phantom{}^{L}T$ valued in the maximal torus. Additionally, in this case the Drinfeld compactification of the moduli space of $B$-structures $\Bun_{B}$ has a fairly manageable geometry. 
\subsubsection{Geometric Eisenstein Series}
Throughout, we will let $\phi_{T}: W_{\mathbb{Q}_{p}} \rightarrow \phantom{}^{L}T(\Lambda)$ be an $L$-parameter valued in the $L$-group of $T$, where $\Lambda/\bb{Z}_{\ell}[\sqrt{p}]$ is a coefficient system Setup \ref{assumption: coefficients}. Our aim is to construct an eigensheaf (Definition \ref{defeigsheaf}) with eigenvalue $\phi$, the composition of $\phi_{T}$ with the natural embedding $\phantom{}^{L}T(\Lambda) \rightarrow \phantom{}^{L}G(\Lambda)$. This will be an object in $\Dlis(\Bun_{G},\Lambda)$ the category of lisse-\'etale solid $\Lambda$-sheaves, as defined in \cite[Chapter~VII]{FS} on $\Bun_{G}$, the moduli stack of $G$-bundles on the Fargues-Fontaine curve $X$. We recall that this is equipped with a locally closed Harder-Narasimhan (HN) stratification $j_{b}: \Bun_{G}^{b} \hookrightarrow \Bun_{G}$ for $b \in B(G)$, the Kottwitz set of $G$. Moreover, the category of sheaves on the strata $\Dlis(\Bun_{G}^{b},\Lambda)$ identifies with $\D(J_{b}(\bb{Q}_{p}),\Lambda)$ the unbounded derived category of smooth representations of $J_{b}(\bb{Q}_{p})$ on $\Lambda$-modules, where $J_{b}$ is an inner form of a Levi $M_{b}$ of $G$ attached to $b$.

We do not work directly with the subcategory $\Dlis(\Bun_{G},\Lambda)$ of solid sheaves in our argument, as the usual six functors are not as well behaved in this case. Instead, we will first restrict to the case where $\Lambda = \mathcal{O}_{F}/\varpi^{n}$ for $F/\bb{Q}_{\ell}$ an algebraic extension such that $\mathcal{O}_{F}$ is a DVR and $\varpi \in \mathcal{O}_{F}$ a uniformizing element in its ring of integers (Setup \ref{assumption: coefficients} (1)). In this case, one has an isomorphism $\Dlis(\Bun_{G},\Lambda) \simeq \D_{\text{\'et}}(\Bun_{G},\Lambda)$ with the usual category of \'etale $\Lambda$-sheaves as defined in \cite{Ecod}. We then construct the lisse-\'etale sheaves for more general coefficient systems such as $\mathcal{O}_{F}$ and $F$ from this sheaf. For the first part of this section, we will always assume that $\Lambda = \mathcal{O}_{F}/\varpi^{n}$ unless otherwise stated, and to simplify the notation, we adopt the convention that, when working within this context, we will denote the derived category of \'etale sheaves on a $v$-stack or diamond $Z$ by simply writing $\D(Z)$. In order to avoid complications in the $\ell$-modular setting, we will additionally assume that $\ell$ is very decent with respect to $G$ throughout in the following sense: 
\begin{enumerate}
\item We have that $\ell \nmid \pi_{0}(Z(G))$, where $Z(G)$ denotes the center of $G$.
\item We have that $\ell \nmid |Q|$, where $Q$ is the smallest quotient through which the $W_{\mathbb{Q}_{p}}$, the Weil group of $\mathbb{Q}_{p}$, acts on the Langlands dual group $\hat{G}$. 
\end{enumerate}
The former assumption will guarantee that we have access to the spectral action constructed in Fargues-Scholze and more notably the excursion algebra, and the latter assumption will guarantee that various categories of tilting modules of algebraic representations are well-behaved. All our moduli stacks (e.g $\Bun_{G}$ and so on) will be viewed as $v$-stacks over $\Spd(\ol{\bb{F}}_{p})$; in other words, they are groupoid valued functors on the category $\Perf_{\ol{\bb{F}}_{p}}^{\mathrm{op}}$ of affinoid perfectoid $\ol{\bb{F}}_{p}$-algebras.

We let $\mathcal{S}_{\phi_{T}}$ be the eigensheaf on the moduli stack $\Bun_{T}$ parameterizing $T$-bundles on $X$ attached to $\phi_{T}$ by Fargues \cite{Fa1,Fa} and Zou \cite{Zou}. Our aim is to construct an eigensheaf with respect to the parameter $\phi$ by applying a geometric Eisenstein functor to $\mathcal{S}_{\phi_{T}}$. To do this, one needs to show that the relevant geometric objects are well-behaved. Namely, one can show that the moduli stack of $B$-bundles on $X$, denoted $\Bun_{B}$, gives rise to an Artin $v$-stack and the maps in the natural Eisenstein diagram
\begin{equation}{\label{eqn: pullpush}}
\begin{tikzcd}
\Bun_{T} &  \Bun_{B} \arrow[r,"\mathfrak{p}"] \arrow[l,"\mathfrak{q}"] & \Bun_{G}   \\
\end{tikzcd}
\end{equation}
have good geometric properties; $\mathfrak{q}$ is a $\ell$-cohomologically smooth (non-representable) map of Artin $v$-stacks, and $\mathfrak{p}$ is compactifiable and representable in locally spatial diamonds (see \S \ref{bstrucproperties}). It follows that one has a well-defined functor given by $\mathfrak{p}_{!}\mathfrak{q}^{*}:\D(\Bun_{T}) \ra \D(\Bun_{G})$ using the functors constructed in \cite{Ecod}. 

In order to make the functor $\mathfrak{p}_{!}\mathfrak{q}^{*}$ behave well with respect to Verdier duality, one needs to take into account the appropriate shifts and twists coming from the dualizing object on $\Bun_{B}$. Namely, using the cohomological smoothness of $\mf{q}$ and $\Bun_{T}$, it is easy to see that the moduli stack $\Bun_{B}$ is $\ell$-cohomologically smooth of some $\ell$-dimension given by a locally constant function $\dim(\Bun_{B}): |\Bun_{B}| \ra \mathbb{Z}$, where $|\Bun_{B}|$ is the underlying topological space of $\Bun_{B}$. Moreover, the dualizing object on the nose is described in terms of the modulus character. These modulus character twists can be encoded by the sheaf $\Delta_{B}^{1/2}$ on $\Bun_{T}$, whose value on the connected components $\Bun_{T}^{\nu} \simeq [\ast/\ul{T(\bb{Q}_{p})}]$ will be given by the square root of the modulus character $\delta_{B}^{1/2}:T(\mathbb{Q}_{p}) \ra \Lambda^{*}$. Similarly, we write $\Delta_{B}$ for the Hecke eigensheaf whose stalks will be given by $\delta_{B}$. It now follows from \cite[Theorem~1.6]{GH} that the dualizing complex on $\Bun_{B}$ is isomorphic to $\mf{q}^{*}(\Delta_{B})[2\dim(\Bun_{B})]$, and it follows that the sheaf $\IC_{\Bun_{B}} := \mf{q}^{*}(\Delta_{B}^{1/2})[\dim(\Bun_{B})]$ is Verdier self-dual. Therefore,  we define the normalized Eisenstein functor to be
\[ \nmEis(-) := \mf{p}_{!}(\mf{q}^{*}(-) \otimes \IC_{\Bun_{B}}). \]
This is already very suggestive. Indeed, the normalized Eisenstein functor will have stalks related to the normalized parabolic induction. Moreover, just as smooth duality interacts nicely with normalized parabolic induction so does Verdier duality with the normalized Eisenstein functor. In order to study how the normalized Eisenstein functor interacts with Verdier duality, it becomes very natural to want a nice compactification of the morphism $\mf{p}$, as $\Eis$ involves the functor $\mf{p}_{!}$. As seen in \S $1.1$, this can be accomplished by considering an analogue of the Drinfeld compactification $\ol{\Bun}_{B} = \wt{\Bun}_{B}$. This object was studied in the Fargues-Scholze context for more general parabolics in \cite{HHS}. In particular, it was shown there that there exists an Artin $v$-stack $\overline{\Bun}_{B}$ equipped with an open immersion
\[ j: \Bun_{B} \hookrightarrow \overline{\Bun}_{B} \]
such that the map $\overline{\mf{p}}: \overline{\Bun}_{B} \ra \Bun_{G}$ is proper (after restricting to finitely many connected components of $\Bun_{T}$).

Now, we would like to claim that the sheaf $\nmEis(\mathcal{S}_{\phi_{T}})$ is a Hecke eigensheaf with respect to the parameter $\phi$ given by composing $\phi_{T}$ with the natural embedding $\phantom{}^{L}T \ra \phantom{}^{L}G$; however, in analogy with \S $1.1$, the right object to consider in this case is not $\nmEis(\mathcal{S}_{\phi_{T}})$, but rather a compactified version $\ol{\Eis}(\mathcal{S}_{\phi_{T}})$. Unfortunately, there is currently no well-behaved formalism for intersection cohomology in the context of diamonds and $v$-stacks. This prevents us from even defining the kernel sheaf $\IC_{\ol{\Bun}_{B}}$ typically used in the definition of $\ol{\Eis}(-)$ in any naive way. There is however a way out if we impose some conditions on our parameter $\phi_{T}$. To explain this, we write $\gamorb$ for the set of $\Gamma$-orbits of geometric cocharacters. Let $E$ be the splitting field of $G$. Given an element $\nu \in \gamorb$ with reflex $E_{\nu}$ and a representation $\tau$ of $\Gal(E/E_{\nu})$, we can attach to it a representation of $\phantom{}^{L}T$ by inducing the representation of $\hat{T}$ defined by a representative of the orbit $\nu$ in $\mathbb{X}_*(T_{\ol{\mathbb{Q}}_{p}})$ and twisting by the representation $\tau$.  We denote this representation by $\nu_{\tau}$ and if $\tau$ is trivial simply by $\nu$. For all $\alpha \in \gamorb$ defined by the $\Gamma$-orbits of coroots in $\mathbb{X}_{*}(T_{\ol{\mathbb{Q}}_{p}})$, there is a natural $\tau$ defined by the looking at the action of $W_{E_{\alpha}}$ on the root space $\hat{\mf{g}}_{\tilde{\alpha}}$ of $\hat{G}$ on a representative $\tilde{\alpha} \in \cochar$ of $\alpha$, and we denote the representation of $\phantom{}^{L}T$ given by $\alpha$ and this character by $\alpha_{\Ad}$. The key condition is the following.
\begin{definition}{(Definition \ref{normregcond}))}{\label{defn: genericityintro}}
For any $\Lambda/\bb{Z}_{\ell}[\sqrt{p}]$ and a parameter $\phi_{T}: W_{\mathbb{Q}_{p}} \ra \phantom{}^{L}T(\Lambda)$, we impose the following conditions on $\phi_{T}$.
\begin{enumerate} 
    \item For all $\alpha \in \gamorb$ defined by the $\Gamma$-orbits of coroots in $\mathbb{X}_{*}(T_{\ol{\mathbb{Q}}_{p}})$, the Galois cohomology complex $R\Gamma(W_{\mathbb{Q}_{p}},\alpha_{\Ad} \circ \phi_{T})$ is trivial.
    \item If $\chi$ is the character attached to $\phi_{T}$ under local class field theory. We have that $\chi$ is regular. I.e., for all $w \neq 1$ in the relative Weyl group $W_{G}$ of $G$, we have that
    \[ \chi \not\simeq \chi^{w}. \]
\end{enumerate}
If $\phi_{T}$ satisfies (1) then we say $\phi_{T}$ is generic with respect to $G$. If $\phi_{T}$ satisfies (1)-(2) we say that $\phi_{T}$ is generic regular with respect to $G$. We omit $G$ from the terminology when it is clear from context.
\end{definition}
\begin{remark}{\label{nomregvgeneric}}
The choice of calling such parameters generic is motivated by the analogous notion of decomposed generic considered by Caraiani and Scholze \cite[Definition~1.9]{CS}. In particular, we note that if $\phi_{T}$ is unramified and $G = \GL_{n}$ then it is easy to check that \ref{defn: genericityintro} (1) is precisely equivalent to $\phi_{T}$ being decomposed generic. 
\end{remark}
These conditions may appear mysterious; but there are several ways to see why it is the morally correct condition. Perhaps the most compelling comes from local representation theory. As mentioned in \S $1.1$, the compactified Eisenstein series in the function field setting corresponds to the completed Eisenstein series under the function sheaf dictionary, while the non-compactified Eisenstein series corresponds to just the usual Eisenstein series. In particular, the sheaf $\IC_{\ol{\Bun}_{B}}$ encodes the zeros and poles of the meromorphic continuation of the Eisenstein series. If $\chi$ denotes the character of $T(\mathbb{Q}_{p})$ attached to $\phi_{T}$ via local class field theory then we recall that, if $w \in W_{G}$ is an element of the relative Weyl group, the local analogue of the meromorphic continuation of Eisenstein series is the theory of intertwining operators. In particular, suppose that $\Lambda = \ol{\mathbb{Q}}_{\ell}$, then we have maps 
\[ i_{\chi,w}: i_{B}^{G}(\chi) \ra i_{B}^{G}(\chi^{w}) \]
of smooth $G(\mathbb{Q}_{p})$-representations, which can be viewed as meromorphic functions on the set of complex unramified characters by twisting $\chi$. 

Now, using local Tate-duality, it is easy to see that the vanishing condition imposed in (Definition \ref{defn: genericityintro} (1))  will imply that $\chi$ pre-composed with coroots is not trivial or isomorphic to the norm character (up to additional twists if $G$ is not split (see Example \ref{U3ex}). These are precisely the type of conditions one expects to guarantee that the intertwining operators are holomorphic and give rise to an isomorphism. In fact in the appendix, we show that, for $\chi$ attached to a generic parameter $\phi_{T}$, we always have an isomorphism $i_{B}^{G}(\chi) \simeq i_{B}^{G}(\chi^{w})$ for any $w \in W_{G}$ (Proposition \ref{prop: genericinterisom}). 

This suggests that, at least heuristically, we should always have an isomorphism  $\nmEis(\mathcal{S}_{\phi_{T}}) \simeq \ol{\Eis}(\mathcal{S}_{\phi_{T}})$, for $\phi_{T}$ satisfying some version of genericity and regularity and any reasonable definition of $\ol{\Eis}(\mathcal{S}_{\phi_{T}})$. Indeed, this makes sense when we look at the geometry of $\ol{\Bun}_{B}$; in particular, the stack $\ol{\Bun}_{B}$ admits a locally closed stratification given by $\Div^{(\ol{\nu})} \times \Bun_{B}$, where $\Div^{(\ol{\nu})}$ is a certain partially symmetrized version of the mirror curve $\Div^{1}$ parameterizing effective Cartier divisors in $X$ attached to $\ol{\nu}$ inside the coinvariant lattice $\mathbb{X}_{*}(T_{\ol{\mathbb{Q}}_{p}})_{\Gamma}$. Now, the restriction of $\ol{\mf{q}}^{*}(\mathcal{S}_{\phi_{T}})$ to this stratum is given by pulling back a Hecke operator on $\Bun_{T}$ applied to $\mathcal{S}_{\phi_{T}}$ along $\mf{q}$. One can deduce that the factor appearing on the divisor curve $\Div^{(\ol{\nu})}$ will be related to $\alpha_{\Ad} \circ \phi_{T}$, via the Hecke eigensheaf property $T_{\alpha}(\mathcal{S}_{\phi_{T}}) \simeq \alpha \circ \phi_{T} \boxtimes  \mathcal{S}_{\phi_{T}}$ for $\mathcal{S}_{\phi_{T}}$. This should imply that, if one looks at the natural map 
\[ \nmEis(\mathcal{S}_{\phi_{T}}) \ra \ol{\Eis}(\mathcal{S}_{\phi_{T}}) \]
induced by excision, then the cone of this morphism should have a natural filtration with graded pieces tensored by the Galois cohomology groups $R\Gamma(W_{\mathbb{Q}_{p}},\alpha_{\Ad} \circ \phi_{T})$ for all positive coroots. These Galois cohomology groups will in turn vanish for $\phi_{T}$ generic, suggesting an isomorphism of the form $\nmEis(\mathcal{S}_{\phi_{T}}) \simeq \ol{\Eis}(\mathcal{S}_{\phi_{T}})$ in this case.

We can turn this heuristic into actual math. In particular, since we expect an isomorphism of the form $\nmEis(\mathcal{S}_{\phi_{T}}) \simeq \ol{\Eis}(\mathcal{S}_{\phi_{T}})$ after assuming genericity, we should expect for such parameters that $\nmEis(\mathcal{S}_{\phi_{T}})$ behaves well under Verdier duality, is a perverse Hecke eigensheaf with eigenvalue $\phi$, and satisfies the analogue of the functional equation seen in Theorem \ref{classicalfunctequation} in this case at least assuming $\phi_{T}$ is regular, as in Definition \ref{defn: genericityintro} (2). We will show that this is indeed the case.

 Unfortunately, in general to get the Hecke eigensheaf property for $\nmEis(\mathcal{S}_{\phi_{T}})$ we still need more. Write $(-)^{\Gamma}: \cochar \ra \gamorb$ for the natural map from geometric cocharacters to their Galois orbits. We let $E$ be the splitting field of $G$. Consider a geometric dominant cocharacter $\mu \in \domcochar$ with reflex field $E_{\mu}$ and Galois orbit $\mu^{\Gamma} \in \domgamorb$. We set $Q_{\mu} := \Gal(E/E_{\mu})$. Given an indecomposable projective $\Lambda$-module $\tau$ of the finite group $Q_{\mu}$, we have an associated tilting representation $\mathcal{T}_{\mu^{\Gamma},\tau} \in \Rep_{\Lambda}(\phantom{}^{L}G)$ (see \S \ref{s: TiltingEigensheaves}). When $\tau$ is trivial, we denote this simply by $\mathcal{T}_{\mu^{\Gamma}}$. The following condition will guarantee the Hecke eigensheaf property for $\nmEis(\mathcal{S}_{\phi_{T}})$ and the Hecke operator defined by $\mathcal{T}_{\mu^{\Gamma}}$. 
\begin{definition}{(Definition \ref{def: strongmureg})}
For any $\Lambda/\bb{Z}_{\ell}$, a toral parameter $\phi_{T}: W_{\mathbb{Q}_{p}} \ra \phantom{}^{L}T(\Lambda)$, a geometric dominant cocharacter $\mu$, and an indecomposable representation $\tau$ of $Q_{\mu} := \Gal(E/E_{\mu})$ on projective $\Lambda$-modules, we say $\phi_{T}$ is strongly $(\mu,\tau)$-regular if the Galois cohomology complex 
\[ R\Gamma(W_{\mathbb{Q}_{p}},(\mathcal{T}_{\mu^{\Gamma},\tau} \otimes \mathcal{T}_{\mu^{\Gamma},\tau}^{\vee})|^{\nt}_{\phantom{}^{L}T} \circ \phi_{T}) \]
vanishes. Here the superscript $(-)^{\nt}$ indicates that we pass to the summand with non-trivial weights. If $\tau$ is trivial then we omit it from the terminology (i.e. we just say strongly $\mu$-regular). 
\end{definition}
\begin{remark}{\label{rem: genremarks}}
We note that genericity usually implies strong $\mu$-regularity for some suitably chosen cocharacters $\mu \in \domcochar$. In particular, note that if $\mu = (1,0,\ldots,0,0)$ and $G = \GL_{n}$ then, since $\mu$ is minuscule, the possible weights of $V_{\mu}$ in $\hat{T}$ are all given by Weyl group orbits of $\mu$, and therefore the non-trivial differences between the weights will consequently be precisely the coroots of $\GL_{n}$.
\end{remark}
Let's now see how these conditions manifest in our results on Eisenstein series. We begin with the Hecke eigensheaf property. In particular, consider a finite index set $I$ and a representation $V \in \Rep_{\Lambda}(\phantom{}^{L}G^{I})$. Given $(\nu_{i})_{i \in I} \in (\gamorb)^{I}$ with a choice of representative $(\tilde{\nu}_{i})_{i \in I} \in (\cochar)^{I}$, we write $V((\tilde{\nu}_{i})_{i \in I})$ for the weight space of $V$ restricted to $\hat{T}^{I}$, and endow it with the natural action of $W_{E_{(\nu_{i})_{i \in I}}} := \prod_{i \in I} W_{E_{\nu_{i}}} \subset W_{\bb{Q}_{p}}^{I}$ induced by the Galois action on $V$.  We consider the associated Hecke operator.
\[ T_{(\tilde{\nu}_{i})_{i \in I}}: \D(\Bun_{T}) \ra \D(\Bun_{T})^{BW_{E_{(\nu_{i})_{i \in I}}}}. \]
By applying excision to the aforementioned locally closed stratification of $\ol{\Bun}_{B}$ and combining it with the geometric Satake correspondence of Fargues-Scholze \cite[Chapter~VI]{FS}, we can show the following result. 
\begin{theorem}{(Theorem \ref{filteigensheaf}, Corollary \ref{appliedfilteigsheaf})}
For $\mathcal{F} \in \D(\Bun_{T})$, $I$ a finite index set, and $V \in \Rep_{\Lambda}(\phantom{}^{L}G^{I})$ with associated Hecke operator $T_{V}$, the sheaf $T_{V}(\nmEis(\mathcal{F}))$ on $\Bun_{G}$ with continuous $W_{\mathbb{Q}_{p}}^{I}$-action has a $W_{\mathbb{Q}_{p}}^{I}$-equivariant filtration indexed by $(\nu_{i})_{i \in I} \in (\gamorb)^{I}$ with graded isomorphic to
\[  \Ind_{W_{E_{(\nu_{i})_{i \in I}}}}^{W_{\bb{Q}_{p}}^{I}}(\nmEis(T_{(\tilde{\nu}_{i})_{i \in I}}(\mathcal{F})) \otimes V((\tilde{\nu}_{i})_{i \in I})). \]
Moreover, the filtration is natural with respect to morphisms of finite index sets $I$ and the fusion product, morphisms of representations $V \in \Rep(\phantom{}^{L}G^{I})$, and compatible with compositions/exterior tensor products. In the above, when tensoring the sheaves with continuous $W_{E_{(\nu_{i})_{i \in I}}}$ together we endow the resulting object with the diagonal $W_{E_{(\nu_{i})_{i \in I}}}$-action.
\end{theorem}
This result tells us that $\nmEis(\mathcal{S}_{\phi_{T}})$ is a "filtered eigensheaf" in the sense that if we knew the filtration split then it would be an eigensheaf, and Definition 1.4 will be designed to force this splitting for specific representations.

Even without knowing that the filtration on $\nmEis(\mathcal{S}_{\phi_{T}})$ splits, the filtered eigensheaf property can already be used to tell us a lot about the structure of $\nmEis(\mathcal{S}_{\phi_{T}})$. Assume $\Lambda = \ol{\bb{F}}_{\ell}$ temporarily. In particular, given $b \in B(G)$, the restriction $\nmEis(\mathcal{S}_{\phi_{T}})|_{\Bun_{G}^{b}}$ to the HN-stratum $\Bun_{G}^{b}$ defines a complex of smooth $J_{b}(\mathbb{Q}_{p})$-representations, and we are interested in describing this restriction. The fact that $\nmEis(\mathcal{S}_{\phi_{T}})$ is a filtered Hecke eigensheaf with eigenvalue $\phi$ implies that, if $\rho$ is an irreducible constituent of this restriction then the Fargues-Scholze parameter $\phi_{\rho}^{\mathrm{FS}}$ must be equal to $\phi$ under the twisted embedding $\phantom{}^{L}J_{b} \ra \phantom{}^{L}G$. If we believe that the Fargues-Scholze local Langlands correspondence is the true local Langlands correspondence then this seems to suggest that $\rho$ should be given by a normalized parabolic induction of the character $\chi$ attached to $\phi_{T}$ via local class field theory. In particular, we note, by deformation theory, that a $\ol{\bb{F}}_{\ell}$-toral parameter being generic implies that every lift of $\phi_{T}$ to $\ol{\mathbb{Z}}_{\ell}$ factors through $\phantom{}^{L}T$ and that the induced $\ol{\mathbb{Q}}_{\ell}$-parameter cannot come from the semi-simplification of a parameter with non-trivial monodromy (Lemma \ref{regmonodromy}). This imposes a very rigid constraint on $J_{b}$. In particular, the Borel $B \cap M_{b}$ of $M_{b}$ should transfer to a Borel $B_{b}$ of $J_{b}$. The elements where this occurs are the elements in the image of the map $B(T) \ra B(G)$, called the unramified elements $B(G)_{\mathrm{un}}$, as studied in the work of Xiao-Zhu \cite{XZ}. 
\begin{corollary}{(Corollary \ref{noredvanishing})}{\label{cor noredvanishingintro}}
For $\phi_{T}$ a generic parameter, $\Lambda = \ol{\bb{F}}_{\ell}$, and $b \in B(G)$, assuming compatibility of the Fargues-Scholze and the conjectural local Langlands correspondence (Assumption \ref{compatibility}), the restriction $\nmEis(\mathcal{S}_{\phi_{T}})|_{\Bun_{G}^{b}}$ vanishes unless $b \in B(G)_{\mathrm{un}}$.  
\end{corollary}
We now assume compatibility (Assumption \ref{compatibility}) for the rest of the section. The previous corollary tells us that if we are interested in understanding the complex of $J_{b}(\mathbb{Q}_{p})$-representations defined by the stalks $\nmEis(\mathcal{S}_{\phi_{T}})|_{\Bun_{G}^{b}}$ then we can restrict to the case where $b \in B(G)_{\mathrm{un}}$, and here we expect the stalks to be given by the inductions $i_{B_{b}}^{J_{b}}(\chi) \otimes \delta_{P_{b}}^{-1/2}$, where $\delta_{P_{b}}$ is the modulus character of the parabolic $P_{b}$ with Levi factor $M_{b}$ transferred to $J_{b}$.\footnote{These twists by the modulus character come from the fact that the excursion algebra on $\Bun_{G}$ acts on a smooth irreducible representation $\rho \in \D(J_{b}(\mathbb{Q}_{p}),\Lambda) \simeq \D(\Bun_{G}^{b}) \subset \D(\Bun_{G})$ via the Fargues-Scholze parameter $\phi_{\rho}^{FS}: W_{\mathbb{Q}_{p}} \ra \phantom{}^{L}J_{b}(\Lambda)$ of a smooth irreducible representation $\rho$ of $J_{b}$, composed with the \emph{twisted} embedding $\phantom{}^{L}J_{b}(\Lambda) \ra \phantom{}^{L}G(\Lambda)$, as in \cite[Section~IX.7.1]{FS}.}. This is indeed the case. To understand this, we use that each element $b \in B(G)_{\mathrm{un}}$ has a unique reduction $b_{T} \in B(T)$ such that the slopes are $G$-dominant. The set of elements in $B(T)$ mapping to $b$ can be described as $w(b_{T})$, where we identify $w \in W_{b} := W_{G}/W_{M_{b}}$ with a set of representatives of minimal length in $W_{G}$. The connected components of $\Bun_{B}$ and $\Bun_{T}$ are indexed by elements in $B(T) \simeq \mathbb{X}_{*}(T_{\ol{\mathbb{Q}}_{p}})_{\Gamma}$, giving a direct sum decomposition:
\[ \nmEis(\mathcal{S}_{\phi_{T}}) = \bigoplus_{\ol{\nu} \in B(T)} \nmEis^{\ol{\nu}}(\mathcal{S}_{\phi_{T}}) \]
After restricting to $\Bun_{G}^{b}$, the point is that only the summands indexed by $\ol{\nu} = w(b_{T})$ for $w \in W_{b}$ survive. It is fairly easy to see this when $\ol{\nu} = rb_{T}$, where $r := w_{0}^{G}w_{0}^{M_{b}}$ is the product of the element of longest length. In particular, the connected component $\Bun_{B}^{\ol{\nu}}$ will parametrize split $B$-structures since the HN-slopes are dominant (recalling that these are the negatives of the isocrystal slopes). Using this, it is easy to see that $\mf{p}_{!}\mf{q}^{*}(\chi)$ will be (essentially) given by compactly supported functions of $J_{b}/B_{b}(\mathbb{Q}_{p})$ which transform under $B_{b}(\mathbb{Q}_{p})$ via $\chi$. In other words, the unnormalized induction $\Ind_{B_{b}}^{J_{b}}(\chi)$. Carried out precisely, one finds that the exact formula becomes
\[ \nmEis^{\ol{\nu}}(\mathcal{S}_{\phi_{T}}) \simeq j_{b!}(i_{B_{b}}^{J_{b}}(\chi^{r^{-1}}) \otimes \delta_{P_{b}}^{-1/2})[-\langle 2\hat{\rho}, \nu_{b} \rangle]   \]
. Now, what about the connected components $\ol{\nu} = w(b_{T})$ with $w \neq r$? Here the HN-slopes of $\ol{\nu}$ are at least partially anti-dominant, and therefore $\Bun_{B}^{w(b_{T})}$ will parametrize some non-split extensions. Nonetheless, by combining Corollary \ref{cor noredvanishingintro} with the regularity of $\chi$ with an analogue of the geometric lemma for geometric Eisenstein functors (Proposition \ref{bruhatstrata}), one finds that $\nmEis^{w(b_{T})}(\mathcal{S}_{\phi_{T}})$ behaves similarly to the contribution of the connected component given by the HN-dominant reduction, allowing one to conclude an isomorphism:
\[ \nmEis^{w(b_{T})}(\mathcal{S}_{\phi_{T}}) \simeq j_{b!}(i_{B_{b}}^{J_{b}}(\chi^{w^{-1}}) \otimes \delta_{P_{b}}^{-1/2})[-\langle 2\hat{\rho}, \nu_{b} \rangle],   \]
This parallel behavior between the HN-dominant connected component and the connected components in its Weyl group orbit is no accident. In analogy with \S $1.1$, we expect, for a choice of representative $\tilde{w} \in N(T)$ of $w \in W_{G}$ in the relative Weyl group, to have an isomorphism
\[ \nmEis(\mathcal{S}_{\phi_{T}}) \simeq \nmEis(\mathcal{S}_{\phi_{T}}^{\tilde{w}}) \]
where $\mathcal{S}_{\phi_{T}}^{\tilde{w}}$ is the pullback of $\mathcal{S}_{\phi_{T}}$ along the map $\Bun_{T} \ra \Bun_{T}$ induced by $\tilde{w}$, assuming $\mathcal{S}_{\phi_{T}}$ or equivalently $\chi$ is regular and that $\phi_{T}$ is generic, so that  $\nmEis(\mathcal{S}_{\phi_{T}}) \simeq \ol{\nmEis}(\mathcal{S}_{\phi_{T}})$ for a putative definition of $\ol{\nmEis}(\mathcal{S}_{\phi_{T}})$. This involution sends the connected component $\Bun_{T}^{b_{T}}$ to $\Bun_{T}^{w(b_{T})}$ and sends the character $\chi$ to $\chi^{w}$. In particular, we see that, by our previous description of stalks, this gives the precise analogue of Theorem \ref{classicalfunctequation} recalling that genericity should guarantee an isomorphism $\nmEis(\mathcal{S}_{\phi_{T}}) \simeq \ol{\nmEis}(\mathcal{S}_{\phi_{T}})$ for any good definition of the RHS. 

Similar reasoning allows us to understand the interaction of $\nmEis(\mathcal{S}_{\phi_{T}})$ with Verdier duality. In particular, if, for an Artin $v$-stack $Z$, we denote Verdier duality by $\mathbb{D}_{Z}$ then we have the following.
\begin{theorem}{(Theorem \ref{commverddual})}
For $\phi_{T}$ a parameter with generic regular mod $\ell$-reduction and $\Lambda \in \{\ol{\bb{F}}_{\ell},\ol{\bb{Z}}_{\ell}\}$, there is an isomorphism of objects in $\D(\Bun_{G})$
\[ \mathbb{D}_{\Bun_{G}}(\nmEis(\mathcal{S}_{\phi_{T}})) \simeq \nmEis(\mathbb{D}_{\Bun_{T}}(\mathcal{S}_{\phi_{T}})),  \]
where we note that $\mathbb{D}_{\Bun_{T}}(\mathcal{S}_{\phi_{T}}) \simeq \mathcal{S}_{\phi_{T}^{\vee}}$, if $\phi_{T}^{\vee}$ is the parameter dual to $\phi_{T}$.
\end{theorem}
We summarize the above discussion as follows. 
\begin{theorem}{(Corollary \ref{normstalksdescription})}
Consider $\phi_{T}$ a toral parameter with generic regular mod $\ell$-reduction with associated character $\chi: T(\mathbb{Q}_{p}) \rightarrow \Lambda^{*}$. Given $b \in B(G)_{\mathrm{un}}$, we consider $J_{b}$, $M_{b}$, $B_{b}$, and $W_{b}$ as defined above. For $b \in B(G)$, the stalk $\nmEis(\mathcal{S}_{\phi_{T}})|_{\Bun_{G}^{b}} \in \D(\Bun_{G}^{b}) \simeq \D(J_{b}(\mathbb{Q}_{p}),\Lambda)$ is isomorphic to
\begin{enumerate}
\item  $\bigoplus_{w \in W_{b}} i_{B_{b}}^{J_{b}}(\chi^{w^{-1}}) \otimes \delta_{P_{b}}^{-1/2}[-\langle 2\hat{\rho},\nu_{b} \rangle]$ if $b \in B(G)_{\mathrm{un}}$,
\item $0$ if $b \notin B(G)_{\mathrm{un}}$.  
\end{enumerate}
In particular, $\nmEis(\mathcal{S}_{\phi_{T}})$ is a perverse sheaf on $\Bun_{G}$ with respect to the standard $t$-structure defined by the HN-strata using Theorem 1.9. 
\end{theorem}
We note that the previous Corollaries imply that the stalks of the sheaf $\nmEis(\mathcal{S}_{\phi_{T}})$ are valued in smooth admissible representation when $\phi_{T}$ has generic regular mod $\ell$-reduction. This implies that the sheaf $\nmEis(\mathcal{S}_{\phi_{T}})$ is ULA with respect to the structure map $\Bun_{G} \ra \ast$, using the characterization given in \cite[Theorem~V.7.1]{FS}. This ULA property allows us to extend the construction of $\nmEis(\mathcal{S}_{\phi_{T}})$ to integral and rational coefficients, where passing to the world of lisse-\'etale solid sheaves is a non-trivial matter because of the difference between $\ell$-adic and discrete topologies. The description of the stalks, the filtered eigensheaf property, and the commutation with Verdier duality all extend in a natural way to these coefficient systems using these finiteness properties, and we would now like to say that the filtered eigensheaf property implies it is a genuine eigensheaf under the conditions on $\phi_{T}$. For a representation $V \in \Rep_{\Lambda}(\phantom{}^{L}G)$, the filtered eigensheaf property tells us that $T_{V}(\nmEis(\mathcal{S}_{\phi_{T}}))$ has a filtration whose graded pieces are given by the direct summands of $V|_{\phantom{}^{L}T}$, it suffices to show that the Weil group action between these different summands is trivial. However, this is implied by the vanishing of the entire complex
\[ R\Gamma(W_{\mathbb{Q}_{p}},(V \otimes V^{\vee})|^{\nt}_{\phantom{}^{L}T} \circ \phi_{T}), \]
and this was precisely the kind of vanishing result that strong  $(\mu,\tau)$-regularity of $\phi_{T}$ guaranteed. Moreover, by the vanishing of the $H^{0}$ the splitting will be unique. Now, as noted in Remark \ref{rem: genremarks}, we see that strong $\mu$-regularity is usually guaranteed under generic for a sufficiently nicely chosen cocharacter. We would like to use this to conclude that the filtration splits in more generality. Suppose we are given cocharacters $\mu_{1},\ldots,\mu_{k}$ and we know that the filtration splits for the $\mathcal{T}_{\mu_{i}^{\Gamma}}$ then, using the compatibilities of the filtration, we can show the splitting for any representation $V$ realized as a direct summand of $\bigotimes_{k = 1}^{m} \mathcal{T}_{\mu_{k}^{\Gamma}}^{\otimes n_{i}}$, but in this case we cannot guarantee that this splitting is unique. 

We need to be a bit careful when running the argument sketched above. In particular, if $\Lambda = \ol{\mathbb{Q}}_{\ell}$, then the category $\Rep_{\Lambda}(\phantom{}^{L}G)$ is semi-simple with irreducible objects parametrized by $\Gamma$-orbits of dominant cocharacters $\domgamorb$ and Galois representations of the reflex field and the above argument goes through. If $\Lambda = \ol{\bb{Z}}_{\ell}$ or is torsion this is no longer true, and this is precisely why we work with tilting modules in Definition 1.4. Indeed, in these cases, we can replace $\Rep_{\Lambda}(\phantom{}^{L}G)$ by a sub-category of tilting modules $\Tilt_{\Lambda}(\phantom{}^{L}G)$ (\cite{Mat},\cite[Appendix~E]{JanRepresentationsofAlgebraicGroups}), which will behave like a semisimple category with indecomposable objects by the objects $\mathcal{T}_{\mu^{\Gamma},\tau} \in \Tilt_{\Lambda}(\phantom{}^{L}G)$. Extending the theory of tilting modules to the full $L$-group $\phantom{}^{L}G$ is a bit subtle. However, this is precisely what our assumption that $\ell \nmid |Q|$ appearing in the very decent condition above will allow us to do. This category is preserved under taking tensor products, and therefore we can define the notion of a "tilting eigensheaf" (Definition \ref{deftilteigsheaf}) by replacing $\Rep_{\Lambda}(\phantom{}^{L}G)$ with $\Tilt_{\Lambda}(\phantom{}^{L}G)$ in the usual definition. This allows us to define the following. 
\begin{definition}{\label{def: mureg}}
For a finite index set $I$, a tuple of cocharacters $(\mu_{i})_{i \in I} \in (\domcochar)^{I}$, and indecomposable projective representations $\tau_{i} \in \mathrm{Irr}(Q_{\mu_{i}})$ for $Q_{\mu_{i}} := \Gal(E/E_{\mu_{i}})$, we say that $\phi_{T}: W_{\mathbb{Q}_{p}} \ra \phantom{}^{L}T(\Lambda)$ is $(\mu_{i},\tau_{i})_{i \in I}$-regular if the filtration on $T_{V}(\nmEis(\mathcal{S}_{\phi_{T}}))$ splits for the tilting module $V = \boxtimes_{i \in I} \mathcal{T}_{\mu_{i}^{\Gamma},\tau_{i}}$. If the $\tau_{i}$ is trivial then we omit it from the terminology (i.e we just say $(\mu_{i})_{i \in I}$-regularity in this case).
\end{definition}
\begin{remark}
The argument sketched above allows us to show if $\phi_{T}$ is $(\mu_{1i})_{i \in I}$ and $(\mu_{2i})_{i \in I}$ regular then this implies that $\phi_{T}$ is regular for the cocharacter attached to any highest weight tilting module occurring in the tensor product $\boxtimes_{i \in I} \mathcal{T}_{\mu_{1i}^{\Gamma}} \otimes \mathcal{T}_{\mu_{2i}^{\Gamma}}$ (Proposition \ref{prop: muregdirectsum}). In particular, this implies that $\phi_{T}$ is $(\mu_{1i} + \mu_{2i})_{i \in I}$-regular by considerations of highest weight. This property often allows us to see that we get $\mu$-regularity just under the generic hypothesis on $\phi_{T}$. For example, for $G = \GL_{n}$, and $\Lambda = F$ for $F/\bb{Q}_{\ell}$ an algebraic extension then, under the standing residual genericity hypothesis, genericity implies $\mu$-regularity for all $\mu$ (Corollary \ref{cor: GLnisideal}).
\end{remark}
Our main theorem is then as follows.  
\begin{theorem}{(Theorem \ref{Eiseigsheaf})}
For $\Lambda/\bb{Z}_{\ell}[\sqrt{p}]$ as in Setup \ref{assumption: coefficients}, we consider $\phi_{T}: W_{\mathbb{Q}_{p}} \ra \phantom{}^{L}T(\Lambda)$ an integral parameter with generic regular mod $\ell$-reduction (Definition \ref{defn: IntegralParameter}) and write $\phi$ for its induced $L$-parameter. There then exists a perverse sheaf $\nmEis(\mathcal{S}_{\phi_{T}}) \in \Dlis(\Bun_{G},\Lambda)$ which is a filtered eigensheaf with eigenvalue $\phi$. If $V \in \Tilt_{\Lambda}(\phantom{}^{L}G^{I})$ is a sum of tilting modules $\boxtimes_{i \in I} \mathcal{T}_{(\mu_{i}^{\Gamma},\tau_{i})}$ for geometric dominant cocharacters $\mu_{i}$, $\tau_{i}$ is an indecomposable projective representation of $Q_{\mu_{i}} := \Gal(E/E_{\mu_{i}})$, and $\phi_{T}$ is $(\mu_{i},\tau_{i})_{i \in I}$-regular (resp. strongly $(\mu_{i},\tau_{i})_{i \in I}$-regular), the filtration on $T_{V}(\nmEis(\mathcal{S}_{\phi_{T}}))$ splits (resp. splits uniquely on direct summands), and we have an isomorphism
\[ T_{V}(\nmEis(\mathcal{S}_{\phi_{T}})) \simeq \nmEis(\mathcal{S}_{\phi_{T}}) \boxtimes r_{V} \circ \phi^{\boxtimes I} \]
of sheaves in $\Dlis(\Bun_{G})^{BW_{\mathbb{Q}_{p}}^{I}}$. In particular, if $\phi_{T}$ is $(\mu,\tau)$-regular for all geometric dominant cocharacters $\mu$ and $\tau$ an indecomposable projective representation of $Q_{\mu}$ then $\nmEis(\mathcal{S}_{\phi_{T}})$ is a weak tilting eigensheaf. For $b \in B(G)$, the stalk $\nmEis(\mathcal{S}_{\phi_{T}})|_{\Bun_{G}^{b}} \in \D(\Bun_{G}^{b}) \simeq \D(J_{b}(\mathbb{Q}_{p}),\Lambda)$ is isomorphic to
\begin{enumerate}
\item  $\bigoplus_{w \in W_{b}} i_{B_{b}}^{J_{b}}(\chi^{w^{-1}}) \otimes \delta_{P_{b}}^{-1/2}[-\langle 2\hat{\rho},\nu_{b} \rangle]$ if $b \in B(G)_{\mathrm{un}}$,
\item $0$ if $b \notin B(G)_{\mathrm{un}}$.  
\end{enumerate}
Moreover, if $\mathbb{D}_{\Bun_{G}}$ denotes Verdier duality on $\Bun_{G}$, we have an isomorphism
\[ \mathbb{D}_{\Bun_{G}}(\nmEis(\mathcal{S}_{\phi_{T}})) \simeq \nmEis(\mathcal{S}_{\phi_{T}^{\vee}}) \]
of sheaves in $\Dlis(\Bun_{G},\Lambda)$.
\end{theorem}
\begin{remark}{\label{uniquesplitting}}
The notion of a weak tilting eigensheaf means that we always have isomorphisms 
\[ \eta_{V,I}: T_{V}(\nmEis(\mathcal{S}_{\phi_{T}})) \simeq \nmEis(\mathcal{S}_{\phi_{T}}) \boxtimes r_{V} \circ \phi^{\boxtimes I} \]
for $V \in \Tilt(\phantom{}^{L}G^{I})$ and a finite index set $I$, but do not necessarily know that the desired compatibilities with respect to $I$ and $V$. Even though we know these compatibilities for the filtration, it is not necessarily clear that the splitting we produce through our argument respects these compatibilities without assuming strong $\mu$-regularity or in fact a stronger version known as strong $V$-regularity (Definition \ref{def: strongmureg}) for all $V \in \Tilt_{\Lambda}(\phantom{}^{L}G^{I})$. Only knowing the compatibilities of the splittings under such restrictive conditions is a bit unfortunate; fortunately, for most of the applications to local Shtuka spaces with one leg it suffices to only know a splitting exists.  
\end{remark}
This eigensheaf has several surprising applications to the cohomology of local Shimura varieties and shtuka spaces. To formalize this, we define, for $b \in B(G)$, a complex of $J_{b}(\mathbb{Q}_{p})$-representations denoted $\Red_{b,\phi}$. If $b \notin B(G)_{\mathrm{un}}$ we set this to be equal to $0$ and if $b \in B(G)_{\mathrm{un}}$ to be equal to $\bigoplus_{w \in W_{b}} i_{B_{b}}^{J_{b}}(\chi^{w^{-1}}) \otimes \delta_{P_{b}}^{-1/2}[-\langle 2\hat{\rho},\nu_{b}\rangle]$. Now let's consider $\mu$ a geometric dominant cocharacter of $G$ with reflex field $E$ and set $B(G,\mu) \subset B(G)$ to be the subset of $\mu$-admissible elements (Definition \ref{bgmudefinition}). We let $\mathcal{T}_{\mu}$ be the associated highest weight tilting module of $\hat{G}$. This defines a representation of $W_{E} \ltimes \hat{G}$ with associated Hecke operator $T_{\mu}$. We write $r_{\mu}: Q_{\mu} \ltimes \hat{G} \ra \GL(\mathcal{T}_{\mu})$ for the associated map. We consider the cohomology of the local shtuka spaces $\Sht(G,b,\mu)_{\infty}$, as defined in \cite{SW}. In particular, the tilting representation $\mathcal{T}_{\mu}$ attached to $\mu$ defines a sheaf $\mathcal{S}_{\mu}$ on $\Sht(G,b,\mu)_{\infty}$ via geometric Satake, and we can consider the complex $R\Gamma_{c}(G,b,\mu)$ of $J_{b}(\mathbb{Q}_{p}) \times G(\mathbb{Q}_{p}) \times W_{E}$-modules attached to cohomology valued in this sheaf. 

We can use $\nmEis(\mathcal{S}_{\phi_{T}})$ to describe the cohomology of $R\Gamma_{c}(G,b,\mu)$. Assume that $\phi_{T}$ is $\mu$-regular, the Hecke eigensheaf property then tells us that we have an isomorphism
\[ T_{\mu}(\nmEis(\mathcal{S}_{\phi_{T}})) \simeq \nmEis(\mathcal{S}_{\phi_{T}}) \boxtimes r_{\mu} \circ \phi|_{W_{E}} \]
of sheaves with continuous $W_{E}$-action. If we restrict to the open HN-stratum $j_{\mathbf{1}}: \Bun_{G}^{\mathbf{1}} \ra \Bun_{G}$ defined by the trivial element $\mathbf{1} \in B(G)$ then this gives an isomorphism
\[ j_{\mathbf{1}}^{*}T_{\mu}(\nmEis(\mathcal{S}_{\phi_{T}})) \simeq i_{B}^{G}(\chi) \boxtimes r_{\mu} \circ \phi|_{W_{E}} \]
of complexes of $G(\mathbb{Q}_{p}) \times W_{E}$-modules. Now the point is that only the elements $b \in B(G,\mu)$ occur as a modifications $\mathcal{F}_{b} \ra \mathcal{F}_{G}^{0}$ of type $\mu$, where $\mathcal{F}_{G}^{0}$ is the trivial $G$-bundle. Therefore, only these stalks contribute to the LHS. By applying excision to the locally closed stratification given by $\Bun_{G}^{b}$ for $b \in B(G,\mu)$, we find that the LHS has a filtration with graded pieces isomorphic to $j_{\mathbf{1}}^{*}T_{\mu}(j_{b!}(\Red_{b,\phi}))$, but these are related to the isotypic parts $R\Gamma^{\flat}_{c}(G,b,\mu)[\Red_{b,\phi}]$. From the above analysis, we deduce the following. 
\begin{theorem}{(Theorem \ref{avgingformula})} 
For $\phi_{T}: W_{\mathbb{Q}_{p}} \ra \phantom{}^{L}T(\Lambda)$ an integral toral parameter such that its mod $\ell$-reduction is generic regular and any geometric dominant cocharacter $\mu$ such that $\phi_{T}$ is $\mu$-regular, we have an equality
\[ \sum_{b \in B(G,\mu)} [R\Gamma^{\flat}_{c}(G,b,\mu)[\Red_{b,\phi}]] = [r_{\mu} \circ \phi|_{W_{E}} \boxtimes i_{B}^{G}(\chi)] \]
in the Grothendieck group $K_{0}(G(\mathbb{Q}_{p}) \times W_{E},\Lambda)$ of smooth admissible $G(\mathbb{Q}_{p})$-representations with a continuous action of $W_{E}$.
\end{theorem}
If we now consider the case where $\Lambda = \ol{\mathbb{Q}}_{\ell}$ then it follows that the averaging formula is valid for all generic regular and $\mu$-regular parameters $\phi_{T}: W_{\mathbb{Q}_{p}} \ra \phantom{}^{L}T(\ol{\mathbb{Q}}_{\ell})$, which admit a $\ol{\mathbb{Z}}_{\ell}$-lattice. Moreover, with $\ol{\mathbb{Q}}_{\ell}$-coefficients, we can interpret both sides as trace forms on $K_{0}(T(\mathbb{Q}_{p}),\ol{\mathbb{Q}}_{\ell})$, and use that the set of characters obtained from such parameters is Zariski dense in the variety of unramified characters for split $G$ to conclude the following more general claim.
\begin{theorem}{(Theorem \ref{qellavgingformula})} 
For $\phi_{T}: W_{\mathbb{Q}_{p}} \ra \phantom{}^{L}T(\ol{\mathbb{Q}}_{\ell})$ any toral parameter and $\mu$ any geometric dominant cocharacter of split $G$, we have an equality
\[ \sum_{b \in B(G,\mu)} [R\Gamma^{\flat}_{c}(G,b,\mu)[\Red_{b,\phi}]] = [r_{\mu} \circ \phi|_{W_{E}} \boxtimes i_{B}^{G}(\chi)] \]
in $K_{0}(G(\mathbb{Q}_{p}) \times W_{E},\ol{\mathbb{Q}}_{\ell})$.
\end{theorem}
If $\mu$ is minuscule and $G = \mathrm{GL}_{n}$ then this recovers special cases of an averaging formula of Shin \cite{Shin}, which was formalized for more general reductive groups by Alexander Bertoloni-Meli \cite{BM}. In particular, for all $\chi$ the induction $i_{B}^{G}(\chi)$ defines a class $[i_{B}^{G}(\chi)] \in K_{0}^{st}(G(\mathbb{Q}_{p}),\ol{\mathbb{Q}}_{\ell})$ in the subgroup of the Grothendieck group with stable character sum. To such a class, the averaging formula gives a description of the RHS in terms of an average over $B(G,\mu)$ of the isotypic parts of $R\Gamma_{c}(G,b,\mu)$ with respect to $\Red_{b}^{\mf{c}}(i_{B}^{G}(\chi))$, where $\mf{c}$ is a refined endoscopic datum (Definition \ref{def: refinedendoscopy}). In Appendix~\ref{Classicavg},
we verify this agreement for the trivial endoscopic datum. Such formulae are typically proven in the minuscule case by stabilizing the trace formula on the Igusa varieties indexed by $b \in B(G,\mu)$, and our analysis gives a more conceptual explanation for them. 

We recall that, in the proof of the averaging formula, we used excision to produce a filtration whose graded pieces were isomorphic to
\[ j_{\mathbf{1}}^{*}T_{\mu}(j_{b!}(\Red_{b,\phi})) \simeq j_{\mathbf{1}}^{*}T_{\mu}(j_{b!}(\nmEis(\mathcal{S}_{\phi_{T}})|_{\Bun_{G}^{b}}))   \]
By using the isomorphism $\mathbb{D}_{\Bun_{G}}(\nmEis(\mathcal{S}_{\phi_{T}})) \simeq \nmEis(\mathcal{S}_{\phi_{T}^{\vee}})$, we can show (see Corollary \ref{rhoduality}) that we have an isomorphism: $j_{\mathbf{1}}^{*}T_{\mu}(j_{b!}(\Red_{b,\phi})) \simeq j_{\mathbf{1}}^{*}T_{\mu}(j_{b*}(\Red_{b,\phi}))$ of objects in $\Dlis(\Bun_{G},\Lambda)$. This implies that the excision spectral sequence degenerates, allowing us to conclude the following refined averaging formula.
\begin{theorem}{(Theorem \ref{refinedaveragingformula})}
For $\phi_{T}$ an integral parameter with generic regular mod $\ell$-reduction, and $\mu$ any geometric dominant cocharacter such that $\phi_{T}$ is $\mu$-regular, we have an isomorphism
\[ \bigoplus_{b \in B(G,\mu)_{\mathrm{un}}} \bigoplus_{w \in W_{b}} R\Gamma^{\flat}_{c}(G,b,\mu)[\rho_{b,w}][-\langle 2\hat{\rho},\nu_{b} \rangle] \simeq i_{B}^{G}(\chi) \boxtimes r_{\mu} \circ \phi|_{W_{E}} \]
of complexes of $G(\mathbb{Q}_{p}) \times W_{E}$-modules, where $\rho_{b,w} := i_{B_{b}}^{J_{b}}(\chi^{w^{-1}}) \otimes \delta_{P_{b}}^{-1/2}$ and $\chi: T(\bb{Q}_{p}) \ra \Lambda^{*}$ is the smooth character attached to $\phi_{T}$ via the geometric normalization of class field theory. In particular, the cohomology groups of $R\Gamma^{\flat}_{c}(G,b,\mu)[\rho_{b,w}][-\langle 2\hat{\rho},\nu_{b} \rangle]$ are concentrated in degree $0$.
\end{theorem}
We now assume that $\phi_{T}$ is an integral parameter with generic regular mod $\ell$ reduction in all that follows. The previous theorem leads to a very explicit descriptions of the complexes $R\Gamma^{\flat}_{c}(G,b,\mu)[\rho_{b,w}]$ and the degrees of cohomology they sit in. 
\begin{corollary}{(Corollary \ref{inductisotypic})}
For $\mu$ a geometric dominant cocharacter with reflex field $E$ such that $\phi_{T}$ is $\mu$-regular, fixed $b \in B(G,\mu)_{\mathrm{un}}$, and varying $w \in W_{b}$, the complex $R\Gamma^{\flat}_{c}(G,b,\mu)[\rho_{b,w}]$ is isomorphic to $\phi_{b,w}^{\mu} \boxtimes \sigma[\langle 2\hat{\rho},\nu_{b} \rangle]$, for $\phi_{b,w}^{\mu}$ a representation of $W_{E}$ and $\sigma$ a sub-quotient of $i_{B}^{G}(\chi)$. Moreover, we have an isomorphism 
\[ \bigoplus_{b \in B(G,\mu)_{\mathrm{un}}} \bigoplus_{w \in W_{b}} \phi_{b,w}^{\mu} \simeq r_{\mu} \circ \phi|_{W_{E}} \]
of $W_{E}$-representations.
\end{corollary}

It is now natural to wonder what the representations $\phi_{b,w}^{\mu}$ exactly are. It was already observed by Xiao-Zhu \cite{XZ} that the elements of the set $B(G,\mu)_{\mathrm{un}}$ correspond to Weyl group orbits of weights of the highest weight module $\mathcal{T}_{\mu}$ or rather its restriction $\mathcal{T}_{\mu}|_{\hat{G}^{\Gamma}}$ (Corollary \ref{bgmuweights}). If we let $b_{T} \in B(T)$ be the dominant reduction of an element $b \in B(G,\mu)_{\mathrm{un}}$ then the orbit of the character $b_{T}$ under the Weyl group $W_{G}$ can be described as $w(b_{T})$ for $w \in W_{b}$ varying. For varying $b \in B(G,\mu)_{\mathrm{un}}$, this describes the set of non-zero weights which can occur in the representation $\mathcal{T}_{\mu}|_{\hat{G}^{\Gamma}}$. In particular, given such a $\ol{\nu} \in \mathbb{X}_{*}(T_{\ol{\mathbb{Q}}_{p}})_{\Gamma}$, we can look at the direct sum of weight spaces
\[ \bigoplus_{\substack{\nu \in \mathbb{X}_{*}(T_{\ol{\mathbb{Q}}_{p}}) \\ \nu_{\Gamma} = \ol{\nu}}} \mathcal{T}_{\mu}(\nu) \]
and this coincides with the weight space of $\mathcal{T}_{\mu}|_{\hat{G}^{\Gamma}}(\ol{\nu})$ via the isomorphism $\mathbb{X}_{*}(T_{\ol{\mathbb{Q}}_{p}})_{\Gamma} \simeq \mathbb{X}^{*}(\hat{T}^{\Gamma})$. The refined averaging formula suggests the following relationship. 
\begin{conjecture}{(Conjecture \ref{weilgroupaction})}
For all geometric dominant cocharacters $\mu$ such that $\phi_{T}$ is $\mu$-regular, an unramified element $b \in B(G,\mu)_{\mathrm{un}}$, and a Weyl group element $w \in W_{b}$, we have an isomorphism 
\[ \bigoplus_{\substack{\wt{w(b_{T})} \in \mathbb{X}_{*}(T_{\ol{\mathbb{Q}}_{p}}) \\
\wt{w(b_{T})}_{\Gamma} = w(b_{T})}} \wt{w(b_{T})} \circ \phi_{T}|_{W_{E'}} \otimes \mathcal{T}_{\mu}(\wt{w(b_{T})}) \simeq \phi_{b,w}^{\mu}|_{W_{E'}} \] 
of $W_{E'}$-representations, where $E'|E$ denotes the splitting field of $G$.
\end{conjecture}
We verify this conjecture in some particular cases, by noting that the contribution from the $\mu$-ordinary locus can be explicitly computed using a shtuka analogue of Boyer's trick \cite{Boy}, as studied by Gaisin-Imai \cite{IG}. To do this, we note that we have a distinguished element in $B(G,\mu)_{\mathrm{un}}$ called the $\mu$-ordinary element, which we denote by $b_{\mu}$. It is the maximal element in the partial ordering on $B(G,\mu)$, and we let  $b_{\mu_{T}}$ be its dominant reduction. The conjecture suggests that this should correspond to the Weyl group orbit of the highest weight of $\mathcal{T}_{\mu}$. In this case, the space $\Sht(G,b_{\mu},\mu)_{\infty,\mathbb{C}_{p}}$ with its $G(\mathbb{Q}_{p}) \times J_{b}(\mathbb{Q}_{p})$-action is determined by the space  $\Sht(T,b_{\mu_{T}},\mu)_{\infty,\mathbb{C}_{p}}$ with its $T(\mathbb{Q}_{p}) \times T(\mathbb{Q}_{p})$-action. In particular, using this we can deduce the following isomorphism (see Proposition \ref{muordinary}):
\[ R\Gamma^{\flat}_{c}(G,b_{\mu},\mu)[\rho_{b_{\mu},w}] \simeq i_{B}^{G}(\chi^{rw^{-1}}) \boxtimes w(\mu) \circ \phi_{T}|_{W_{E}}[\langle 2\hat{\rho}, \nu_{b} \rangle], \]
where $w \in W_{b}$ is a minimal length representative and $r = w_{0}^{G}w_{0}^{M_{b}}$ for $w_{0}^{G}$ (resp. $w_{0}^{M_{b}}$) the element of longest length in $W_{G}$ (resp. $W_{M_{b}}$).

This calculation reveals a very direct connection with the theory
of intertwining operators. The existence of 
intertwining isomorphisms, together with irreducibility of the inductions $i_{B}^{G}(\chi)$ under the assumption of generic regularity,
is established independently in Section~\ref{s: SecondAdjointness} (Corollary \ref{cor:principal-series-detection}), using a specialization argument and the straightforward 
characteristic-zero result of Appendix~\ref{append: irred of principal series}. These
results are needed in our calculation of the stalks
of geometric Eisenstein series.

However, the refined averaging formula gives a different,
geometric perspective on these isomorphisms. In
particular, the previous isomorphism identifies the
$\mu$-ordinary contributions to local shtuka cohomology
with Weyl conjugates of principal series representations,
whereas the filtered eigensheaf property describes the
same Hecke complex in terms of the original principal
series. Comparing these descriptions gives a geometric description of the intertwining isomorphisms. Although the geometric results used in this comparison rely on an independent argument for their existence (Proposition \ref{prop:opposite-borel-comparison}), the comparison reveals how intertwining operators interact with the cohomology of local shtuka spaces.
\begin{theorem}{(Corollary \ref{intertwiningoperators})}
Assume \ref{compatibility}. Let
$\Lambda \in \{\ol{\bb{F}}_{\ell},\ol{\bb{Q}}_{\ell}\}$
and let $\phi_T$ be an integral toral parameter with
generic regular mod $\ell$-reduction, with associated
character $\chi$. Then, for every $w \in W_G$, we
have an isomorphism
\[
i_{\chi,w}: i_B^G(\chi) \simeq i_B^G(\chi^w)
\]
of smooth $G(\mathbb{Q}_p)$-representations. Moreover, the $i_B^G(\chi)$ are irreducible in this case. 
\end{theorem}
This showcases the strong connection between the theory of geometric Eisenstein series and the theory of intertwining operators and the Langlands quotient that has been our philosophical guide throughout. 

Throughout our results, we have introduced various technical conditions on $\phi_{T}$. We suspect that some of these conditions (in particular, the $(\mu,\tau)$-regularity) are artifacts of the proofs we have used to overcome the technical geometry of $\Bun_{B}$ and its compactifications in the context of analytic geometric and diamonds. While the conditions are manageable for specific applications to specific groups it leaves one wanting for a more conceptually clear picture. In particular, we conjecture that the following is true, which (modulo checking the compatibilities of the isomorphisms in the eigensheaf property) our methods show for $\GL_{n}$ and integral parameters with generic regular mod $\ell$-reduction (assuming $(\ell,n!) = 1$ in the torsion and integral cases (Corollary \ref{cor: GLnisideal})
\begin{conjecture}{\label{conj: idealworld}}
For any $\Lambda/\bb{Z}_{\ell}[\sqrt{p}]$ and $\phi_{T}: W_{\mathbb{Q}_{p}} \ra \phantom{}^{L}T(\Lambda)$ a generic regular $L$-parameter, there exists a sheaf $\nmEis(\mathcal{S}_{\phi_{T}}) \in \Dlis(\Bun_{G},\Lambda)$ which is a perverse Hecke eigensheaf with eigenvalue $\phi$ such that one has an isomorphism $\mathbb{D}_{\Bun_{G}}(\nmEis(\mathcal{S}_{\phi_{T}})) \simeq \nmEis(\mathcal{S}_{\phi_{T}^{\vee}})$, and its stalk at all $b \in B(G)$ is isomorphic to $\Red_{b,\phi}$.
\end{conjecture}
This conjecture would follow from knowing the existence of $\IC_{\ol{\Bun}_{B}}$ and in turn the compactified Eisenstein functor $\ol{\Eis}$ with all the various desiderata proven by Braverman-Gaitsgory \cite{BG} in the function field setting. In particular, we expect that $\ol{\Eis}(-)$ should commute with Verdier duality, and satisfy the functional equation under condition \ref{normregcond} (2). Moreover, by the analogue of the results of \cite{BG2}, there should be a natural map
\[  \nmEis(\mathcal{S}_{\phi_{T}}) \ra  \ol{\Eis}(\mathcal{S}_{\phi_{T}}) \]
whose cone should be given by Eisenstein functors tensored with complexes admitting a filtration isomorphic to $R\Gamma(W_{\mathbb{Q}_{p}},\alpha_{\Ad} \circ \phi_{T})$ for $\alpha$ a $\Gamma$-orbit of coroots of $G$. In particular, we should have an isomorphism $\ol{\Eis}(\mathcal{S}_{\phi_{T}}) \simeq \nmEis(\mathcal{S}_{\phi_{T}})$ precisely when $\phi_{T}$ is generic.  

In \S $2$, we start by defining the set of unramified elements in $B(G)$ and discussing their relationship with highest weight theory, as in \cite{XZ}. In \S $3$, we review the construction of eigensheaves on $\Bun_{T}$ attached to parameters $\phi_{T}$, introducing the conditions on our parameter $\phi_{T}$ and working through some useful lemmas and examples related to them. In \S $4$, we review the geometric Satake correspondence of Fargues-Scholze, recalling the key results and relating the highest weight theory of $\phantom{}^{L}G$ to the cohomology of semi-infinite Schubert cells. In \S $5$, we introduce Drinfeld's compactifications over the Fargues-Fontaine curve and recall its basic geometric properties and discuss a key diagram that will be used in the proof of the Hecke eigensheaf property. In \S $6$, we will study how the Eisenstein functor interacts with Hecke operators, establishing Theorem 1.6. In \S $7$, we review some important results from the paper \cite{HHS} on geometric second adjointness and the preservation of ULA objects by Eisenstein functors, as well as establish irreducibility of the representations $\rho_{b,w}$ under the assumption of generic regularity. In \S $8$, we will carry out the computation of the stalks of the Eisenstein series $\nmEis(\mathcal{S}_{\phi_{T}})$ and show the commutation with Verdier duality, establishing Theorem 1.8 and Theorem 1.9. In \S $9$, we describe the theory of tilting modules for the $L$-group $\phantom{}^{L}G$, constructing $\nmEis(\mathcal{S}_{\phi_{T}})$ with integral and rational coefficients and showing Theorem 1.12. Finally, in \S $10$, we deduce the applications to the cohomology of local shtuka spaces showing Theorem 1.16, Corollary 1.17, and Theorem 1.19. 
\section*{Acknowledgments}
I would like to thank my advisor David Hansen for his continual support throughout this project, as well as many useful comments and suggestions. Special thanks also go to Peter Scholze for very helpful feedback and corrections on an earlier draft. I would also like to thank Alexander Bertoloni-Meli, Arthur-C\'esar Le-Bras, Eric Chen, Dennis Gaitsgory, Thomas Haines, Tasho Kaletha, Teruhisa Koshikawa, Robert Kurinczuk, Sam Mundy, Marko Tadi\'{c}, Naoki Imai, Eva Viehmann, Geordie Williamson, Liang Xiao, Xinwen Zhu, and Konrad Zou, for helpful exchanges which provided insights that led to the development and completion of this paper. Lastly, I would like to thank the MPIM for their hospitality throughout part of the completion of this project, as well as the anonymous referee for many detailed comments. 
\section*{AI Disclosure} After a few rounds of the referee process, the author used 10 A.I. agents to identify additional mistakes in the article, as well as had A.I. review the paper after subsequent revisions. In addition to various small mistakes, the A.I. correctly identified that the map $r$ used in the proof of Proposition \ref{splitmap} was not in general an isomorphism in contrast to what was previously claimed by the author and that one needed to check the inductions $\rho_{b,w}$ are irreducible in the calculation of the stalks of the Eisenstein series. The A.I. provided an initial draft of the proof of Lemma \ref{lem:unit-from-residue-tube} in terms of explicit calculations with polynomials; however, this was subsequently rewritten using the language of formal schemes by the author.
\section{Notation}
\begin{itemize}
\item Let $\ell \neq p$ be distinct primes. 
\item Let $G/\bb{Q}_{p}$ be a quasi-split connected reductive group. We will often denote the splitting field of $G$ by $E$.
\item We let $\ol{\mathbb{Q}}_{\ell}$ denote the algebraic closure of the $\ell$-adic numbers, with residue field $\ol{\mathbb{F}}_{\ell}$ and ring of integers $\ol{\mathbb{Z}}_{\ell}$. 
\item Let $\Gamma$ be the absolute Galois group of $\mathbb{Q}_{p}$, and let $W_{\mathbb{Q}_{p}} \subset \Gamma$ be the Weil group of $\mathbb{Q}_{p}$.
\item We set $\mathcal{L}_{\mathbb{Q}_{p}} := W_{\mathbb{Q}_{p}} \times \mathrm{SL}_{2}(\ol{\mathbb{Q}}_{\ell})$ to be the Weil-Deligne group. 
\item Fix choices $A \subset T \subset B \subset G$ of maximal split torus, maximal torus, and Borel. We use $U$ to denote the unipotent radical of $B$.
\item We let $W_{G}$ be the relative Weyl group of $G$ and $w_{0}$ be the element of longest length. 
\item We write $\Ind_{B}^{G}(-)$ for the unnormalized parabolic induction functor from $B$ to $G$. We let $\delta_{B}$ be the modulus character defined by $B$ so that $i_{B}^{G}(-) := \Ind_{B}^{G}(- \otimes \delta_{B}^{1/2})$ is the normalized induction. In other words, $\delta_{B}$ is defined by the transformations of the space of right Haar measures.
\item Let $\Breve{\mathbb{Q}}_{p}$ be the completion of the maximal unramified extension of $\mathbb{Q}_{p}$ with Frobenius $\sigma$. For $E/\mathbb{Q}_{p}$ a finite extension, we set $\Breve{E}$ to be the compositum $E\Breve{\mathbb{Q}}_{p}$.
\item Set $\mathbb{C}_{p}$ to be the completion of the algebraic closure of $\mathbb{Q}_{p}$.
\item Let $B(G) = G(\Breve{\mathbb{Q}}_{p})/(g \sim hg\sigma(h)^{-1})$ denote the Kottwitz set of $G$.
\item For $b \in B(G)$, we write $J_{b}$ for the $\sigma$-centralizer of $b$.
\item We will always work over the fixed base $\ast := \Spd{\ol{\mathbb{F}}_{p}}$, unless otherwise stated.
\item Let $\Perf$ denote the category of (affinoid) perfectoid spaces in characteristic $p$ over $\ast$ endowed with the $v$-topology. For $S \in \Perf$, let $\Perf_{S}$ denote the category of affinoid perfectoid spaces over it. We write $\Perf_{\ol{\bb{F}}_{p}}$ for the category of affinoid perfectoids defined by perfectoid $\ol{\bb{F}}_{p}$-algebras.
\item For $S \in \Perf$, let $X_{S}$ denote the relative (schematic) Fargues-Fontaine curve over $S$.
\item For $\Spa{(F,\mathcal{O}_{F})} \in \Perf$ a geometric point, we will often drop the subscript on $X_{F}$ and just write $X$ for the associated Fargues-Fontaine curve. 
\item For $b \in B(G)$, we write $\mathcal{F}_{b}$ for the associated $G$-bundle on $X$.
\item For $S \in \Perf$, we let $\mathcal{F}_{G}^{0}$ denote the trivial $G$-bundle on $X_{S}$.
\item We will denote our coefficients by $\Lambda$. Throughout the paper, they will be one of the following unless otherwise stated.
\begin{setup}{\label{assumption: coefficients}}
We will assume throughout that the coefficient system $\Lambda/\mathbb{Z}_{\ell}[\sqrt{p}]$ is one of the following:
\begin{enumerate}
\item $\Lambda$ is an $\ell$-torsion ring of the form $\Lambda = \mathcal{O}_{F}/\varpi^{n}$ for some $n \geq 1$, $F/\bb{Q}_{\ell}$ a (possibly infinite) algebraic extension such that $\mathcal{O}_{F}$ is a DVR and $\varpi \in \mathcal{O}_{F}$ a uniformizing element in its ring of integers.
\item $\Lambda = \mathcal{O}_{F}$ for $F/\bb{Q}_{\ell}$ a finite algebraic extension.
\item $\Lambda = \mathcal{O}_{F}$ for $F/\bb{Q}_{\ell}$ a (possibly infinite) algebraic extension.
\item $\Lambda = F$ for $F/\bb{Q}_{\ell}$ a (possibly infinite) algebraic extension.
\end{enumerate}
\end{setup}
We note that assuming  $\Lambda/\mathbb{Z}_{\ell}[\sqrt{p}]$ fixes a choice of square root of $p$ in $\Lambda$. We define all half Tate twists and square roots of the norm character with respect to this choice. 
\item For an Artin $v$-stack $X$, we write $\D_{\blacksquare}(X,\Lambda)$ for the condensed $\infty$-category of solid $\Lambda$-valued sheaves on $X$, and write $\Dlis(X,\Lambda) \subset \D_{\blacksquare}(X,\Lambda)$ for the full sub-category of $\Lambda$-valued lisse-\'etale sheaves, as defined in \cite[Chapter~VII]{FS}.
\item For a $v$-stack or diamond $X$, when working with torsion coefficients, we will indicate this by just writing $\D(X) := \D_{\text{\'et}}(X,\Lambda)$ for the category of \'etale $\Lambda$-sheaves on $X$, as defined \cite{Ecod}. If $X$ is an Artin $v$-stack (\cite[Definition~IV.V.1]{FS}) admitting a separated $\ell$-cohomologically smooth surjection $U \ra X$ from a locally spatial diamond $U$ such that the \;etale site has a basis with bounded $\ell$-cohomological dimension (which will always be the case for our applications) then we will regard it as a condensed $\infty$-category via the identification $\Dlis(X,\Lambda) \simeq \D(X)$ when viewed as objects in $\D_{\blacksquare}(X,\Lambda)$ \cite[Proposition~VII.6.6]{FS}.
\item For a locally closed immersion of $v$-stacks $V \hookrightarrow X$, we will occasionally denote the $*$-pullback along this immersion by $(-)|_{V}: \D(X) \ra \D(V)$.
\item We let $\hat{G}$ denote the Langlands dual group of $G$ with fixed splitting $(\hat{T},\hat{B},\{X_{\alpha}\})$.
\item If $E$ denotes the splitting field of $G$ then the action of $W_{\mathbb{Q}_{p}}$ factors through $Q:= W_{\mathbb{Q}_{p}}/W_{E}$. We let $\phantom{}^{L}G := \hat{G} \rtimes Q$ denote the $L$-group. 
\item Throughout, we will assume that, for our fixed $G$, $\ell$ is very decent in the sense mentioned in the introduction that $\ell \nmid \pi_{0}(Z(G))$ and that $\ell \nmid |Q|$. 
\item For $I$ a finite index set, we let $\Rep_{\Lambda}(\phantom{}^{L}G^{I})$ denote the category of algebraic representations of $\phantom{}^{L}G^{I}$ on finite projective $\Lambda$-modules.
\item To any condensed $\infty$-category $\mathcal{C}$ and any condensed group $\Gamma$, we write $\mathcal{C}^{B\Gamma}$ for the category of objects with continuous $\Gamma$-action, as defined in \cite[Section~IX.1]{FS}.
\item We will let $\Div^{1} := \Spd{\Breve{\mathbb{Q}}_{p}}/\mathrm{Frob}^{\mathbb{Z}}$ denote the mirror curve, and, for a finite extension $E/\mathbb{Q}_{p}$, we write $\Div^{1}_{E}$ for the base-change to $E$.
\item For $I$ a finite index set, we let $\Div^{I}$ denote $|I|$-copies of the mirror curve. For $n \in \mathbb{N}_{\geq 1}$, we let $\Div^{(n)} = (\Div^{1})^{n}/S_{n}$, denote the $n$th symmetric power of the mirror curve, where $S_{n}$ is the symmetric group on $n$ letters.  
\item For a reductive group $H/\mathbb{Q}_{p}$, we write $\D(H(\mathbb{Q}_{p}),\Lambda)$ for the unbounded derived category of smooth $\Lambda$-representations. 
\item We say a map of $v$-stacks $f: X \rightarrow Y$ is representable in nice diamonds if it is representable in locally spatial diamonds, is compactifiable, and (locally) $\mathrm{tr.deg}(f) < \infty$.
\item All $6$-functors will be implicitly derived unless otherwise stated.
\item For a locally pro-$p$ group $H$, we write $\ul{H}$ for the functor sending $S \in \Perf$ to $\mathrm{Cont}(|S|,H)$, the set of continuous maps from the underlying topological space of $S$ to $H$. 
\begin{remark} 
At various points, we will need to consider the functors $f_{!}: \D(X) \ra \D(Y)$ and $f^{!}: \D(Y) \ra \D(X)$ for certain "stacky" morphisms of Artin $v$-stacks $f: X \ra Y$. The correct definitions of these functors in this case are given in the work of \cite{DHW}. In particular, they extend the $6$-functors studied in \cite{Ecod,FS} to fine maps \cite[Definition~1.3]{DHW} of decent $v$-stacks \cite[Definition~1.2]{DHW}. In general being a decent $v$-stack is stronger than being Artin. However, it is easy to check that for all the stacks (resp. morphisms) we consider these functors for they will be decent (resp. fine). To see this, one can use which states that if $f: X \ra Y$ is a map of $v$-stacks which is representable in nice diamonds and $Y$ is decent then $X$ is also decent and $f$ is fine. In the cases we consider, one can apply this if one takes $Y = \Bun_{G}$. To see that $\Bun_{G}$ is decent, one can use the charts studied in \cite[Section~V.3]{FS}, and take advantage of the fact that the maps defining the charts are formally smooth \cite[Definition~IV.3.1]{FS} by \cite[Proposition~IV.4.24]{FS}. This in particular allows one to see that these charts map strictly surjectively \cite[Definition~4.1]{DHW} to $\Bun_{G}$. It remains to explain why the maps appearing in our context are fine, to do this one can combine the previous analysis with \cite[Proposition~4.10 (iii)]{DHW}, which says that fine morphisms satisfy the 2 out of 3 property.  
\end{remark}
\item When speaking about such fine maps of decent $v$-stacks we will often just cite theorems that only apply to the setting where $f$ is representable in nice diamonds, and leave it to the reader to check that one can deduce the analogous results from the cited result and the formal properties of the $6$-functors defined in \cite{DHW}. 
\item Given a decent $v$-stack $X \ra \ast$ such that $X$ is fine over $\ast$, we let $K_{X} := f^{!}(\Lambda) \in \D(X)$ denote the dualizing object of $X$. Similarly, for $\mathcal{F} \in \D(X)$, we will write $R\Gamma_{c}(X,\mathcal{F}) := f_{!}(\mathcal{F}) \in D(\Lambda)$. We write $\mathbb{D}_{X}(-) := R\mathcal{H}om(-,K_{X})$ for the Verdier duality functor. For a fine map $f: X \ra S$ of decent $v$-stacks, we write $\mathbb{D}_{X/S} := R\mathcal{H}om(-,f^{!}(\Lambda))$ for relative Verdier duality.
\item We will use the geometric normalization of local class field theory. For $n \in \mathbb{Z}$, we write $(n)$ for the $n$th power of the $\ell$-adic cyclotomic character of $W_{\mathbb{Q}_{p}}$. We note that, under this normalization, $(1)$ is sent to the norm character $|\cdot|: \mathbb{Q}_{p}^{*} \ra \Lambda^{*}$, which acts trivially on $\mathbb{Z}_{p}^{*} $ and sends $p$ to $p^{-1} \in \Lambda^{*}$.
\end{itemize}
Before introducing the rest of the notation, we discuss the relationship between unramified elements in $B(G)$ and the representation theory of the dual group. 
\subsection{Unramified Elements in $B(G)$ and Highest Weight Theory}{\label{sec: unramifiedelments}}
In this section, we will study the set of unramified elements in the Kottwitz set of $G$. As we will show, these elements are connected to the highest weight theory of the Langlands dual group $\hat{G}$, as discussed in \cite[Section~4.2.1]{XZ}. First, we recall that the Kottwitz set $B(G)$ of a connected reductive group $G/\mathbb{Q}_{p}$ is equipped with two maps:
\begin{itemize}
    \item The slope homomorphism
    \[ \nu: B(G) \rightarrow \mathbb{X}_{*}(T_{\overline{\mathbb{Q}}_{p}})^{+,\Gamma}_{\mathbb{Q}} \]
    \[ b \mapsto \nu_{b} \]
    where $\Gamma := \mathrm{Gal}(\overline{\mathbb{Q}}_{p}/\mathbb{Q}_{p})$ and 
    $\mathbb{X}_{*}(T_{\overline{\mathbb{Q}}_{p}})_{\mathbb{Q}}^{+}$ is the set of rational dominant cocharacters of $G$. 
    \item The Kottwitz invariant
    \[ \kappa_{G}: B(G) \rightarrow \pi_{1}(G)_{\Gamma} \]
    where $\pi_1(G)$ denotes the algebraic fundamental group of Borovoi.
\end{itemize}
Now, given a geometric cocharacter $\mu$ of $G$ with $E$ a Galois extension containing the reflex field, we can define the element:
\[ \tilde{\mu} := \frac{1}{[E:\mathbb{Q}_{p}]} \sum_{\gamma \in \mathrm{Gal}(E/\mathbb{Q}_{p})} \gamma(\mu) \in \mathbb{X}_{*}(T_{\overline{\mathbb{Q}}_{p}})^{+,\Gamma}_{\mathbb{Q}} \]
We let $\mu^{\flat}$ be the image of $\mu$ in $\pi_1(G)_{\Gamma} \simeq X^*(Z(\hat{G})^{\Gamma})$. Via the isomorphism $B(G)_{\mathrm{basic}} \simeq \pi_{1}(G)_{\Gamma}$, we regard it as a basic element of $B(G)$, which are the minimal elements in the natural partial ordering on $B(G)$. Now we recall that, for a torus $T$, we have an isomorphism $B(T) \simeq \mathbb{X}_{*}(T_{\ol{\mathbb{Q}}_{p}})_{\Gamma}$. We can use this isomorphism to give a nice description of a certain piece of $B(G)$. 
\begin{definition}{\cite[Section~4.2.1]{XZ}}
We let $B(G)_{\mathrm{un}} \subset B(G)$ denote the image of the natural map $B(T) \ra B(G)$. We refer to this as the set of unramified elements. 
\end{definition}
We now have the following Lemma. We write $(-)_{\Gamma}$ for the natural quotient map $\mathbb{X}_{*}(T_{\ol{\mathbb{Q}}_{p}}) \ra \mathbb{X}_{*}(T_{\ol{\mathbb{Q}}_{p}})_{\Gamma}$.
\begin{lemma}{\cite[Lemma~4.2.2]{XZ}}{\label{unrelements}}
Let $\mathbb{X}_{*}(T_{\ol{\mathbb{Q}}_{p}})_{\Gamma} \simeq B(T) \ra B(G)$ be the natural map. Then this induces an isomorphism:
\[ \mathbb{X}_{*}(T_{\ol{\mathbb{Q}}_{p}})_{\Gamma}/W_{G} \simeq B(G)_{\mathrm{un}} \]
\end{lemma}
\begin{proof}
Strictly speaking, the proof given by Xiao-Zhu is only in the case that $G$ is unramified. We remedy this now. Note that it is clear that this map is surjective, so it suffices to check injectivity. Let $\mu_{1},\mu_{2} \in \mathbb{X}_{*}(T_{\ol{\mathbb{Q}}_{p}})$  be two elements with $b_{1},b_{2}$ their images in $B(G)$ under the natural composite
\[ \mathbb{X}_{*}(T_{\ol{\mathbb{Q}}_{p}}) \ra \mathbb{X}_{*}(T_{\ol{\mathbb{Q}}_{p}})_{\Gamma} \simeq B(T) \ra B(G), \]
and suppose that $b_{1} = b_{2}$. Since $\kappa_{G}(b_{1}) = \kappa_{G}(b_{2})$, the difference $\mu_{1}-\mu_{2}$ lies in the sum of the coroot lattice and $\sum_{\gamma\in\Gamma}(\gamma-1)\cochar$. Changing $\mu_{1}$ by an element of the latter subgroup, we may therefore assume that $\mu_{1}-\mu_{2}=\alpha$ for an element $\alpha$ of the coroot lattice. Since the slope homomorphisms of $\nu_{b_{1}}$ and $\nu_{b_{2}}$ are equal by assumption, we can assume, after conjugating by an element of $W_{G}$, that $\tilde{\mu}_{1} = \tilde{\mu}_{2}$. Therefore, it follows that, if $E'/\bb{Q}_{p}$ is a finite Galois extension containing the reflex fields of $\mu_{1}$, $\mu_{2}$, and $\alpha$, we have an equality
\[ \sum_{g \in \Gal(E'/\mathbb{Q}_{p})} g(\alpha) = 0 \]
which in turn implies that 
\[ \sum_{g \in \Gal(E'/\mathbb{Q}_{p})} (1 - g)(\alpha) = |\Gal(E'/\mathbb{Q}_{p})|\alpha \]
The coroot lattice has a basis of simple coroots permuted by $\Gamma$, so its group of coinvariants is torsion-free. The preceding equality therefore implies that the image of $\alpha$ in these coinvariants, and hence in $\mathbb X_*(T_{\overline{\mathbb Q}_p})_\Gamma$, is zero.
\end{proof}
Now we would like to describe $\mathbb{X}_{*}(T_{\ol{\mathbb{Q}}_{p}})_{\Gamma}/W_{G}$ slightly differently. To do this, we consider the natural pairing 
\[ \langle -,- \rangle:  \mathbb{X}_{*}(T_{\ol{\mathbb{Q}}_{p}})_{\Gamma} \times \mathbb{X}^{*}(T_{\ol{\mathbb{Q}}_{p}})^{\Gamma} \ra \mathbb{Z} \]
induced by the usual pairing between cocharacters and characters. We let $\hat{\Delta} \subset \mathbb{X}^{*}(T_{\ol{\mathbb{Q}}_{p}})$ (resp. $\Delta \subset \mathbb{X}_{*}(T_{\ol{\mathbb{Q}}_{p}})$) be the set of (absolute) simple roots (resp. coroots) of $G$. Then we define
$\mathbb{X}_{*}(T_{\ol{\mathbb{Q}}_{p}})_{\Gamma}^{+}$ to be the set of elements in $\mathbb{X}_{*}(T_{\ol{\mathbb{Q}}_{p}})_{\Gamma}$ whose inner product with $\Im{(\hat{\Delta} \ra \mathbb{X}^{*}(T_{\ol{\mathbb{Q}}_{p}})^{\Gamma})}$ under the natural averaging map is non-negative. The natural map 
\[ \mathbb{X}_{*}(T_{\ol{\mathbb{Q}}_{p}})_{\Gamma}^{+} \ra  \mathbb{X}_{*}(T_{\ol{\mathbb{Q}}_{p}})_{\Gamma}/W_{G} \]
is an isomorphism. We also note that we have a natural partial ordering on $\mathbb{X}_{*}(T_{\ol{\mathbb{Q}}_{p}})_{\Gamma}$. In particular, given  $\ol{\nu}, \ol{\nu}' \in \mathbb{X}_{*}(T_{\ol{\mathbb{Q}}_{p}})_{\Gamma}$, we say that $\ol{\nu} \geq \ol{\nu}'$ if $\ol{\nu} - \ol{\nu}'$ is a non-negative integral combination of $\alpha_{\Gamma}$ for $\alpha \in \Delta$. We note that we have a natural injective order preserving map: 
\begin{equation}{\label{eqn: NaturalMapInDiagram}}
 \mathbb{X}_{*}(T_{\ol{\mathbb{Q}}_{p}})^{+}_{\Gamma} \ra \mathbb{X}_{*}(T_{\ol{\mathbb{Q}}_{p}})_{\mathbb{Q}}^{\Gamma,+} \times \pi_{1}(G)_{\Gamma} 
\end{equation}
With this, we can reformulate the previous lemma as follows. 
\begin{lemma}{\cite[Lemma~4.2.3]{XZ}}{\label{unrweyl}}
The following diagram is commutative and respects the partial ordering 
\[ \begin{tikzcd}
&  B(G)_{\mathrm{un}} \arrow[r,"\simeq"] \arrow[d] & \mathbb{X}_{*}(T_{\ol{\mathbb{Q}}_{p}})^{+}_{\Gamma} \arrow[d,"(\ref{eqn: NaturalMapInDiagram})"]     \\
& B(G) \arrow[r,"\nu \times \kappa"]  & \mathbb{X}_{*}(T_{\ol{\mathbb{Q}}_{p}})_{\mathbb{Q}}^{\Gamma,+} \times \pi_{1}(G)_{\Gamma}, 
\end{tikzcd} \]
where the left vertical arrow is the natural inclusion and the top horizontal arrow is the isomorphism of \ref{unrelements}.
\end{lemma}
Now recall that, for $\mu$ a geometric dominant cocharacter of $G$, we have the following.
\begin{definition}{\label{bgmudefinition}}
We define $B(G,\mu) \subset B(G)$ to be subset of $b \in B(G)$ for which $\nu_{b} \leq \tilde{\mu}$ with respect to the Bruhat ordering and $\kappa(b) = \mu^{\flat}$. 
\end{definition}
The previous lemma allows us to interpret the unramified elements in this set as follows.
\begin{corollary}{\cite[Corollary~4.2.4]{XZ}}
Under the identifications $\mathbb{X}_{*}(T_{\ol{\mathbb{Q}}_{p}})_{\Gamma}^{+} \simeq \mathbb{X}_{*}(T_{\ol{\mathbb{Q}}_{p}})_{\Gamma}/W_{G} \simeq B(G)_{\mathrm{un}}$, we have an equality:  
\[ B(G,\mu)_{\mathrm{un}} := B(G)_{\mathrm{un}} \cap B(G,\mu) = \{\lambda_{\Gamma} \in \mathbb{X}_{*}(T_{\ol{\mathbb{Q}}_{p}})_{\Gamma}^{+}\text{ } | \text{ } \lambda_{\Gamma} \leq \mu_{\Gamma} \}\]
\end{corollary}
We now would like to connect this set with the highest weight theory for $\hat{G}$. If the group is not split then the unramified elements are naturally connected with the highest weight theory of the subgroup $\hat{G}^{\Gamma}$. Even though $\hat{G}^{\Gamma}$ is possibly disconnected its representation theory behaves like a connected reductive group. To see this, first we note that the subgroup $\hat{T}^{\Gamma}$ defined by the maximal torus has character group isomorphic to $\mathbb{X}_{*}(T_{\ol{\mathbb{Q}}_{p}})_{\Gamma}$, and the partial order described above allows one to talk about the highest weight of a representation. In particular, if we let $\hat{T}^{\Gamma,\circ}$ (resp. $\hat{G}^{\Gamma,\circ}$) denote the neutral component of $\hat{T}^{\Gamma}$ (resp. $\hat{G}^{\Gamma}$). Then one can use that the natural map $\hat{T}^{\Gamma}/\hat{T}^{\Gamma,\circ} \ra \hat{G}^{\Gamma}/\hat{G}^{\Gamma,\circ}$ is an isomorphism (\cite[Lemma~4.6]{Zhu}) to see that usual highest weight theory extends to $\hat{G}^{\Gamma}$. In particular, we have the following.
\begin{lemma}{\cite[Lemma~4.10]{Zhu}}{\label{dualgrouphighestweight}}
For $\ol{\mu} \in \mathbb{X}_{*}(T_{\ol{\mathbb{Q}}_{p}})^{+}_{\Gamma}$, there is a unique up to isomorphism irreducible representation of $V_{\ol{\mu}} \in \Rep_{\ol{\mathbb{Q}}_{\ell}}(\hat{G}^{\Gamma})$ of highest weight $\ol{\mu}$, which gives rise to all the irreducible representations in $\Rep_{\ol{\mathbb{Q}}_{\ell}}(\hat{G}^{\Gamma})$ for varying $\ol{\mu}$. Moreover, the multiplicity of the $\ol{\mu}$ weight space in $V_{\ol{\mu}}$ is $1$, and the non-zero weights $\nu \in \mathbb{X}_{*}(T_{\ol{\mathbb{Q}}_{p}})_{\Gamma}$ of $V_{\ol{\mu}}$ lie in the convex hull of the $W_{G}$-orbit of $\ol{\mu}$.
\end{lemma}
To a geometric dominant cocharacter $\mu$, we can attach an irreducible representation $V_{\mu} \in \Rep_{\ol{\mathbb{Q}}_{\ell}}(\hat{G})$. This defines a natural representation of $\hat{G} \rtimes W_{E_{\mu}}$ as in \cite[Lemma~2.1.2]{Kott1}, where $E_{\mu}$ is the reflex field of $\mu$. An element $\nu \in \cochar$ defines a representation of $\hat{T}$, and we write $V_{\mu}(\nu)$ for the corresponding weight space of $V_{\mu}$. If we consider the restriction $V_{\mu}|_{\hat{G}^{\Gamma}}$ then the weight space $V_{\mu}(\nu)$ gives rise to a $\nu_{\Gamma}$ weight space, where we write $(-)_{\Gamma}: \cochar \ra \coinv$ for the map given by taking coinvariants. Using this, it is easy to see we have the following relationship.
\begin{lemma}{\label{lemma: coinv versus orbs}}
For $\mu \in \domcochar$ and $\ol{\nu} \in \coinv$, we have the following equality:
\[ \dim(V_{\mu}(\ol{\nu})) = \sum_{\substack{\nu \in \cochar \\ \nu_{\Gamma} = \ol{\nu}}} \dim(V_{\mu}(\nu)) \]
\end{lemma}
We will combine this lemma with the following, which follows from the above discussion. 
\begin{corollary}{\label{bgmuweights}}{(\cite[Lemma~4.26]{XZ})}
For $\mu$ a geometric dominant cocharacter, under the identification $\mathbb{X}_{*}(T_{\ol{\mathbb{Q}}_{p}})_{\Gamma}^{+} \simeq B(G)_{\mathrm{un}}$ the elements $\ol{\nu} \in B(G,\mu)_{\mathrm{un}}$ correspond to $W_{G}$-orbits of the possible non-zero weights in $V_{\mu}|_{\hat{G}^{\Gamma}}$.  
\end{corollary}
Let's study this now more carefully. For $b \in B(G)$, we want to use the above discussion to understand the fiber of the map: 
\[ i: B(T) \ra B(G)_{\mathrm{un}} \subset B(G) \]
Recall that, given $b \in B(G)$, since $G$ is quasi-split the $\sigma$-centralizer $J_{b}$ is an extended pure inner form (in the sense of \cite[Section~5.2]{Kott}) of a Levi subgroup $M_{b}$ of $G$ \cite[Section~6.2]{Kott}, which is the centralizer of the slope homomorphism $\nu_{b}$ of $b$. We make the following definition. 
\begin{definition}
For $b \in B(G)$, we let $W_{M_{b}}$ denote the relative Weyl group of $M_{b}$ and set $W_{b} := W_{G}/W_{M_{b}}$. We will fix a set of representatives $w \in W_{G}$ of minimal length, as in \cite[Section~2.11]{BZ}, and abuse notation by writing $w$ for both the representative and the corresponding element. 
\end{definition}
When combining the above discussion with Lemma \ref{unrelements}, we can deduce the following Corollary. 
\begin{corollary}{\label{weylgrouporbits}}
For fixed $b \in B(G)_{\mathrm{un}}$, the fiber $i^{-1}(b)$ has a unique element, denoted $b_{T}$, whose $\kappa$-invariant lies in $\mathbb{X}_{*}(T_{\ol{\mathbb{Q}}_{p}})_{\Gamma}^{+}$. Moreover, we have an equality $i^{-1}(b) = \{ w(b_{T})\text{ } |\text{ } w \in W_{b} \}$.
\end{corollary} 
By \cite[Lemma~4.26]{GH}, for a standard parabolic $P$ with Levi
factor $M$, $b \in B(G)$ admits a basic reduction to $M$ if
and only if there exists $w\in W_G$ with $w(M)\subset M_b$
such that $w(P)\cap M_b$ transfers to a parabolic subgroup
of $J_b$. Specializing to $P=B$, and using the dominant $T$-reduction
from Corollary~2.12, gives the following criterion.
\begin{lemma}{\label{jbborel}}
An element $b \in B(G)$ lies in $B(G)_{\mathrm{un}}$ if and only if $B \cap M_{b}$ defines, via the inner twisting, a Borel subgroup of $J_{b}$.
\end{lemma} 
\begin{remark}
We note that, when $b \in B(G)_{\mathrm{un}}$, we have an isomorphism $M_{b} \simeq J_{b}$ because $J_{b}$ will be a quasi-split inner form of $M_{b}$.
\end{remark}
We will end with an important observation about the map $(-)_{\Gamma}: \gamorb \ra \coinv$ from orbits to coinvariants. Let $\tilde{\mathcal{J}}$ (resp. $\mathcal{J}$) denote the set of vertices of the absolute (resp. relative) Dynkin diagram of $G$. For $i \in \tilde{\mathcal{J}}$, we write $\tilde{\alpha}_{i} \in \cochar$ for the corresponding simple absolute coroot. We recall that $\Gamma$ permutes the $\tilde{\alpha}_{i}$, and the orbits under $\Gamma$ are in bijection with elements of $\mathcal{J}$; namely, the average over the orbit is a positive rational
multiple of the corresponding reduced relative simple
coroot $\alpha_{i,A}$ (see \cite[Lemma~3.2]{Haines2018}). Therefore, for each $i \in \mathcal{J}$, we obtain an element $\alpha_{i} \in \coinv$ given by the common image of the elements in the orbit corresponding to $i \in \mathcal{J}$ under the map $(-)_{\Gamma}$. This allows us to make the following definition.
\begin{definition}{\label{def: gammaorbsimp}}
We denote the group of coinvariants by $\Lambda_{G,B} := \coinv$, and set $\Lambda_{G,B}^{pos}$ to be the semi-group spanned by the elements $\alpha_{i} \in \coinv$ corresponding to the $\Gamma$-orbit of coroots indexed by $i \in \mathcal{J}$.
\end{definition}
Now we introduce the rest of the notation motivated by the discussion above. 
\begin{itemize}
\item We let $\hat{\Lambda}_{G}^{+} := \mathbb{X}^{*}(T_{\ol{\mathbb{Q}}_{p}})^{\Gamma,+}$ be the set of Galois invariant dominant characters.
\item To a dominant character $\hat{\lambda} \in \bb{X}^{*}(T_{\ol{\bb{Q}}_{p}})^{+}$, we get an associated highest weight Weyl $G_{\ol{\bb{Q}}_{p}}$-module, denoted $\mathcal{V}^{\hat{\lambda}}$. It has a fixed highest weight vector $v^{\hat{\lambda}} \in \mathcal{V}^{\hat{\lambda}}$, and, given a pair of such weights $\hat{\lambda}_1, \hat{\lambda}_2$, there is a canonical map
\[ \mathcal{V}^{\hat{\lambda}_1 + \hat{\lambda}_2} \rightarrow \mathcal{V}^{\hat{\lambda}_1} \otimes \mathcal{V}^{\hat{\lambda}_{2}} \]
which takes $v^{\hat{\lambda}_1 + \hat{\lambda}_2}$ to $v^{\hat{\lambda}_1} \otimes v^{\hat{\lambda}_2}$. If $\hat{\lambda} \in \hat{\Lambda}_{G}^{+}$, then this defines a highest weight Weyl module of $G$ and we will also write $\mathcal{V}^{\hat{\lambda}}$ for this Weyl module.
\item For $S \in \Perf$ and a $G$-bundle $\mathcal{F}_{G}$ on $X_{S}$, we write $\mathcal{V}_{\mathcal{F}_{G}}$ for the vector bundle on $X_{S}$ induced by $\mathcal{F}_{G}$ and the algebraic representation $\rho_{\mathcal{V}}: G \ra \GL(V)$ under the Tannakian interpretation of $G$-torsors (see \cite[Appendix~A]{SW}) (i.e $\mathcal{V}_{\mathcal{F}_{G}}$ is the pushout $\mathcal{F}_{G} \times^{G} V)$).
\item Let $\mathcal{J}$ be the set of vertices of the relative Dynkin diagram of $G/\mathbb{Q}_{p}$. For each $i \in \mathcal{J}$, we denote the corresponding element in the coinvariant lattice by $\alpha_i \in \mathbb{X}_{*}(T_{\ol{\mathbb{Q}}_{p}})_{\Gamma}$, as in Definition \ref{def: gammaorbsimp}. We denote the (reduced) simple coroots corresponding to $i \in \mathcal{J}$ by  $\alpha_{i,A} \in X_{*}(A)$. Similarly, for $i \in \mathcal{J}$, we denote the reduced simple roots by $\hat{\alpha}_{i} \in \mathbb{X}^{*}(A)^{+}$.
\item We consider the natural pairing 
\[ \langle -,- \rangle: \hat{\Lambda}_{G} \times \Lambda_{G,B} \ra \mathbb{Z}. \]
\item We will regularly use the natural quotient map
\[ (-)_{\Gamma}: \cochar \ra \coinv \]
as well as the map 
\[ (-)^{\Gamma}: \cochar \ra \gamorb \]
from cocharacters to their Galois orbits. We note that $(-)_{\Gamma}$ factorizes over $(-)^{\Gamma}$ as a map of sets.
\item For a geometric dominant cocharacter $\mu \in \domcochar$ with reflex field $E_{\mu}$, we write $V_{\mu}$ for the natural representation of $W_{E_{\mu}}/W_{E} \ltimes \hat{G}$ of highest weight $\mu$, as in \cite[Lemma~2.1.2]{Kott1} and define $Q_{\mu} := W_{E_{\mu}}/W_{E}$. For a representation $\tau$ of $Q_{\mu}$, we write $V_{\mu^{\Gamma},\tau} \in \Rep_{\Lambda}(\phantom{}^{L}G)$ for the induction of the tensor product $V_{\mu} \otimes \tau$ of $Q_{\mu} \ltimes \hat{G}$ representations to $\phantom{}^{L}G$ (Here we regard $\tau$ as a representation of $Q_{\mu} \ltimes \hat{G} $ via inflation along the natural homomorphism $Q_{\mu} \ltimes \hat{G} \ra Q_{\mu}$). We note that this only depends on the associated $\Gamma$-orbit $\mu^{\Gamma}$ of $\mu$. 
\item For $b \in B(G)$, we let $J_{b}$ be the extended pure inner form of $M_{b}$ considered above. If $b \in B(G)_{\mathrm{un}}$, we write $B_{b}$ for the Borel defined by $B \cap M_{b}$ under the inner twisting, as in Lemma \ref{jbborel}. 
\item Given $b \in B(G)_{\mathrm{un}}$, we note that, by Lemma \ref{unrelements}, there exists a unique element $b_{T} \in B(T)$ with dominant slope homomorphism with respect to the choice of Borel. We refer to this as the dominant reduction. Similarly, we write $b_{T}^{-}$ for the unique element with anti-dominant slopes and will refer to this as the HN-dominant reduction of $b$ (Recall that there is a minus sign when comparing isocrystal slopes and HN-slopes).
\item We set $W_{b} := W_{G}/W_{M_{b}}$, where $W_{M_{b}}$ (resp. $W_{G}$) is the relative Weyl group of $M_{b}$ (resp. $G$). We identify $w \in W_{b}$ with a representative of minimal length in $W_{G}$ throughout.
\end{itemize} 
\section{Geometric Local Class Field Theory}{\label{GCFT}}
In this section, we will temporarily allow for a general coefficient ring $\Lambda/\bb{Z}_{\ell}[\sqrt{p}]$.
\subsection{Hecke Eigensheaves for Tori}{\label{s: HeckeEigensheavesforTori}}
We recall that all our moduli stacks (e.g $\Bun_{G}$ and so on) will be viewed as $v$-stacks over $\Spd(\ol{\bb{F}}_{p})$ unless otherwise stated; in other words, they are groupoid valued functors on the category $\Perf_{\ol{\bb{F}}_{p}}^{\mathrm{op}}$ of affinoid perfectoid spaces defined by perfectoid $\ol{\bb{F}}_{p}$-algebras.

In this section, we want to talk about geometric local class field theory. Namely, given a torus $T$ with $L$-parameter 
\[ \phi_{T}: W_{\mathbb{Q}_{p}} \ra \phantom{}^{L}T(\Lambda) \]
where $\phantom{}^{L}T$ denotes the Langlands dual group of $T$, we want to construct a Hecke eigensheaf, denoted $\mathcal{S}_{\phi_{T}}$, on the moduli stack $\Bun_{T}$. We recall that, for a representation $V \in \Rep_{\Lambda}(\phantom{}^{L}G^{I})$ of $I$-copies of the $L$-group of $G$ for some finite index set $I$, a Hecke operator is a map defined by the correspondence
\[ \begin{tikzcd}
& & \arrow[dl,"h^{\leftarrow} \times \pi"] \Hck_{G}^{I} \arrow[dr,"h^{\rightarrow}"] & & \\
& \Bun_{G} \times \Div^{I} & & \Bun_{G}   & 
\end{tikzcd} \]
where $\Hck_{G}^{I}$ is the functor that parametrizes, for $S \rightarrow \Div^{I}$ defining a tuple of Cartier divisors in the relative Fargues-Fontaine $X_{S}$ over $S$, corresponding to characteristic $0$ untilts $S_{i}^{\sharp}$ for $i \in I$ of $S$, a pair of $G$-torsors $\mathcal{E}_{1}$, $\mathcal{E}_{2}$ together with an isomorphism
\[ \beta:\mathcal{E}_{1}|_{X_{S} \setminus \bigcup_{i \in I} S_{i}^{\sharp}} \xrightarrow{\simeq} \mathcal{E}_{2}|_{X_{S} \setminus \bigcup_{i \in I} S_{i}^{\sharp}}\]
where $h^{\ra}((\mathcal{E}_{1},\mathcal{E}_{2},(S_{i}^{\sharp})_{i \in I})) = \mathcal{E}_{1}$ and $(h^{\la} \times \pi)((\mathcal{E}_{1},\mathcal{E}_{2},\beta,(S_{i}^{\sharp})_{i \in I})) = (\mathcal{E}_{2},(S_{i}^{\sharp})_{i \in I})$. For an algebraic representation $V \in \Rep_{\Lambda}(^{L}G^{I})$, the geometric Satake correspondence of Fargues-Scholze \cite[Chapters~VI,IX.2]{FS} furnishes a sheaf $\mathcal{S}_{V} \in \D_{\blacksquare}(\Hck,\Lambda)$. Using this, we can define the Hecke operator as
\[ T_{V}: \Dlis(\Bun_{G},\Lambda) \ra \D_{\blacksquare}(\Bun_{G} \times \Div^{I},\Lambda) \]
\[ \mathcal{F} \mapsto (h^{\la} \times \pi)_{\natural}((h^{\ra})^{*}(\mathcal{F}) \otimes \mathcal{S}_{V})\]
where $(h^{\la} \times \pi)_{\natural}$ is the natural push-forward (we note that with torsion coefficients this can be replaced with usual pushforward using \cite[Proposition~VII.5.3]{FS} and the properness of $h^{\la} \times \pi$ after restricting to the support of $\mathcal{S}_{V}$). I.e the left adjoint to $*$-pullback in the category of solid $\Lambda$-sheaves \cite[Proposition~VII.3.1]{FS}. 
This induces a functor 
\[ T_{V}: \Dlis(\Bun_{G},\Lambda) \ra \Dlis(\Bun_{G},\Lambda)^{BW_{\mathbb{Q}_{p}}^{I}} \]
via a version of Drinfeld's Lemma \cite[Theorem~I.7.2,Proposition~IX.2.1,Corollary~IX.2.3]{FS}. When $\Lambda$ is $\ell$-torsion, this is essentially the statement that $\Lambda$-valued local systems on $\Div^{I}$ are equivalent to continuous representations of $W_{\mathbb{Q}_{p}}^{I}$ on finite projective $\Lambda$-modules \cite[Proposition~VI.9.2]{FS}. In this case, we will freely pass between this perspective of local systems and $W_{\mathbb{Q}_{p}}^{I}$-representations.

With this in hand, we can define what it means for a sheaf on $\Bun_{G}$ to be a Hecke eigensheaf. 
\begin{definition}{\label{defeigsheaf}}
Given a $L$-parameter $\phi: W_{\mathbb{Q}_{p}} \ra \phantom{}^{L}G(\Lambda)$ continuous with respect to the $\ell$-adic topology, we say a sheaf $\mathcal{S}_{\phi} \in \Dlis(\Bun_{G},\Lambda)$ is a Hecke eigensheaf with eigenvalue $\phi$ if, for all finite index sets $I$ and $V \in \Rep_{\Lambda}(\phantom{}^{L}G^{I})$ with associated map $r_{V}: \phantom{}^{L}G^{I} \ra \GL(V)$, there exists isomorphisms
\[ \eta_{V,I}: T_{V}(\mathcal{S}_{\phi}) \simeq \mathcal{S}_{\phi} \boxtimes r_{V} \circ \phi^{\boxtimes I} \]
of sheaves in $\Dlis(\Bun_{G},\Lambda)^{BW_{\mathbb{Q}_{p}}^{I}}$, the isomorphism is natural with respect to morphisms of finite index sets $I$ and the fusion product, morphisms of representations $V \in \Rep(\phantom{}^{L}G^{I})$, and exterior products/compositions. We will similarly say that $\mathcal{S}_{\phi}$ is a weak eigensheaf with eigenvalue $\phi$ if we only know the existence of these isomorphisms. 
\end{definition}
\begin{remark}
The assignment $V\mapsto T_V$ is compatible with exterior products and fusion; after restricting the Weil actions along the diagonal, it is monoidal for tensor product and composition. It is functorial in $(V,I)$. In particular, given two representations $V,W \in \Rep_{\Lambda}(\phantom{}^{L}G)$, we have a natural isomorphism
\begin{equation}{\label{eqn: compositio/exteriorproduct}}
 T_{V}T_{W}(-) \simeq T_{V\boxtimes W}(-),
\end{equation}
as well as an isomorphism 
\begin{equation}{\label{eqn: diagonalrestriction}}
T_{V \boxtimes W}(-)|_{\Delta} \simeq T_{V \otimes W}(-)
\end{equation}
given by the fusion product and similarly for any map of finite index sets $I \ra J$. The required compatibilities for the isomorphisms $\eta_{V,I}$ with respect to composition/exterior product in Definition \ref{defeigsheaf} are formulated with respect to the isomorphism (\ref{eqn: compositio/exteriorproduct}) and the required compatibilities with respect to morphisms of finite index set $I \ra J$ are given by the isomorphisms generalizing (\ref{eqn: diagonalrestriction}).
\end{remark}
Let's elucidate what the notion of Hecke eigensheaf means for tori. Recall that representations of $\phantom{}^{L}T^{I}$ are parametrized by a tuple of Galois orbits $(\nu_{i})_{i \in I} \in (\gamorb)^{I}$ together with a tuple of representations $\tau_{i}$ of $Q_{\nu_{i}} := W_{E_{\nu_{i}}}/W_{E}$, where $E_{\nu_{i}}$ denotes the reflex field of $\nu_{i}$ and $E$ is the splitting field of $G$. We denote the associated representations of $\phantom{}^{L}T$ by $\nu_{i\tau}$. For simplicity, we will take the $\tau_{i}$ to be trivial and just write $\nu_{i}$. Similarly, one has a decomposition
\[ \Hck_{T}^{I} = \bigcup_{(\nu_{i})_{i \in I} \in (\gamorb)^{I}} \Hck^{I}_{T,(\nu_{i})_{i \in I}}  \]
of $\Hck_{T}^{I}$ into closed substacks, where $\Hck^{I}_{T,(\nu_{i})_{i \in I}}$ parametrizes a modification $\mathcal{E}_{1} \dashrightarrow \mathcal{E}_{2}$ of meromorphy given by a representative of the Galois orbit $\nu_{i}$ over the $I$ Cartier divisors in $\Div^{I}$. If one lets $\mathcal{S}_{(\nu_{i})_{i \in I}}$ be the sheaf defined by the representation of $\phantom{}^{L}T^{I}$ corresponding to $(\nu_{i})_{i \in I}$, then this sheaf is supported on the subspace $\Hck^{I}_{T,(\nu_{i})_{i \in I}}$. Therefore, for studying the Hecke operator $T_{(\nu_{i})_{i \in I}}$, we can restrict the Hecke correspondence to the diagram:
\[ \begin{tikzcd}
& & \arrow[dl,"h_{(\nu_{i})_{i \in I}}^{\leftarrow} \times \pi"] \Hck^{I}_{T,(\nu_{i})_{i \in I}} \arrow[dr,"h_{(\nu_{i})_{i \in I}}^{\rightarrow}"] & & \\
& \Bun_{T} \times \Div^{I} & & \Bun_{T}  & 
\end{tikzcd} \]
We can consider the following base-change of $\Div^{I}$:
\[ \Div^{I}_{E_{(\nu_{i})_{i \in I}}} := \prod_{i \in I} \Div^{1}_{E_{\nu_{i}}}. \]
We note that, since a modification of $T$-bundles is uniquely determined by the locus of meromorphy,  if we set
$\widetilde{\Hck}^{I}_{T,(\nu_{i})_{i \in I}} \simeq \Bun_{T} \times \Div^{I}_{E_{(\nu_{i})_{i \in I}}}$ after fixing a choice of representatives $\tilde{\nu}_{i} \in \cochar$ of the $\nu_{i}$ then we have a natural map $a_{(\nu_{i})_{i \in I}}: \widetilde{\Hck}^{I}_{T,(\nu_{i})_{i \in I}} \ra \Hck_{T,(\nu_{i})_{i \in I}}$ and the Satake sheaf $\mathcal{S}_{(\nu_{i})_{i \in I}}$ is isomorphic to the pushforward $a_{(\nu_{i})_{i \in I}*}(\Lambda)$. We have a map denoted
\begin{equation}{\label{eqn: ReflexFieldHeckeOperator}}
 h_{(\tilde{\nu}_{i})_{i \in I}}^{\ra}: \Bun_{T} \times \Div^{I}_{E_{(\nu_{i})_{i \in I}}} \ra \Bun_{T}, 
\end{equation}
where, given $(\mathcal{F}_{T},(D_{i})_{i \in I})$, we denote the resulting $T$-bundle under applying this map as $\mathcal{F}_{T}(\sum_{i \in I} -\tilde{\nu}_{i}D_{i})$. This will satisfy the identity that $h^{\ra}_{(\tilde{\nu}_{i})_{i \in I}} := h^{\ra}_{(\nu_{i})_{i \in I}} \circ a_{(\nu_{i})_{i \in I}}$.  

Moreover, the map $(h_{(\nu_{i})_{i \in I}}^{\la} \times \pi) \circ a_{(\nu_{i})_{i \in I}}$ will correspond to the base-change of the finite \'etale morphism 
\begin{equation}{\label{eqn: FiniteEtaleCover}}
q_{(\nu_{i})_{i \in I}}: \Div^{I}_{E_{(\nu_{i})_{i \in I}}} \ra \Div^{I},    
\end{equation}
where we note that pushing forward corresponds to inducing the $\prod_{i \in I} W_{E_{\nu_{i}}}$ action to $\prod_{i \in I} W_{\mathbb{Q}_{p}}$. The Hecke operators attached to a tuple $(\nu_{i},\tau_{i})_{i \in I}$ by tensoring by the natural local system on $\Div^{I}_{E_{(\nu_{i})_{i \in I}}}$ defined by the $(\tau_{i})_{i \in I}$ before pushing forward along the map $a_{(\nu_{i})_{i \in I}}$.

Now, via local class field theory, there is a character $\chi: T(\mathbb{Q}_{p}) \ra \Lambda^{*}$ attached to $\phi_{T}$. Moreover, each connected component $\Bun_{T}^{\nu}$ for varying $\nu \in B(T) \simeq \Lambda_{G,B}$ is isomorphic to the classifying stack $[\ast/\ul{T(\mathbb{Q}_{p})}]$. As a consequence, we may interpret $\chi$ as a sheaf on the connected components $j_{\ol{\nu}}: \Bun_{T}^{\ol{\nu}} \ra \Bun_{T}$ for $\ol{\nu} \in B(T)$. One might hope that considering
\[ \mathcal{S}_{\phi_{T}} := \bigoplus_{\ol{\nu} \in B(T)} j_{\ol{\nu}!}(\chi) \] 
the sheaf on $\Bun_{T}$ whose restriction to each connected component is equal to $\chi$ gives rise to the desired Hecke eigensheaf. This is indeed the case. In particular, via the realization of local class field theory in the torsion of Lubin-Tate formal groups together with projection formula applied to the proper morphism $a_{(\nu_{i})_{i \in I}}$ we have the following proposition. 
\begin{proposition}{\cite[Section~9.2]{Fa},\cite{Zou}}{\label{torieigsheaf}}
The sheaf $\mathcal{S}_{\phi_{T}}$ is an eigensheaf with eigenvalue $\phi_{T}$. In particular, for all $(\nu_{i})_{i \in I} \in (\gamorb)^{I}$, with a choice of representatives $(\tilde{\nu}_{i})_{i \in I} \in (\cochar)^{I}$, we have an isomorphism
\[ T_{(\nu_{i})_{i \in I}}(\mathcal{S}_{\phi_{T}}) \simeq \boxtimes_{i \in I} \nu_{i} \circ \phi_{T} \otimes \mathcal{S}_{\phi_{T}} \]
of objects in $\Dlis(\Bun_{T},\Lambda)^{BW_{\mathbb{Q}_{p}}^{I}}$. More precisely, if $\tilde{\nu}_{i}$ is a representative of the $\Gamma$-orbit of $\nu_{i}$ for all $i \in I$, we have an isomorphism 
\[ (h_{(\tilde{\nu}_{i})_{i \in I}}^{\ra})^{*}(\mathcal{S}_{\phi_{T}}) \simeq \boxtimes_{i \in I} \tilde{\nu}_{i} \circ \phi_{T}|_{W_{E_{\nu_{i}}}} \otimes \mathcal{S}_{\phi_{T}}, \]
in $\Dlis(\Bun_{T},\Lambda)^{BW_{E_{(\nu_{i})_{i \in I}}}}$ which, after applying $q_{(\nu_{i})_{i \in I}*}$ (equivalently inducing the Weil group action along the inclusion of subgroups $W_{E_{(\nu_{i})_{i \in I}}} \subset W_{\bb{Q}_{p}}^{I}$), gives rise to the previous identification of induced representations. Here $W_{E_{(\nu_{i})_{i \in I}}} := \prod_{i \in I} W_{E_{\nu_{i}}}$ and $q_{(\nu_{i})_{i \in I}}: \Div^{I}_{E_{(\nu_{i})_{i \in I}}} \ra \Div^{I}$ is as in (\ref{eqn: FiniteEtaleCover})
\end{proposition}
\begin{remark}{\label{rem: HeckeOperatoronReflexField}}
We set $W_{E_{(\nu_{i})_{i \in I}}} := \prod_{i \in I} W_{E_{\nu_{i}}}$. For choices of representative $\tilde{\nu}_{i}$ of the Galois orbits $\nu_{i}$, we will write $T_{(\tilde{\nu}_{i})_{i \in I}}: \Dlis(\Bun_{T}) \ra \Dlis(\Bun_{T})^{BW_{E_{(\nu_{i})_{i \in I}}}}$ for the operator induced by the $*$-pullback $h^{\ra*}_{(\tilde{\nu}_{i})_{i \in I}}$, where $h^{\ra}_{(\tilde{\nu}_{i})_{i \in I}}$ is as in (\ref{eqn: ReflexFieldHeckeOperator}). 
\end{remark}
A special role will be played by the eigensheaf attached to the parameter
\[ 2\hat{\rho}\circ |\cdot|^{1/2}: W_{\mathbb{Q}_{p}} \ra \phantom{}^{L}T(\Lambda) \]
where $2\hat{\rho}$ denotes the cocharacter of $\hat{T}$ defined by the sum of all positive roots with respect to the choice of Borel and $|\cdot|: W_{\mathbb{Q}_{p}} \ra \Lambda^{*}$ is the norm character. We note that the value of this sheaf on each connected component is given by the character $\delta_{B}^{1/2}$, where $\delta_{B}$ denotes the modulus character defined by $B$. This leads to the following definition.
\begin{definition}{\label{modulussheaf}} 
We let $\Delta_{B}^{1/2}$ be the eigensheaf on $\Bun_{T}$ attached to the parameter
\[ 2\hat{\rho}\circ |\cdot|^{1/2}: W_{\mathbb{Q}_{p}} \ra \phantom{}^{L}T(\Lambda) \]
via Proposition \ref{torieigsheaf}. Similarly, we write $\Delta_{B}$ for the eigensheaf attached to $2\hat{\rho}\circ |\cdot|$, where the stalks of this sheaf are given by $\delta_{B}$.
\end{definition}
The key point is that (up to shifts) the pullback of this eigensheaf to the moduli stack $\Bun_{B}$ gives rise to a sheaf which we will denote by $\IC_{\Bun_{B}}$. We will see later that this sheaf is Verdier self-dual on $\Bun_{B}$ and therefore tensoring by it will give rise to the morally correct definition of the Eisenstein functor. We note that, given a parameter $\phi_{T}$ with associated eigensheaf $\mathcal{S}_{\phi_{T}}$, the tensor product $\mathcal{S}_{\phi_{T}} \otimes \Delta_{B}^{1/2}$ will be the eigensheaf attached to the tensor product $\phi_{T} \otimes \hat{\rho} \circ |\cdot|$ of $L$-parameters (the $L$-parameter whose value is equal to the product of the parameters). It therefore follows from Proposition \ref{torieigsheaf} that the following is true.
\begin{corollary}{\label{twistedeigsheaf}}
For all $(\nu_{i})_{i \in I} \in (\gamorb)^{I}$, we have an isomorphism
\[ T_{(\nu_{i})_{i \in I}}(\mathcal{S}_{\phi_{T}} \otimes \Delta_{B}^{1/2}) \simeq  \boxtimes_{i \in I} (\nu_{i} \circ \phi_{T})(\langle \hat{\rho}, \nu_{i\Gamma} \rangle) \otimes (\mathcal{S}_{\phi_{T}} \otimes \Delta_{B}^{1/2}) \]
of objects in $\Dlis(\Bun_{T},\Lambda)^{BW^{I}_{\mathbb{Q}_{p}}}$. Moreover, for a choice of representative $\tilde{\nu}_{i} \in \cochar$ of $\nu_{i}$, we have an isomorphism 
\[ T_{(\tilde{\nu}_{i})_{i \in I}}(\mathcal{S}_{\phi_{T}} \otimes \Delta_{B}^{1/2}) \simeq  \boxtimes_{i \in I} (\tilde{\nu}_{i} \circ \phi_{T}|_{W_{E_{\nu_{i}}}})(\langle \hat{\rho}, \tilde{\nu}_{i} \rangle) \otimes (\mathcal{S}_{\phi_{T}} \otimes \Delta_{B}^{1/2}), \]
of objects in $\Dlis(\Bun_{T},\Lambda)^{BW_{E_{(\nu_{i})_{i \in I}}}}$.
\end{corollary}
Now we discuss the various conditions that we will impose on our parameter $\phi_{T}$, as well as discuss their relationship with the irreducibility of principal series through various examples. 
\subsection{Genericity, Regularity, and the Irreducibility of Principal Series}
A natural condition that will appear in our study of Eisenstein series is genericity and regularity. Genericity will impose a very rigid constraint on the stalks of our Eisenstein functors and regularity will allow us to determine precisely what they are. To explain the precise conditions, we recall that given an element $\nu \in \gamorb$ with reflex $E_{\nu}$ and a representation $\tau$ of $Q_{\nu} := \Gal(E/E_{\nu})$, we can attach to it a representation of $\phantom{}^{L}T$ by inducing the representation of $\hat{T}$ defined by a representative of the orbit $\nu$ in $\mathbb{X}_*(T_{\ol{\mathbb{Q}}_{p}})$ and tensoring by the representation $\tau$ viewed as a representation of $\hat{T} \rtimes Q_{\mu}$ via the projection map $\hat{T} \rtimes Q_{\nu} \ra Q_{\nu}$. Here we recall that $E$ is the splitting field of $G$. We denote this representation by $\nu_{\tau}$ and if $\tau$ is trivial simply by $\nu$. For all $\alpha \in \gamorb$ defined by the $\Gamma$-orbits of coroots in $\mathbb{X}_{*}(T_{\ol{\mathbb{Q}}_{p}})$, there is a natural $\tau$ defined by the looking at the action of $W_{E_{\alpha}}$ on the root space $\hat{\mf{g}}_{\alpha}$ of $\hat{G}$ and we denote the resulting representation of $\phantom{}^{L}T$ by $\alpha_{\Ad}$.  Now we have the following.
\begin{definition}{\label{normregcond}}
For any $\Lambda/\bb{Z}_{\ell}$ and a parameter $\phi_{T}: W_{\mathbb{Q}_{p}} \ra \phantom{}^{L}T(\Lambda)$, we impose the following conditions on $\phi_{T}$. 
\begin{enumerate}
    \item For all $\Gamma$-orbits $\alpha \in \gamorb$ of coroots in $\mathbb{X}_{*}(T_{\ol{\mathbb{Q}}_{p}})$, the Galois cohomology complex $R\Gamma(W_{\mathbb{Q}_{p}},\alpha_{\Ad} \circ \phi_{T})$ is trivial. 
    \item If $\chi$ is the character attached to $\phi_{T}$ under local class field theory then we require that $\chi$ is regular. I.e. for all $w \in W_{G}$ non-trivial, we have that
    \[ \chi \not\simeq \chi^{w}.  \]
\end{enumerate}
If $\phi_{T}$ satisfies (1) then we say $\phi_{T}$ is generic with respect to $G$. If $\phi_{T}$ satisfies (1)-(2) then we say that $\phi_{T}$ is generic regular with respect to $G$. We omit $G$ from the terminology when it is clear from context.
\end{definition}
These conditions are related to the irreducibility of the normalized induction $i_{B}^{G}(\chi)$. To explain this, let's translate all these conditions to the character $\chi$. Given $\alpha \in \gamorb$ a $\Gamma$-orbit of coroots with reflex field $E_{\alpha}$ and representative $\tilde{\alpha}$ viewed as a character of $\hat{T}$, local class field theory \cite{Lang} gives us a map
\[ E_{\alpha}^{*} \ra T(E_{\alpha}) \]
attached to $\tilde{\alpha} \circ \phi_{T}|_{W_{E_{\alpha}}}$. If we post-compose with the norm map $\mathrm{Nm}_{E_{\alpha}/\mathbb{Q}_{p}}$ then we get a map
\[ E_{\alpha}^{*} \ra T(\mathbb{Q}_{p}). \]
 We write $\chi^{0}_{\tilde{\alpha}}: E_{\alpha}^{*} \ra \Lambda^{*}$ for this map composed with $\chi$. Let $\varepsilon_{\widetilde{\alpha}}^{G}: W_{E_{\alpha}} \ra \Lambda^{*}$ be the character defined by the natural action of $W_{E_{\alpha}}$ on the root space $\hat{\mf{g}}_{\tilde{\alpha}}$, and set $\varepsilon_{\widetilde{\alpha}}\colon E_{\alpha}^{*} \rightarrow \Lambda^{\times}$ to be the character attached to it by class field theory. We set $\chi_{\tilde{\alpha}} := \chi_{\tilde{\alpha}}^{0} \otimes \varepsilon_{\tilde{\alpha}}$ It follows by Shapiro's lemma that we have an isomorphism:
\[ R\Gamma(W_{E_{\alpha}},\tilde{\alpha} \circ \phi_{T} \otimes \varepsilon_{\widetilde{\alpha}}^{G}|_{W_{E_{\alpha}}}) \simeq R\Gamma(W_{\mathbb{Q}_{p}},\alpha_{\Ad} \circ \phi_{T}). \]
For field coefficients $\Lambda$, by local Tate-duality and the fact that the Euler-Poincar\'e characteristic of this complex is $0$, we note that the vanishing of $R\Gamma(W_{E_{\alpha}},\tilde{\alpha} \circ \phi_{T} \otimes \varepsilon_{\widetilde{\alpha}}^{G}|_{W_{E_{\alpha}}})$ is the same as saying that $\tilde{\alpha} \circ \phi_{T} \otimes \varepsilon_{\widetilde{\alpha}}^{G}|_{W_{E_{\alpha}}}$ is not isomorphic to the trivial or cyclotomic character. Using this, we can see that the above conditions on $\phi_{T}$ are equivalent to the following conditions on $\chi$ when $\Lambda$ is a field.
\begin{definition}{\label{def: charnormreg}}
For any field $\Lambda/\bb{Z}_{\ell}$ and any smooth character $\chi: T(\mathbb{Q}_{p}) \ra \Lambda^{*}$ consider the following conditions on $\chi$. 
\begin{enumerate}
    \item For all $\Gamma$-orbits $\alpha \in \gamorb$ of positive coroots in $\mathbb{X}_{*}(T_{\ol{\mathbb{Q}}_{p}})$, the character $\chi_{\tilde{\alpha}}$ defined above is not isomorphic to the trivial representation $\mathbf{1}$ or $|\cdot|^{\pm 1}_{E_{\alpha}}$, where $|\cdot|_{E_{\alpha}}$ is the norm character on the reflex $E_{\alpha}$ of $\alpha$ and $\tilde{\alpha} \in \cochar$ is a choice of representative of $\alpha \in \gamorb$.
    \item  We have that $\chi$ is regular. I.e. for all $w \in W_{G}$ non-trivial, we have that 
    \[ \chi \not\simeq \chi^{w}. \]
\end{enumerate}
We say that $\chi$ is generic if (1) holds, and that it is generic regular if (1)-(2) hold. 
\end{definition}
We now illustrate how this condition is related to irreducibility of $i_{B}^{G}(\chi)$ in some examples. We will assume that $\Lambda = \ol{\mathbb{Q}}_{\ell}$ in all of the examples for simplicity. 
\begin{Example}{\label{Gl2irred}}{($G = \GL_{2}$)}
We can write $\chi := \chi_{1} \otimes \chi_{2}$ for $\chi_{i}: \mathbb{Q}_{p}^{*} \ra \ol{\mathbb{Q}}_{\ell}^{*}$ smooth characters and $i = 1,2$. We see that $\chi$ being generic implies that $\chi_{1}\chi_{2}^{-1} \not\simeq \mathbf{1}$ and $\chi_{1}\chi_{2}^{-1} \not\simeq |\cdot|^{\pm 1}$, and this latter condition guarantees that the normalized parabolic induction $i_{B}^{\GL_{2}}(\chi_{1} \otimes \chi_{2})$ is irreducible. Let's also look at Condition (2) in this case. Suppose it fails, then we have an isomorphism:
\[ \chi_{1}(t_{1})\chi_{2}(t_{2}) \simeq \chi_{1}(t_{2})\chi_{2}(t_{1}) \]
Evaluating at $(t_{1},t_{2}) = (t,1)$ this would imply that $\chi_{1}\chi_{2}^{-1} \simeq \mathbf{1}$, which would contradict $\chi$ being generic.  The same is true for $\GL_{n}$ with completely analogous proof.
\end{Example}
In particular, we see that $\chi$ being generic is enough to guarantee irreducibility and regularity for $\GL_{n}$.

In general, genericity does not always imply regularity. In particular, Condition \ref{normregcond} (2) seems to be related to the irreducibility of some unitary principal series representations.
\begin{Example}{\label{Sl2ex}}($G = \mathrm{SL}_{2}$)
In this case, $\chi: \mathbb{Q}_{p}^{*} \ra \ol{\mathbb{Q}}_{\ell}^{*}$ is a character of $\mathbb{Q}_{p}^{*}$. The induction $i_{B}^{G}(\chi)$ will be irreducible if and only if $\chi \not\simeq |\cdot|^{\pm 1}$ and $\chi$ is not a nontrivial quadratic character. The condition that $\chi \not\simeq |\cdot|^{\pm 1}$ is guaranteed by $\chi$ being generic but the condition $\chi^{2} \not\simeq \mathbf{1}$ is not. However, we note that $\chi$ being regular enforces the additional condition that $\chi^{2} \not\simeq \mathbf{1}$ guaranteeing irreducibility in this case.
\end{Example}
We will need the irreducibility of generic regular
principal series in our study of the stalks of
geometric Eisenstein series. We establish this independently in Section~\ref{s: SecondAdjointness}. The essential characteristic-zero input is the
following classical comparison, proven in Appendix~\ref{append: irred of principal series}.
\begin{proposition}{(Proposition \ref{prop: genericinterisom})}
If $\chi: T(\mathbb{Q}_{p}) \ra \ol{\mathbb{Q}}_{\ell}^{*}$ is a generic character then, for all $w \in W_{G}$, we have an isomorphism $i_{B}^{G}(\chi) \simeq i_{B}^{G}(\chi^{w})$. 
\end{proposition}
In fact, for regular characters $\chi$, the existence of such isomorphisms is equivalent to irreducibility. More generally, we show that such isomorphisms exist assuming $\chi$ is generic (Proposition \ref{prop: genericinterisom}), but this does not guarantee irreducibility of certain unitary principal series as seen when $G = \SL_{2}$. This only follows from generic regularity.

Using the Langlands classification, the proof of this proposition will essentially reduce to a calculation of the set of $\chi$ for which $i_{B}^{G}(\chi)$ is irreducible in rank $1$, where it reduces to Example \ref{Sl2ex} and the following example, which illustrates the behavior of the generic condition in the non-split case. 
\begin{Example}{\label{U3ex}}($G = \mathrm{U}_{3}/E$) Let $E/\mathbb{Q}_{p}$ be a quadratic extension. We write $\ol{(-)}$ for the non-trivial automorphism of $E$ over $\mathbb{Q}_{p}$. If $e_{1},e_{2},e_{3}$ is the standard basis for the cocharacter lattice $\mathbb{X}_{*}(T_{\ol{\mathbb{Q}}_{p}})$ then $\ol{(-)}$ acts by
\[ e_{1} \longleftrightarrow -e_{3} \]
\[ e_{2} \longleftrightarrow -e_{2} \]
It follows that the simple coroot $\alpha_{1} := e_{1} - e_{2}$ is sent to the simple coroot $\alpha_{2} := e_{2} - e_{3}$ under $\ol{(-)}$. Thus, the $\Gamma$-orbits of positive coroots in $\mathbb{X}_{*}(T_{\ol{\mathbb{Q}}_{p}})$ are given by $\{\alpha_{1},\alpha_{2}\}$ with reflex field $E$ and $\alpha_{1} + \alpha_{2}$ with reflex field $\mathbb{Q}_{p}$. Now recall that the maximal torus $T(\mathbb{Q}_{p}) \subset \mathrm{U}_{3}(\mathbb{Q}_{p})$ is isomorphic to $E^{*} \times E^{1}$, via the embedding 
\[ E^{*} \times E^{1} \ra \mathrm{U}_{3}(\mathbb{Q}_{p}) \]
\[ (t,s) \mapsto \begin{pmatrix} t & 0 & 0 \\
                             0 & s & 0 \\
                             0 & 0 & \ol{t}^{-1} \end{pmatrix} \]
where $E^{1}$ denotes the set of norm $1$ elements. Then if we write the character $\chi(t,s): E^{*} \times E^{1} \ra \Lambda^{*}$ as $\chi(t,s) = \chi_{1}(t)\chi_{2}(ts\ol{t}^{-1})$ the reducibility of $i_{B}^{G}(\chi)$ depends solely on $\chi_{1}$, as in \cite[Page~173]{Rog}, where here it reduces to the analogous question for $\mathrm{SU}_{3}$, and there the reducibility points were studied in \cite[Section~7]{Keys1}. The induction $i_{B}^{G}(\chi)$ is reducible if and only if one of the following hold:
\begin{enumerate}
    \item $\chi_{1} = \eta|\cdot|_{E}^{\pm 1/2}$, where $\eta|_{\mathbb{Q}_{p}^{*}} = \eta_{E/\mathbb{Q}_{p}}$,
    \item $\chi_{1} = |\cdot|_{E}^{\pm 1}$,
    \item $\chi_{1}|_{\mathbb{Q}_{p}^{*}}$ is trivial, but $\chi_{1}$ is not.
\end{enumerate}
Here $|\cdot|_{E}$ is the norm character of $E$, and $\eta_{E/\mathbb{Q}_{p}}: \mathbb{Q}_{p}^{*} \ra \Lambda^{*}$ is the unique quadratic character with kernel given by $\mathrm{Nm}_{E/\mathbb{Q}_{p}}(E^{*})$. We note that the cocharacters $\chi_{\tilde{\alpha}}^{0}$ of $T(\mathbb{Q}_{p})$ defined by representatives of the $\Gamma$-orbits of positive roots are
\[ \{\alpha_{1},\alpha_{2}\}: E^{*} \ra T(\mathbb{Q}_{p}) = E^{*} \times E^{1} \]
\[ t \mapsto (t,t^{-1}\ol{t}) \]
and 
\[ \{\alpha_{1} + \alpha_{2}\}: \bb{Q}_{p}^{*} \ra E^{*} \times E^{1} \]
\[ t \mapsto (t,1) \]
Moreover, the modifying character $\varepsilon_{\tilde{\alpha}}$ is trivial in the first case, since the reflex field is the same as the splitting field, and it is given by the quadratic character $\eta_{E/\bb{Q}_{p}}$ in the second case since $\ol{(-)}$ sends $\alpha_{1} + \alpha_{2}$ to $\alpha_{1} + \alpha_{2}$. Similarly, we see that $W_{\bb{Q}_{p}}/W_{E}$ acts on the root space $\hat{\mf{g}}_{\alpha_{1} + \alpha_{2}}$ by a quadratic character. Using this, we see that $\chi$ being generic with respect to the Galois orbit $\{\alpha_{1},\alpha_{2}\}$ implies $\chi_{1} \not\simeq \mathbf{1}$ and $\chi_{1} \not\simeq |\cdot|^{\pm 1}_{E}$, which implies reducibility point $(2)$ cannot occur. Moreover, $\chi$ being generic with respect to the Galois orbit $\{\alpha_{1} + \alpha_{2}\}$ implies that $\chi_{1}|_{\bb{Q}_{p}^{*}} \eta_{E/\bb{Q}_{p}} \not\simeq \mathbf{1},|\cdot|_{\bb{Q}_{p}}^{\pm 1}$. This immediately excludes reducibility points (1), by noting that the restriction of $|\cdot|_{E}^{\pm \frac{1}{2}}$ to $\bb{Q}_{p}^{*}$ identifies with $|\cdot|^{\pm 1}$. Moreover, we can easily check that regularity excludes reducibility point (3).
\end{Example} 
The connection between genericity and irreducibility of non-unitary principal series fits in nicely with the general Langlands philosophy. In particular, since we are inducing from a Borel, we expect that a parameter should have monodromy if it arises as a constituent of the reducible induction of a non-unitary character. We saw in the above examples that this shouldn't occur when the parameter $\phi_{T}$ attached to $\chi$ is generic. We can analyze when a parameter comes from the semi-simplification of a parameter with non-trivial monodromy and relate this to genericity.
\begin{lemma}{\label{regmonodromy}}
We let $\phi: \mathcal{L}_{\mathbb{Q}_{p}} \ra \phantom{}^{L}G(\ol{\mathbb{Q}}_{\ell})$ be an $L$-parameter (not necessarily Frobenius semi-simple). Suppose that $\phi: \mathcal{L}_{\mathbb{Q}_{p}} \ra \phantom{}^{L}G(\ol{\mathbb{Q}}_{\ell})$ has non-trivial monodromy, and that the semi-simplification $\phi^{\mathrm{ss}}$ (see Assumption \ref{compatibility}) factors through $\phantom{}^{L}T \ra \phantom{}^{L}G$ via the natural embedding. If we write $\phi_{T}: W_{\mathbb{Q}_{p}} \ra \phantom{}^{L}T(\ol{\mathbb{Q}}_{\ell})$ for the parameter induced by $\phi^{\mathrm{ss}}$ then $\phi_{T}$ is not generic. 
\end{lemma}
\begin{proof}
If $\phi$ has non-trivial monodromy and the semi-simplification factors through $\phantom{}^{L}T$, there exists a lift 
\[ \begin{tikzcd}
&  \phantom{}^{L}B(\ol{\mathbb{Q}}_{\ell}) \arrow[d] & \\
W_{\mathbb{Q}_{p}} \arrow[ur,"\tilde{\phi}",dashed] \arrow[r,"\phi_{T}"] & \phantom{}^{L}T(\ol{\mathbb{Q}}_{\ell}) &  \\ 
\end{tikzcd}\]
of $\phi_{T}$, which is not given by the inclusion $\phantom{}^{L}T(\ol{\mathbb{Q}}_{\ell}) \ra \phantom{}^{L}B(\ol{\mathbb{Q}}_{\ell})$. The existence of such an extension implies that there exists a non-trivial class in $R\Gamma(W_{\mathbb{Q}_{p}},\alpha_{\Ad} \circ \phi_{T})$, which would in turn imply $\phi_{T}$ is not generic. 
\end{proof}
We finish this section by describing the Hecke eigensheaf property on symmetric powers of the curve for the sheaf $\mathcal{S}_{\phi_{T}}$ attached to a parameter $\phi_{T}$.
\subsection{The Symmetrized Hecke Eigensheaf Property}{\label{sec: geomconsofweakgen}}
We let $\ol{\nu}$ be an element of $\Lambda_{G,B}^{pos} \setminus \{0\}$. We can write this as a linear combination $\sum_{i \in \mathcal{J}} n_{i}\alpha_{i}$ for non-negative integers $n_{i}$, where the $\alpha_{i}$ correspond to the Galois orbits of simple absolute roots as in Definition \ref{def: gammaorbsimp}. Given such an $\alpha_{i}$, we can consider the reflex field $E_{i}$ of the associated Galois orbit, and define the following partially symmetrized curve:
\[ \Div^{(\ol{\nu})} := \prod_{i \in \mathcal{J}} \Div_{E_{i}}^{(n_{i})}. \]
Points of this curve correspond to tuples of Cartier divisors $D_{i}$ over $E_{i}$ of degree $n_{i}$ for all $i \in \mathcal{J}$. We can consider the map
\[ h_{(\ol{\nu})}^{\ra}: \Bun_{T} \times \Div^{(\ol{\nu})} \ra \Bun_{T} \]
given by sending $(\mathcal{F}_{T},(D_{i})_{i \in \mathcal{J}})$ to $\mathcal{F}_{T}(\sum_{i \in \mathcal{J}} -\alpha_{i} \cdot D_{i})$, where we are identifying $\alpha_{i}$ with its corresponding $\Gamma$-orbit.

This partially symmetrized mirror curve $\Div^{(\ol{\nu})}$ behaves a bit strangely if $G$ is not split. To illustrate this, consider the following example. 
\begin{Example}
Let $G = \U_{3}$ be a unitary group in $3$ variables attached to a quadratic extension $E/\mathbb{Q}_{p}$ and write $\alpha_{1}$ and $\alpha_{2}$ for the two absolute positive simple roots. We recall, as in Example \ref{U3ex}, that the Galois group exchanges $\alpha_{1}$ and $\alpha_{2}$. Therefore, they both map to a unique element $\alpha \in \coinv$ which generates the monoid $\Lambda_{G,B}^{pos}$ by Definition. Consider the element $\ol{\nu} := 2\alpha \in \coinv$. We note that we have an equality $\Div^{(\ol{\nu})} = \Div^{(2)}_{E}$ in this case. The pre-image of $2\alpha$ under the natural map $(-)_{\Gamma}: \gamorb \ra \coinv$ consists of two elements after restricting to the monoid of positive coroots: the $\Gamma$-orbit $\{2\alpha_{1},2\alpha_{2}\}$ with reflex field $E$ and the $\Gamma$-orbit of $\{\alpha_{1} + \alpha_{2}\}$ with reflex field $\mathbb{Q}_{p}$. We saw in the previous section that the space of modifications defined by the $\Gamma$-orbit $\{2\alpha_{1},2\alpha_{2}\}$ is given by $\Bun_{T} \times \Div^{1}_{E}$, correspondingly we have a natural map
\[ \triangle_{\{2\alpha_{1},2\alpha_{2}\}}: \Div^{1}_{E} \xrightarrow{\triangle} \Div^{2}_{E} \ra \Div^{(2)}_{E} \]
given by the diagonal embedding composed with the quotient map. It is easy to check we have an equality $h_{\{2\alpha_{1},2\alpha_{2}\}}^{\ra}(-) = h_{\ol{\nu}}^{\ra} \circ (\mathrm{id} \times \triangle_{\{2\alpha_{1},2\alpha_{2}\}})$. Perhaps more interestingly, attached to the $\Gamma$-orbit $\{\alpha_{1} + \alpha_{2}\}$, we have a twisted diagonal map
\[ \triangle_{\{\alpha_{1} + \alpha_{2}\}}: \Div^{1}_{\mathbb{Q}_{p}} \ra \Div^{(2)}_{E} \]
given by sending a Cartier divisor $D$ to its pre-image under the natural finite-\'etale covering $X_{S,E} \ra X_{S}$ of Fargues-Fontaine curves induced by the extension $E/\mathbb{Q}_{p}$. By \cite[Proposition~1.5]{LiHue}, this map defines a closed embedding whose image lies in the complement of the image of $\triangle_{\{2\alpha_{1},2\alpha_{2}\}}$ in $\Div^{(2)}_{E}$, and we similarly see that we have a relationship $h_{(\ol{\nu})}^{\ra} \circ (\mathrm{id} \times \triangle_{\{\alpha_{1} + \alpha_{2}\}})(-) = h_{\{\alpha_{1} + \alpha_{2}\}}^{\ra}(-)$. 
\end{Example}
The previous example illustrates that we can understand $\Bun_{T} \times \Div^{(\ol{\nu})}$ as realizing the set of all modifications specified by Galois orbits $\nu \in \gamorb$ such that $\nu_{\Gamma} = \ol{\nu}$. This is indeed a general phenomenon. In particular, using that $\Gamma$ permutes the simple absolute coroots of $G$, given any such $\nu$ with reflex field $E_{\nu}$, we can define a twisted diagonal embedding 
\[ \Delta_{\nu}: \Div^{1}_{E_{\nu}} \ra \Div^{(\ol{\nu})} \]
such that we have a relationship $h_{\nu}^{\ra}(-) := h_{\ol{\nu}}^{\ra} \circ (\mathrm{id} \times \triangle_{\nu})(-)$. 

If $\ol{\nu} = \alpha_{i}$ for $i \in \mathcal{J}$ the map $\Delta_{\nu}$ is an isomorphism for the unique $\Gamma$-orbit of simple coroots corresponding to $\alpha_{i}$. Therefore, the pullback of $\mathcal{S}_{\phi_{T}}$ along $h_{(\ol{\nu})}^{\ra}$ is isomorphic to $\mathcal{S}_{\phi_{T}} \boxtimes \tilde{\alpha}_{i} \circ \phi_{T}$ for a choice of representative $\tilde{\alpha}_{i}$ of the $\Gamma$-orbit corresponding to $\alpha_{i}$. In general, recall that given a local system $\mathbb{L}$ and $n$ a positive integer, we can consider the symmetric powers
\[ \mathbb{L}^{(n)} := \pi_{*}(\boxtimes_{i = 1}^{n} \mathbb{L})^{S_{n}} \]
where $\pi$ denotes the push-forward along the quotient map by $S_{n}$:
\[ \pi : (\Div^1)^{n} \ra \Div^{(n)} \]
Using this, we can define a local system on $\Div^{(\ol{\nu})}$ given by
\begin{equation}{\label{eqn: symmetrizedlocalsystem}}
 E_{\phi_{T}}^{(\ol{\nu})} := \boxtimes_{i \in \mathcal{J}} E_{\phi_{i}}^{(n_{i})}   
\end{equation}
where $\phi_{i}$ is the local system on $\Div^{1}_{E_{i}}$ corresponding to the character $\tilde{\alpha}_{i} \circ \phi_{T}|_{W_{E_{i}}}$ of $W_{E_{i}}$ for $E_{i}$ the reflex field of the Galois orbit corresponding to $\alpha_{i}$. With this in hand, we can describe the pullback of $\mathcal{S}_{\phi_{T}}$ along $h_{(\ol{\nu})}^{\ra}$ as the sheaf $E_{\phi_{T}}^{(\ol{\nu})} \boxtimes \mathcal{S}_{\phi_{T}}$ by using that Hecke operators are monoidal, and the natural compatibilities of the eigensheaf. 

We will now review the next ingredient in our calculations of geometric Eisenstein series, the Geometric Satake correspondence. 
\section{Geometric Satake and the Affine Grassmannian}{\label{geomsatakesection}}
As we will now consider some more sophisticated sheaf theory for simplicity from now on stick to the case where $\Lambda/\bb{Z}_{\ell}[\sqrt{p}]$ is a torsion coefficient system as in \ref{assumption: coefficients} (1) until \S 9 unless otherwise stated. We recall that in this case will denote the derived category of \'etale $\Lambda$-sheaves on a $v$-stack or diamond $Z$ by simply writing $\D(Z)$ and we recall that we have a natural isomorphism $\D(Z) \simeq \Dlis(Z,\Lambda)$ in this case for any reasonable Artin $v$-stack (in particular all objects we consider) by \cite[Proposition~VII.6.6]{FS}. 
\subsection{The Geometric Satake Correspondence}
We will now recall some facts about the geometric Satake correspondence for the $B_{dR}^{+}$ Grassmannian, as proven in \cite[Chapter~VI]{FS}. For any finite set $I$, we consider the local Hecke stack 
\[ \pi_{G}: \Hckloc_{G}^{I} \ra \Div^{I} \]
For a point $S \ra \Div^{I}$ with $S \in \Perf_{\ol{\bb{F}}_{p}}$, we can consider the global sections of the completion of the structure sheaf $\mathcal{O}_{X_{S}}$ at the union of the $I$ Cartier divisors in $X_{S}$ defined by $S$. This defines a ring which we denote by $B^{+}$, and inverting $D$, we get a ring which we denote by $B$. The mapping sending $S \in \Div^{I}$ to $G(B^{+})$ and $G(B)$ defines \'etale sheaves on $\Perf_{\ol{\bb{F}}_{p}}$, which we denote by $L^{+}_{\Div^{I}}G$ and $L_{\Div^{I}}G$, respectively. We note that, for $I = \{\ast\}$ and $S = \Spa(F,\mathcal{O}_{F}) \ra \Div^{I}$ a geometric point with associated untilt $(C,\mathcal{O}_{C})$, we have $B^{+} = B_{dR}^{+}(C,\mathcal{O}_{C})$ and $B = B_{dR}(C,\mathcal{O}_{C})$, the usual de-Rham period rings attached to the untilt. For simplicity, we will often just drop the subscript $\Div^{I}$ and just write $L^{+}G$ and $LG$ for these \'etale sheaves. By \cite[Proposition~19.1.2]{SW}, the Hecke stack can be described as the quotient:
\[ [L^{+}G\backslash LG/L^{+}G] \ra \Div^{I} \]
In other words, for $S \in \Perf_{\ol{\bb{F}}_{p}}$ mapping to $\Div^{I}$, $\Hckloc_{G}^{I}$ will parametrize a pair of $G$-bundles over the formal completion of the tuple of divisors $D_{S,i}$ defined by the map $S \ra \Div^{I}$ together with a trivialization away from the $D_{S,i}$. It follows that this has a map from the global Hecke stack considered in \S \ref{GCFT}, by restricting to formal completions. 

To study this, we can uniformize this by the quotient 
\[ \Gr_{G}^{I} := LG/L^{+}G \ra \Div^{I} \]
which is the Fargues-Fontaine analogue of the Beilinson-Drinfeld Grassmannian. Using Beauville-Laszlo \cite[Lemma~5.29]{SW}, this can be interpreted as parameterizing modifications $\mathcal{F}_{G} \dashrightarrow \mathcal{F}_{G}^{0}$ over $X_{S}$ away from the tuple of Cartier divisors defined by the projection to $\Div^{I}$. It follows by the results of \cite[Lecture~XX]{SW} that the Beilinson-Drinfeld Grassmannian is a well-behaved geometric object; it can be written as a closed union of subsheaves that are proper and representable in spatial diamonds over $\Div^{I}$, given by bounding the meromorphy of the modifications. We can consider the category
\[ \D(\Hckloc_{G}^{I})^{bd}  \]
of $\Lambda$-\'etale sheaves on $\Hckloc_{G}^{I}$ with support in one of the aforementioned subsets quasi-compact over $\Div^{I}$. This carries a monoidal structure coming from the diagram:
\[ \Hckloc_{G}^{I} \times_{\Div^{I}} \Hckloc_{G}^{I} \xleftarrow{(p_{1},p_{2})} L^{+}G\backslash LG \times^{L^{+}G} LG/L^{+}G \xrightarrow{m} \Hckloc_{G}^{I} \]
Here the middle space can be interpreted as parameterizing a pair of $G$-bundles $\mathcal{E}_{1},\mathcal{E}_{2}$ together with a pair of modifications $\beta_{1}: \mathcal{E}_{1} \dashrightarrow \mathcal{E}_{0}$ and $\beta_{2}: \mathcal{E}_{0} \dashrightarrow \mathcal{E}_{2}$ to/from the trivial bundle with the same locus of meromorphy. The maps $p_{i}$ are the natural projections remembering only the data $(\mathcal{E}_{i},\beta_{i})$ for $i = 1,2$, and $m$ is defined by sending $(\mathcal{E}_{i},\beta_{i})_{i = 1,2}$ to $\beta_{2} \circ \beta_{1}: \mathcal{E}_{1} \dashrightarrow \mathcal{E}_{2}$. Given $A,B \in \D(\Hckloc_{G}^{I})^{bd}$, we define 
\[ A \star B := Rm_{*}(p_{1}^{*}(A) \otimes p_{2}^{*}(B)) \in \D(\Hckloc_{G}^{I})^{bd}\]
the convolution of $A$ and $B$ \cite[Section~VI.8]{FS}. Since the map $m$ is a fibration in $\Gr_{G}^{I}$ it is proper after restricting to the support of any $A, B \in \D(\Hckloc_{G}^{I})^{bd}$, so by proper base-change it gives a well-defined associative monoidal structure on $\D(\Hckloc_{G}^{I})^{bd}$. One can further refine our category of sheaves as follows. In particular, we note that the locus of $\Hckloc_{G}^{I}$ where the meromorphy is equal to $\mu_{i} \in \mathbb{X}_{*}(T_{\ol{\mathbb{Q}}_{p}})^{+}$ at the $i$th Cartier divisor for $i \in I$ is uniformized by a $\ell$-cohomologically smooth diamond of relative dimension $\sum_{i \in I} \langle 2\hat{\rho}, \mu_{i} \rangle$ over $\Div^{I}$ by \cite[Proposition~VI.2.4]{FS}. One can check that this gives rise to a well-defined perverse $t$-structure on $\D(\Hckloc_{G}^{I})^{bd}$ over $\Div^{I}$ given by insisting that $!$-restriction (resp. $*$-restriction) to these strata sit in cohomological degrees $\geq$ (resp. $\leq$) $-\sum_{i \in I} \langle 2\hat{\rho}, \mu_{i} \rangle$, and that convolution preserves perversity \cite[Proposition~VI.8.1]{FS}. With this in hand, we arrive at the key definition.
\begin{definition}{\cite[Definition~I.6.2]{FS}}{\label{defsatakecat}}
We define the Satake category 
\[ \Sat_{G}^{I} \subset \D(\Hckloc_{G}^{I})^{bd} \]
as the category of all $A \in \D(\Hckloc_{G}^{I})^{bd}$ which are flat perverse (i.e for all $\Lambda$-modules $M$ $A \otimes M$ is also perverse) and ULA over $\Div^{I}$, as defined in \cite[Chapter~V.7]{FS}. 
\end{definition}
The ULA and flatness property in the above definition has the key consequence that the pullback of $A \in \Sat_{G}^{I}$ to $\Gr_{G}^{I}$ composed with the push-forward to $\Div^{I}$ has cohomology sheaves valued in $\Lambda$-valued local systems on $\Div^{I}$. In particular, using an analogue of Drinfeld's lemma \cite[Proposition~VI.9.2]{FS}, this gives rise to a fiber functor
\[ F_{G}^{I}: \Sat_{G}^{I} \ra \Rep_{W_{\mathbb{Q}_{p}}^{I}}(\Lambda) \]
\[ A \mapsto \bigoplus_{i} \mathcal{H}^{i}(R\pi_{G*}(A)) \]
where $\Rep_{W_{\mathbb{Q}_{p}}^{I}}(\Lambda)$ denotes the category of continuous representations of $W_{\mathbb{Q}_{p}}^{I}$ on finite projective $\Lambda$-modules, and $R\pi_{G*}$ is the functor given by pulling back to $\Gr_{G}^{I}$ and taking the push-forward to $\Div^{I}$, as in \cite[Definition/Proposition~VI.7.10]{FS}. Now, by using the factorization structure on these Grassmannians, one can also construct an analogue of the fusion product  \cite[Section~VI.9]{FS}. Let us recall its construction. We consider a partition $I = I_{1} \sqcup \ldots \sqcup I_{k}$ of this index set. We consider the open subspace
\[ \Div^{I;I_{1},\ldots,I_{k}} \subset \Div^{I} \]
parameterizing a tuple of Cartier divisors $(D_{i})_{i \in I}$ such that $D_{i}$ and $D_{i'}$ are disjoint whenever $i,i' \in I = I_{1} \sqcup \ldots \sqcup I_{k}$ lie in different $I_{j}$s, and let $\Hckloc_{G}^{I;I_{1},\ldots,I_{k}}$ be the base-change of $\Hckloc_{G}^{I}$ to this open subspace. We can define the subcategory $\Sat_{G}^{I;I_{1},\ldots,I_{k}} \subset \D(\Hckloc_{G}^{I;I_{1},\ldots,I_{k}})^{\mathrm{bd}}$ in an analogous manner to Definition \ref{defsatakecat}. We have a natural restriction map 
\[ \Sat_{G}^{I} \ra \Sat_{G}^{I;I_{1},\ldots,I_{k}} \]
which, by \cite[Proposition~VI.9.3]{FS}, defines a fully faithful embedding. There is also an identification 
\[ \Hckloc^{I}_{G} \times_{\Div^{I}} \Div^{I;I_{1},\ldots,I_{k}} \simeq \prod_{j = 1}^{k} \Hckloc_{G}^{I_{j}} \times_{\prod_{j} \Div^{I_{j}}} \Div^{I;I_{1},\ldots,I_{k}} \]
giving a natural map 
\[ \Sat_{G}^{I_{1}} \times \ldots \times \Sat_{G}^{I_{k}} \ra \Sat_{G}^{I;I_{1},\ldots,I_{k}} \]
via taking exterior products. Then, by \cite[Definition/Proposition~VI.9.4]{FS}, this lies in the full subcategory $\Sat_{G}^{I}$, and the resulting map 
\[ \Sat_{G}^{I_{1}} \times \ldots \times \Sat_{G}^{I_{k}} \ra \Sat_{G}^{I} \]
is called the fusion product. Now consider the composite 
\[ \Sat_{G}^{I} \times \Sat_{G}^{I} \ra \Sat_{G}^{I \sqcup I} \ra \Sat_{G}^{I} \]
where the first map is the fusion product, and the last map is given by restricting along the diagonal embedding $\Hckloc^{I} \ra \Hckloc^{I \sqcup I}$. Then this is naturally isomorphic to the convolution product. By this comparison between fusion and convolution, one can deduce that the convolution product is in fact symmetric monoidal, and that the functor $F_{G}^{I}$ takes this monoidal structure to the usual tensor product on $\Rep_{W_{\mathbb{Q}_{p}}^{I}}(\Lambda)$. Now, using Tannakian duality, one deduces the following.
\begin{theorem}{\cite[Theorem~I.6.3]{FS}}{\label{geomsatake}}
For a finite index set $I$, the category $\Sat_{G}^{I}$ is naturally in $I$ identified with $\Rep_{\Lambda}((\hat{G} \rtimes W_{\bb{Q}_{p}})^{I})$ the category of algebraic representations of  $(\hat{G} \rtimes W_{\bb{Q}_{p}})^{I}$ on finite projective $\Lambda$-modules.
\end{theorem}
\begin{remark}
One needs to be a bit careful here. In particular, $\phantom{}^{L}G$ as defined here differs from the usual definition of the $L$-group; namely, the usual action of $W_{\mathbb{Q}_{p}}$ on $\hat{G}$ is twisted by a cyclotomic character (see \cite[Section~VI.11]{FS} for more details). We also note that the category $\Rep_{\Lambda}(\phantom{}^{L}G^{I})$ differs from $\Rep_{\Lambda}((\hat{G} \rtimes W_{\bb{Q}_{p}})^{I})$  as in the former we only allow for representations of $\hat{G} \rtimes Q$, where $Q := \Gal(E/\bb{Q}_{p})$ is the Galois group of the splitting field. In particular, we have an inclusion $\Rep_{\Lambda}(\phantom{}^{L}G^{I}) \subset \Rep_{\Lambda}((\hat{G} \rtimes W_{\bb{Q}_{p}})^{I})$, but this is not an equivalence of categories (even if $G$ is split).
\end{remark}
One of the key points that also plays an important role in the proof of this theorem is that this construction respects the natural map 
\[ \mathrm{res}_{T}^{I,G}:  \Rep_{\Lambda}((\hat{G} \rtimes W_{\bb{Q}_{p}})^{I}) \ra \Rep_{\Lambda}((\hat{T} \rtimes W_{\bb{Q}_{p}})^{I})  \]
given by restricting a representation of $(\hat{G} \rtimes W_{\bb{Q}_{p}})^{I}$ to $(\hat{T} \rtimes W_{\bb{Q}_{p}})^{I}$ along the natural embedding: 
\[ (\hat{T} \rtimes W_{\bb{Q}_{p}})^{I} \ra (\hat{G} \rtimes W_{\bb{Q}_{p}})^{I} \]
To explain this, we need to explain how this functor is realized in the geometry of Beilinson-Drinfeld Grassmannians. First note that, as in \S \ref{GCFT}, we have a decomposition
\[ \Gr_{T}^{I} =  \bigcup_{(\nu_{i})_{i \in I} \in (\gamorb)^{I}} \Gr_{T,(\nu_{i})_{i \in I}}^{I} \]
into closed subspaces, where $\Gr_{T,(\nu_{i})_{i \in I}}$ parametrizes modifications of $T$-bundles with meromorphy equal to the Galois orbit defined by $\nu_{i}$ at a Cartier divisor $D_{i}$ for all $i \in I$ (in the sense that its meromorphy is given by the set of representatives of $\nu_{i} \in \gamorb$ in $\cochar$), and $\Div^{I}_{E_{(\nu_{i})_{i \in I}}} := \prod_{i \in I} \Div^{1}_{E_{\nu_{i}}}$, where $E_{\nu_{i}}$ denotes the reflex field of the Galois orbit $\nu_{i}$. Moreover, for each $(\nu_{i})_{i \in I} \in (\gamorb)^{I}$, there is a proper morphism  $a_{(\nu_{i})_{i \in I}}: \Div^{I}_{E_{(\nu_{i})_{i \in I}}} \ra \Gr_{T,(\nu_{i})_{i \in I}}$ which is an isomorphism if $G$ is split. The Satake sheaves corresponding to the representations defined by the tuples $(\nu_{i},\tau_{i})_{i \in I}$ are defined by taking the pushforward along $a_{(\nu_{i})_{i \in I}}$ of the external product of the local systems defined by the $\tau_{i}$.

In particular, we see that Theorem \ref{geomsatake} is very simple in this case, as $F_{T}^{I}$ will induce an equivalence between $\Sat_{T}^{I}$ and $(\cochar)^{I}$-graded objects in $\Rep_{\Lambda}(W_{\mathbb{Q}_{p}}^{I})$ under this isomorphism at least when $T$ is split (in general, we may reduce to this case).
\begin{remark}{\label{rem: factorizationconfusing}}
We note that the subspaces $i_{(\nu_{i})_{i \in I}}: \Gr^{I}_{T,(\nu_{i})_{i \in I}} \ra \Gr_{T}^{I}$ are not disjoint for $(\nu_{i})_{i \in I} \in (\gamorb)^{I}$ if $I$ is not a singleton. In particular, if $|I| = 2$ and $\nu = \nu_{1} + \nu_{2}$ then the strata corresponding $(\nu,0)$ and $(\nu_{1},\nu_{2})$ will intersect along the diagonal. Nonetheless, in light of the identification, $\Sat_{T}^{I} \simeq \Rep((\hat{T} \rtimes W_{\bb{Q}_{p}})^{I})$ of Theorem \ref{geomsatake}. If we write $\mathcal{S}_{V}$ for the Satake sheaf attached to $V \in \Rep(\phantom{}^{L}T^{I})$ the weight decomposition of $\phantom{}^{L}T$ implies that we have a direct sum decomposition
\[ \mathcal{S}_{V} \simeq \bigoplus_{(\nu_{i})_{i \in I} \in \gamorb^{I}}  \mathcal{S}_{V(\nu_{i})_{i \in I}} \]
of sheaves in $\D(\Gr_{T}^{I})$, where $\mathcal{S}_{V((\nu_{i})_{i \in I})}$ is the Satake sheaf corresponding to the weight space $V((\nu_{i})_{i \in I})$ corresponding to the Galois orbit $(\nu_{i})_{i \in I} \in (\gamorb)^{I}$. This will be supported on the strata $\Gr_{T,(\nu_{i})_{i \in I}}^{I}$. 
\end{remark}
We consider the diagram 
\[ \begin{tikzcd}
&  \Gr_{B}^{I} \arrow[dl,"p"] \arrow[dr,"q"] & \\
\Gr_{T}^{I}& & \Gr_{G}^{I} \\ 
\end{tikzcd}\]
and define the functor:
\[ p_{!}q^{*}: \D(\Gr_{G}^{I})^{bd} \ra \D(\Gr_{T}^{I})^{bd} \]
As we will discuss further in the next section, the fibers of the morphism $p$ over the connected components $\Gr_{T,(\nu_{i})_{i \in I}}^{I}$ give rise to a locally closed non-disjoint stratification of $\Gr_{G}^{I}$ which embed via the morphism $q$ (cf.  \cite[Example~VI.3.4]{FS}). These are the so called semi-infinite Schubert cells. If one considers $\mathbb{G}_{m}$ acting on $\Gr_{G}^{I}$ via a suitably chosen cocharacter $\mu$ composed with the $L^{+}G$ action on $\Gr_{G}^{I}$ then one can identify these semi-infinite cells with a union of $\mathbb{G}_{m}$-orbits (the attractor of the $\mathbb{G}_{m}$-action), and the fixed points will be precisely the closed subspaces:
\[ \Gr_{T}^{I} =  \bigcup_{(\nu_{i})_{i \in I} \in (\gamorb)^{I}} \Gr_{T,(\nu_{i})_{i \in I}}^{I}. \]
This allows one to apply a diamond analogue \cite[Theorem~IV.6.5]{FS} of Braden's hyperbolic localization theorem \cite{Brad}. In particular, since sheaves in $\Sat_{G}^{I}$ pullback to $L^{+}G$-equivariant sheaves on $\Gr_{G}^{I}$, they will be $\mathbb{G}_{m}$-equivariant. From this, one can deduce that $p_{!}q^{*}$ is a hyperbolic localization with respect to this $\mathbb{G}_{m}$-action and will therefore preserve perverse (after appropriate shifts), flat, and ULA objects over $\Div^{I}$. We therefore get an induced functor
\[ \mathrm{CT}[\mathrm{deg}]: \Sat_{G}^{I} \ra \Sat_{T}^{I} \]
called the constant term functor, as in \cite[Proposition~VI.7.13]{FS}. Here $\mathrm{deg}$ is the function
\[ \mathrm{deg}: |\Gr_{T}^{I}| \ra \mathbb{Z} \]
which has value $-\langle 2\hat{\rho}, |(\nu_{i})_{i \in I}| \rangle$ on the connected component indexed by $(\nu_{i})_{i \in I}$, where $|(\nu_{i})_{i \in I}| := \sum_{i \in I} \ol{\nu}_{i\Gamma} \in \Lambda_{G,B}$. Now, using the Tannakian reconstruction procedure given in the construction of $F_{G}^{I}$ \cite[Section~VI.11]{FS}, these constant terms can be used to show the following Proposition.  
\begin{proposition}{\label{constantterm}}
For all finite index sets $I$, there exists a commutative diagram
\[\begin{tikzcd}
& \Sat_{G}^{I} \arrow[d,"F_{G}^{I}"]  \arrow[r,"\mathrm{CT}{[\mathrm{deg}]}"] & \Sat_{T}^{I}  \arrow[d,"F_{T}^{I}"] & \\
& \Rep_{\Lambda}((\hat{G} \rtimes W_{\bb{Q}_{p}})^{I})  \arrow[r,"\mathrm{res}^{I,G}_{T}"] & \Rep_{\Lambda}((\hat{T} \rtimes W_{\bb{Q}_{p}})^{I})  &
\end{tikzcd} \]
where here $F_{G}^{I}$ (resp. $F_{T}^{I}$) is the equivalence given by Theorem \ref{geomsatake} for $G$ (resp. $T$). 
\end{proposition}
Moreover, it follows by \cite[Proposition~VI.9.6]{FS}, that the fusion product respects the constant term functors. 
\begin{proposition}{\label{fusionconstant}}
For all finite index sets $I$, with a partition $I_{1} \sqcup \ldots \sqcup I_{k}$, we have a commutative diagram 
\[\begin{tikzcd}
& \arrow[d,"\CT^{I_{1}}{[\deg]} \times \ldots \times \CT^{I_{k}}{[\deg]}"] \Sat_{G}^{I_{1}} \times \ldots \times \Sat_{G}^{I_{k}}  \arrow[rr] & & \Sat_{G}^{I} \arrow[d,"\CT^{I}{[\deg]}"] \\
& \Sat_{T}^{I_{1}} \times \ldots \times \Sat_{T}^{I_{k}}  \arrow[rr] & & \Sat_{T}^{I}  &
\end{tikzcd} \]
which commutes functorially in $I$, where the top (resp. bottom) horizontal arrow is given by the fusion product for $G$ (resp. $T$).
\end{proposition}
We now turn our attention to the semi-infinite cells. 
\subsection{The Cohomology of Semi-Infinite Cells}
Theorem \ref{geomsatake}, Proposition \ref{constantterm}, and Proposition \ref{fusionconstant} have implications for the cohomology of spaces related to moduli spaces of $B$-structures. We will record this now. We let $E/\mathbb{Q}_{p}$ be the splitting field of $G$. Then we have an identification $\Gr_{G,E}^{I} \simeq \Gr_{G_{E}}^{I}$ over the base change $\Div^{I}_{E}$ of $\Div^{I}$ to $E$. An appropriate subcategory of sheaves on this space will be equivalent to algebraic representations of $I$ copies of the dual group $\hat{G}$. First, we recall that we have the following natural stratification of $\Gr_{G,E}^{I}$.
\begin{definition}
For $(\lambda_{i})_{i \in I} \in (\mathbb{X}_{*}(T_{\ol{\mathbb{Q}}_{p}})^{+})^{I}$, we let $\Gr_{G,(\lambda_{i})_{i \in I},E}^{I}$ (resp. $\Gr^{I}_{G,\leq (\lambda_{i})_{i \in I},E}$) be the locally closed (resp. closed) subfunctor of $\Gr_{G}$ parameterizing modifications
\[ \mathcal{F}_{G}^{0} \ra \mathcal{F}_{G} \]
between the trivial $G$-bundle $\mathcal{F}_{G}^{0}$ and a $G$-bundle $\mathcal{F}_{G}$ of meromorphy equal to (resp. less than) $\sum_{D_{i} = D_{j}} \lambda_{i}$ at a Cartier divisor $D_{j}$, for some fixed $j \in I$.
\end{definition}
As mentioned in the previous section, $\Gr^{I}_{G,(\lambda_{i})_{i \in I},E}$ is representable in nice diamonds and is $\ell$-cohomologically smooth of dimension equal to $\sum_{i \in I} \langle 2\hat{\rho}, \lambda_{i} \rangle$ over $\Div^{I}_{E}$, by \cite[Proposition~VI.2.4]{FS}. Similarly, by \cite[Proposition~19.2.3]{SW}, $\Gr^{I}_{G,\leq (\lambda_{i})_{i \in I},E}$ is representable in nice diamonds and proper over $\Div^{I}_{E}$. Fix $V\in\Rep_\Lambda(\hat G^I)$ with a fixed central character, and choose geometric dominant cocharacters $\lambda_i$ such that $\mu_i\leq\lambda_i$ for every joint weight $(\mu_i)_{i\in I}$ of $V$. Attached to this, we get a $\Lambda$-valued perverse sheaf $\mathcal{S}_{V}$ supported on $\Gr^{I}_{G,\leq (\lambda_{i})_{i \in I},E}$ by Theorem \ref{geomsatake} and \cite[Proposition~VI.7.5]{FS}. We now fix a geometric point $x \ra \Div^{I}_{E}$. In what follows, for a space $?$ over $\Div^{I}_{E}$ we write $_{x}?$ for the base-change to $x$. Since a local system on $\Div^{I}_{E}$ will be determined by the $W_{E}^{I}$-representation given by its pullback to this geometric point, looking at this pullback will be sufficient for most calculations. We now consider another stratification of $\Gr_{G,E}^{I}$.
\begin{definition}{\label{defsemiinfinite}}
Consider the natural map:
\[ p: \Gr_{B,E}^{I} \ra \Gr_{T,E}^{I} \]
For $(\nu_{i})_{i \in I} \in (\cochar)^{I}$, we set $\mathrm{S}^{I}_{G,(\nu_{i})_{i \in I},E}$ to be the fiber of $\Gr_{B,E}^{I}$ over the closed subspace $\Gr_{T,(\nu_{i})_{i \in I},E}^{I}$ in $\Gr_{T,E}^{I}$ parameterizing modifications of type $(\nu_{i})_{i \in I} \in (\cochar)^{I}$. We note, by \cite[Proposition~VI.3.1]{FS} and its proof, that the natural map 
\[ q: \Gr_{B,E}^{I}  \ra \Gr_{G,E}^{I} \]
is a surjection (and bijective if $I = \{\ast\}$) on geometric points and it defines a locally closed embedding on the $\mathrm{S}_{G,(\nu_{i})_{i \in I},E}^{I}$. Therefore, the spaces $\mathrm{S}_{G,(\nu_{i})_{i \in I},E}^{I}$ for varying $(\nu_{i})_{i \in I} \in (\cochar)^{I}$, form locally-closed subspaces of $\Gr_{G,E}^{I}$ and every geometric point lies in a stratum (though not necessarily a unique one!). We write $p_{(\nu_{i})_{i \in I}}: \mathrm{S}_{G,(\nu_{i})_{i \in I},E}^{I} \ra \Gr_{T,(\nu_{i})_{i \in I},E}^{I} \simeq \Div^{I}_{E}$ and $q_{(\nu_{i})_{i \in I}}: \mathrm{S}_{G,(\nu_{i})_{i \in I},E}^{I} \ra \Gr_{G,E}^{I}$ for the restrictions of $p$ and $q$ to these subspaces.
\end{definition}
\begin{remark}{\label{rem: semiinfinitecellinterpretation}}
In particular, given a modification $\beta: \mathcal{F}_{G}^{0} \dashrightarrow \mathcal{F}_{G}$ of $G$-bundles, by the Tannakian formalism this is equivalent to specifying a set of meromorphic maps
\[ \mathcal{V}^{\hat{\lambda}}_{\mathcal{F}_{G}^{0}} \dashrightarrow \mathcal{V}^{\hat{\lambda}}_{\mathcal{F}_{G}} \]
for all dominant characters $\hat{\lambda} \in \bb{X}^{*}(T_{\ol{\bb{Q}}_{p}})^{+}$ satisfying the Pl\"ucker relations (cf. Definition \ref{defdrinfeld}). Here the highest-weight bundles and maps are interpreted on the \(B_{\mathrm{dR}}^+\)-formal neighborhoods of the modification divisors, using their chosen \(E\)-structures.

The trivial $G$-bundle $\mathcal{F}_{G}^{0}$ admits a natural split $B$-structure whose associated $T$-bundle is the trivial $T$-bundle $\mathcal{F}_{T}^{0}$, and this defines a set of maps
\[ \mathcal{L}^{\hat{\lambda}}_{\mathcal{F}_{T}^{0}} \hookrightarrow \mathcal{V}^{\hat{\lambda}}_{\mathcal{F}^{0}_{G}} \]
where $\mathcal{L}^{\hat{\lambda}} := (\mathcal{V}^{\hat{\lambda}})^{U}$ and $U$ is the unipotent radical of $B$. For a set of divisors $(D_{i})_{i \in I} \in \Div_{E}^{I}$, we can also consider the meromorphic map
\[ \mathcal{L}^{\hat{\lambda}}_{\mathcal{F}_{T}^{0}}(\sum_{i \in I} \langle \hat{\lambda},\nu_{i} \rangle \cdot D_{i}) \dashrightarrow \mathcal{L}^{\hat{\lambda}}_{\mathcal{F}_{T}^{0}} \ra \mathcal{V}^{\hat{\lambda}}_{\mathcal{F}_{G}^{0}} \dashrightarrow \mathcal{V}^{\hat{\lambda}}_{\mathcal{F}_{G}} \]
defined by modifying the $T$-bundles by $\nu_{i}$ at $D_{i}$ for all $i \in I$. We claim that $\beta$ defines a point in $\mathrm{S}_{G,(\nu_{i})_{i \in I},E}^{I}$ if and only if this map does not have a zero or pole for all $\hat{\lambda}$. This is easy to see. In particular, if this map does not have a zero or a pole then the maps
\[\mathcal{L}^{\hat{\lambda}}_{\mathcal{F}_{T}^{0}}(\sum_{i \in I} \langle \hat{\lambda},\nu_{i} \rangle \cdot D_{i}) \ra  \mathcal{V}^{\hat{\lambda}}_{\mathcal{F}_{G}} \]
define a $B$-structure $\mathcal{F}_{B}$ on  $\mathcal{F}_{G}$ whose underlying $T$-bundle is given by $\mathcal{F}_{T}^{0}(\sum_{i \in I} \nu_{i}  \cdot D_{i})$. This corresponds to $\mathcal{F}_{T}$ for a modification $\mathcal{F}_{T}^{0} \dashrightarrow \mathcal{F}_{T}$ lying in $\mathrm{S}^{I}_{G,(\nu_{i})_{i \in I},E}$. Moreover, the construction induces a map of $B$-bundles
\[ \mathcal{F}_{B}^{0} \dashrightarrow \mathcal{F}_{B} \]
which when applying $\times^{B} G$ gives the modification $\beta$ and when applying $\times^{B} T$ gives rise to a modification defining a point in the closed subspace $\Gr_{T,(\nu_{i})_{i \in I},E}^{I}$ over the point attached to $(D_{i})_{i \in I}$ in $\Div^{I}_{E}$, as explained above. In other words, $\mathcal{F}_{B}^{0} \dashrightarrow \mathcal{F}_{B}$ defines an element of the stratum $\mathrm{S}^{I}_{G,(\nu_{i})_{i \in I},E}$. 
\end{remark}
For $V \in \Rep_{\Lambda}(\hat{G}^{I})$, we again consider the sheaf $\mathcal{S}_{V}$ on $\Gr^{I}_{G,\leq (\lambda_{i})_{i \in I},E}$, we note by Proposition \ref{constantterm} and the fact that constant term functors preserve Satake sheaves, that we  have a direct sum decomposition
\begin{equation}{\label{eqn: directsumdecomposition}}
 p_{!}q^{*}\mathcal{S}_{V}[\deg] \simeq \bigoplus_{(\nu_{i})_{i \in I} \in \cochar^{I}} \mathcal{S}_{V((\nu_{i})_{i \in I})}, 
\end{equation}
in terms of the Satake sheaves occurring in the weight decomposition of $V$. Here the sheaves $\mathcal{S}_{(V(\nu_{i})_{i \in I})}$ are supported on the locally closed subspace $\Gr_{T,(\nu_{i})_{i \in I},E}^{I} \ra \Gr_{T,E}^{I}$, as in Remark \ref{rem: factorizationconfusing}.  More precisely, let's  pullback to a fixed geometric point $x \ra \Div^{I}_{E}$ defined by $\Spa(C)$ for $C$ an algebraically closed perfectoid field corresponding to pairwise distinct Cartier divisors. If we write $p_{(\nu_{i})_{i \in I}}: \phantom{}_{x}\mathrm{S}_{G,(\nu_{i})_{i \in I},E}^{I} \ra \phantom{}_{x}\Gr^{I}_{T,(\nu_{i})_{i \in I},E} \simeq \Spa(C)$ for the induced map on the closed subspaces indexed by $(\nu_{i})_{i \in I} \in (\cochar)^{I}$, we note (\ref{eqn: directsumdecomposition}) corresponds to the isomorphism 
\[ p_{!}q^{*}(\mathcal{S}_{V}) \simeq \bigoplus_{(\nu_{i})_{i \in I} \in (\cochar)^{I}} p_{(\nu_{i})_{i \in I}!}(\mathcal{S}_{V}|_{\phantom{}_{x}\mathrm{S}^{I}_{G,(\nu_{i})_{i \in I},E}}) = \bigoplus_{(\nu_{i})_{i \in I} \in (\cochar)^{I}} \mathbb{H}^{*}_{c}(\phantom{}_{x}\mathrm{S}^{I}_{G,(\nu_{i})_{i \in I}},\mathcal{S}_{V}|_{\phantom{}_{x}\mathrm{S}_{G,(\nu_{i})_{i \in I}}}) \]
(cf. \cite[Example~VI.3.6]{FS}) coming from the fact that the strata $\Gr_{T} = \bigcup_{(\nu_{i})_{i \in I}} \Gr_{T,(\nu_{i})_{i \in I},E}^{I}$ are disjoint on this locus. However, this is simply the constant term functor in the previous section. In particular, by Proposition \ref{constantterm}, we deduce the following.
\begin{corollary}{\label{highweightcohom}} 
For $V \in \Rep(\hat{G}^{I})$, a geometric point $x \ra \Div^{I}_{E}$ defining pairwise distinct Cartier divisors, and all tuples $(\nu_{i})_{i \in I} \in (\cochar)^{I}$, we have an isomorphism
\[ \mathbb{H}^{-\langle 2\hat{\rho}, |(\nu_{i})_{i \in I}| \rangle}_{c}(_{x}\mathrm{S}^{I}_{G,(\nu_{i})_{i \in I},E},\mathcal{S}_{V}|_{_{x}\mathrm{S}^{I}_{G,(\nu_{i})_{i \in I},E}}) \simeq V((\nu_{i})_{i \in I}) \otimes \boxtimes_{i \in I} \Lambda(\langle\hat{\rho},\nu_{i}\rangle) \]
of $W_{E}^{I}$-modules, where we endow the RHS with the diagonal $W_{E}^{I}$-action.
\end{corollary} 
This will be the key proposition required for the proof of the filtered Hecke eigensheaf property. More specifically, to show the compatibilities of the filtered eigensheaf property, we need to show that this isomorphism is functorial in $I$. In particular, consider a map of finite index sets $\pi: I \ra J$. For $j \in J$, we set $I_{j} := \pi^{-1}(j)$ and consider the natural map $\Delta_{IJ}: \Div^{J}_{E} \ra \Div^{I}_{E}$, which diagonally embeds the $j$th copy of $\Div^{1}_{E}$ in $\Div^{J}_{E}$ into $\Div^{I_{j}}_{E}$. Then, by the relationship between fusion product and tensor product under Theorem \ref{geomsatake}, we have an identification
\[ \Delta_{IJ}^{*}(\mathcal{S}_{V}) \simeq \mathcal{S}_{\Delta^{*}_{IJ}(V)} \]
of sheaves on $\Gr_{G,E}^{J}$, where $\Delta^{*}_{IJ}(V)$ is given by restriction along the corresponding map $\hat{G}^{J} \ra \hat{G}^{I}$. Now, by Proposition \ref{fusionconstant}, we have the following. 
\begin{corollary}{\label{fusionconstant2}}
For all finite index sets $I,J$ with a map $f: I \ra J$, a tuple $(\nu_{j})_{j \in J} \in (\cochar)^{J}$, a representation $V \in \Rep(\hat{G}^{I})$, and a geometric point $x \ra \Div_{E}^{J}$ defining pair-wise disjoint Cartier divisors, the identification $\Delta_{IJ}^{*}(\mathcal{S}_{V}) \simeq \mathcal{S}_{\Delta^{*}_{IJ}(V)}$ and Corollary \ref{highweightcohom} gives an isomorphism 
\[ \mathbb{H}^{-\langle 2\hat{\rho},|(\nu_{j})_{j \in J}| \rangle}_{c}(_{x}\mathrm{S}^{J}_{G,(\nu_{j})_{j \in J},E},\mathcal{S}_{\Delta^{*}_{IJ}(V)}|_{\phantom{}_{x}\mathrm{S}^{J}_{G,(\nu_{j})_{j \in J},E}}) \simeq \bigoplus_{\substack{(\nu_{i})_{i \in I} \in (\cochar)^{I} \\ \sum_{i \in I_{j}} \nu_{i} = \nu_{j}}} V((\nu_{i})_{i \in I}) \otimes \boxtimes_{j \in J} \Lambda(\langle\hat{\rho},\nu_{j}\rangle)  \]
of $W_{E}^{J}$-modules, where the action on the RHS is via the natural map $\Delta_{IJ}: W_{E}^{J} \ra W_{E}^{I}$.
\end{corollary}
We note that the previous result has some very useful geometric consequences. Let's explain this in the case that $I = \{\ast\}$ is a singleton for a fixed geometric point $x \ra \Div^{1}_{E}$. Fix $\lambda \in \mathbb{X}_{*}(T_{\ol{\mathbb{Q}}_{p}})^{+}$, and consider the highest weight representation $V_{\lambda} \in \Rep_{\Lambda}(\hat{G})$ defined by $\lambda$. Since the sheaf $\mathcal{S}_{\lambda}$ is supported on $_{x}\Gr_{G,\leq \lambda}$, we can deduce the following. 
\begin{corollary}{\label{intnonempty}}
For $\nu \in \cochar$ and $\lambda \in \domcochar$ with associated highest weight representation $V_{\lambda} \in \Rep_{\Lambda}(\hat{G})$, if the weight space $V_{\lambda}(\nu)$ is non-trivial then the intersection $_{x}\mathrm{S}_{G,\nu,E} \cap \phantom{}_{x}\Gr_{G,\leq \lambda,E}$ is non-empty. 
\end{corollary}
\begin{remark}
One can also see this by using the Iwasawa decomposition of $G$ and working explicitly with the loop group of $G$ (see the analysis proceeding \cite[Proposition~6.4]{XS} and \cite[Remark~6.5]{XS}). For example, one can show that the intersection $_{x}\Gr_{G,\leq \lambda,E} \cap \phantom{}_{x}\mathrm{S}_{G,\lambda,E}$ is simply the point given by $\xi^{\lambda}$, where $\xi \in B_{dR}^{+}(C,\mathcal{O}_{C})$ is the uniformizing parameter defined by the geometric point $x$. This corresponds to the highest weight space $V_{\lambda}(\lambda)$. 
\end{remark}
\section{Moduli stacks of $B$-structures}
In the first part of the section, we review the structure of the moduli stack of $B$-structures $\Bun_{B}$ and its basic geometric properties. This will allow us to define the geometric Eisenstein functor. For the proof of the Hecke eigensheaf property, it will be natural to consider a compactification of $\Bun_{B}$, denoted $\ol{\Bun}_{B}$ which is the Drinfeld compactification. We will review the basic properties of this compactification in the second part. 
\subsection{The Geometry of $\Bun_{B}$}{\label{bstrucproperties}}
We will start by collecting some basic facts about the moduli stack $\Bun_{B}$ parameterizing, for $S \in \Perf_{\ol{\bb{F}}_{p}}$, the groupoid of $B$-bundles on $X_{S}$. Given a $B$-bundle $\mathcal{G}_{B}$, we can send it to the induced $T$-bundle and $G$-bundle via the natural maps
\begin{equation*}
\begin{tikzcd}
&  B \arrow[r] \arrow[d] & G  \\
& T  &  
\end{tikzcd}
\end{equation*}
which induces a diagram of $v$-stacks:
\begin{equation*}
\begin{tikzcd}
&  \Bun_{B} \arrow[r,"\mathfrak{p}"] \arrow[d,"\mathfrak{q}"] & \Bun_{G}  \\
& \Bun_{T}  &  .
\end{tikzcd}
\end{equation*}
Let's first start by breaking this up into connected components. As seen in \S $3$, the connected components of $\Bun_{T}$ are indexed by elements $\ol{\nu} \in B(T) \simeq \mathbb{X}_{*}(T_{\ol{\mathbb{Q}}_{p}})_{\Gamma} = \Lambda_{G,B}$. This allows us to define the following.
\begin{definition}
For $\ol{\nu} \in B(T)$, we write $\Bun_{B}^{\ol{\nu}}$ for the pre-image of the connected component $\Bun_{T}^{\ol{\nu}}$ defined by $\ol{\nu}$ under the map $\mf{q}$. We write $\mf{p}^{\ol{\nu}}: \Bun_{B}^{\ol{\nu}} \ra \Bun_{G}$ and $\mf{q}^{\ol{\nu}}: \Bun_{B}^{\ol{\nu}} \ra \Bun^{\ol{\nu}}_{T}$ for the restriction of $\mf{p}$ and $\mf{q}$ to $\Bun_{B}^{\ol{\nu}}$, respectively.
\end{definition}
We claim that this induces a decomposition of the moduli stack $\Bun_{B}$ into connected components. To each element $\ol{\nu}$, we define the integer $d_{\ol{\nu}} := \langle 2\hat{\rho}, \ol{\nu} \rangle$, where $2\hat{\rho}$ is the sum of all positive roots with respect to the choice of Borel. We note that, if $\ol{\nu}$ is anti-dominant with respect to the choice of Borel, $d_{\ol{\nu}}$ is negative. This will be the case where the HN-slopes are dominant so the $B$-bundles will split, and the negative dimension comes from quotienting out by the torsor of splittings. On the other hand, if $\ol{\nu}$ is dominant then the connected component will parametrize non-split $B$-structures and we see that the dimension will be positive. We have the following claim.
\begin{proposition}{\label{prop: qsmooth}}
The map $\mf{q}$ is a $\ell$-cohomologically smooth (non-representable) morphism of Artin $v$-stacks in the sense of \cite[Definition~IV.1.11]{FS}. In particular, for $\ol{\nu} \in \Lambda_{G,B}$, the map $\mf{q}^{\ol{\nu}}$ is pure of $\ell$-dimension equal to $d_{\ol{\nu}}$, in the sense of \cite[Definition~IV.1.17]{FS}.
\end{proposition}
\begin{proof}
This follows from \cite[Proposition 4.8]{Ham1}, where we note that $\Bun_{T}$ is an Artin $v$-stack that is $\ell$-cohomologically smooth of $\ell$-dimension $0$ (see also \cite[Lemma~4.1 (ii)]{AL}).  
\end{proof}
In particular, this implies using \cite[Proposition~23.11]{Ecod}, that $\mf{q}$ is a universally open morphism of Artin $v$-stacks. Moreover, one can check that the fibers of this morphism are connected (see the proof of \cite[Lemma~4.1 (ii)]{AL} and \cite[Proposition~3.16]{Ham1}). As a consequence, we can deduce that, since $\bigsqcup_{\ol{\nu} \in \Lambda_{G,B}} \Bun_{T}^{\ol{\nu}}$ is a decomposition of $\Bun_{T}$ into connected components, the following is true.
\begin{corollary}{\label{conncomponents}}
The connected components of $\Bun_{B}$ are given by $\Bun_{B}^{\ol{\nu}}$, for varying $\ol{\nu} \in B(T)$. 
\end{corollary} 
We now comment briefly on the geometry of the map $\mf{p}$. In particular, we have the following.
\begin{lemma}{\cite[Lemma~4.1 (ii)]{AL}}{\label{lemma: repinnicediamonds}}
The map $\mf{p}$ is representable in nice diamonds. 
\end{lemma}
These properties allow us to define the Eisenstein functor using the derived functors defined in \cite{Ecod}.  
\begin{definition}{\label{defeisfunctor}}
We define a locally constant function 
\[ \dim(\Bun_{B}): |\Bun_{B}| \ra \mathbb{Z} \]
\[ x \in |\Bun_{B}^{\ol{\nu}}| \mapsto d_{\ol{\nu}} \]
and with it the unnormalized Eisenstein functor
\[ \Eis: \D(\Bun_{T}) \ra \D(\Bun_{G}) \]
\[ \mathcal{F} \mapsto \mf{p}_{!}(\mf{q}^{*}(\mathcal{F})[\dim(\Bun_{B})]). \]
Here we recall that we are working with $\ell$-torsion coefficients so that $\mf{p}_{!}$ is well-defined by Lemma \ref{lemma: repinnicediamonds} and the results of \cite{Ecod}.
\end{definition}
Even though $\Bun_{B}$ is $\ell$-cohomologically smooth (by the smoothness of $\Bun_{T}$ and Proposition \ref{prop: qsmooth}), the sheaf $\Lambda[\dim(\Bun_{B})]$ is not Verdier self-dual. In particular, the dualizing object $K_{\Bun_{B}}$ on $\Bun_{B}$ is not isomorphic to $\Lambda[2\dim(\Bun_{B})]$; this is only $v$-locally true on $\Bun_{B}$. To elucidate the problem, note that, given $\ol{\nu} \in \Lambda_{G,B}$, we can consider the natural map $\mathfrak{q}^{\ol{\nu}}: \Bun_{B}^{\ol{\nu}} \ra \Bun_{T}^{\ol{\nu}} = [\ast/\ul{T(\mathbb{Q}_{p})}]$. Given a character $\kappa_{\ol{\nu}}$ of $T(\mathbb{Q}_{p})$, we can pull this character back along $\mathfrak{q}^{\ol{\nu}}$ to get a sheaf on $\Bun_{B}^{\ol{\nu}}$. These characters give us sheaves on $\Bun_{B}^{\ol{\nu}}$, which $v$-locally will be constant, but are not constant on the nose. The main result of \cite{GH} specialized to the case of the Borel is as follows.
\begin{theorem}{\label{dualob}}{\cite[Theorem~1.6]{GH}}
The dualizing object on $\Bun_{B}$ is isomorphic to $\mf{q}^{*}(\Delta_{B})[2\dim(\Bun_{B})]$, with $\Delta_{B} \in \D(\Bun_{T})$ as in Definition \ref{modulussheaf}. 
\end{theorem}
The key consequence of this theorem for us. 
\begin{corollary}{\label{verdierselfdual}}
The sheaf $\mf{q}^{*}(\Delta_{B}^{1/2})[\dim(\Bun_{B})]$ on $\Bun_{B}$ is Verdier self-dual. 
\end{corollary}
This motivates the definition of the normalized Eisenstein functor.
\begin{definition}{\label{icdef}}
We let $\IC_{\Bun_{B}} := \mf{q}^{*}(\Delta_{B}^{1/2})[\dim(\Bun_{B})]$. We define the normalized Eisenstein functor:
\[ \nmEis: \D(\Bun_{T}) \ra \D(\Bun_{G}) \]
\[ \mathcal{F} \mapsto \mf{p}_{!}(\mf{q}^{*}(\mathcal{F}) \otimes \IC_{\Bun_{B}})\]
In particular, we note that we have a natural isomorphism:
\[ \nmEis(-) \simeq \Eis(- \otimes \Delta_{B}^{1/2}). \]
\end{definition}
Since the definition of the Eisenstein functor involves the functor $\mf{p}_{!}$, it is natural to consider a compactification of the morphism $\mf{p}$ in order to study its finer properties. This leads us to the Drinfeld compactification. 
\subsection{The Drinfeld Compactification}
\subsubsection{The Definition and Basic Properties}
We assume temporarily that the derived group of $G$ is simply connected. We recall that classically (curve over a finite or complex field) there is a rather straight-forward way of compactifying the map:
\[ \mathfrak{p}: \Bun_B \rightarrow \Bun_G \]
This is called a Drinfeld Compactification of $\mathfrak{p}$, denoted $\ol{\Bun}_B$. Its main property is that there exists an open immersion $\Bun_{B} \rightarrow \ol{\Bun}_{B}$, and it has a map $\ol{\mathfrak{p}}: \ol{\Bun}_B \rightarrow \Bun_G$ extending $\mathfrak{p}$, which is proper after restricting to finitely many connected components. First, as a warm up, let us explain the construction when $G = \GL_2$. For this, we recall that $\Bun_B$ can be viewed as parameterizing tuples $(\mathcal{M},\mathcal{L},\kappa: \mathcal{L} \hookrightarrow \mathcal{M})$, where $\mathcal{M}$ is a rank $2$ vector bundle, $\mathcal{L}$ is a rank $1$ vector bundle, and $\kappa$ is an injective bundle map. To compactify this space, we will allow $\kappa$ to be a map of $\mathcal{O}_{X_S}$-modules whose pullback to each geometric point is an injective map of coherent sheaves on $X$. In other words, we allow $\mathcal{M}/\mathcal{L}$ to have torsion. For a general $G$, the idea is to apply the Tannakian formalism. In particular, given a $G$-bundle $\mathcal{F}_G$ we get, for all $\hat{\lambda} \in \hat{\Lambda}_G^+$, an induced highest weight bundle, denoted $\mathcal{V}_{\mathcal{F}_G}^{\hat{\lambda}}$. A point of $\Bun_B$ mapping to $\mathcal{F}_G$ via $\mathfrak{p}$ then defines a set of line subbundles
$\kappa^{\hat{\lambda}}: \mathcal{L}^{\hat{\lambda}} \hookrightarrow \mathcal{V}^{\hat{\lambda}}_{\mathcal{F}_G}$ which satisfy some Plücker relations. Using this interpretation, $\ol{\Bun}_B$ can then be defined as classifying $G$-bundles $\mathcal{F}_G$ together with a system of maps $\ol{\kappa}^{\hat{\lambda}}: \mathcal{L}^{\hat{\lambda}} \rightarrow \mathcal{V}_{\mathcal{F}_G}^{\hat{\lambda}}$ for all $\hat{\lambda} \in \Lambda_G^+$, which are injective after pulling back to a geometric point and satisfy the same Plücker relations.

We now explain how to construct the aforementioned compactification $\ol{\Bun}_B$ of $\Bun_B$ over $\Bun_G$ in the Fargues-Fontaine setting. In order to describe the Drinfeld compactification, we note that, for $S \in \Perf_{\ol{\bb{F}}_{p}}$, $\Bun_B$ can be viewed as a stack parameterizing triples:
\begin{enumerate}
\item A $G$-bundle $\mathcal{F}_G$ on $X_{S}$.
\item A $T$-bundle $\mathcal{F}_T$ on $X_{S}$.
\item A $G$-equivariant map $\kappa: \mathcal{F}_G \rightarrow G/U \times^T \mathcal{F}_T$. 
\end{enumerate}
Moreover (3) can be described as a collection of injective bundle maps on $X_S$, $\kappa^{\mathcal{V}}: (\mathcal{V}^{U})_{\mathcal{F}_{T}} \rightarrow \mathcal{V}_{\mathcal{F}_{G}}$ for every $G$-module $\mathcal{V}$ satisfying the following Plücker relations (see \cite[Theorem~1.1.2]{BG} and the discussion in \cite[Section~1.2.1.]{BG}: the key point is to apply the algebraic Peter-Weyl theorem (see \cite[Part II, Proposition~4.20]{JanRepresentationsofAlgebraicGroups}) describing $k[G]^{U}$ as a direct sum of $V^{\vee} \otimes (V)^{U}$ for $V$ an irreducible representation of $G$ over $k$ to rewrite $T$-points of $\ol{G/U}$ for $T$ an affine scheme in terms of maps $\kappa^{V}$ satisfying the Plücker relations. Restricting to the open subscheme $G/U$ gives subbundle embeddings. One can then use this description of the $T$-points to see the desired description by descending along a common \'etale cover trivializing $\mathcal{F}_{G}$ and $\mathcal{F}_{T}$.)
\begin{enumerate}
\item For the trivial representation $\mathcal{V}$, $\kappa^{\mathcal{V}}$ must be the identity map $\mathcal{O}_{X_S} \rightarrow \mathcal{O}_{X_S}$.
\item For a $G$-module map $\mathcal{V}^1 \rightarrow \mathcal{V}^2$, the induced square
\[\begin{tikzcd}
&   ((\mathcal{V}^1)^U) _{\mathcal{F}_T}\arrow[d] \arrow[r, "\kappa^{\mathcal{V}^1}"] & \mathcal{V}^1_{\mathcal{F}_G}  \arrow[d] & \\
& ((\mathcal{V}^2)^U)_{\mathcal{F}_T}  \arrow[r, "\kappa^{\mathcal{V}^2}"] & \mathcal{V}^2_{\mathcal{F}_G} &
\end{tikzcd}
\]
commutes.
\item For two $G$-modules $\mathcal{V}^1$ and $\mathcal{V}^2$, we have that the diagram
\[\begin{tikzcd}
&   ((\mathcal{V}^1)^U \otimes (\mathcal{V}^2)^U)_{\mathcal{F}_T} \arrow[d] \arrow[r,"\kappa^{\mathcal{V}^1} \otimes \kappa^{\mathcal{V}^2}"] & \mathcal{V}^1_{\mathcal{F}_G} \otimes \mathcal{V}^2_{\mathcal{F}_G} \arrow[d,"id"] & \\
& ((\mathcal{V}^1 \otimes \mathcal{V}^2)^U)_{\mathcal{F}_T} \arrow[r, "\kappa^{\mathcal{V}^1 \otimes \mathcal{V}^2}"] & \mathcal{V}^1_{\mathcal{F}_G} \otimes \mathcal{V}^2_{\mathcal{F}_G} &
\end{tikzcd}\]
commutes.
\end{enumerate}
Here we recall that $\mathcal{V}_{\mathcal{F}_{G}}$ is the vector bundle on $X_{S}$ given by $\mathcal{F}_{G}$ and the homomorphism $\rho_{\mathcal{V}}: G \ra \GL(V)$ defined by $V$. Similarly, $(\mathcal{V}^{U})_{\mathcal{F}_{T}}$ is given by the algebraic $T$-representation $\mathcal{V}^{U}$ and the $T$-bundle $\mathcal{F}_{T}$ on $X_{S}$. 

As mentioned above, the idea will now be to introduce torsion in the above definition. In particular, we have the following definition for $\ol{\Bun}_{B}$.
\begin{definition}{\label{defdrinfeld}}
We define $\ol{\Bun}_{B}$ to be the $v$-stack parameterizing, for $S \in \Perf_{\ol{\bb{F}}_{p}}$, triples $(\mathcal{F}_G,\mathcal{F}_T, \ol{\kappa}^{\mathcal{V}})$, where $\ol{\kappa}^{\mathcal{V}}$ is a map of $\mathcal{O}_{X_S}$-modules defined for every $G$-module $\mathcal{V}$
\[ (\mathcal{V}^U)_{\mathcal{F}_T} \rightarrow \mathcal{V}_{\mathcal{F}_G} \]
satisfying the following conditions:
\begin{itemize}
\item For every geometric point $s \ra S$, the pullback of $\ol{\kappa}^{\mathcal{V}}$ to the Fargues-Fontaine curve over $s$ is an injection of coherent sheaves.
\item The Plücker relations hold in the following sense:
\begin{enumerate}
\item For the trivial representation $\mathcal{V}$, $\ol{\kappa}^{\mathcal{V}}$ is the identity map $\mathcal{O} \rightarrow \mathcal{O}$.
\item For a $G$-module map $\mathcal{V}^1 \rightarrow \mathcal{V}^2$, the induced square
\[\begin{tikzcd}
&   ((\mathcal{V}^1)^U)_{\mathcal{F}_T} \arrow[d] \arrow[r, "\ol{\kappa}^{\mathcal{V}^1}"] & \mathcal{V}^1_{\mathcal{F}_G}  \arrow[d] & \\
& ((\mathcal{V}^2)^U)_{\mathcal{F}_T}  \arrow[r, "\ol{\kappa}^{\mathcal{V}^2}"] & \mathcal{V}^2_{\mathcal{F}_G} &
\end{tikzcd}
\]
commutes.
\item For two $G$-modules $\mathcal{V}^1$ and $\mathcal{V}^2$, we have that the diagram
\[\begin{tikzcd}
&   ((\mathcal{V}^1)^U \otimes (\mathcal{V}^2)^U)_{\mathcal{F}_T} \arrow[d] \arrow[r,"\ol{\kappa}^{\mathcal{V}^1} \otimes \ol{\kappa}^{\mathcal{V}^2}"] & \mathcal{V}^1_{\mathcal{F}_G} \otimes \mathcal{V}^2_{\mathcal{F}_G} \arrow[d,"id"] & \\
& ((\mathcal{V}^1 \otimes \mathcal{V}^2)^U)_{\mathcal{F}_T} \arrow[r, "\ol{\kappa}^{\mathcal{V}^1 \otimes \mathcal{V}^2}"] & \mathcal{V}^1_{\mathcal{F}_G} \otimes \mathcal{V}^2_{\mathcal{F}_G} &
\end{tikzcd}\]
commutes.
\end{enumerate}
\end{itemize}
\end{definition}
\begin{remark}
To simplify the notation, we will write $\mathcal{L}^{\hat{\lambda}} := (\mathcal{V}^{\hat{\lambda}})^{U}$ and $\ol{\kappa}^{\hat{\lambda}}$ for the embedding attached to the highest weight module of $G$ of highest weight $\hat{\lambda} \in \hat{\Lambda}_{G}^{+}$. We note that, analogous to $\Bun_{B}$, the $v$-stack $\ol{\Bun}_{B}$ can be viewed as an open subfunctor of the moduli space of tuples $(\kappa: \mathcal{F}_{G} \ra \ol{G/U} \times^{T} \mathcal{F}_{T},\mathcal{F}_{G},\mathcal{F}_{T})$, where $\kappa$ is a $G$-equivariant map and $\ol{G/U}$ denotes the affine closure of the quasi-affine $G/U$. In particular, such a tuple will define a set of maps $(\mathcal{V}^{U})_{\mathcal{F}_{T}} \ra \mathcal{V}_{\mathcal{F}_{G}}$ for all $G$-modules $\mc{V}$ satisfying the Pl\"ucker conditions, and $\ol{\Bun}_{B}$ is the open subfunctor where these maps are injections of coherent sheaves after pulling back to each geometric point.
\end{remark}
This gives rise to a well-defined $v$-stack, and, using this description, we get well-defined morphisms $\ol{\mathfrak{p}}: \ol{\Bun}_B \rightarrow \Bun_G$ and $\ol{\mathfrak{q}}: \ol{\Bun}_B \rightarrow \Bun_T$ via projecting the data $(\mathcal{F}_{G},\mathcal{F}_{T},\ol{\kappa})$ to the first and second factor, respectively. We also get a natural map $j: \Bun_{B} \ra \ol{\Bun}_{B}$, which will be an open immersion by \cite[Theorem~1.2.2 (i)]{HHS}. We now recall the stratification of $\overline{\Bun}_{B}$ For an element $\ol{\nu} \in \Lambda_{G,B}^{pos} \setminus \{0\}$, we write $\ol{\nu} = \sum_{i \in \mathcal{J}} n_{i}\alpha_{i}$ as a positive linear combination of the elements corresponding to $\Gamma$-orbits of simple positive coroots $\alpha_{i}$. We let $\Div^{(\ol{\nu})}$ be the partially symmetrized power of the mirror curve attached to it, as in \S \ref{sec: geomconsofweakgen}. We have a map of Artin $v$-stacks
\[ j_{\ol{\nu}}: \Div^{(\ol{\nu})} \times \Bun_{B} \ra \ol{\Bun}_{B} \]
sending a tuple $(\{(D_{i})_{i \in \mathcal{J}}\}, \mathcal{F}_{G},\mathcal{F}_{T},\kappa^{\hat{\lambda}})$, to the tuple
\[ (\mathcal{F}_{G},\mathcal{F}_{T}(-\sum_{i \in J}\alpha_{i} \cdot D_{i}), \ol{\kappa}^{\hat{\lambda}}) \]
where $\ol{\kappa}^{\hat{\lambda}}$ is the natural composition
\[ (\mathcal{L}^{\hat{\lambda}})_{\mathcal{F}_{T}}(-\sum_{i \in J}\langle \alpha_{i},\hat{\lambda} \rangle \cdot D_{i}) \ra \mathcal{L}^{\hat{\lambda}}_{\mathcal{F}_{T}} \xrightarrow{\kappa^{\hat{\lambda}}} (\mathcal{V}^{\hat{\lambda}})_{\mathcal{F}_{G}} \]
defined by the unique effective modification of $T$-bundles of the specified meromorphy and support. We now make the following definition. 
\begin{definition}{\label{defstrata}}
For $\ol{\nu} \in \Lambda_{G,B}^{\mathrm{pos}} \setminus \{0\}$, we define the $v$-stack $\phantom{}_{\ol{\nu}}\ol{\Bun}_{B}$ (resp. $_{\geq \ol{\nu}}\ol{\Bun}_{B}$) to be the locus where, for all $\hat{\lambda} \in \hat{\Lambda}_{G}^{+}$, the cokernels $\mathcal{V}^{\hat{\lambda}}/\Im{(\ol{\kappa}^{\hat{\lambda}})}$ have torsion of length equal to (resp. greater than or equal to) $\langle \hat{\lambda}, \ol{\nu} \rangle$ after pulling back to any geometric point. By \cite[Lemma~4.14 (i)]{HHS}, $_{\ol{\nu}}\ol{\Bun}_{B}$ is a locally closed substack of $\ol{\Bun}_{B}$, and the closure of $_{\ol{\nu}}\ol{\Bun}_{B}$ in $\ol{\Bun}_{B}$ is contained in $_{\geq \ol{\nu}}\ol{\Bun}_{B}$. Moreover, it follows by \cite[Proposition~4.2.8]{HHS} that the map $j_{\ol{\nu}}: \Div^{(\ol{\nu})} \times \Bun_{B} \ra \ol{\Bun}_{B}$ is a locally closed embedding with image isomorphic to $\phantom{}_{\ol{\nu}}\ol{\Bun}_{B}$.
\end{definition}
\begin{remark}{\label{rem: conjugatingbythecenter}}
As explained in \cite[Section~6.3]{HHS}, if the derived group of $G$ is not simply connected, one may define the Drinfeld compactification $\ol{\Bun}_{B}$ by considering a central extension 
\[ 0 \ra Z \ra \tilde{G} \ra G \ra 0, \]
where $\tilde{G}$ has simply connected derived group. One then considers the preimage $\tilde{B} \subset \tilde{G}$ of $B$ and defines $\ol{\Bun}_{B}$ to be the stack quotient of $\ol{\Bun}_{\tilde{B}}$ by the natural action of $\Bun_{Z}$ given by tensoring by $Z$-bundles. Similarly, for $\ol{\nu} \in \Lambda_{G,B}^{\mathrm{pos}}$, one may define a locally closed strata $j_{\bar\nu}:
\operatorname{Div}^{(\bar\nu)}\times\operatorname{Bun}_B
\xrightarrow{\sim}{}_{\bar\nu}\overline{\operatorname{Bun}}_B
\hookrightarrow\overline{\operatorname{Bun}}_B$, by quotienting the corresponding locally closed strata in $\ol{\Bun}_{\tilde{B}}$ by the action of $\Bun_{Z}$, where $\Bun_{Z}$ acts trivially on $\Div^{\ol{\nu}}$. Using \cite[Proposition~6.3.3]{HHS}, one sees that this satisfies the analogue of the above properties.
\end{remark}
With these strata defined, we now discuss a key diagram that will be important for proving the filtered Hecke eigensheaf property.
\subsubsection{The Key Diagram}{\label{keydiagram}}
We would now like to describe how Hecke correspondences on $\Bun_{G}$ interact with pullback along the map $\ol{\mf{p}}: \ol{\Bun}_{B} \ra \Bun_{G}$. This will be used in \S 6 to show the filtered Hecke eigensheaf property for the geometric Eisenstein series, analogous to the analysis carried out in \cite[Section~3]{BG}. We fix a finite index set $I$, and consider the Hecke stacks $\Hck^{I}_{G,E}$ base-changed to the field $E$ over which $G$ splits. We consider the usual diagram
\[ \Bun_{G} \times \Div^{I}_{E}  \xleftarrow{h_{G}^{\leftarrow} \times \pi} \Hck_{G,E}^{I} \xrightarrow{h_{G}^{\rightarrow}} \Bun_{G} \]
as in \S \ref{GCFT} and \S \ref{geomsatakesection}. We now fix a tuple of geometric dominant cocharacters $(\lambda_{i})_{i \in I} \in (\mathbb{X}_{*}(T_{\ol{\mathbb{Q}}_{p}})^{+})^{I}$, and restrict to the locus $\Hck_{G,\leq (\lambda_{i})_{i \in I},E}$ where the meromorphy of this modification is bounded by $(\lambda_{i})_{i \in I}$. We define the $v$-stack $\overline{Z}^{I}_{(\lambda_{i})_{i \in I}}$ by the Cartesian diagram: 
\begin{equation*}
\begin{tikzcd}
&  \overline{Z}^{I}_{(\lambda_{i})_{i \in I}} \arrow[r,"'h_{G}^{\rightarrow}"] \arrow[d,"'\ol{\mf{p}}"] & \overline{\Bun}_{B}  \arrow[d,"\overline{\mathfrak{p}}"] \\
& \Hck^{I}_{G,\leq (\lambda_{i})_{i \in I},E} \arrow[r,"h_{G}^{\rightarrow}"] & \Bun_{G}.
\end{tikzcd}
\end{equation*}
By definition, $\overline{Z}^{I}_{(\lambda_{i})_{i \in I}}$ parametrizes a pair of $G$-bundles $(\mathcal{F}_{G},\mathcal{F}_{G}')$ together with a modification $\mathcal{F}_{G} \dashrightarrow \mathcal{F}_{G}'$ with meromorphy bounded by $\lambda_{i}$ at Cartier divisors $D_{i}$ for $i \in I$ and an enhanced $B$-structure (in the sense of Definition \ref{defdrinfeld}) on $\mathcal{F}_{G}$ specified by maps $\overline{\kappa}^{\hat{\lambda}}$ for $\hat{\lambda} \in \hat{\Lambda}_{G}^{+}$. The fact that the modification $\mathcal{F}_{G} \dashrightarrow \mathcal{F}_{G}'$ has meromorphy bounded by $(\lambda_{i})_{i \in I}$ implies that, for all $\hat{\lambda} \in \hat{\Lambda}_{G}^{+}$, we have an inclusion:
\[ \mathcal{V}^{\hat{\lambda}}_{\mathcal{F}_{G}'}(-\sum_{i \in I} \langle \lambda_{i\Gamma}, \hat{\lambda} \rangle \cdot D_{i}) \subset \mathcal{V}^{\hat{\lambda}}_{\mathcal{F}_{G}} \subset \mathcal{V}^{\hat{\lambda}}_{\mathcal{F}_{G}'}(-\sum_{i \in I} \langle  w_{0}(\lambda_{i\Gamma}), \hat{\lambda} \rangle \cdot D_{i}) \]
(cf. \cite[Section~2.1.4]{BG}). Therefore, the embeddings \[ \ol{\kappa}^{\hat{\lambda}}: \mathcal{L}^{\hat{\lambda}}_{\mathcal{F}_{T}} \hookrightarrow \mathcal{V}^{\hat{\lambda}}_{\mathcal{F}_{G}} \]
give rise to a map: 
\[ \ol{\kappa}^{'\hat{\lambda}}: \mathcal{L}^{\hat{\lambda}}_{\mathcal{F}_{T}}(\sum_{i \in I} \langle w_{0}(\lambda_{i\Gamma}),\hat{\lambda} \rangle \cdot D_{i}) \hookrightarrow \mathcal{V}^{\hat{\lambda}}_{\mathcal{F}_{G}'}. \]
This defines for us a morphism 
\[ \phi_{(\lambda_{i})_{i \in I}}: \ol{Z}^{I}_{(\lambda_{i})_{i \in I}} \rightarrow \overline{\Bun}_{B} \times \Div^{I}_{E} \]
which records the point in $\ol{\Bun}_{B} \times \Div^{I}_{E}$ defined by the pair $(\ol{\kappa}^{'\hat{\lambda}},(D_{i})_{i \in I})$. This sits in a commutative diagram
\begin{equation} 
\begin{tikzcd}
& \ol{\Bun}_{B} \times \Div^{I}_{E} \arrow[d,"\ol{\mf{p}} \times \mathrm{id}"]  & \arrow[l,"\phi_{(\lambda_{i})_{i \in I}}"] \ol{Z}^{I}_{(\lambda_{i})_{i \in I}} \arrow[r,"'h_{G}^{\rightarrow}"] \arrow[d,"\phantom{}^{'}\ol{\mf{p}}"] & \ol{\Bun}_{B}  \arrow[d,"\ol{\mathfrak{p}}"] \\
& \Bun_{G} \times \Div^{I}_{E} & \arrow[l,"h_{G}^{\leftarrow} \times \pi"] \Hck^{I}_{G,\leq (\lambda_{i})_{i \in I},E} \arrow[r,"h_{G}^{\rightarrow}"] & \Bun_{G}, 
\end{tikzcd}
\end{equation}
where we note that the left square is not Cartesian. We will not use the space $\ol{Z}^{I}_{(\lambda_{i})_{i \in I}}$ at all in our arguments, but it should be important for future applications. For our purposes, we consider $Z^{I}_{(\lambda_{i})_{i \in I}}$, the space obtained by replacing $\ol{\Bun}_{B}$ with $\Bun_{B}$ on the right hand side of the diagram. This sits in an analogous diagram
\begin{equation}{\label{eqn: keydiagram}}
\begin{tikzcd}
& \ol{\Bun}_{B} \times \Div^{I}_{E} \arrow[d,"\ol{\mf{p}} \times \mathrm{id}"]  & \arrow[l,"\phi_{(\lambda_{i})_{i \in I}}"] Z^{I}_{(\lambda_{i})_{i \in I}} \arrow[r,"'h_{G}^{\rightarrow}"] \arrow[d,"\phantom{}^{'}\mf{p}"] & \Bun_{B}  \arrow[d,"\mathfrak{p}"] \\
& \Bun_{G} \times \Div^{I}_{E} & \arrow[l,"h_{G}^{\leftarrow} \times \pi"] \Hck^{I}_{G,\leq (\lambda_{i})_{i \in I},E} \arrow[r,"h_{G}^{\rightarrow}"] & \Bun_{G} 
\end{tikzcd}
\end{equation}
As we will see, the proof of the filtered Hecke eigensheaf Property will ultimately reduce to contemplating the fibers of the morphism $\phi_{(\lambda_{i})_{i \in I}}$. For notation simplicity, we will describe  the pullback of this diagram to a geometric point $\Spa(F,\mathcal{O}_{F}) \ra \Div^{I}_{E}$ with the precise description of the analogues over $\Div^{I}_{E}$ being left to the reader. We denote the resulting space by $_{x}\overline{Z}^{I}_{(\lambda_{i})_{i \in I}}$. It sits in a diagram of the form
\begin{equation*}
\begin{tikzcd}
& \phantom{}_{x}\ol{\Bun}_{B} \arrow[d,"\ol{\mf{p}}"] & \arrow[l,"\phi_{(\lambda_{i})_{i \in I}}"] _{x}\ol{Z}^{I}_{(\lambda_{i})_{i \in I}} \arrow[r,"'h_{G}^{\rightarrow}"] \arrow[d,"'\ol{\mf{p}}"] & \overline{\Bun}_{B}  \arrow[d,"\overline{\mathfrak{p}}"] \\
& \phantom{}_{x}\Bun_{G} & \arrow[l,"h_{G}^{\leftarrow} \times \pi"] _{x}\Hck^{I}_{G,\leq (\lambda_{i})_{i \in I},E} \arrow[r,"h_{G}^{\rightarrow}"] & \Bun_{G} 
\end{tikzcd}
\end{equation*}
where $_{x}\Hck^{I}_{G,\leq (\lambda_{i})_{i \in I},E}$ is the Hecke stack parameterizing modifications at the tuple of Cartier divisors $(D_{i})_{i \in I}$ corresponding to $x$ and $\phantom{}_{x}\ol{\Bun}_{B}$ (resp. $\phantom{}_{x}\Bun_{G}$) denotes the base change of $\ol{\Bun}_{B}$ (resp. $\Bun_{G}$) to $\Spa(F,\mathcal{O}_{F})$. Consider a tuple $(\nu_{i})_{i \in I} \in (\cochar)^{I}$ lying in the span of the positive coroots and write $\ol{\nu} := \sum_{i \in I} \nu_{i\Gamma}$. Let $E_{\nu_{i}}$ denote the reflex field of $\nu_{i}$. We view the geometric point $x \ra \Div^{I}$ as a geometric point of $\Div^{(\ol{\nu})}$ via composing with the map
\[ \Delta_{(\nu_{i})_{i \in I}}: \Div^{I}_{E} \ra \prod_{i \in I} \Div^{1}_{E_{\nu_{i}}}  \xrightarrow{\prod_{i \in I} \Delta_{\nu_{i}}} \prod_{i \in I} \Div^{(\nu_{i\Gamma})} \ra \Div^{(\ol{\nu})}, \] 
where the first map is the natural finite \'etale covering induced by the extension $E/E_{i}$, the last map is given by taking the union of Cartier divisors, and $\Delta_{\nu_{i}}$ is the twisted diagonal embedding described in \S \ref{sec: geomconsofweakgen}. We set $_{x,(\nu_{i})_{i \in I}}\ol{\Bun}_{B}$ to be the pullback of the locally closed stratum $\phantom{}_{\ol{\nu}}\ol{\Bun}_{B} \simeq \Div^{(\ol{\nu})} \times \Bun_{B}$ to this geometric point. The substack $_{x,(\nu_{i})_{i \in I}}\ol{\Bun}_{B}$ corresponds to the locus where the embeddings
\[ \mathcal{L}^{\hat{\lambda}}_{\mathcal{F}_{T}} \rightarrow \mathcal{V}^{\hat{\lambda}}_{\mathcal{F}_{G}} \]
have a zero of order given by $\langle \nu_{i}, \hat{\lambda} \rangle$ at $D_{i}$ for all dominant characters $\hat{\lambda}$ of $G$ and all $i \in I$, and a zero nowhere else. We now consider the open substack $_{x,0}\ol{\Bun}_{B} = \phantom{}_{x}\Bun_{B}$. Since the maps $\kappa$ have no zero at the Cartier divisors corresponding to $x$, it follows that they define an $L^{+}B$-torsor of trivializations of the $B$-bundles on the formal completion at the Cartier divisors over the stack $_{x,0}\ol{\Bun}_{B}$, which we denote by $_{x}\mathcal{B}$. We then consider the map
\[ i_{(\nu_{i})_{i \in I}}: \ol{\Bun}_{B} \times \Div^{I}_{E} \rightarrow \ol{\Bun}_{B} \times \Div^{I}_{E} \]
given by sending $(\mathcal{F}_{G},\mathcal{F}_{T},\ol{\kappa}^{\hat{\lambda}},(D_{i})_{i \in I})$ to the object $(\mathcal{F}_{G},\mathcal{F}_{T}(-\sum_{i \in I} \nu_{i} \cdot D_{i}),\mathcal{L}^{\hat{\lambda}}_{\mathcal{F}_{T}}(- \sum_{i \in I} \langle \nu_{i\Gamma}, \hat{\lambda} \rangle \cdot D_{i}) \hookrightarrow \mathcal{L}^{\hat{\lambda}}_{\mathcal{F}_{T}} \hookrightarrow \mathcal{V}^{\hat{\lambda}}_{\mathcal{F}_{G}},(D_{i})_{i \in I})$, where the map $\mathcal{L}^{\hat{\lambda}}_{\mathcal{F}_{T}}(- \sum_{i \in I} \langle \nu_{i\Gamma}, \hat{\lambda} \rangle \cdot D_{i}) \hookrightarrow \mathcal{L}^{\hat{\lambda}}_{\mathcal{F}_{T}}$ is specified by the $\nu_{i}$ at the $D_{i}$. Note that the map $i_{(\nu_{i})_{i \in I}}$ defines an isomorphism between the pullbacks $_{x,0}\ol{\Bun}_{B}$ and $_{x,(\nu_{i})_{i \in I}}\ol{\Bun}_{B}$. Therefore, by transport of structure, we get a $L^{+}B$-torsor, denoted $_{x}\mathcal{B}^{(\nu_{i})_{i \in I}}$, over $_{x,(\nu_{i})_{i \in I}}\ol{\Bun}_{B}$. For $(\nu_{i})_{i \in I} \in (\cochar)^{I}$ (resp. $(\nu'_{i})_{i \in I} \in (\cochar)^{I}$) lying in the span of the positive coroots, we let $\phantom{}_{x}Z^{I,\ast,(\nu_{i})_{i \in I}}_{(\lambda_{i})_{i \in I}}$ (resp. $\phantom{}_{x}Z^{I,(\nu'_{i})_{i \in I},\ast}_{(\lambda_{i})_{i \in I}}$) be the fibers of $'h_{G}^{\rightarrow}$ (resp. $\phi_{(\lambda_{i})_{i \in I}}$) over  $_{x,(\nu_{i})_{i \in I}}\ol{\Bun}_{B} \ra \ol{\Bun}_{B}$ (resp. $_{x,(\nu'_{i})_{i \in I}}\ol{\Bun}_{B} \ra \ol{\Bun}_{B}$). We write $\phantom{}_{x}Z^{I,(\nu'_{i})_{i \in I},(\nu_{i})_{i \in I}}_{(\lambda_{i})_{i \in I}}$ for their intersection. We now have the following Lemma describing these subspaces, which is an analogue of \cite[Lemma~3.3.6]{BG}.
\begin{lemma}{\label{phifibers}}
For tuples $(\nu_{i})_{i \in I},(\nu'_{i})_{i \in I} \in (\cochar)^{I}$ lying in the span of the positive coroots, geometric dominant cocharacters $(\lambda_{i})_{i \in I} \in (\mathbb{X}_{*}(T_{\ol{\mathbb{Q}}_{p}})^{+})^{I}$, and a geometric point $x \ra \Div^{I}_{E}$, the following is true.
\begin{enumerate}
    \item There is an isomorphism 
    \[ _{x}Z^{I,\ast,(\nu'_{i})_{i \in I}}_{(\lambda_{i})_{i \in I}} \simeq
    \phantom{}_{x}\Gr^{I}_{G,\leq (\lambda_{i})_{i \in I},E} \times^{L^{+}B} \phantom{}_{x}\mathcal{B}^{(\nu'_{i})_{i \in I}} \]
    of stacks over $\phantom{}_{x,(\nu'_{i})_{i \in I}}\ol{\Bun}_{B}$, where the $L^{+}B$ action on $_{x}\Gr^{I}_{G,\leq (\lambda_{i})_{i \in I}}$ is given by the inclusion $L^{+}B \hookrightarrow L^{+}G$.
    \item Under the identification in (1), the substack $\phantom{}_{x}Z^{I,(\nu_{i})_{i \in I},(\nu'_{i})_{i \in I}}_{(\lambda_{i})_{i \in I}} \hookrightarrow \phantom{}_{x}Z^{I,\ast,(\nu'_{i})_{i \in I}}_{(\lambda_{i})_{i \in I}}$ identifies with the substack
    \[ (\phantom{}_{x}\Gr^{I}_{G,\leq (\lambda_{i})_{i \in I},E} \cap \phantom{}_{x}\mathrm{S}^{I}_{G,(w_{0}(\lambda_{i}) + \nu_{i} - \nu'_{i})_{i \in I},E}) \times^{L^{+}B} \phantom{}_{x}\mathcal{B}^{(\nu'_{i})_{i \in I}} \subset \phantom{}_{x}\Gr^{I}_{G,\leq (\lambda_{i})_{i \in I},E} \times^{L^{+}B} \phantom{}_{x}\mathcal{B}^{(\nu'_{i})_{i \in I}} \]
\item By (1), there is an isomorphism
\[
{}_x Z^{I,*,(\nu_i')_{i\in I}}
_{(-w_0(\lambda_i))_{i\in I}}
\simeq
{}_x\Gr^I_{G,\leq(-w_0(\lambda_i))_{i\in I},E}
\times^{L^+B}
{}_x\mathcal B^{(\nu_i')_{i\in I}}
\]
of stacks over ${}_{x,(\nu_i')_{i\in I}}\ol{\Bun}_B$.

Reversing modifications gives a natural morphism
\[
{}_x Z^{I,(\nu_i)_{i\in I},*}_{(\lambda_i)_{i\in I}}
\longrightarrow
{}_x Z^{I,*,(\nu_i)_{i\in I}}
_{(-w_0(\lambda_i))_{i\in I}},
\]
intertwining $\phi_{(\lambda_i)_{i\in I}}$ with
${}'h_G^{\rightarrow}$ for the inverse bound.
Using the preceding identification with
$(\nu_i')_{i\in I}$ replaced by $(\nu_i)_{i\in I}$,
pulling back the stratum
${}_{x,(\nu_i')_{i\in I}}\ol{\Bun}_B$
along the original map ${}'h_G^{\rightarrow}$
corresponds precisely to intersecting the Grassmannian
with
${}_x\mathrm S^I_{G,
(\nu_i'-\nu_i-w_0(\lambda_i))_{i\in I},E}$.
More precisely, the resulting substack identifies with
\[
\begin{aligned}
{}_x Z^{I,(\nu_i)_{i\in I},(\nu_i')_{i\in I}}
_{(\lambda_i)_{i\in I}}
&\simeq
(
{}_x\Gr^I_{G,\leq(-w_0(\lambda_i))_{i\in I},E}
\cap
{}_x\mathrm S^I_{G,
(\nu_i'-\nu_i-w_0(\lambda_i))_{i\in I},E})
\times^{L^+B}
{}_x\mathcal B^{(\nu_i)_{i\in I}}
\\
&\hookrightarrow
{}_x\Gr^I_{G,\leq(-w_0(\lambda_i))_{i\in I},E}
\times^{L^+B}
{}_x\mathcal B^{(\nu_i)_{i\in I}}
\\
&\simeq
{}_x Z^{I,*,(\nu_i)_{i\in I}}
_{(-w_0(\lambda_i))_{i\in I}}.
\end{aligned}
\]
\end{enumerate}
\end{lemma}
\begin{proof}
This is essentially an exercise in unwinding definitions, by using the moduli interpretation of the semi-infinite cells provided in Remark \ref{rem: semiinfinitecellinterpretation}. We sketch the proof. Fix $S \in \Perf_{\ol{\bb{F}}_{p}}$ and an $S$-point of ${}_{x,(\nu_i')_{i\in I}}\ol{\Bun}_{B}$ given by a triple $(\mathcal{F}_{T},\mathcal{F}_{G},\ol{\kappa})$.  For (1), an
$S$-point of
\[
{}_x Z^{I,*,(\nu_i')_{i\in I}}_{(\lambda_i)_{i\in I}}
\]
lying over the fixed $S$-point of ${}_{x,(\nu_i')_{i\in I}}\ol{\Bun}_{B}$  is given by the additional data of a modification
\[
\mathcal{F}_G \dashrightarrow \mathcal{F}_G'
\]
at the divisors $(D_{i})_{i \in I}$ whose meromorphy is bounded
by $(\lambda_i)_{i\in I}$.

Choosing an $S$-point of
${}_x\mathcal B^{(\nu_i')_{i\in I}}$ lying over the fixed point of $\phantom{}_{x,(\nu'_{i})_{i \in I}}\ol{\Bun}_{B}$ gives a trivialization of the enhanced $B$-structure on $\mathcal{F}_{G}$ in a
formal neighborhood of the divisors of $X_{S}$ determined by $x$ and in turn a trivialization of $\mathcal{F}_{G}$, by definition of ${}_x\mathcal B^{(\nu_i')_{i\in I}}$. With respect to this trivialization, the modification $\mathcal{F}_{G} \dashrightarrow \mathcal{F}_{G}'$ becomes exactly a point of the affine Grassmannian ${}_x\Gr^I_{G,\leq (\lambda_i)_{i\in I},E}$. However, we note that acting by $L^{+}B$ on a choice of trivialization changes the resulting point of ${}_x\Gr^I_{G,\leq (\lambda_i)_{i\in I},E}$ by the action of $L^{+}B \ra L^{+}G$ on ${}_x\Gr^I_{G,\leq (\lambda_i)_{i\in I}}$. Consequently, descending along the $L^+B$-torsor gives an isomorphism
\[ _{x}Z^{I,\ast,(\nu'_{i})_{i \in I}}_{(\lambda_{i})_{i \in I}} \xrightarrow{\simeq} 
    \phantom{}_{x}\Gr^{I}_{G,\leq (\lambda_{i})_{i \in I},E} \times^{L^{+}B} \phantom{}_{x}\mathcal{B}^{(\nu'_{i})_{i \in I}}, \]
as stacks over ${}_{x,(\nu_i')_{i\in I}}\ol{\Bun}_B$.

For (2), we consider a $S$-point of $\phantom{}_{x}Z^{I,*,(\nu'_{i})_{i \in I}}_{(\lambda_{i})_{i \in I}}$ lying over the fixed $S$-point of  ${}_{x,(\nu_i')_{i\in I}}\ol{\Bun}_{B}$. This means that the embeddings
\[ \ol{\kappa}^{\hat{\lambda}}: \mathcal{L}^{\hat{\lambda}}_{\mathcal{F}_{T}} \hookrightarrow \mathcal{V}^{\hat{\lambda}}_{\mathcal{F}_{G}}, \] 
have a zero of order $\nu'_{i}$ at $D_{i}$ and are regular everywhere else. In other words, for all $\hat{\lambda} \in \bb{X}^{*}(T_{\ol{\mathbb{Q}}_{p}})^{+}$, we have that the meromorphic map
\begin{equation}{\label{eqn: meromorphicmap}}
 \mathcal{L}^{\hat{\lambda}}_{\mathcal{F}_{T}}(\sum_{i \in I} \langle \nu'_{i}, \hat{\lambda} \rangle \cdot D_{i}) \dashrightarrow \mathcal{L}^{\hat{\lambda}}_{\mathcal{F}_{T}} \xrightarrow{\ol{\kappa}^{\hat{\lambda}}} \mathcal{V}^{\hat{\lambda}}_{\mathcal{F}_{G}} 
\end{equation}
does not have a zero or pole\footnote{Here we note that, as in Remark \ref{rem: semiinfinitecellinterpretation}, we have implicitly base-changed along $G_{E} \ra G$ and are allowing for all highest weight modules $\mathcal{V}^{\hat{\lambda}}$ attached to $\hat{\lambda} \in \bb{X}^{*}(T_{\ol{\bb{Q}}_{p}})^{+}$ where the maps are interpreted on $\bb{B}_{\mathrm{dR}}^{+}$-formal neighborhood of the divisors. We note that the maps $\kappa^{\mc{V}}$ are a priori only defined for a representation of $G$. However, any representation of $G_{E}$ may be realized as a direct summand of the base-change of the restriction of scalars of the representation from $E$ to $\bb{Q}_{p}$, allowing us to extend the definition to this case.}. On the other hand, the map $\phi$ was given by considering the enhanced $B$-structure 
\[ \ol{\kappa}^{'\hat{\lambda}}: \mathcal{L}^{\hat{\lambda}}_{\mathcal{F}_{T}}(\sum_{i \in I} \langle w_{0}(\lambda_{i}),\hat{\lambda} \rangle \cdot D_{i}) \hookrightarrow \mathcal{V}^{\hat{\lambda}}_{\mathcal{F}_{G}'} \]
induced by the modification $\mathcal{F}_{G}' \dashrightarrow \mathcal{F}_{G}$ parametrized by $\phantom{}_{x}Z^{I,?,(\nu'_{i})_{i \in I}}_{(\lambda_{i})_{i \in I}}$. If our $S$-point  maps to $\phantom{}_{x,(\nu_{i})_{i \in I}}\ol{\Bun}_{B}$ (i.e it lies in $\phantom{}_{x}Z^{I,(\nu_{i})_{i \in I},(\nu'_{i})_{i \in I}}_{(\lambda_{i})_{i \in I}}$) then this implies that the meromorphic map
\begin{equation}{\label{eqn: meromorphicmapII}}
\mathcal{L}^{\hat{\lambda}}_{\mathcal{F}_{T}}(\sum_{i \in I} \langle w_{0}(\lambda_{i}) + \nu_{i},\hat{\lambda} \rangle \cdot D_{i})  \dashrightarrow \mathcal{L}^{\hat{\lambda}}_{\mathcal{F}_{T}}(\sum_{i \in I} \langle w_{0}(\lambda_{i}),\hat{\lambda} \rangle \cdot D_{i}) \xrightarrow{\ol{\kappa}^{'\hat{\lambda}}} \mathcal{V}^{\hat{\lambda}}_{\mathcal{F}_{G}'} 
\end{equation}
does not have a zero or pole. Now, under the identification shown in (1), we choose a trivialization of the meromorphic map (\ref{eqn: meromorphicmap}) at the divisors $D_{i}$ and we note that a modification of $T$-bundles defined by embeddings $\mathcal{L}^{\hat{\lambda}}_{\mathcal{F}_{T}^{0}} \rightarrow \mathcal{L}^{\hat{\lambda}}_{\mathcal{F}_{T}}(\sum_{i \in I} \langle \nu'_{i}, \hat{\lambda} \rangle \cdot D_{i})$ for all $\hat{\lambda}$ necessarily has a zero of order $\nu'_{i}$ at $D_{i}$ and is regular everywhere else. The regularity of the maps (\ref{eqn: meromorphicmap}) and (\ref{eqn: meromorphicmapII}) therefore implies that the map
\[ \mathcal{L}^{\hat{\lambda}}_{\mathcal{F}_{T}^{0}}(\sum_{i \in I} \langle \nu_{i} - \nu'_{i} + w_{0}(\lambda_{i}), \hat{\lambda} \rangle \cdot D_{i}) \rightarrow \mathcal{V}^{\hat{\lambda}}_{\mathcal{F}_{G}'} \]
defined by the split $B$-structure on the trivial $G$-bundle and the modification $\mathcal{F}_{G}^{0} \dashrightarrow \mathcal{F}_{G}$ in $\phantom{}_{x}\Gr^{I}_{G,\leq (\lambda_{i})_{i \in I},E}$ corresponding to our $S$-point in $\phantom{}_{x}Z^{I,(\nu_{i})_{i \in I},(\nu'_{i})_{i \in I}}_{(\lambda_{i})_{i \in I}}$ does not have a zero or a pole. However, as explain in Remark \ref{rem: semiinfinitecellinterpretation}, this is precisely the condition saying the modification lies in the strata $\phantom{}_{x}\mathrm{S}^{I}_{G,(w_{0}(\lambda_{i}) + \nu_{i} - \nu'_{i})_{i \in I},E}$, as desired. Point (3) follows by the exact same reasoning as (1) and (2), where one trivializes the bundle $\mathcal{F}_{G}'$ in the modification $\mathcal{F}_{G} \dashrightarrow \mathcal{F}_{G}'$ instead of $\mathcal{F}_{G}$.
\end{proof}
We now apply this diagram to prove the filtered eigensheaf property.
\section{The Filtered Eigensheaf Property}
\subsection{Proof of the Filtered Eigensheaf Property}
We would now like the describe how Hecke correspondences on $\Bun_{G}$ interact with the Eisenstein functor. This will be used to show the Hecke eigensheaf property. Our analysis is heavily inspired by \cite[Section~3]{BG}, where an analogous claim is proven in the classical case. However, unlike the arguments there, we cannot appeal to the decomposition theorem, as the usual formalism of weights doesn't exist in this context. Nonetheless, we still have the excision spectral sequence, which will give us a filtration on the Eisenstein series. In \S \ref{consteigsheaf}, we will show that, under the condition of $\mu$-regularity (Definition \ref{def: muregularity}), this filtration splits for the Hecke operator defined by $V_{\mu^{\Gamma}} \in \Rep_{\Lambda}(\phantom{}^{L}G)$ when applied to $\mathcal{F} = \mathcal{S}_{\phi_{T}}$. 

Our goal is the following Theorem. 
\begin{theorem}{\label{filteigensheaf}}
For $\mathcal{F} \in \D(\Bun_{T})$, $I$ a finite index set, and $V \in \Rep_{\Lambda}(\phantom{}^{L}G^{I})$, the sheaf $T_{V}(\Eis(\mathcal{F}))$ has a $W_{\mathbb{Q}_{p}}^{I}$-equivariant filtration indexed by $(\nu_{i})_{i \in I} \in (\gamorb)^{I}$ with graded isomorphic to
 \[
 \Ind_{W_{E_{(\nu_i)_{i\in I}}}}^{W_{\mathbb Q_p}^{I}}
 \left(
   \Eis(T_{(\tilde\nu_i)_{i\in I}}(\mathcal F))
   \otimes V((\tilde\nu_i)_{i\in I})
   \otimes \boxtimes_{i\in I}
   \Lambda(-\langle\hat\rho,\tilde\nu_i\rangle)
\right), \]
 
as sheaves in $\D(\Bun_{G})^{BW_{\mathbb{Q}_{p}}^{I}}$. Here $\tilde{\nu}_{i} \in \cochar$ is a choice of representative of $\nu_{i} \in \gamorb$, $W_{E_{(\nu_{i})_{i \in I}}} := \prod_{i \in I} W_{E_{\nu_{i}}}$ where $E_{\nu_{i}}$ is the reflex field of $\nu_{i}$, $T_{\tilde{\nu}_{i}}$ is the Hecke operator of Remark \ref{rem: HeckeOperatoronReflexField}, and $V((\tilde{\nu}_{i})_{i \in I})$ is the weight space of $V$ restricted to $\hat{T}^{I}$ equipped with its natural action of $W_{E_{(\nu_{i})_{i \in I}}}$ induced by the action of the Weil group on $V$. Here $\boxtimes_{i \in I} \Lambda(-\langle \hat{\rho},\tilde{\nu}_{i}\rangle)$ is the $W_{E_{(\nu_{i})_{i \in I}}}$-representation defined by the indicated power of the cyclotomic character on $W_{E_{\nu_{i}}}$. In the above, we endow the tensor products of $W_{E_{(\nu_{i})_{i \in I}}}$-representations with the diagonal $W_{E_{(\nu_{i})_{i \in I}}}$-action.

Moreover, the filtration is natural with respect to morphisms of finite index sets $I$ and the fusion product, morphisms of representations $V \in \Rep_{\Lambda}(\phantom{}^{L}G^{I})$, and compatible with respect to compositions and exterior tensor products in $V$. 
\end{theorem}
We first assume that the derived group of $G$  is simply connected. We let $E/\mathbb{Q}_{p}$ be an extension over which $G$ splits. We will first determine the Hecke correspondence base-changed to $E$ and abusively write $T_{V}(\Eis(\mathcal{F})) \in \D(\Bun_{G})^{BW_{E}^{I}}$ for the sheaf determined by this. The full structure of $T_{V}(\Eis(\mathcal{F}))$ as a $W_{\bb{Q}_{p}}^{I}$-representation will follow by descent, which we explain below.

The Hecke correspondence in this case is determined by a representation $V \in \Rep_{\Lambda}(\hat{G}^{I})$ of $I$-copies of the dual group. By taking direct sums and using that we have a natural isomorphism $T_{V \oplus W} \simeq T_{V} \oplus T_{W}$ induced by the analogous isomorphism $\mathcal{S}^{H}_{V \oplus W} \simeq \mathcal{S}_{V}^{H} \oplus \mathcal{S}_{W}^{H}$, we can assume without loss of generality that $V$ has a fixed central character. Choose $\lambda_i \in \domcochar$ such that
$\mu_i \leq \lambda_i$ for every weight
$(\mu_i)_{i \in I}$ of $V$. Such bounds exist because
the weights in each coordinate have the same central character. We let $\mathcal{S}_{V}^{H}$ be the $\Lambda$-valued sheaf on $\Hck^{I}_{G,\leq (\lambda_{i})_{i \in I},E}$ defined via Theorem \ref{geomsatake} and pulling back to the global Hecke stack, where we use the superscript to distinguish it from the analogous sheaves $\mathcal{S}_{V}$ on the affine Grassmannian appearing in \S \ref{geomsatakesection}. Our aim is to construct a $W_{E}^{I}$-equivariant filtration on the sheaf:
\[ T_{V}(\Eis(\mathcal{F})) =  (h^{\leftarrow}_{G} \times \pi)_{!}(h_{G}^{\rightarrow*}(\mf{p}_{!}\mf{q}^{*}(\mathcal{F})[\dim(\Bun_{B})]) \otimes \mathcal{S}_{V}^{H}) \in \D(\Bun_{G})^{BW_{E}^{I}}. \]
This will be accomplished by contemplating the diagram
\[ 
\begin{tikzcd}
& \ol{\Bun}_{B} \times \Div^{I}_{E} \arrow[d,"\ol{\mf{p}} \times \mathrm{id}"] & \arrow[l,"\phi_{(\lambda_{i})_{i \in I}}"] Z^{I}_{(\lambda_{i})_{i \in I}} \arrow[r,"'h_{G}^{\rightarrow}"] \arrow[d,"'\mf{p}"] & \Bun_{B}  \arrow[d,"\mathfrak{p}"] \\
& \Bun_{G} \times \Div^{I}_{E} & \arrow[l,"h_{G}^{\leftarrow} \times \pi"] \Hck^{I}_{G,\leq (\lambda_{i})_{i \in I},E} \arrow[r,"h_{G}^{\rightarrow}"] & \Bun_{G} 
\end{tikzcd}
 \]
as defined in \S \ref{keydiagram} (\ref{eqn: keydiagram}). Using base-change on the right Cartesian square, we get an isomorphism  
\[ T_{V}(\Eis(\mathcal{F})) \simeq (h^{\leftarrow}_{G} \times \pi)_{!}(\phantom{}^{'}\mf{p}_{!}\phantom{}^{'}h_{G}^{\rightarrow*}\mf{q}^{*}(\mathcal{F})[\dim(\Bun_{B})] \otimes \mathcal{S}_{V}^{H}) \]
but, applying the projection formula with respect to $\mf{p}'$, this becomes
\[ T_{V}(\Eis(\mathcal{F})) \simeq  (h^{\leftarrow}_{G} \times \pi)_{!}\phantom{}^{'}\mf{p}_{!}('h_{G}^{\rightarrow*}(\mf{q}^{*}(\mathcal{F})[\dim(\Bun_{B})]) \otimes \phantom{}^{'}\mf{p}^{*}(\mathcal{S}_{V}^{H})) \]
We define
\[ K_{V} :=  \phantom{}^{'}h_{G}^{\rightarrow*}(\Lambda[\dim(\Bun_{B})]) \otimes \phantom{}^{'}\mf{p}^{*}(\mathcal{S}_{V}^{H}) \]
allowing us to rewrite our formula nicely as
\begin{equation}{\label{eqn: HeckeEigensheafequation}}
T_{V}(\Eis(\mathcal{F})) \simeq  (h^{\leftarrow}_{G} \times \pi)_{!} \phantom{}^{'}\mf{p}_{!}('h_{G}^{\rightarrow*}\mf{q}^{*}(\mathcal{F}) \otimes K_{V}) =  (\ol{\mf{p}} \times \mathrm{id})_{!}\phi_{(\lambda_{i})_{i \in I}!}('h_{G}^{\rightarrow*}\mf{q}^{*}(\mathcal{F}) \otimes K_{V})
\end{equation}

Now we would like to reduce the claim to applying excision to $\phi_{(\lambda_{i})_{i \in I}!}(K_{V})$ with respect to a locally closed stratification of $\ol{\Bun}_{B} \times \Div^{I}_{E}$. The claim should then follow from Corollary \ref{highweightcohom} and Lemma \ref{phifibers}. In order to do this, let's further rewrite the formula. In the following defect-indexed formulas, the $\nu_i$ range over the monoid generated by the positive coroots; terms outside this monoid are understood to be zero. We recall that, for such $(\nu_{i})_{i \in I}$, we have a map
\[ i_{(\nu_{i})_{i \in I}}: \ol{\Bun}_{B} \times \Div^{I}_{E} \rightarrow \ol{\Bun}_{B} \times \Div^{I}_{E} \]
sending the tuple $(\mathcal{F}_{G},\mathcal{F}_{T},\ol{\kappa}^{\hat{\lambda}},(D_{i})_{i \in I})$ to the tuple $(\mathcal{F}_{G},\mathcal{F}_{T}(-\sum_{i \in I} \nu_{i} \cdot D_{i}),\mathcal{L}^{\hat{\lambda}}_{\mathcal{F}_{T}}(- \sum_{i \in I} \langle \nu_{i}, \hat{\lambda} \rangle \cdot D_{i}) \hookrightarrow \mathcal{L}^{\hat{\lambda}}_{\mathcal{F}_{T}} \hookrightarrow \mathcal{V}^{\hat{\lambda}}_{\mathcal{F}_{G}},(D_{i})_{i \in I})$. We also have the map
\[ h_{(\nu_{i})_{i \in I}}^{\ra}: \Bun_{T} \times \Div^{I}_{E} \simeq \Hck^{I}_{T,(\nu_{i})_{i \in I}} \rightarrow \Bun_{T} \]
as in \S \ref{GCFT} which is given by modifying a $T$-bundle by $(\nu_{i})_{i \in I}$ at a tuple of divisors $(D_{i})_{i \in I}$ defining a point in $\Div^{I}_{E}$. Then, for $\mathcal{F} \in \D(\Bun_{T})$, we recall that we have an identification $(h_{(\nu_{i})_{i \in I}}^{\ra})^{*}(\mathcal{F}) = T_{(\nu_{i})_{i \in I}}(\mathcal{F})$ of sheaves in $\D(\Bun_{T})^{BW_{E}^{I}}$. Now, we can verify the following easy Lemma, which follows from the definition of $\phi_{(\lambda_{i})_{i \in I}}$.
\begin{lemma}{\label{lemma: factormaps}}
The following is true.
\begin{enumerate}
    \item The maps $h_{(w_{0}(\lambda_{i}))_{i \in I}}^{\ra} \circ (\ol{\mf{q}} \times \mathrm{id}) \circ \phi_{(\lambda_{i})_{i \in I}}$ and $\mf{q} \circ \phantom{}^{'}h_{G}^{\rightarrow}$ from $Z^{I}_{(\lambda_{i})_{i \in I}}$ to $\Bun_{T}$ coincide.  
    \item For every tuple $(\nu_{i})_{i \in I}$ of such effective defects, the maps $(\ol{\mf{q}} \times \mathrm{id}) \circ i_{(\nu_{i})_{i \in I}}$ and $(h_{(\nu_{i})_{i \in I}}^{\ra} \times \mathrm{id}) \circ (\ol{\mf{q}} \times \mathrm{id})$ from $\ol{\Bun}_{B} \times \Div^{I}_{E}$ to $\Bun_{T} \times \Div^{I}_{E}$ coincide.
\end{enumerate}
\end{lemma}
Now, with this in hand, let's revisit equation (\ref{eqn: HeckeEigensheafequation}):
\[ T_{V}(\Eis(\mathcal{F})) \simeq (\ol{\mf{p}} \times \mathrm{id})_{!} \phi_{(\lambda_{i})_{i \in I}!}('h_{G}^{\rightarrow*}\mf{q}^{*}(\mathcal{F}) \otimes K_{V}) \]
Using Lemma \ref{lemma: factormaps} (1), we have that 
\[ 'h_{G}^{\rightarrow*}\mf{q}^{*}(\mathcal{F}) \simeq \phi_{(\lambda_{i})_{i \in I}}^{*}(\ol{\mf{q}} \times \mathrm{id})^{*}(h_{(w_{0}(\lambda_{i}))_{i \in I}}^{\ra})^{*}(\mathcal{F}) \]
substituting this in and applying projection formula with respect to $\phi_{(\lambda_{i})_{i \in I}}$, we can rewrite the RHS as
\[(\ol{\mf{p}} \times \mathrm{id})_{!}((\ol{\mf{q}} \times \mathrm{id})^{*}(h_{(w_{0}(\lambda_{i}))_{i \in I}}^{\ra})^{*}(\mathcal{F}) \otimes \phi_{(\lambda_{i})_{i \in I}!}(K_{V})) \]
Now we claim that we have the following description of $\phi_{(\lambda_{i})_{i \in I}!}(K_{V})$.
\begin{theorem}{\label{phifilt}}
The sheaf $\phi_{(\lambda_{i})_{i \in I}!}(K_{V}) \in \D(\ol{\Bun}_{B})^{BW_{E}^{I}}$ has a $W_{E}^{I}$-equivariant filtration indexed by tuples $(\nu_{i})_{i \in I}$ in the monoid generated by the positive coroots. The graded pieces are isomorphic to

 \[
(i_{(\nu_i)_{i\in I}})_!(j\times\mathrm{id})_!
(\Lambda[\dim(\Bun_B)])
\otimes V((w_0(\lambda_i)+\nu_i)_{i\in I})
\otimes\boxtimes_{i\in I}
\Lambda(-\langle\hat\rho,w_0(\lambda_i)+\nu_i\rangle).
\]  

\end{theorem}
Assuming this for now, we get that $T_{V}(\Eis(\mathcal{F}))$ has a $W_{E}^{I}$-equivariant filtration with graded pieces given by
\begin{equation*}
\begin{split}
 (\ol{\mf{p}} \times \mathrm{id})_{!}((\ol{\mf{q}} \times \mathrm{id})^{*}h_{(w_{0}(\lambda_{i}))_{i \in I}}^{\ra*}(\mathcal{F}) \otimes i_{(\nu_{i})_{i \in I}!}(j \times \mathrm{id})_{!}(\Lambda[\dim(\Bun_{B})])) \otimes & \\ V((w_0(\lambda_i)+\nu_i)_{i \in I})
\otimes \boxtimes_{i \in I}
\Lambda(-\langle\hat{\rho},w_0(\lambda_i)+\nu_i\rangle)
\end{split}
\end{equation*}
Applying projection formula with respect to $i_{(\nu_{i})_{i \in I}}$, we obtain
\begin{equation*}
\begin{split}
  (\ol{\mf{p}} \times \mathrm{id})_{!}i_{(\nu_{i})_{i \in I}!}(i_{(\nu_{i})_{i \in I}}^{*}(\ol{\mf{q}} \times \mathrm{id})^{*}h_{(w_{0}(\lambda_{i}))_{i \in I}}^{\ra*}(\mathcal{F}) \otimes (j \times \mathrm{id})_{!}(\Lambda[\dim(\Bun_{B})])) \otimes &\\ V((w_0(\lambda_i)+\nu_i)_{i \in I})
\otimes \boxtimes_{i \in I}
\Lambda(-\langle\hat{\rho},w_0(\lambda_i)+\nu_i\rangle)
\end{split}
\end{equation*}
but now, by Lemma \ref{lemma: factormaps} (2), we have
\[ i_{(\nu_{i})_{i \in I}}^{*}(\ol{\mf{q}} \times \mathrm{id})^{*}h_{(w_{0}(\lambda_{i}))_{i \in I}}^{\ra*}(\mathcal{F}) \simeq  (\ol{\mf{q}} \times \mathrm{id})^{*}h^{\ra*}_{(w_{0}(\lambda_{i}) + \nu_{i})_{i \in I}}(\mathcal{F}) \simeq (\ol{\mf{q}} \times \mathrm{id})^{*}(T_{(w_{0}(\lambda_{i}) + \nu_{i})_{i \in I}}(\mathcal{F}))   \]
so substituting this into the previous formula we get 
\begin{equation*}
\begin{split}
 (\ol{\mf{p}} \times \mathrm{id})_{!}i_{(\nu_{i})_{i \in I}!}((\ol{\mf{q}} \times \mathrm{id})^{*}(T_{(w_{0}(\lambda_{i}) + \nu_{i})_{i \in I}}(\mathcal{F})) \otimes (j \times \mathrm{id})_{!}(\Lambda[\dim(\Bun_{B})])) \otimes & \\  V((w_0(\lambda_i)+\nu_i)_{i \in I})
\otimes \boxtimes_{i \in I}
\Lambda(-\langle\hat{\rho},w_0(\lambda_i)+\nu_i\rangle)
\end{split}
\end{equation*}
Now $i_{(\nu_{i})_{i \in I}}$ does nothing to the $G$-bundle $\mathcal{F}_{G}$ and the copy of $\Div^{I}_{E}$. Therefore, this becomes
\begin{equation*}
\begin{split}
 (\ol{\mf{p}} \times \mathrm{id})_{!}((\ol{\mf{q}} \times \mathrm{id})^{*}(T_{(w_{0}(\lambda_{i}) + \nu_{i})_{i \in I}}(\mathcal{F})) \ \otimes (j \times \mathrm{id})_{!}(\Lambda[\dim(\Bun_{B})])) & \\ \otimes  V((w_0(\lambda_i)+\nu_i)_{i \in I})
\otimes \boxtimes_{i \in I}
\Lambda(-\langle\hat{\rho},w_0(\lambda_i)+\nu_i\rangle) 
\end{split}
\end{equation*}
which is just
\[  \Eis_{B}(T_{(w_{0}(\lambda_{i}) + \nu_{i})_{i \in I}}(\mathcal{F})) \otimes V((w_0(\lambda_i)+\nu_i)_{i \in I})
\otimes \boxtimes_{i \in I}
\Lambda(-\langle\hat{\rho},w_0(\lambda_i)+\nu_i\rangle) \]
by an application of projection formula to $j \times \mathrm{id}$.  Every weight $(\mu_i)_{i \in I}$ of $V$ satisfies
$w_0(\lambda_i) \leq \mu_i$, so reindexing by
$\mu_i=w_0(\lambda_i)+\nu_i$ gives the desired result. Thus, to construct the filtration all we have to do is to prove Theorem \ref{phifilt}. 

\begin{proof}{(Theorem \ref{phifilt})}
For a tuple $(\nu_{i})_{i \in I}$ of effective defects, let $\ol{\nu} := \sum_{i \in I} \nu_{i\Gamma}$. Consider the locally closed stratum $_{\ol{\nu}}\ol{\Bun}_{B} \times \Div^{I}_{E} \simeq \Bun_{B} \times \Div^{(\ol{\nu})} \times \Div^{I}_{E} \subset \ol{\Bun}_{B} \times \Div^{I}_{E}$. We have the natural projection
\[ p_{1}: \phantom{}_{\ol{\nu}}\ol{\Bun}_{B} \times \Div^{I}_{E} \simeq \Bun_{B} \times \Div^{(\ol{\nu})} \times \Div^{I}_{E} \ra \Div^{(\ol{\nu})}, \]
as well as the map 
\[ p_{2}: \phantom{}_{\ol{\nu}}\ol{\Bun}_{B} \times \Div^{I}_{E} \ra \Div^{I}_{E} \xrightarrow{\Delta_{(\nu_{i})_{i \in I}}} \Div^{(\ol{\nu})}, \]
where the first map is the natural projection and the second map is the natural projection followed by $\Delta_{(\nu_{i})_{i \in I}}$, as defined in \S \ref{keydiagram}. We note that $\Div^{(\ol{\nu})}$ is proper using \cite[Proposition~II.1.21]{FS} and in particular separated. It follows that if we let $\phantom{}_{(\nu_{i})_{i \in I}}(\ol{\Bun}_{B} \times \Div^{I}_{E}) \ra \phantom{}_{\ol{\nu}}\ol{\Bun}_{B} \times \Div^{I}_{E}$ be the pullback of the  morphism $\Delta_{\Div^{(\ol{\nu})}}: \Div^{(\ol{\nu})} \ra \Div^{(\ol{\nu})} \times \Div^{(\ol{\nu})}$ along $(p_{1},p_{2})$ that $\phantom{}_{(\nu_{i})_{i \in I}}(\ol{\Bun}_{B} \times \Div^{I}_{E}) \ra \phantom{}_{\ol{\nu}}\ol{\Bun}_{B} \times \Div^{I}_{E}$ is a closed immersion. Therefore, we see that the composite $\phantom{}_{(\nu_{i})_{i \in I}}(\ol{\Bun}_{B} \times \Div^{I}_{E}) \ra \phantom{}_{\ol{\nu}}\ol{\Bun}_{B} \times \Div^{I}_{E} \ra \ol{\Bun}_{B} \times \Div^{I}_{E}$ is a locally closed immersion parameterizing $(\ol{\kappa}^{\hat{\lambda}}: (\mathcal{L}^{\hat{\lambda}})_{\mathcal{F}_{T}} \ra (\mathcal{V}^{\hat{\lambda}})_{\mathcal{F}_{G}}, \hat{\lambda} \in \hat{\Lambda}_{G}^{+},(D_{i})_{i \in I})$ such that $\Coker(\ol{\kappa}^{\hat{\lambda}})$ has torsion of length $\langle \hat{\lambda}, \nu_{i\Gamma} \rangle$ at $D_{i}$ for all $\hat{\lambda} \in \hat{\Lambda}_{G}^{+}$ and $i \in I$ (adding together the multiplicities when the $D_{i}$ are equal) and does not have torsion away from these divisors. If we let $j_{(\nu_{i})_{i \in I}} := i_{(\nu_{i})_{i \in I}} \circ (j \times \mathrm{id})$ then we see that this maps isomorphically onto $\phantom{}_{(\nu_{i})_{i \in I}}(\ol{\Bun}_{B} \times \Div^{I}_{E})$, and for varying effective defect tuples $(\nu_{i})_{i \in I}$ these form  locally closed subspaces of $\ol{\Bun}_{B} \times \Div^{I}_{E}$. We now apply excision (see \cite[Section~6.3.1]{gleason2026schematicanalyticconstructionslocal} and \cite[Section~2.4.1]{caraiani2026intersectioncohomologyigusastacks} for why this exists in the present context) with respect to the locally closed stratification $j_{\ol{\nu}}: \phantom{}_{\ol{\nu}}\ol{\Bun}_{B} \times \Div^{I}_{E} \ra \ol{\Bun}_{B} \times \Div^{I}_{E}$ for varying $\ol{\nu} \in \Lambda_{G,B}^{\mathrm{pos}}$ and the sheaf $\phi_{(\lambda_{i})_{i \in I}!}K_{V}$. Since the maps $j_{\ol{\nu}}$ are defined over the projection to $\Div^{I}_{E}$ this gives a $W_{E}^{I}$-equivariant filtration with graded equal to
\[ j_{\ol{\nu}!}j_{\ol{\nu}}^{*}\phi_{(\lambda_{i})_{i \in I}!}K_{V}. \]

Now, to describe this, we use Lemma~\ref{phifibers}(3) (or rather its proof) with
$(\nu_i')_{i\in I}=0$. Put
$\mu_i=w_0(\lambda_i)+\nu_i$. After trivializing the
saturated target $B$-bundle, this tells us that the fibers of
$\phi_{(\lambda_i)_{i\in I}}$ embed into  
\[
 \Gr^I_{G,\leq(-w_0(\lambda_i))_{i\in I},E} \] 
over $\Div^{I}_{E}$. Moreover, if we further pullback along $j_{(\nu_{i})_{i \in I}}$, this corresponds to looking at the locally closed subspace
\[
 \Gr^I_{G,\leq(-w_0(\lambda_i))_{i\in I},E}
 \cap
 \mathrm S^I_{G,(-\mu_i)_{i\in I},E} \ra \Div^{I}_{E}.
\]
In particular, we see that pulling back along the strata $\phantom{}_{\ol{\nu}}\ol{\Bun}_{B} \times \Div^{I}_{E} \ra \ol{\Bun}_{B}$ corresponds to the (locally closed in $\Gr^{I}_{G,\leq (-w_{0}(\lambda_{i}))_{i \in I},E}$) subspace
\[
 \Gr^I_{G,\leq(-w_0(\lambda_i))_{i\in I},E}
 \cap
 \mathrm S^I_{G,-\mu_{\Gamma}} \ra \Div^{I}_{E}.
\]
Here $\mu_{\Gamma} := \ol{\nu} + \sum_{i \in I} w_{0}(\lambda_{i\Gamma})$, and this denotes the locus in the affine Grassmannian where the induced modification of the trivial $T$-bundle as in Remark \ref{rem: semiinfinitecellinterpretation} has isomorphism class given by $-\mu_{\Gamma}$ after pulling back to geometric points. Therefore, the $!$-pushforward of $\mathcal{S}_{V^{\sigma}}$ restricted to $\phantom{}_{\ol{\nu}}\ol{\Bun}_{B} \times \Div^{I}_{E} \ra \ol{\Bun}_{B}$   identifies with the contribution to the constant-term functor of \S \ref{geomsatake} coming from the weights of $\hat{T}$, which, when further restricted to $\hat{T}^{\Gamma}$ are given by $-\mu_{\Gamma} \in \coinv \simeq \bb{X}^{*}(\hat{T}^{\Gamma})$. Using this description of $\phi_{(\lambda_{i})_{i \in I}}$ in families over $\Div^{I}_{E}$ and descending along the torsor of trivializations  of the $B$-bundle using the $L^{+}G$-equivariance of $\mathcal{S}_{V^{\sigma}}$ and the $L^{+}B$-equivariance of the constant term (cf. \cite[Proposition~VI.7.2]{FS}), we may therefore apply the direct sum decomposition (\ref{eqn: directsumdecomposition}) to $V^\sigma$ in order to deduce that there is a $W_{E}^{I}$-equivariant decomposition
\[ j_{\ol{\nu}!}j_{\ol{\nu}}^{*}\phi_{(\lambda_{i})_{i \in I}!}K_{V} \simeq \bigoplus_{\substack{(\nu_{i})_{i \in I} \in \cochar^{I} \\ \sum_{i \in I} \nu_{i\Gamma} = \ol{\nu}}} Q_{(\nu_{i})_{i \in I}}, \]
where $Q_{(\nu_{i})_{i \in I}}$ corresponds to the contribution of the weight space $-(\mu_{i})_{i \in I}$ to $V^{\sigma}|_{\hat{T}^{I}}$. In particular, by definition we have that $j_{(\nu_{i})_{i \in I}!}j_{(\nu_{i})_{i \in I}}^{*}Q_{(\nu_{i})_{i \in I}} \xrightarrow{\simeq} Q_{(\nu_{i})_{i \in I}}$, since this contribution comes from the substack $\Gr^I_{G,\leq(-w_0(\lambda_i))_{i\in I},E}
 \cap
 \mathrm S^I_{G,(-\mu_i)_{i\in I},E} \ra \Gr^I_{G,\leq(-w_0(\lambda_i))_{i\in I},E}$, which corresponds to the pullback along $j_{(\nu_{i})_{i \in I}}$, as described above. 

Here, we recall that, by definition we have
\[ K_{V} := \phantom{}^{'}\mf{p}^{*}(\mathcal{S}_{V}^{H}) \otimes ('h_{G}^{\rightarrow})^{*}(\Lambda[\dim(\Bun_{B})]). \]
To deal with the shifts by $\dim(\Bun_{B})$, we write $d_{\mathrm{source}}$ for $\ell$-dimension of the connected component of $\Bun_{B}$ that $\phantom{}^{'}h_{G}^{\ra}$ maps to, and we write $d_{\mathrm{target}}$ for $\ell$-dimension of the connected component of $\Bun_{B}$ defined by the saturations of the maps $\ol{\kappa}^{\hat{\lambda}'}$ defining $\phi_{(\lambda_{i})_{i \in I}}$. One easily checks using Proposition \ref{prop: qsmooth} that we have a relationship: 
\begin{equation}{\label{eqn: relationshipoftheconnectedcomponent}}
d_{\mathrm{source}} = d_{\mathrm{target}} + a,
\end{equation}
where $a = \sum_{i \in I} \langle 2\hat{\rho},\mu_{i} \rangle$ for $\mu_{i} = w_{0}(\lambda_{i}) + \nu_{i}$.

We now determine the $W_{E}^{I}$-action on the graded pieces. To do this, we can consider the pullback to a geometric point $x = \Spa(C,\mathcal{O}_{C}) \ra \Div^{I}_{E}$ corresponding to pairwise disjoint Cartier divisors, where we can regard $x$ as a point of $\Div^{(\ol{\nu})}$, as in \S \ref{keydiagram}. By Lemma \ref{phifibers} (3), the pullback to $x$ identifies with
\[ R\Gamma_{c}(\phantom{}_{x}\Gr^{I}_{G,\leq (-w_{0}(\lambda_{i}))_{i \in I},E} \cap  \phantom{}_{x}\mathrm{S}_{G,(-\nu_{i} - w_{0}(\lambda_{i}))_{i \in I},E}, \mathcal{S}_{V^{\sigma}}|_{\phantom{}_{x}\Gr^{I}_{G,\leq (-w_{0}(\lambda_{i}))_{i \in I},E} \cap \phantom{}_{x}\mathrm{S}_{G,(-\nu_{i} -w_{0}(\lambda_{i}))_{i \in I},E}})[d_{\mathrm{source}}] \]
where $(-)^{\sigma}: \Rep(\hat{G}^{I}) \ra \Rep(\hat{G}^{I})$ is the map induced by the Chevalley involution\footnote{Here we note the importance of distinguishing between the sheaf $\mathcal{S}_{V}^{H}$ on the Hecke stack and the sheaf $\mathcal{S}_{V}$ on the affine Grassmannian. Indeed, if we consider $\Hck_{G,\leq \mu}$ as a moduli space of modifications $\mathcal{E}_{1} \dashrightarrow \mathcal{E}_{2}$ of meromorphy $\leq \mu$ then if we rigidify the $G$-bundle defined by $\mathcal{E}_{1}$ to be trivial we obtain the Grassmannian $\Gr_{G,\leq \mu}$. On the other hand, if we rigidify $\mathcal{E}_{2}$ then we obtain $\Gr_{G,\leq -w_{0}(\mu)}$, the affine Grassmannian attached to the dominant inverse. However, the natural sheaves on $\Hck_{G,\leq \mu}$ furnished by geometric Satake will correspond to the highest weight modules with weight $\leq \mu$. }.

Now, Corollary \ref{highweightcohom}, shows that $j_{(\nu_{i})_{i \in I}}^{*}Q_{(\nu_{i})_{i \in I}}$ identifies with
\[ V^\sigma((-w_0(\lambda_i)-\nu_i)_{i \in I})
\otimes \boxtimes_{i \in I}
\Lambda(-\langle\hat{\rho},w_0(\lambda_i)+\nu_i\rangle)[-\langle 2\hat{\rho}, \sum_{i \in I} (w_{0}(\lambda_{i}) + \nu_{i}) \rangle + d_{\mathrm{source}}], \]
which is equal to
\[ V^\sigma((-w_0(\lambda_i)-\nu_i)_{i \in I})
\otimes \boxtimes_{i \in I}
\Lambda(-\langle\hat{\rho},w_0(\lambda_i)+\nu_i\rangle)[-a + d_{\mathrm{source}}], \]
as desired. In general, if the derived group of $G$ is not simply connected, we may argue in the same way using the stratification of $\ol{\Bun}_{B}$ of Remark \ref{rem: conjugatingbythecenter}. In particular, if we choose a central cover $0 \ra Z \ra \tilde{G} \ra G \ra 0$, where $\tilde{G}$ has simply connected derived group. We can use that this stratification pulls back to the strata on $\ol{\Bun}_{\tilde{B}}$ of the pulled back Borel. It is easily verified  that the resulting description given by the above analysis descends along the $\Bun_{Z}$ torsor, giving us the claim in this case.

We now briefly explain how one obtains the full $W_{\bb Q_p}^{I}$-action from the calculation over the splitting field. The defect stratification used above is indexed by the coinvariant lattice, so the excision filtration is compatible with this descent datum.
In the preceding calculation, reindex
the graded pieces of $T_V(\Eis(\mathcal F))$ by the weights
$\mu_i=w_0(\lambda_i)+\nu_i$.
The descent datum carries the piece indexed by
$(\mu_i)_{i\in I}$ to the piece indexed by its Galois conjugate.
For an orbit $(\nu_i)_{i\in I}\in(\gamorb)^I$, choose a
representative $(\tilde\nu_i)_{i\in I}\in(\cochar)^I$.
The graded piece
\[
\Eis(T_{(\tilde\nu_i)_{i\in I}}(\mathcal F))
\otimes V((\tilde\nu_i)_{i\in I})
\otimes\boxtimes_{i\in I}
\Lambda(-\langle\hat\rho,\tilde\nu_i\rangle)
\]
is preserved by its stabilizer, and hence descends to
$\prod_{i\in I}\Div^1_{E_{\tilde\nu_i}}$.

Now when we compute the full Hecke operator as a $W_{\bb{Q}_{p}}^{I}$-representation this contribution to the filtration will receive a further pushforward along
\[
 \prod_{i\in I}\Div^1_{E_{\tilde\nu_i}}
 \longrightarrow
 \Div^I
\]
giving that the summand is ultimately
\[
\mathrm{Ind}_{W_{E_{(\nu_i)_{i\in I}}}}^{W_{\bb Q_p}^{I}}
(\Eis(T_{(\tilde\nu_i)_{i\in I}}(\mathcal F))
\otimes V((\tilde\nu_i)_{i\in I})
\otimes\boxtimes_{i\in I}
\Lambda(-\langle\hat\rho,\tilde\nu_i\rangle)).
\]
as a sheaf with $W_{\bb{Q}_{p}}^{I}$-action, as desired.
\end{proof} 
It remains to see that this filtration satisfies the desired
compatibilities. We work with the Hecke correspondence base-changed
to $E$ as before. The general claim will follow by descent, as above.

We explain the proof of the compatibility for morphisms of finite
index sets with the proof of the other compatibilities being
analogous and strictly easier. Consider a map of finite index sets
$\pi: I \ra J$. For $j \in J$, we set $I_j := \pi^{-1}(j)$ and
consider the natural map
$\Delta_{IJ}: \Div^J_E \ra \Div^I_E$, which diagonally embeds the
$j$th copy of $\Div^1_E$ in $\Div^J_E$ into $\Div^{I_j}_E$.
Attached to this, we have a Cartesian diagram
\[
\begin{tikzcd}
\ol{\Bun}_B \times \Div^J_E
  \arrow[d,"\mathrm{id} \times \Delta_{IJ}"]
&
Z^J_{(\lambda_j)_{j \in J}}
  \arrow[l,"\phi_{(\lambda_j)_{j \in J}}"]
  \arrow[d,"\tilde{\Delta}_{IJ}"]
\\
\ol{\Bun}_B \times \Div^I_E
&
Z^I_{(\lambda_i)_{i \in I}}
  \arrow[l,"\phi_{(\lambda_i)_{i \in I}}"]
\end{tikzcd}
\]
where $\lambda_j := \sum_{i \in I_j}\lambda_i$ for all $j \in J$.
Base change gives us a natural isomorphism:
\begin{equation}\label{eqn: factorizationequation}
(\mathrm{id} \times \Delta_{IJ})^*
\phi_{(\lambda_i)_{i \in I}!}(K_V)
\simeq
\phi_{(\lambda_j)_{j \in J}!}
\tilde{\Delta}_{IJ}^*(K_V).
\end{equation}
However, by the relationship between fusion product and tensor
product under Theorem \ref{geomsatake}, we deduce that
$\tilde{\Delta}_{IJ}^*(K_V) \simeq K_{\Delta_{IJ}^*V}$,
where $\Delta_{IJ}^*V$ denotes restriction along the corresponding
map $\hat{G}^J \ra \hat{G}^I$.
In the following formulas, the $W_E^I$-representations appearing
over $\Div^J_E$ are understood to be restricted along the natural
map $W_E^J \ra W_E^I$.

We now compare the two filtrations on the LHS and the RHS of this
isomorphism. To do this, we define $(\nu_j)_{j \in J}$ by
$\nu_j := \sum_{i \in I_j}\nu_i$. We note that we have a natural
Cartesian diagram
\[
\begin{tikzcd}
\ol{\Bun}_B \times \Div^J_E
  \arrow[r,"i_{(\nu_j)_{j \in J}}"]
  \arrow[d,"\mathrm{id} \times \Delta_{IJ}"]
&
\ol{\Bun}_B \times \Div^J_E
  \arrow[d,"\mathrm{id} \times \Delta_{IJ}"]
\\
\ol{\Bun}_B \times \Div^I_E
  \arrow[r,"i_{(\nu_i)_{i \in I}}"]
&
\ol{\Bun}_B \times \Div^I_E.
\end{tikzcd}
\]
On the LHS of (\ref{eqn: factorizationequation}), we have a
filtration with graded pieces isomorphic to
\[
\begin{aligned}
&(\mathrm{id} \times \Delta_{IJ})^*
i_{(\nu_i)_{i \in I}!}
(j \times \mathrm{id})_!(\Lambda[\dim(\Bun_B)])
\\
&\qquad \otimes
V((w_0(\lambda_i)+\nu_i)_{i \in I})
\otimes
\boxtimes_{i \in I}
\Lambda(-\langle\hat{\rho},w_0(\lambda_i)+\nu_i\rangle),
\end{aligned}
\]
which is naturally isomorphic to
\[
\begin{aligned}
&i_{(\nu_j)_{j \in J}!}
(\mathrm{id} \times \Delta_{IJ})^*
(j \times \mathrm{id})_!(\Lambda[\dim(\Bun_B)])
\\
&\qquad \otimes
V((w_0(\lambda_i)+\nu_i)_{i \in I})
\otimes
\boxtimes_{i \in I}
\Lambda(-\langle\hat{\rho},w_0(\lambda_i)+\nu_i\rangle)
\end{aligned}
\]
by base-change applied to the previous Cartesian square.
We can further rewrite this as
\[
\begin{aligned}
&i_{(\nu_j)_{j \in J}!}
(j \times \mathrm{id})_!(\Lambda[\dim(\Bun_B)])
\\
&\qquad \otimes
V((w_0(\lambda_i)+\nu_i)_{i \in I})
\otimes
\boxtimes_{j \in J}
\Lambda(-\langle\hat{\rho},w_0(\lambda_j)+\nu_j\rangle).
\end{aligned}
\]
On the other hand, for such a $(\nu_j)_{j \in J}$, the RHS of
(\ref{eqn: factorizationequation}) has a filtration with graded
pieces isomorphic to
\[
\begin{aligned}
&i_{(\nu_j)_{j \in J}!}
(j \times \mathrm{id})_!(\Lambda[\dim(\Bun_B)])
\\
&\qquad \otimes
(\Delta_{IJ}^*V)((w_0(\lambda_j)+\nu_j)_{j \in J})
\otimes
\boxtimes_{j \in J}
\Lambda(-\langle\hat{\rho},w_0(\lambda_j)+\nu_j\rangle),
\end{aligned}
\]
but now note that
\[
(\Delta_{IJ}^*V)((w_0(\lambda_j)+\nu_j)_{j \in J})
\simeq
\bigoplus_{\substack{
  (\nu_i)_{i \in I} \in (\cochar)^I\\
  \sum_{i \in I_j}\nu_i=\nu_j\ \text{for all }j \in J}}
V((w_0(\lambda_i)+\nu_i)_{i \in I}).
\]
Therefore, the graded piece indexed by $(\nu_j)_{j \in J}$ on
the RHS of (\ref{eqn: factorizationequation}) has a split
filtration with graded pieces isomorphic to the graded pieces
coming from the filtration on the LHS of
(\ref{eqn: factorizationequation}).
The compatibility of these two filtrations now follows from
Proposition \ref{fusionconstant} and Corollary \ref{fusionconstant2},
and the fact that the filtration came from restricting the sheaf
$\mathcal{S}_V^H$ to semi-infinite cells.
 
 Now, we can reap the fruit of this section using the filtered eigensheaf property to get some control on the stalks of $\nmEis(\mathcal{S}_{\phi_{T}})$.
\subsection{Consequences of the Filtered Eigensheaf Property}
First, we note, by applying Theorem \ref{filteigensheaf} when $\mathcal{F} = \mathcal{S}_{\phi_{T}} \otimes \Delta_{B}^{1/2}$ together with Corollary \ref{twistedeigsheaf}, we obtain the following. 
\begin{corollary}{\label{appliedfilteigsheaf}}
For all finite index sets $I$ and $V = \boxtimes_{i \in I} V_{i} \in \Rep_{\Lambda}(\phantom{}^{L}G^{I})$, the sheaf $T_{V}(\nmEis(\mathcal{S}_{\phi_{T}}))$ admits a $W_{\mathbb{Q}_{p}}^{I}$-equivariant filtration indexed by $(\nu_{i})_{i \in I} \in (\gamorb)^{I}$. The filtration's graded pieces are isomorphic to $\nmEis(\mathcal{S}_{\phi_{T}}) \otimes \Ind_{W_{E_{(\nu_{i})_{i \in I}}}}^{W_{\bb{Q}_{p}}^{I}}(\boxtimes_{i \in I} (\tilde{\nu}_{i} \circ \phi_{T}|_{W_{E_{\nu_{i}}}}) \otimes V_{i}(\tilde{\nu}_{i}))$, where $V_{i}(\tilde{\nu}_{i})$ is equipped with the natural $W_{E_{\nu_{i}}}$-action induced the action of the Weil group on $V$, and $\tilde{\nu}_{i} \in \cochar$ is a choice of representative of $\nu_{i} \in \gamorb$. The filtration is natural with respect to morphisms of finite index sets $I$ and the fusion product, morphisms of representations $V \in \Rep(\phantom{}^{L}G^{I})$, as well as compatible with compositions and exterior tensor products in $V$. 
\end{corollary}
In particular, we note that the direct sum of the graded pieces of the filtration on $T_{V}(\nmEis(\mathcal{S}_{\phi_{T}}))$ is isomorphic to
\[ \nmEis(\mathcal{S}_{\phi_{T}}) \otimes \bigoplus_{(\nu_{i})_{i \in I} \in (\gamorb)^{I}} \Ind_{W_{E_{(\nu_{i})_{i \in I}}}}^{W_{\bb{Q}_{p}}^{I}}(\boxtimes_{i \in I} (\tilde{\nu}_{i} \circ \phi_{T}|_{W_{E_{\nu_{i}}}}) \otimes V_{i}(\tilde{\nu}_{i}))  \simeq  \nmEis(\mathcal{S}_{\phi_{T}}) \boxtimes r_{V} \circ \phi^{\boxtimes I}, 
\]
as sheaves in $\D(\Bun_{G})^{BW^{I}_{\mathbb{Q}_{p}}}$, where $\phi$ is the parameter $\phi_{T}$ composed with the natural embedding $\phantom{}^{L}T \ra \phantom{}^{L}G$. In other words, $\nmEis(\mathcal{S}_{\phi_{T}})$ is a filtered eigensheaf with eigenvalue $\phi$ and we use this terminology throughout. We now would like to use this to deduce some consequences about the stalks of the Eisenstein series $\nmEis(\mathcal{S}_{\phi_{T}})$.  

To this aim, in order to avoid complications in the $\ell$-modular setting, we will additionally assume for the rest of the paper that $\ell$ is very decent with respect to $G$ in the following sense:
\begin{enumerate}
\item We have that $\ell \nmid \pi_{0}(Z(G))$, where $Z(G)$ denotes the center of $G$.
\item We have that $\ell \nmid |Q|$, where $Q$ is the smallest quotient through which the $W_{\mathbb{Q}_{p}}$, the Weil group of $\mathbb{Q}_{p}$, acts on the Langlands dual group $\hat{G}$. 
\end{enumerate}
However, the second assumption will not be used in a meaningful way until \S \ref{s: TiltingEigensheaves} and is not required for the main results.

In particular, consider some Schur irreducible subquotient $A$ of some cohomology sheaf of $\nmEis(\mathcal{S}_{\phi_{T}})$. Under our assumption on $\ell$ (in particular the assumption that $\ell \nmid \pi_{0}(Z(G)$), the excursion algebra will define endomorphisms of $A$ which determine and are determined by the parameter $\phi_{A}^{\mathrm{FS}}$, as in \cite[Propositions~I.9.1,I.9.3]{FS}. Since the excursion algebra is determined by natural transformations of Hecke operators, the filtered Hecke eigensheaf property tells us that these scalars must be specified by $\phi$, so that we have an equality: $\phi = \phi_{A}^{\mathrm{FS}}$, as conjugacy classes of semi-simple parameters. In particular, if we consider $b \in B(G)$ and look at the restriction $\nmEis(\mathcal{S}_{\phi_{T}})|_{\Bun_{G}^{b}}$ to the locally closed HN-stratum $\Bun_{G}^{b} \subset \Bun_{G}$ indexed by $b$. Then, by \cite[Proposition~V.2.2]{FS}, we have a natural isomorphism $\D(\Bun_{G}^{b}) \simeq \D(J_{b}(\mathbb{Q}_{p}),\Lambda)$. By the previous discussion and compatibility of the Fargues-Scholze correspondence with restriction to $J_{b}$ \cite[Section~IX.7.1]{FS}, we deduce that any irreducible constituent $\rho$ of the restriction $\nmEis(\mathcal{S}_{\phi_{T}})|_{\Bun_{G}^{b}}$ has Fargues-Scholze parameter $\phi_{\rho}^{\mathrm{FS}}: W_{\mathbb{Q}_{p}} \ra \phantom{}^{L}J_{b}(\Lambda)$ equal to $\phi$ under the appropriately Tate twisted embedding:
\[ \phantom{}^{L}J_{b}(\Lambda) \ra \phantom{}^{L}G(\Lambda) \]
Now, since our parameter $\phi$ is induced from the maximal torus, we would like to say that this is impossible unless $J_{b}$ itself admits a Borel, which is in turn equivalent to assuming that $b$ is unramified. Here we need to be a bit careful. In particular, if we consider $G = \GL_{2}$ and $b$ the element of slope $\frac{1}{2}$ then $J_{b} = D^{*}_{\frac{1}{2}}$ the units in the quaternion division algebra. The trivial representation $\mathbf{1}$ of $D^{*}_{\frac{1}{2}}$ has Fargues-Scholze parameter given by 
\[ W_{\mathbb{Q}_{p}} \ra \GL_{2}(\Lambda) \]
\[ g \mapsto \begin{pmatrix} |g|^{1/2} & 0 \\ 0 & |g|^{-1/2} \end{pmatrix} \]
This parameter is induced from a maximal torus of $\GL_{2}$; however, it is not generic. In particular, the composite of this parameter with the unique simple root defined by the upper triangular Borel gives a Galois representation isomorphic to the norm/cyclotomic character $|\cdot|$. Therefore, one might hope that assuming compatibility of some suitably nice form of the local Langlands correspondence for $G$ with the Fargues-Scholze correspondence together with genericity of $\phi_{T}$ is enough to give us the desired description of the stalks. This is indeed the case. The assumption we need is as follows.
\begin{assumption}{\label{compatibility}}
For a connected reductive group $H/\mathbb{Q}_{p}$, we have
\begin{itemize} 
\item the set $\Pi(H)$ of smooth irreducible $\ol{\mathbb{Q}}_{\ell}$-representations of $H(\mathbb{Q}_{p})$,
\item the set $\Phi(H)$ of conjugacy classes of continuous  maps 
\[ W_{\mathbb{Q}_{p}} \times \SL(2,\ol{\mathbb{Q}}_{\ell}) \ra \phantom{}^{L}H(\ol{\mathbb{Q}}_{\ell}), \]
where $\ol{\mathbb{Q}}_{\ell}$ has the discrete topology, and such that the restriction to $\mathrm{SL}(2,\ol{\mathbb{Q}}_{\ell})$ via an algebraic representation, $W_{\mathbb{Q}_{p}}$ acts through a Frobenius semi-simple homomorphism, and the map respects the action of $W_{\mathbb{Q}_{p}}$ on $\phantom{}^{L}H(\ol{\mathbb{Q}}_{\ell})$, the $L$-group of $H$,
\item the set $\Phi^{\mathrm{ss}}(H)$ of continuous (where again $\ol{\mathbb{Q}}_{\ell}$ has the discrete topology) semi-simple homomorphisms,
\[ W_{\mathbb{Q}_{p}} \ra \phantom{}^{L}H(\ol{\mathbb{Q}}_{\ell}), \]
\item and the semi-simplification map $(-)^{\mathrm{ss}}: \Phi(H) \ra \Phi^{\mathrm{ss}}(H)$ defined by precomposition with
\[ W_{\mathbb{Q}_{p}} \ra W_{\mathbb{Q}_{p}} \times \mathrm{SL}(2,\ol{\mathbb{Q}}_{\ell}) \]
\[ g \mapsto (g,\begin{pmatrix} |g|^{1/2} & 0 \\ 0 & |g|^{-1/2} \end{pmatrix}). \]
\end{itemize} 
Then, we assume that there exists a collection of maps
\[ \mathrm{LLC}_{b}: \Pi(J_{b}) \ra \Phi(J_{b}) \]
\[ \rho \mapsto \phi_{\rho}, \]
for all $b \in B(G)$, satisfying the following properties:
\begin{enumerate} 
\item The diagram
\[ \begin{tikzcd}[ampersand replacement=\&]
            \Pi(J_{b})  \ar[rr, "\mathrm{LLC}_{b}"] \arrow[drr,"\mathrm{LLC}^{\mathrm{FS}}_{b}"] \& \&   \Phi(J_{b}) \ar[d,"(-)^{\mathrm{ss}}"] \\
            \& \& \Phi^{\mathrm{ss}}(J_{b})
\end{tikzcd}\] 
commutes, where $\mathrm{LLC}_{b}^{\mathrm{FS}}$ is the Fargues-Scholze local Langlands correspondence for $J_{b}$. 
\item For $\rho$ a non-zero smooth irreducible representation of $J_{b}(\bb{Q}_{p})$ for some $b \in B(G)$, consider $\phi_{\rho}$ as an element of $\Phi(G)$ given by composing with the twisted embedding $\phantom{}^{L}J_{b}(\ol{\mathbb{Q}}_{\ell}) \simeq \phantom{}^{L}M_{b}(\ol{\mathbb{Q}}_{\ell}) \ra \phantom{}^{L}G(\ol{\mathbb{Q}}_{\ell})$ (as defined in \cite[Section~IX.7.1]{FS}). If $\phi_{\rho}$ factors through the natural (up to $\hat{G}$-conjugacy) embedding $\phantom{}^{L}T \ra \phantom{}^{L}G$ then $b \in B(G)_{\mathrm{un}}$. 
\item If $\rho$ is a non-zero representation of $J_{b}(\bb{Q}_{p})$ for some $b \in B(G)$ such that $W_{\mathbb{Q}_{p}} \times \mathrm{SL}(2,\ol{\mathbb{Q}}_{\ell}) \xrightarrow{\phi_{\rho}} \phantom{}^{L}J_{b}(\ol{\mathbb{Q}}_{\ell}) \ra \phantom{}^{L}G(\ol{\mathbb{Q}}_{\ell})$ factors through the embedding $\phantom{}^{L}T \ra \phantom{}^{L}G$ with induced parameter $\phi_{T}$, then, by (2), the element $b$ is unramified, and we require that $\rho$ is isomorphic to an irreducible subquotient of $i_{B_{b}}^{J_{b}}(\chi^{w}) \otimes \delta_{P_{b}}^{-1/2}$ for some $w \in W_{b}$, where $\chi: T(\bb{Q}_{p}) \ra \ol{\bb{Q}}_{\ell}^{*}$ is the character attached to $\phi_{T}$ by class field theory and $\delta_{P_{b}}$ is the character defined in \cite[Definition~3.14]{GH}. Here the map $\phantom{}^{L}T \ra \phantom{}^{L}G$ is the non-twisted embedding and the map $\phantom{}^{L}J_{b} \ra \phantom{}^{L}G$ is the twisted embedding, as in (2).
\end{enumerate}
\end{assumption}
\begin{remark}
This assumption might seem a bit daunting, but is verifiable in many cases. In particular, the first assumption follows from the compatibility of the Fargues-Scholze correspondence with the Harris-Taylor correspondence for groups of type $A_{n}$ and its inner forms (\cite[Theorem~1.0.3]{KW}). Similarly, for groups of type $C_{2}$ and their inner forms over a unramified extension $L$ with $p > 2$, this follows from the main theorem of \cite{Ham}, and, for odd unramified unitary groups over $\mathbb{Q}_{p}$ this follows from the main theorem of \cite{BMNH}. The methods employed in these two papers should generalize to at least a few other cases. 

Assumption (2) is also a standard and verifiable conjecture in the cases where the local Langlands correspondence is known to exist. If $\rho$ is a representation such that $\phi_{\rho}$ factors through $\phantom{}^{L}T$ then we are claiming that $J_{b}$ has a Borel subgroup. For non quasi-split groups, it is conjectured that one should only consider $L$-parameters coming from the $L$-groups of the Levi subgroups of the non quasi-split group. These are referred to as relevant $L$-parameters. In particular, one expects the $L$-packets of $\LLC_{b}$ over irrelevant $\phi$ to be empty (see for example \cite[Conjecture~A.2]{Kal}), so if $\phi_{\rho}$ factors through $\phantom{}^{L}T$ for some $\rho$ under $\LLC_{b}$ it should imply that the group $J_{b}$ has a Borel $B_{b}$.  Assumption (3) is just the expectation that, when $\phi_{\rho}$ factors through $\phantom{}^{L}T$, the members of the $L$-packet should be given by the irreducible constituents of the parabolic inductions from $T$ to $J_{b}$. The Weyl group twists appear since the Weyl group conjugates of $\phi_{\rho}: W_{\mathbb{Q}_{p}} \ra \phantom{}^{L}J_{b}(\ol{\mathbb{Q}}_{\ell}) \simeq \phantom{}^{L}M_{b}(\ol{\mathbb{Q}}_{\ell})$ all map to the same parameter when viewed as a parameter valued in $\phantom{}^{L}G$, and the modulus twist by $\delta_{P_{b}}^{-1/2}$ appears since we are comparing this to an $L$-parameter of $G$ via the twisted embedding $\phantom{}^{L}J_{b}(\ol{\mathbb{Q}}_{\ell}) \simeq \phantom{}^{L}M_{b}(\ol{\mathbb{Q}}_{\ell}) \ra \phantom{}^{L}G(\ol{\mathbb{Q}}_{\ell})$ (cf. \cite[Page~24-25]{GH}) . 
\end{remark}
Under this assumption, we will deduce our main Corollary of the filtered eigensheaf property.
\begin{corollary}{\label{noredvanishing}} Under Assumption \ref{compatibility}, consider $b \in B(G)$ with corresponding locally closed HN-stratum $\Bun_{G}^{b} \subset \Bun_{G}$. For $\phi$ a generic parameter and $\Lambda = \ol{\bb{F}}_{\ell}$ with $\chi: T(\bb{Q}_{p}) \ra \ol{\bb{F}}_{\ell}^{*}$ the associated character, the following is true.
\begin{enumerate}
\item If $b \notin B(G)_{\mathrm{un}}$ the restriction
\[ \nmEis(\mathcal{S}_{\phi_{T}})|_{\Bun_{G}^{b}} \]
vanishes.
\item If $b \in B(G)_{\mathrm{un}}$ is an unramified element and $\rho$ is a smooth irreducible $\ol{\mathbb{F}}_{\ell}$-representation occurring as a constituent of $\nmEis(\mathcal{S}_{\phi_{T}})|_{\Bun_{G}^{b}}$ then $\rho$ is an irreducible constituent of $i_{B_{b}}^{J_{b}}(\chi^{w}) \otimes \delta_{P_{b}}^{-1/2}$ for some $w \in W_{b}$. 
\end{enumerate}
\end{corollary}
\begin{proof}
As noted above, $\nmEis(\mathcal{S}_{\phi_{T}})|_{\Bun_{G}^{b}}$ will be valued in a complex of smooth $\ol{\mathbb{F}}_{\ell}$-representations of the $\sigma$-centralizer $J_{b}$ of $b$. If we consider a smooth irreducible constituent of this restriction $\rho$ then, as already discussed above, it follows that the Fargues-Scholze parameter $\phi_{\rho}^{\mathrm{FS}}: W_{\mathbb{Q}_{p}} \rightarrow \phantom{}^{L}J_{b}(\ol{\mathbb{F}}_{\ell})$ under the twisted embedding
\[ \phantom{}^{L}J_{b}(\ol{\mathbb{F}}_{\ell}) \rightarrow \phantom{}^{L}G(\ol{\mathbb{F}}_{\ell}) \]
agrees with $\phi$. Using \cite[Lemma~6.8]{Dat}, we can choose $\tilde{\rho}$ a lift of $\rho$ to a smooth irreducible $\ol{\mathbb{Q}}_{\ell}$-representation admitting a $J_{b}(\mathbb{Q}_{p})$-stable $\ol{\mathbb{Z}}_{\ell}$-lattice such that $\rho$ occurs as a subquotient of $\tilde{\rho}$ mod $\varpi$. Since the Fargues-Scholze correspondence is compatible with reduction mod $\varpi$ (by the $\Lambda$-linearity of the maps appearing in \cite[Theorem~IX.5.2]{FS}), it follows that the Fargues-Scholze parameter $\phi_{\tilde{\rho}}^{\mathrm{FS}}$ factors through $\phantom{}^{L}G(\ol{\mathbb{Z}}_{\ell})$ and that it equals $\phi_{\rho}^{\mathrm{FS}}$ mod $\varpi$. Now, since $\phi_{\rho}^{\mathrm{FS}}$ factors through $\phantom{}^{L}T$ and induces a generic parameter, we claim that the same is true for $\phi_{\tilde{\rho}}^{\mathrm{FS}}$. This follows through standard deformation theory. In particular, if
\[ H^{1}(W_{\mathbb{Q}_{p}},\alpha_{\Ad} \circ \phi_{T}) \]
vanishes for all $\Gamma$-orbits $\alpha$ of roots, then any lift of $\phi_{\rho}^{\mathrm{FS}}$ will factor through $\phantom{}^{L}T$, but this vanishing is guaranteed by $\phi_{T}$ being generic (cf. \cite[Lemma~6.2.2]{CS}). It is also easy to see that $\phi_{\tilde{\rho}}^{\mathrm{FS}}$ must be generic since its mod $\varpi$ reduction is. Now, by Assumption \ref{compatibility} (1), we note that $\phi_{\tilde{\rho}}^{\mathrm{FS}}$ is the semi-simplification of $\phi_{\tilde{\rho}}$, the $L$-parameter attached to $\tilde{\rho}$. Therefore, Lemma \ref{regmonodromy} and genericity imply that $\phi_{\tilde{\rho}}$ has trivial restriction to $\SL(2,\ol{\bb{Q}}_{\ell})$. The two claims now follow from Assumptions \ref{compatibility} (2) and (3), respectively. 
\end{proof}
This statement will allow us to give a complete description of the eigensheaf $\nmEis(\mathcal{S}_{\phi_{T}})$ for $\phi_{T}$ satisfying the slightly stronger condition of generic regularity. In particular, as we will see in \S $9$, for the restrictions of the sheaf to $\Bun_{G}^{b}$ for $b \in B(G)_{\mathrm{un}}$, we will always be able to evaluate the stalks in terms of normalized parabolic inductions of Weyl group translates of the character $\chi$,and, by the previous Corollary, we know these are the only possible non-zero stalks. 

Now we turn our attention to studying how the geometric Eisenstein functor interacts with Verdier duality. 
\section{Second Adjointness and the ULA property}{\label{s: SecondAdjointness}}
In this section, we recall two important results from the paper \cite{HHS}. First, we recall the statement of geometric second adjointness. We will consider the following  geometric constant term functors.
\[ \nmCT_{*}: \D(\Bun_{G}) \ra \D(\Bun_{T}) \]
\[  A \mapsto \mf{q}_{*}(\mf{p}^{!}(A) \otimes \IC_{\Bun_{B}}^{-1}) \]
and 
\[ \nmCT_{!}: \D(\Bun_{G}) \ra \D(\Bun_{T}) \]
\[  A \mapsto \mf{q}_{!}(\mf{p}^{*}(A) \otimes \IC_{\Bun_{B}}). \]
These satisfy adjunction relationships
\[ (\nmEis(-),\nmCT_{*}(-)) \]
and
\[ (\nmCT_{!}(-),\nmEis_{*}(-)). \]
Here we recall $\nmEis_{*}(-) := \mf{p}_{*}(\mf{q}^{*}(-) \otimes \IC_{\Bun_{B}})$, and for the second relationship, we have used that $\mf{q}^{!}(\Lambda) \simeq \IC_{\Bun_{B}}^{\otimes 2}$, by Theorem \ref{dualob}. 

We will use the subscripts $(-)_{-}$ to denote the analogous functors formed with respect to the diagram attached to $\Bun_{B^{-}}$, where $B^{-}$ denotes the opposite Borel. We now have the following result, which is a geometric analogue of the classical second adjointness result.
\begin{theorem}{\label{thm: geometricsecondadjointness}}{\cite[Theorem~1.1.1 (3)]{HHS},\cite[Theorem~2.5]{YutaTakayaSecondAdjointness}}
There is a natural isomorphism
\[ \nmCT_{-!}(-) \simeq \nmCT_{*}(-) \]
of functors $\D(\Bun_{G}) \ra \D(\Bun_{T})$.
\end{theorem}
This expression will be used in the next section to show the commutation of Eisenstein series with Verdier duality, by allowing us to rewrite certain constant term functors. We also have the following key finiteness result.
\begin{theorem}{\label{thm: EispreservesULA}}{\cite[Theorem~1.1.2 (1)]{HHS}}
The functor 
\[ \nmEis(-): \D(\Bun_{T}) \ra \D(\Bun_{G}) \]
preserves ULA objects after restricting to finitely many connected components of $\Bun_{T}$.
\end{theorem}
In particular, this will be used in the next section to show that the Eisenstein functor $\nmEis(\mathcal{S}_{\phi_{T}})$ has stalks valued in finite length representations after restricting to connecting components and assuming \ref{compatibility}.

We will also need to establish some results about the principal series representations appearing in the stalks of the Eisenstein series. We start with the following auxiliary lemma. 

\begin{lemma}\label{lem:unit-from-residue-tube}
Let $F/\mathbb{Q}_\ell$ be a finite extension, $\mathcal{O}_{F}$ its ring of integers, $\varpi$ its uniformizer, and $k$ its residue field. We set $A=\mathcal{O}_{F}[X_1^{\pm1},\ldots,X_r^{\pm1}]$, and let $R = A_{\mathfrak{m}}$ be the localization at the maximal ideal $\mf{m} = (\varpi,X_{1} - 1,X_{2} -1,\ldots,X_{r} - 1)$. Let $c\in R$. Suppose the following is true.
\begin{enumerate}
\item The image of $c$ in
      $\operatorname{Frac}(R/\varpi R)$ is nonzero.
\item For every finite extension $F'/F$ and every tuple
      \[
        x=(x_1,\ldots,x_r)\in
        (1+\mathfrak m_{F'})^r,
      \]
      we have $c(x_{1},\ldots,x_{r})\ne0$, where $\mathfrak m_{F'}$ is the
      maximal ideal of $\mathcal O_{K'}$.
\end{enumerate}
Then $c$ is a unit of $R$.
\end{lemma}
\begin{proof}
Let $X=\operatorname{Spec}(A) \ra \Spec(\mathcal{O}_{F})$ and write $s$ for the point in the special fiber $X_{k}$ defined by the ideal $(\varpi,X_1-1,\ldots,X_r-1)$. We argue by contradiction and assume $c\notin R^\times$.

Choose an affine open neighborhood $U\subset X$ of $s$
on which $c$ is regular and whose special fiber is integral.
Since $R=\mathcal O_{X,s}$ is a local ring and we are assuming that $c \notin R^{*}$, we must have that $c$ vanishes on $s$.

Let $\mathfrak U$ denote the formal scheme defined by the $\varpi$-adic completion
of $U$, and consider the closed formal subscheme $\mathfrak Z=V(c)\subset\mathfrak U$.
The first hypothesis implies that the reduction
$\overline c$ is nonzero on the integral special fiber
$U_k$. Therefore, it is a non-zero-divisor.

It follows that $\mathfrak Z$ is flat over $\mathcal{O}_{F}$.
Indeed, if $\varpi a=cb$, reducing modulo $\varpi$
gives $\overline c\,\overline b=0$, hence
$b=\varpi b'$. Canceling $\varpi$ gives $a=cb'$.
Thus, $\mathcal O_{\mathfrak Z}$ has no
$\varpi$-torsion as multiplication by $\varpi$ is injective and is therefore flat over $\mathcal{O}_{F}$.

We now consider the diagram of specialization morphisms
\[
\begin{tikzcd}
\text{|}\mathfrak Z_{\eta}\text{|}
  \arrow[r,hook]
  \arrow[d,"\mathrm{sp}_{\mathfrak Z}"']
&
\text{|}\mathfrak U_{\eta}\text{|}
  \arrow[d,"\mathrm{sp}_{\mathfrak U}"]
\\
 \text{|}Z_k\text{|}
  \arrow[r,hook]
&
\text{|}U_k\text{|},
\end{tikzcd}
\]
where the vertical arrows are the specialization maps
from the topological space of the adic rigid generic fiber to the topological space of the special fiber (see \cite[Definition~4.1]{Gleason25}).

Since $\mathfrak Z$ is flat and topologically of
finite type over $\mathcal{O}_{F}$, its specialization map is surjective \cite[Proposition~4.2]{Gleason25}. However, since $c(s)=0$, we have that $s\in\mathfrak Z_k$. Thus, the open subspace defined by the tube $\mathrm{sp}_{\mathfrak Z}^{-1}(s)
\subset\mathfrak Z_\eta$ is nonempty. As it is a non-empty open subset
of a rigid space of finite type over $F$, it contains
a classical point defined over some finite extension
$F'/F$. In particular, this gives a tuple $(x_1,\ldots,x_r)\in(\mathcal O_{F'}^\times)^r$
satisfying $x_i\equiv1\pmod{\mathfrak m_{F'}}$ for all $i = 1,\ldots,r$ and $c(x_{1},\ldots,x_{r}) = 0$,  which contradicts the second hypothesis. Therefore, $c\in R^\times$.
\end{proof}
We will apply this Lemma in the context of intertwining operators.
\begin{proposition}\label{prop:opposite-borel-comparison}
Let $G/\mathbb{Q}_p$ be quasi-split and connected reductive, with Borel
$B=TU$ and opposite Borel $B^{-}$. Let
$\chi:T(\mathbb{Q}_p)\to \ol{\mathbb{F}}_{\ell}^\times$ be smooth and generic regular character. Then we have an isomorphism of smooth $G(\bb{Q}_{p})$-representations
\[
 i_B^G(\chi)\simeq i_{B^-}^G(\chi).
\]
\end{proposition}

\begin{proof}
The image of $\chi$ is finite: its image on a compact open
subgroup is finite, the quotient by that subgroup is finitely generated,
and every element of $\ol{\bb{F}}_{\ell}^\times$ has finite order. Choose a finite extension
$F/\mathbb Q_\ell$ whose ring of integers $\mathcal{O}_{F}$ contains the chosen
square root $\sqrt{p} \in \ol{\bb{Q}}_{\ell}$ and has residue field $k$ containing the values of
$\chi$. Teichm\"uller lifts give a smooth character
$\widetilde\chi:T(\mathbb Q_p)\to\mathcal O_{F}^\times$ lifting
$\chi$.

Let $T(\mathbb{Q}_p)^0 \subset T(\bb{Q}_{p})$ be the maximal compact subgroup and put
$X =T(\mathbb{Q}_p)/T(\mathbb{Q}_p)^0\simeq \bb{Z}^r$. Set $R=\mathcal O_{F}[X]_{\mathfrak m}$, 
$\mathfrak m=(\varpi,[\lambda]-1:\lambda\in X)$, where $\mathcal{O}_{F}[X]$ is the group ring of $X$ over $\mathcal{O}_{F}$. In particular, it is a polynomial algebra of exactly the form described in Lemma \ref{lem:unit-from-residue-tube}.

We consider the character
\[ \theta: T(\bb{Q}_{p}) \ra R^{*} \]
\[ t \mapsto \widetilde\chi(t)[\overline t], \]
where $\overline t$ denotes the image of $t \in T(\bb{Q}_{p})$ in $X$. For $a \in W_G$, we write
$\Pi_{a} =i_B^G(\theta^{a})$ over $R$.

For $w \in W_{G}$, we consider the ideals $I_{w} \subset R[T(\bb{Q}_{p})]$ defined by the $\theta^{w}$ in the group
algebra $R[T(\mathbb{Q}_p)]$.  Explicitly,
\[
I_{w} =
([t]-\theta^{w}(t)[1]\mid t\in T(\bb{Q}_p)).
\]
The geometric lemma \cite[Section~2.8]{Dat} endows the Jacquet module $r_B^G\Pi_{a}$ a filtration with the characters $\theta^{w}$, each occurring once. We claim this filtration is split under our regularity assumption.

By construction, $I_{w}$ annihilates the character $\theta^{w}$. This implies that the product $\prod_{w \in W_{G}} I_{w}$ annihilates 
$r_{B}^{G}(\Pi_{w})$. In particular, its action factors through
\[
R[T(\bb{Q}_{p}]/\prod_{w\in W_G} I_{w}.
\]
Now the assumed regularity of $\chi$ implies that the ideals $I_{w}$
are pairwise comaximal. Indeed, if $w \ne w'$, the regularity assumption implies that there exists
$t\in T(\mathbb Q_p)$ such that
\[
\theta^{w}(t)-\theta^{w'}(t)\in R^\times.
\]
Since $[t]-\theta^{w}(t)[1]\in I_{w}$ and $[t]-\theta^{w'}(t)[1]\in I_{w'}$ looking at the difference shows that that $I_{w}+I_{w'}$ contains a unit of $R$ implying that $I_{w} + I_{w'} = R$, as desired.

Therefore, the Chinese remainder theorem tells us that
\[
R[T(\bb{Q}_{p})]\Big/\prod_{w\in W_G}I_{w}
\simeq
\prod_{w\in W_G} R[T(\bb{Q}_{p})]/I_{w}
\simeq
\prod_{w\in W_G}R.
\]
For each $w\in W_G$, let $e_w$ denote the
idempotent corresponding to the identity in the
$w$-th factor. These idempotents give a direct
sum decomposition
\begin{equation}{\label{eqn: JacquetmoduleComparison}}
r_B^G\Pi_{a}
\simeq
\bigoplus_{w \in W_G}\theta^{w}.
\end{equation}
for all $a \in W_{G}$. We note that this decomposition is compatible with arbitrary changes
of coefficients $R\to S$, since the character ideals,
their comaximality, and the corresponding idempotent
decomposition are preserved under base change.

Moreover, parabolic induction and Jacquet modules also commute with such base-changes (see 
\cite[Section 5, p.~339]{VigICM22}).  Therefore, Frobenius reciprocity and (\ref{eqn: JacquetmoduleComparison}) consequently give
\begin{equation}{\label{eqn: JacquetmoduleComparisonII}}
 \operatorname{Hom}_{G(\mathbb{Q}_p)}(\Pi_a,\Pi_b)\simeq R,
\end{equation}
for all $a,b \in W_{G}$, compatibly with every  base-change $R\to S$. In particular the endomorphism algebra
of each $\Pi_a$ is $R$.

Let $w_0$ be the longest Weyl element. Choose generators
$f:\Pi_1\to\Pi_{w_0}$ and $g:\Pi_{w_0}\to\Pi_1$ of 
(\ref{eqn: JacquetmoduleComparisonII}) and write $gf=c\,\mathrm{id}$ and in turn that 
$fg=c\,\mathrm{id}$ by (\ref{eqn: JacquetmoduleComparisonII}).

We set $k(X):=\operatorname{Frac}(R/\varpi R)$, and let $\Omega$ be an algebraic closure of $k(X)$.

If $f_\Omega$ and $g_\Omega$ denote the base-changes of $f$ and $g$, they are nonzero, using that $f$ and $g$ were chosen to be generators and the fact that  (\ref{eqn: JacquetmoduleComparisonII}) is compatible with base-change.

Moreover, the base-change  $\theta_\Omega:T(\mathbb{Q}_p) \rightarrow
\Omega^\times$ of $\theta$ is
the generic unramified twist of the reduction $\chi$ of
$\widetilde\chi$.  In particular, by Dat's generic irreducibility Theorem
\cite[Theorem~5.1]{Dat}, the normalized
principal series $\Pi_{1,\Omega}
=i_B^G(\theta_\Omega)$ is irreducible. The same theorem applies to the
Weyl-conjugate character, showing that $\Pi_{w_0,\Omega}
=i_B^G(\theta_\Omega^{w_0})$ is also irreducible.

Since $f_\Omega$ is a nonzero homomorphism between
irreducible representations, it is an isomorphism.
The same argument shows that $g_\Omega$ is an
isomorphism. Consequently, the composition $g_\Omega\circ f_\Omega$
is an isomorphism, and in particular is nonzero. This shows that the image of $c$
in $k(X)$ is nonzero, so that it satisfies hypothesis (1) of Lemma \ref{lem:unit-from-residue-tube}

Next, for a finite extension $F'/F$, consider a map $\psi: X\to\mathcal O_{F'}^\times$ giving rise to an unramified character by precomposing with the projection $T(\bb{Q}_{p}) \ra X$. Suppose that it reduces to the trivial character modulo the maximal ideal. The unramified twist $\widetilde\chi\psi$ is generic regular, since its reduction modulo the maximal ideal is generic regular. Proposition~\ref{prop: genericinterisom} gives an isomorphism between the corresponding base-change of $\Pi_{1}$ and $\Pi_{w_{0}}$ along the characteristic $0$-specialization $R \xrightarrow{\mathrm{ev}_{\psi}} \mathcal{O}_{F'} \ra F'$.

The endomorphism algebra of these representations is a field by
(\ref{eqn: JacquetmoduleComparison}). Hence, every nonzero map between them is an isomorphism. In particular, this implies that $c(\psi)\ne0$ for all such $\psi$. In particular, this implies that $c$ satisfies hypothesis (2) of Lemma~\ref{lem:unit-from-residue-tube} implying that $c\in R^\times$.

Now if we take the reduction of $f$ and $g$ at $\mathfrak m$ and extend coefficients to $\ol{\mathbb{F}}_{\ell}$ this gives the desired isomorphism
$i_B^G(\chi)\simeq i_B^G(\chi^{w_0}) \simeq i_{B^{-}}^{G}(\chi)$.
\end{proof}
This allows us to deduce the result we will need.
\begin{corollary}\label{cor:principal-series-detection}
Let $\chi: T(\bb{Q}_{p}) \ra \ol{\bb{F}}_{\ell}^{*}$ be a smooth character.
\begin{enumerate}
\item If $\chi$ is generic and regular then we have that $i_B^G(\chi)$ is irreducible. Moreover, for every $b\in B(G)_{\mathrm{un}}$ and $w\in W_b$, we have that
\[
\rho_{b,w}
 =i_{B_b}^{J_b}(\chi^{w^{-1}})\otimes\delta_{P_b}^{-1/2}
\]
is irreducible
\item  Suppose we have a bounded complex $A$ of smooth
$J_b(\mathbb{Q}_p)$-representations which has finite-length cohomology with all
irreducible constituents given by the $\rho_{b,w}$, then
\[
 R\operatorname{Hom}(A,\rho_{b,w})=0
 \quad\text{for every }w\in W_b
 \qquad\Longrightarrow\qquad A=0.
\]
\end{enumerate}
\end{corollary}

\begin{proof}
Put $\pi=i_B^G(\chi)$. The geometric lemma and regularity show
that $\chi$ occurs with multiplicity one in $r_B^G\pi$.
If $0\to\sigma\to\pi\to\tau\to0$ had both ends nonzero, Frobenius reciprocity applied to $\sigma\hookrightarrow i_B^G(\chi)$
would give
\[
 \operatorname{Hom}_{T(\mathbb{Q}_p)}(r_B^G\sigma,\chi)\ne0.
\]
On the other hand, Proposition \ref{prop:opposite-borel-comparison} and second adjointness applied to the quotient map
$i_{B^-}^G(\overline\chi)\twoheadrightarrow\tau$ would give
\[
 \operatorname{Hom}_{T(\mathbb{Q}_p)}(\chi,r_B^G\tau)\ne0.
\]
Exactness of $r_B^G$ would then force two occurrences of
$\chi$ in $r_B^G\pi$, a contradiction to the regularity. The remaining claims follows by noting that if $b \in B(G)_{\mathrm{un}}$ then $J_{b} \simeq M_{b}$ a Levi of $G$ and the fact that generic regularity is stable under passage to Levis and conjugation by elements of the Weyl group.

Finally, if $A\ne0$ is as in (2), choose the greatest $n$ with $H^n(A)\ne0$ and a quotient $H^n(A)\twoheadrightarrow\overline\rho_{b,w}$. The
composite
\[
 A \rightarrow H^n(A)[-n]\ \rightarrow \rho_{b,w}[-n]
\]
is nonzero on the $n$th cohomology, contradicting the asserted vanishing.
\end{proof}
\section{Stalks of Geometric Eisenstein Series and Verdier Duality}{\label{stalkssection}}
\subsection{Stalks of Geometric Eisenstein Series}
Now our aim is to explicitly determine the stalks of the sheaf $\nmEis(\mathcal{S}_{\phi_{T}})$, when $\phi_{T}$ is generic regular. Recall this means that $\phi_{T}$ is generic, and, for all $w \in W_{G}$ non-trivial, $\chi \not\simeq \chi^{w}$, where $\chi$ is the smooth character attached to $\phi_{T}$ via local class field theory. Genericity will allow us to apply the results of Corollary \ref{noredvanishing}, and the second condition will appear naturally in the computation of the stalks. It is helpful to treat each connected component separately. In particular, using Corollary \ref{conncomponents}, we consider the decomposition:
\[ \nmEis(\mathcal{S}_{\phi_{T}}) = \bigoplus_{\ol{\nu} \in B(T)} \nmEis^{\ol{\nu}}(\mathcal{S}_{\phi_{T}}) \]
The main result of this section is as follows. 
\begin{theorem}{\label{connnormstalksdescription}}
Consider $\phi_{T}$ a toral parameter with associated character $\chi: T(\mathbb{Q}_{p}) \rightarrow \Lambda^{*}$. We fix $\ol{\nu} \in B(T)$ with image $b \in B(G)_{\mathrm{un}}$, dominant reduction $b_{T}$, and associated Borel $B_{b}$. Using Corollary \ref{weylgrouporbits}, we can write $\ol{\nu} = w(b_{T})$ for a unique $w \in W_{b}$. Assume that $\phi_{T}$ has generic regular mod $\varpi$-reduction and that assumption \ref{compatibility} holds, then we have an isomorphism
\[ \nmEis^{\ol{\nu}}(\mathcal{S}_{\phi_{T}}) \simeq j_{b!}(i_{B_{b}}^{J_{b}}(\chi^{w^{-1}}) \otimes \delta_{P_{b}}^{-1/2})[-\langle 2\hat{\rho},\nu_{b} \rangle] \]
under the identification $\D(J_{b}(\mathbb{Q}_{p}),\Lambda) \simeq \D(\Bun_{G}^{b})$, where $j_{b}: \Bun_{G}^{b} \ra \Bun_{G}$ is the locally closed immersion defined by the HN-stratum corresponding to $b$, and $P_{b}$ is the standard parabolic with Levi factor $M_{b} \simeq J_{b}$.
\end{theorem}
By varying $\ol{\nu}$ over all connected components, we obtain the following.
\begin{corollary}{\label{normstalksdescription}}
Consider $\phi_{T}$ a toral parameter with generic regular mod $\varpi$-reduction, with associated character $\chi: T(\mathbb{Q}_{p}) \rightarrow \Lambda^{*}$. Assume \ref{compatibility}. For $b \in B(G)$, the stalk $\nmEis(\mathcal{S}_{\phi_{T}})|_{\Bun_{G}^{b}} \in \D(\Bun_{G}^{b}) \simeq \D(J_{b}(\mathbb{Q}_{p}),\Lambda)$ is isomorphic to
\begin{enumerate}
\item  $\bigoplus_{w \in W_{b}} i_{B_{b}}^{J_{b}}(\chi^{w^{-1}}) \otimes \delta_{P_{b}}^{-1/2}[-\langle 2\hat{\rho},\nu_{b} \rangle]$ if $b \in B(G)_{\mathrm{un}}$,
\item $0$ if $b \notin B(G)_{\mathrm{un}}$.  
\end{enumerate}
\end{corollary} 
First, consider the following easy Lemma. 
\begin{lemma}{\label{restvanishing}}
For $\ol{\nu} \in B(T)$ with image $b$ in $B(G)$, the restriction
\[ \nmEis^{\ol{\nu}}(\mathcal{S}_{\phi_{T}})|_{\Bun_{G}^{b'}} \]
for $b' \in B(G)$ vanishes unless $b \succeq b'$ in the natural partial ordering on $B(G)$.
\end{lemma}
\begin{proof}
This follows from the observation that the image of $\Bun_{B}^{\ol{\nu}}$ under $\mf{p}^{\ol{\nu}}$ is contained in the open substack $\Bun_{G}^{\leq b}$ parametrizing bundles with associated Kottwitz element less than $b$. In particular, using the Tannakian formalism \cite[Theorems~4.42,4.43]{Zieg}, this reduces to the observation that, for $\GL_{n}$, a bundle $\mathcal{E}$ with a filtration by vector subbundles has Harder-Narasimhan polygon less than or equal to Harder-Narasimhan polygon of the direct sum of the graded pieces of the filtration, which is an easy consequence of the formalism of Harder-Narasimhan reductions (see for example \cite[Corollary~3.4.18]{Ked}). 
\end{proof}
Thus, for a fixed $b' \in B(G)$, Lemma \ref{restvanishing} tells us that $\nmEis(\mathcal{S}_{\phi_{T}})|_{\Bun_{G}^{b'}}$ is a direct sum of $\nmEis^{\ol{\nu}}(\mathcal{S}_{\phi_{T}})|_{\Bun_{G}^{b'}}$ for $\ol{\nu}$ whose image $b \in B(G)$ satisfies $b \succeq b'$. Now the key point is that, under the generic regularity assumption, all the contributions will vanish except when $b' = b$. This is one of the many reasons that generic regularity is absolutely necessary to get a reasonable eigensheaf. In general, all possible $\ol{\nu}$ contribute to $\nmEis(\mathcal{S}_{\phi_{T}})|_{\Bun_{G}^{b'}}$, and  $\nmEis(\mathcal{S}_{\phi_{T}})|_{\Bun_{G}^{b'}}$ will be equal to an infinite direct sum of smooth irreducible representations sitting in infinitely many degrees. We now reduce Theorem \ref{connnormstalksdescription} to two propositions. We first have the following proposition describing the contribution of the split reduction in the connected components $\Bun_{B}^{\ol{\nu}}$, which is a special case of \cite[Theorem~4.26]{GH}.
\begin{proposition}{\label{splitcontribution}}{\cite[Theorem~4.26]{GH}}
Let $\ol{\nu} \in B(T)$ be an element mapping to $b \in B(G)_{\mathrm{un}}$. We write $\ol{\nu} = w(b_{T})$ as above. If $\phi_{T}$ is any toral parameter, we have an isomorphism 
\[  \nmEis^{\ol{\nu}}(\mathcal{S}_{\phi_{T}})|_{\Bun_{G}^{b}} \simeq i_{B_{b}}^{J_{b}}(\chi^{w^{-1}}) \otimes \delta_{P_{b}}^{-1/2}[-\langle 2\hat{\rho},\nu_{b} \rangle]  \]
of complexes of smooth $J_{b}(\mathbb{Q}_{p})$-modules, under the identification $\D(\Bun_{G}^{b}) \simeq \D(J_{b}(\mathbb{Q}_{p}),\Lambda)$, where $w \in W_{b}$ is identified with a representative of minimal length and $\delta_{P_{b}}$ is as in assumption \ref{compatibility}.
\end{proposition}
This tells us that all the claimed contributions to the restriction $\nmEis(\mathcal{S}_{\phi_{T}})|_{\Bun_{G}^{b}}$ appear. All that remains is to show is that there are no additional contributions, and this is precisely what generic regularity will allow us to do.
\begin{proposition}{\label{nonsplitcontribution}}
Assume $\phi_{T}$ has generic regular reduction modulo the maximal ideal $\mf{m}_{F} \subset \mathcal{O}_{F}$ for $\Lambda$ as in Setup \ref{assumption: coefficients} (1) and that assumption \ref{compatibility} holds, then, for all $\ol{\nu} \in B(T)$ mapping to $b \in B(G)_{\mathrm{un}}$, the sheaf $\nmEis^{\ol{\nu}}(\mathcal{S}_{\phi_{T}})$ is only supported on the HN-stratum $\Bun_{G}^{b}$.
\end{proposition}
We now prove this Proposition. 
\subsubsection{The Proof of Proposition \ref{nonsplitcontribution}}
We argue by induction on $b \in B(G)_{\mathrm{un}}$ with respect to the partial ordering on $B(G)$ and the following stronger statement.

"For $b \in B(G)_{\mathrm{un}}$ with dominant reduction $b_{T}$, and $\ol{\nu} = w(b_{T}) \in B(T)$ mapping to $b$ for varying $w \in W_{b}$, we have an isomorphism
\[ \nmEis^{w(b_{T})}(\mathcal{S}_{\phi_{T}}) \simeq j_{b!}(i_{B_{b}}^{J_{b}}(\chi^{w^{-1}}) \otimes \delta_{P_{b}}^{-1/2})[-\langle 2\hat{\rho}, \nu_{b} \rangle] \]
of sheaves in $\D(\Bun_{G})$."

We note that the functor $- \otimes_{\mathcal{O}_{F}/\varpi^{n}} \ol{\bb{F}}_{\ell}$ is conservative by the (derived form) of Nakayama's lemma. Moreover, the construction of Eisenstein series commutes with this functor by projection formula applied to $\mf{p}_{!}$ and the symmetric monoidality of $\mf{q}^{*}$. Using projection formula applied to $j_{b!}$, this allows us to assume without loss of generality that $\Lambda = \ol{\bb{F}}_{\ell}$, which we will assume for the rest of the proof.

The base case will be when $b$ is such that any $b' \in B(G)_{\mathrm{un}}$ satisfying $b \succeq b'$ is equal to $b$. Since the stalk of $\nmEis^{\ol{\nu}}(\mathcal{S}_{\phi_{T}})|_{\Bun_{G}^{b'}}$ will only be non-trivial for $b' \in B(G)$ such that $b \succeq b'$ by Lemma \ref{restvanishing}, the result in this case follows from Proposition \ref{splitcontribution} and Corollary \ref{noredvanishing} (1), where we note that $\phi_{T}$ is generic regular and therefore generic. For the inductive step, assume the claim is true for all $b' \in B(G)_{\mathrm{un}}$ such that $b \succeq_{\neq} b'$. Let $\ol{\nu} \in B(T)$ be an element mapping to $b$. By Proposition \ref{splitcontribution} and Corollary \ref{noredvanishing} (1) again, it suffices to show that the restriction of $\nmEis^{\ol{\nu}}(\mathcal{S}_{\phi_{T}})$ to $\Bun_{G}^{b'}$ vanishes for all such $b'$. By Corollary \ref{noredvanishing} (2), it suffices to show, for all $w \in W_{b'}$, that the complex
\[ \RHom(\nmEis^{\ol{\nu}}(\mathcal{S}_{\phi_{T}})|_{\Bun_{G}^{b'}},i_{B_{b'}}^{J_{b'}}(\chi^{w^{-1}}) \otimes \delta_{P_{b'}}^{-1/2}) =  \RHom(\nmEis^{\ol{\nu}}(\mathcal{S}_{\phi_{T}}),j_{b'*}(i_{B_{b'}}^{J_{b'}}(\chi^{w^{-1}}) \otimes \delta_{P_{b'}}^{-1/2})) \]
is trivial. In particular, since, by Theorem \ref{thm: EispreservesULA}, we know that $\nmEis^{\ol{\nu}}(\mathcal{S}_{\phi_{T}})$ is ULA and it follows that $\nmEis^{\ol{\nu}}(\mathcal{S}_{\phi_{T}})|_{\Bun_{G}^{b'}}$ is admissible in the sense that, for all compact open $K \subset J_{b'}(\mathbb{Q}_{p})$, $\nmEis^{\ol{\nu}}(\mathcal{S}_{\phi_{T}})^{K}|_{\Bun_{G}^{b'}}$ is a perfect complex. By Corollary \ref{noredvanishing} (2) and \cite[II.5.13]{Vig1}, we know that there are only finitely many possibilities for smooth irreducible constituents of $\nmEis^{\ol{\nu}}(\mathcal{S}_{\phi_{T}})|_{\Bun_{G}^{b'}}$. Therefore, by choosing $K \subset J_{b'}(\mathbb{Q}_{p})$ sufficiently small (so that every such constituent has a non-zero fixed vector), we deduce that $\nmEis^{\ol{\nu}}(\mathcal{S}_{\phi_{T}})|_{\Bun_{G}^{b'}}$ is a bounded complex with finite length cohomology, which reduces us to showing that the previous complex is trivial by Corollary \ref{cor:principal-series-detection}. To do this, let $\ol{\nu}' = w(b'_{T})$ be the element mapping to $b' \in B(G)$ defined by $w \in W_{b'}$. Our inductive hypothesis tells us that we have an isomorphism
\[ j_{b'!}(i_{B_{b'}}^{J_{b'}}(\chi^{w^{-1}}) \otimes \delta_{P_{b'}}^{-1/2})[-\langle 2\hat{\rho},\nu_{b'} \rangle] \simeq \nmEis^{\ol{\nu}'}(\mathcal{S}_{\phi_{T}}) \]
varying over $\ol{\nu}'$ mapping to $b' \in B(G)$.  Using Theorem \ref{dualob} and our inductive hypothesis together with the fact that generic regularity is stable under replacing $\phi_{T}$ by $\phi_{T}^{\vee}$, it follows that we have an isomorphism $\nmEis_{*}^{\ol{\nu}'}(\mathcal{S}_{\phi_{T}}) \simeq j_{b'*}(i_{B_{b'}}^{J_{b'}}(\chi^{w^{-1}}) \otimes \delta_{P_{b'}}^{-1/2})[-\langle 2\hat{\rho},\nu_{b'} \rangle]$. Therefore, it then suffices to show, for all $\ol{\nu}'$ mapping to $b' \in B(G)_{\mathrm{un}}$ and $\ol{\nu}$ mapping to $b \in B(G)_{\mathrm{un}}$, that 
\[ \RHom(\nmEis^{\ol{\nu}}(\mathcal{S}_{\phi_{T}}),\nmEis^{\ol{\nu}'}_{*}(\mathcal{S}_{\phi_{T}})) \]
is trivial. To aid our analysis, we consider the following unnormalized version of the constant term functor
\[ \CT^{\ol{\nu}}_{*}(-) := \mf{q}^{\ol{\nu}}_{*} \circ \mf{p}^{\ol{\nu}!}(-)[-\dim(\Bun_{B}^{\ol{\nu}})]: \D(\Bun_{G}) \ra \D(\Bun_{T}^{\ol{\nu}}) \]
which is in particular the right adjoint of the unnormalized Eisenstein functor $\Eis^{\ol{\nu}}(-) := \mf{p}_{!}(\mf{q}^{*}(-)[\dim(\Bun_{B}^{\ol{\nu}})])$. The unnormalized version of the functors considered in \S 8. Writing $\nmEis^{\ol{\nu}}(\mathcal{S}_{\phi_{T}})$ as $\Eis^{\ol{\nu}}(\mathcal{S}_{\phi_{T}} \otimes \Delta_{B}^{1/2})$ and using adjunction, it suffices to show that the complex 
\[ \RHom_{T(\mathbb{Q}_{p})}(\chi \otimes \delta_{B}^{1/2},\CT_{*}^{\ol{\nu}} \circ \nmEis^{\ol{\nu}'}_{*}(\mathcal{S}_{\phi_{T}})) \]
is trivial in $\D(\Bun_{T}^{\ol{\nu}}) \simeq \D(T(\mathbb{Q}_{p}),\Lambda)$. To show this, we first look at the diagram
\[\begin{tikzcd}
& \Bun_{B}^{\ol{\nu}} \times_{\Bun_{G}} \Bun_{B}^{\ol{\nu}'} \arrow[r,"\phantom{}^{'}\mf{p}^{\ol{\nu}}"] \arrow[d,"\phantom{}^{'}\mf{p}^{\ol{\nu}'}"]  &  \Bun_{B}^{\ol{\nu}'} \arrow[d,"\mf{p}^{\ol{\nu}'}"]  \arrow[r,"\mf{q}^{\ol{\nu}'}"] & \Bun_{T}^{\ol{\nu}'} \\
& \Bun_{B}^{\ol{\nu}} \arrow[r,"\mf{p}^{\ol{\nu}}"] \arrow[d,"\mf{q}^{\ol{\nu}}"] & \Bun_{G}   & \\
& \Bun_{T}^{\ol{\nu}} & & 
\end{tikzcd}\]
and note, by base-change, that we have a natural isomorphism
\[ \mf{q}^{\ol{\nu}}_{*} \circ \mf{p}^{\ol{\nu}!} \circ \mf{p}^{\ol{\nu}'}_{*} \circ \mf{q}^{\ol{\nu}'*}(-) \simeq  \mf{q}_{*}^{\ol{\nu}} \circ \phantom{}^{'}\mf{p}^{\ol{\nu}'}_{*} \circ \phantom{}^{'}\mf{p}^{\ol{\nu}!} \circ \mf{q}^{\ol{\nu}'*}(-) \]
of derived functors $\D(\Bun_{T}^{\ol{\nu}'}) \ra \D(\Bun_{T}^{\ol{\nu}})$. This tells us that $\CT^{\ol{\nu}} \circ \nmEis^{\ol{\nu}'}_{*}(\mathcal{S}_{\phi_{T}})$ is the direct image of the complex
\[ \phantom{}^{'}\mf{p}^{\ol{\nu}!}(\mf{q}^{\ol{\nu}'*}(\mathcal{S}_{\phi_{T}} \otimes \Delta_{B}^{1/2}))[\dim(\Bun_{B}^{\ol{\nu}'}) -\dim(\Bun_{B}^{\ol{\nu}})] \]
on $\Bun_{B}^{\ol{\nu}} \times_{\Bun_{G}} \Bun_{B}^{\ol{\nu}'}$ onto $\Bun_{T}^{\ol{\nu}}$. One can easily (by writing as a moduli space of sections from the curve to $B\backslash G/B$ and applying the Bruhat decomposition, see the proof of Proposition \ref{bruhatstrata} below) check that the space $\Bun_{B}^{\ol{\nu}} \times_{\Bun_{G}} \Bun_{B}^{\ol{\nu}'}$ has a locally closed stratification given by the generic relative position of the two bundles
\[ \bigsqcup_{w \in W_{G}} (\Bun_{B}^{\ol{\nu}} \times_{\Bun_{G}} \Bun_{B}^{\ol{\nu}'})_{w}, \]
which we denote by $Z_{w}^{\ol{\nu},\ol{\nu}'}$ for varying $w \in W_{G}$. Using the excision spectral sequence, this implies that $\CT_{*}^{\ol{\nu}} \circ \nmEis^{\ol{\nu}'}_{*}(\mathcal{S}_{\phi_{T}})$ also admits a filtration whose graded pieces we write as $(\CT_{*}^{\ol{\nu}} \circ \nmEis^{\ol{\nu}'}_{*}(\mathcal{S}_{\phi_{T}}))_{w}$. 

We now clarify that we use the conventions
\[
\chi^w(t)=\chi(\dot w^{-1}t\dot w), \]
\[ \mathcal F_T^w
=
\mathcal F_T\times^{T,\,(t\mapsto\dot w^{-1}t\dot w)}T.
\]
Thus, if $\mathcal F_T$ belongs to
$\operatorname{Bun}_T^{\bar\nu}$, then $\mathcal F_T^w$ belongs to
$\operatorname{Bun}_T^{w^{-1}(\bar\nu)}$.

Consider the following claim. 
\begin{proposition}{\label{bruhatstrata}}
Let $\ol{\nu}$ and $\ol{\nu}'$ be two elements mapping to $b$ and $b'$ in $B(G)_{\mathrm{un}}$, respectively. 
\begin{enumerate}
    \item Suppose that $b \neq b'$ then the stack $Z_{w}^{\ol{\nu},\ol{\nu}'}$ is empty if $w = 1$.
    \item If $w \neq 1$ then $(\CT_{*}^{\ol{\nu}} \circ \nmEis^{\ol{\nu}'}_{*}(\mathcal{S}_{\phi_{T}}))_{w}$ in $\D(T(\mathbb{Q}_{p}),\Lambda)$ is isomorphic to $\chi^{w^{-1}} \otimes \delta_{B}^{1/2}$ tensored by an iterated extension of complexes with trivial $T(\bb{Q}_{p})$-action.
\end{enumerate}
\end{proposition}
First, let's finish the proof of Proposition \ref{nonsplitcontribution} assuming this. By the above discussion, it suffices to show that the complex
\[ \RHom_{T(\mathbb{Q}_{p})}(\chi \otimes \delta_{B}^{1/2},(\CT_{*}^{\ol{\nu}} \circ \nmEis^{\ol{\nu}'}_{*}(\mathcal{S}_{\phi_{T}}))_{w}) \]
is trivial for all $w \in W_{G}$. This is trivial if $w = 1$ by point (1). If $w \neq 1$, then it follows from point (2) and the fact that the existence of an isomorphism $\chi \otimes \delta_{B}^{1/2} \simeq \chi^{w^{-1}} \otimes \delta_{B}^{1/2}$ would imply an isomorphism $\chi \simeq \chi^{w^{-1}}$, which would contradict the assumed regularity of $\chi$. Therefore, since $\phi_{T}$ is generic regular by assumption, the claim follows. Let's now finish up by reducing this Proposition to a simpler claim, which we will prove in the next section. Proposition \ref{bruhatstrata} is an analogue of \cite[Proposition~10.8]{BG2} and the idea behind its proof is the same. 
\begin{proof}{(Proposition \ref{bruhatstrata})}
We first begin by elucidating the geometry of the spaces $Z_{w}^{\ol{\nu},\ol{\nu}'}$ a bit more. For $S \in \Perf_{\ol{\bb{F}}_{p}}$, note that a $S$-point of $\Bun_{B}^{\ol{\nu}} \times_{\Bun_{G}} \Bun_{B}^{\ol{\nu}'}$ corresponds to a pair of $B$-structures on a $G$-bundle $\mathcal{F}_{G}$ on $X_{S}$. Namely, it parametrizes a pair $\mathcal{F}_{B}^{1}$ (resp. $\mathcal{F}_{B}^{2}$) of two $B$-structures on a $G$-bundle $\mathcal{F}_{G}$ whose reduction to $T$, denoted $\mathcal{F}_{T}^{1}$ (resp. $\mathcal{F}_{T}^{2}$) is isomorphic to $\mathcal{F}_{\ol{\nu}}$ (resp. $\mathcal{F}_{\ol{\nu}'}$) after pulling back to any geometric point of $S$. We can think of it as parameterizing sections
\[ X_{S} \ra B\backslash G/B\]
such that the degree is of the specified form. More transparently, we can think of a point of $\Bun_{B}^{\ol{\nu}} \times_{\Bun_{G}} \Bun_{B}^{\ol{\nu}'}$ as the $B$-bundle $\mathcal{F}_{B}^{2}$ together with a section 
\[ s: X_{S} \ra \mathcal{F}_{B}^{2} \times^{B} G/B \]
where $B$ acts via conjugation on $G/B$. For $\hat{\lambda} \in \hat{\Lambda}_{G}^{+}$, we recall that by interpreting $\mathcal{V}^{\hat{\lambda}}$ as global sections of the appropriately twisted bundle on $G/B$ corresponding to $\hat{\lambda}$ under Borel-Weil-Bott, every point in $G/B$ gives rise to a line $\ell^{\hat{\lambda}} \subset \mathcal{V}^{\hat{\lambda}}$. We consider the $B$-stable subspace $\mathcal{V}^{\hat{\lambda}}_{\geq w} \subset \mathcal{V}^{\hat{\lambda}}$ consisting of weights greater than or equal to $w(\hat{\lambda})$ and  $\mathcal{V}^{\hat{\lambda}}_{> w} \subset \mathcal{V}^{\hat{\lambda}}_{\geq w}$ the codimension $1$ subspace consisting of weights strictly greater than $w(\hat{\lambda})$. We let $(G/B)_{w} := BwB/B$ be the locally closed Schubert cell attached to $w \in W_{G}$. We write $(G/B)_{\leq w}$ for its closure. The closure is stratified by the Schubert cells indexed by elements $w' \in W_{G}$ with $w'$ with $w' \leq w$ in the Bruhat order. Then the line $\ell^{\hat{\lambda}} \subset \mathcal{V}^{\hat{\lambda}}$ will correspond to a point in $(G/B)_{\leq w}$ if and only if it belongs to $\mathcal{V}^{\hat{\lambda}}_{\geq w}$ for all $\hat{\lambda} \in \hat{\Lambda}_{G}^{+}$. Moreover, the point belongs to the stratum $(G/B)_{w}$ if and only the projection to $\mathcal{V}^{\hat{\lambda}}_{\geq w}/\mathcal{V}^{\hat{\lambda}}_{> w}$ is non-zero. This allows us to explain what it means to lie in the locally closed stratum $Z_{w}^{\ol{\nu},\ol{\nu}'}$. In particular, lying in this stratum is equivalent to the condition that $s$ factors through $\mathcal{F}_{B}^{2} \times^{B} (G/B)_{\leq w}$ and is not contained in any closed stratum defined by $\mathcal{F}_{B}^{2} \times^{B} (G/B)_{\leq w'}$ for any $w' < w$ in the Bruhat order. This implies that the section $s$ determines a set of line subbundles 
\[ \mathcal{L}^{\hat{\lambda}}_{\mathcal{F}_{T}^{1}} \ra  (\mathcal{V}^{\hat{\lambda}}_{\geq w})_{\mathcal{F}_{B}^{2}} \]
for all $\hat{\lambda} \in \hat{\Lambda}_{G}^{+}$. Moreover, via the inclusion $\mathcal{V}^{\hat{\lambda}}_{\geq w} \subset \mathcal{V}^{\hat{\lambda}}$, these give the Pl\"ucker description of the $B$-structure $\mathcal{F}_{B}^{1}$ such that $\mathcal{F}_{G} \simeq \mathcal{F}_{B}^{1} \times^{B} G \simeq \mathcal{F}_{B}^{2} \times^{B} G$. For this to define a point in $Z_{w}^{\ol{\nu},\ol{\nu}'}$, these need to satisfy the condition that the induced map
\[ \mathcal{L}^{\hat{\lambda}}_{\mathcal{F}_{T}^{1}} \ra  (\mathcal{V}^{\hat{\lambda}}_{\geq w})_{\mathcal{F}_{B}^{2}} \ra (\mathcal{V}^{\hat{\lambda}}_{\geq w}/\mathcal{V}^{\hat{\lambda}}_{> w})_{\mathcal{F}_{B}^{2}} = (\mathcal{L}^{\hat{\lambda}})_{(\mathcal{F}_{T}^{2})^{w}} \]
is non-zero map of $\mathcal{O}_{X_{S}}$-modules (cf. \cite[Page~14]{BFGM},\cite[Page~48]{BG2},\cite[Propositions~4.3.2,4.4.2]{Schi}). In particular, since $\mathcal{L}^{\hat{\lambda}}_{\mathcal{F}_{T}^{1}}$ is a line bundle, the map being non-zero implies it is a fiberwise injective map of line bundles. Now, recalling our choice of Borel, if we define $\theta:=\nu-w^{-1}(\nu')$ then the support of the torsion of the cokernel of this map of line bundles determines a point in $\Div^{(\theta)}$, by the assumptions on the degrees and \cite[Lemma~4.2.9]{HHS}. For this stratum to be non-empty, we must have that $\theta \in \Lambda_{G,B}^{pos}$. Therefore, for all non-empty strata, we have a map:
\[ \pi_{w}: Z^{\ol{\nu},\ol{\nu}'}_{w} \ra \Div^{(\theta)}. \]
Now, with these preparations out of the way, let's start with the proof. For Point (1), note that if $w = 1$ then we have an injective map of line bundles 
\[ \mathcal{L}^{\hat{\lambda}}_{\mathcal{F}_{T}^{1}} \ra \mathcal{L}^{\hat{\lambda}}_{\mathcal{F}_{T}^{2}} \]
for all $\hat{\lambda}$, which gives rise to the embeddings defined by $\mathcal{F}_{B}^{1}$ when composed with the injections of bundles $\mathcal{L}^{\hat{\lambda}}_{\mathcal{F}_{T}^{2}} \ra \mathcal{V}^{\hat{\lambda}}_{\mathcal{F}_{G}}$ defined by $\mathcal{F}_{B}^{2}$, by construction.

Since both $\mathcal{L}^{\hat{\lambda}}_{\mathcal F_T^1}
\hookrightarrow
\mathcal{V}^{\hat{\lambda}}_{\mathcal F_G}$ and $\mathcal{L}^{\hat{\lambda}}_{\mathcal F_T^2}
\hookrightarrow
V^{\hat{\lambda}}_{\mathcal F_G}$ are subbundle embeddings, their cokernels are
torsion-free. In particular, if the injection
\[
\mathcal{L}^{\hat{\lambda}}_{\mathcal F_T^1}
\hookrightarrow
\mathcal{L}^{\hat{\lambda}}_{\mathcal F_T^2}
\]
were not an isomorphism, its nonzero torsion
cokernel would inject into the torsion-free
cokernel of $\mathcal{L}^{\hat{\lambda}}_{\mathcal F_T^1}
\hookrightarrow
\mathcal{V}^{\hat{\lambda}}_{\mathcal F_G}$,
which is impossible. Therefore, $\mathcal{F}_{T}^{1} \simeq \mathcal{F}_{T}^{2}$, but this contradicts our assumption that $\ol{\nu}, \ol{\nu}' \in B(T)$ map to $b \neq b'$ in $B(G)$. Therefore, we have established point (1). For point (2), we write $\mf{q}_{1}$ (resp. $\mf{q}_{2}$) for the natural projections of $Z_{w}^{\ol{\nu},\ol{\nu}'}$ to $\Bun_{T}^{\ol{\nu}}$ (resp. $\Bun_{T}^{\ol{\nu}'}$). We note, by the above discussion, that $\mf{q}_{2}$ is equal to the composition 
\[
Z_w^{\bar\nu,\bar\nu'}
\xrightarrow{\mf{q}_1\times\pi_w}
\operatorname{Bun}_T^{\bar\nu}\times\operatorname{Div}^{(\theta)}
\xrightarrow{h^{\ra}_{(-\theta)}}
\operatorname{Bun}_T^{w^{-1}(\bar\nu')}
\xrightarrow{w^{-1}}
\operatorname{Bun}_T^{\bar\nu'}.
\]
where $h^{\ra}_{(-\theta)}$ is the map sending $(\mathcal{F}_{T},(D_{i})_{i \in \mathcal{J}})$ to the bundle $\mathcal{F}_{T}(\sum_{i \in \mathcal{J}} \alpha_{i} \cdot D_{i})$ (Note the different sign used in defining these operators compared to the ones considered in \S \ref{GCFT}) and the last map is given by conjugation by $w$ on $\Bun_{T}$. Recall, $(\CT_{*}^{\ol{\nu}} \circ \nmEis^{\ol{\nu}'}_{*}(\mathcal{S}_{\phi_{T}}))_{w}$ is given (up to a shift) by the sheaf
\[ \mf{q}_{1*} \circ \phantom{}^{'}\mf{p}^{\ol{\nu}!} \circ \mf{q}^{\ol{\nu}'*}(\chi \otimes \delta_{B}^{1/2}).  \]
We can write this as the Verdier dual of $\mf{q}_{1!} \circ \phantom{}^{'}\mf{p}^{\ol{\nu}*} \circ \mf{q}^{\ol{\nu}'!}(\chi^{-1} \otimes (\delta_{B}^{1/2})^{-1}) \simeq \mf{q}_{1!} \circ \phantom{}^{'}\mf{p}^{\ol{\nu}*} \circ \mf{q}^{\ol{\nu}'*}(\chi^{-1} \otimes \delta_{B}^{1/2})$, where the last isomorphism is Theorem \ref{dualob}. Replacing $\chi$ by $\chi^{-1}$, this reduces us to showing that $\mf{q}_{1!} \circ \phantom{}^{'}\mf{p}^{\ol{\nu}*} \circ \mf{q}^{\ol{\nu}'*}(\chi \otimes \delta_{B}^{1/2}) \simeq \mf{q}_{1!} \circ \mf{q}_{2}^{*}(\chi \otimes \delta_{B}^{1/2})$ is isomorphic to $\chi^{w^{-1}}\otimes\delta_B^{-1/2}$-tensored by a complex of $\Lambda$-modules. Using the above factorization of $\mf{q}_{2}$, we rewrite this as
\[
\mf{q}_{1!}(\mf{q}_1\times\pi_w)^*
(h^\to_{(-\theta)})^*(w^{-1})^*
(\chi\otimes\delta_B^{1/2}) \simeq
\mf{q}_{1!}(\mf{q}_1\times\pi_w)^*
(
(\chi^{w^{-1}}\otimes(\delta_B^{1/2})^{w^{-1}}\bigr)
\boxtimes
(
E^{(\theta)}_{
\phi_T^{w^{-1}}\otimes(\hat\rho^{w^{-1}}\circ|\cdot|)
}
)^\vee),
\]
Here
\((E^{(\theta)}_{\phi_T^{w^{-1}}\otimes
(\hat\rho^{w^{-1}}\circ|\cdot|)})^\vee\)
denotes the naive dual of the local system
\(E^{(\theta)}_{\phi_T^{w^{-1}}\otimes
(\hat\rho^{w^{-1}}\circ|\cdot|)}\)
 defined by (\ref{eqn: symmetrizedlocalsystem}). However, now, we claim, we are reduced to the following, which is an analogue of \cite[Proposition~10.10]{BG2}
\begin{proposition}{\label{splitmap}}
The direct image $(\mf{q}_{1} \times \pi_{w})_{!}(\Lambda)$ is isomorphic to an iterated extension of complexes pulled back from $\Div^{(\theta)}$ along the projection map $\Bun_{T}^{\ol{\nu}} \times \Div^{(\theta)} \ra \Div^{(\theta)}$ tensored by $(\delta_B^{-1}\otimes\delta_{B'})^{w^{-1}}$, where $B' := TU' \subset B$ for $U' \subset U$ the subgroup generated  by $\alpha > 0$ such that $w^{-1}(\alpha) < 0$.
\end{proposition}
\end{proof}
In particular, by applying projection formula, this tells us that $\mf{q}_{1!} \circ \mf{q}_{2}^{*}(\chi \otimes \delta_{B}^{1/2})$ is isomorphic to $\chi^{w^{-1}}\otimes(\delta_B^{1/2})^{w^{-1}}
\otimes(\delta_B^{-1}\otimes\delta_{B'})^{w^{-1}}$ tensored by an iterated extension of complexes pulled back from $\Div^{(\theta)}$. However, we now easily check that we have an equality
\[
(\delta_B^{1/2})^{w^{-1}}
\otimes(\delta_B^{-1}\otimes\delta_{B'})^{w^{-1}}
\simeq\delta_B^{-1/2}.
\]
so that we indeed conclude that $\mf{q}_{1!} \circ \mf{q}_{2}^{*}(\chi \otimes \delta_{B}^{1/2})$ is isomorphic to $\chi^{w^{-1}} \otimes \delta_{B}^{-1/2}$ tensored by an iterated extension of complexes pulled back from $\Div^{(\theta)}$, as desired.

We will now prove Proposition \ref{splitmap} by relating the spaces $Z^{\ol{\nu},\ol{\nu}'}_{w}$ to some variants of what are called Zastava or semi-infinite flag spaces in the classical literature, as first studied over function fields by Feigin, Finkelberg, Kuznetsov, and Mirkovi\'c \cite{FFKM,FM}. 
\subsubsection{Zastava Spaces}
We recall that the unipotent part $U'$ of the subgroup $B' \subset B$ appearing in proposition \ref{splitmap} acts on the Schubert cell $(G/B)_{\leq w}$ and use this to define the $w$-twisted version of the Zastava space. 
\begin{definition}
For $\theta \in \Lambda_{G,B}^{pos}$, we let $W_{w}^{\theta} \ra \Div^{(\theta)}$ be the $v$-sheaf parameterizing, for $S \in \Perf_{\ol{\bb{F}}_{p}}$, a triple
\[ (\mathcal{F}_{U'},s,D) \]
of the datum:
\begin{itemize}
    \item A $U'$-bundle $\mathcal{F}_{U'}$ on $X_{S}$. 
    \item A section $s: X_{S} \ra \mathcal{F}_{U'} \times^{U'} (G/B)_{\leq w}$ that does not lie in $(G/B)_{\leq w'}$ for any $w' < w$ in the Bruhat order.
    \item An $S$-point $D$, $S \ra \Div^{(\theta)}$ of $\Div^{(\theta)}$ such that the induced non-zero ($\implies$ fiberwise injective) maps of line bundles
    \[ \mathcal{L}^{\hat{\lambda}} \ra (\mathcal{V}^{\hat{\lambda}}_{\geq w})_{\mathcal{F}_{U'}} \ra (\mathcal{V}^{\hat{\lambda}}_{\geq w}/\mathcal{V}^{\hat{\lambda}}_{> w})_{\mathcal{F}_{U'}} = \mathcal{O}_{X_{S}}  \]
    given by $s$ for all $\hat{\lambda} \in \hat{\Lambda}_{G}^{+}$ have cokernel determined by $D$. Here the equality $(\mathcal{V}^{\hat{\lambda}}_{\geq w}/\mathcal{V}^{\hat{\lambda}}_{> w})_{\mathcal{F}_{U'}} = \mathcal{O}_{X_{S}}$ follows from the fact that $\mathcal{V}^{\hat{\lambda}}_{\geq w}/\mathcal{V}^{\hat{\lambda}}_{> w}$ is a $1$-dimensional $B$-representation through which $U$ and in turn $U'$ acts trivially (\cite[Lemma~4.3.4 (a)]{Schi}). 
\end{itemize}
\end{definition}
Classically, the usual Zastava space in the literature is the same datum as above in the case that $w = w_{0}$ together with a level structure on the bundle $\mathcal{L}^{\hat{\lambda}}$ so that it encodes information about enhanced $B$-structures on one of the factors. Its importance is that it provides a local model for the singularities of the Drinfeld compactification $\ol{\Bun}_{B}$. The space we have defined above in the case that $w = w_{0}$ is the open part of the Zastava space which ($\ell$-cohomologically smooth locally) models just the open subspace $\Bun_{B} \hookrightarrow \ol{\Bun}_{B}$ (see \cite[Proposition~5.4.4]{HHS}, where we note that $\phantom{}_{0}\tilde{Z}^{\theta}$ in the notation of \emph{loc.cit} agrees with $W_{w_{0}}^{\theta}$ when $P = B$) . As seen in our description of $Z^{\ol{\nu},\ol{\nu}'}_{w}$ in the previous section, this will clearly have a relationship to the spaces we are interested in describing. Let us first just consider the case of the element of longest length. We claim that the following is true.
\begin{lemma}{\label{ZastavaCartesian}}
Let $w = w_{0}$ be the element of longest length and $\theta = \ol{\nu} - w(\ol{\nu}') \in \Lambda_{G,B}^{pos}$ for $\ol{\nu}'$ and $\ol{\nu}$. We write $\eta: \mc{P} \ra \Bun_{T}^{\ol{\nu}} \times \Div^{(\theta)} =: \mc{B}$ for the torsor, which, for $S \in \Perf$ defining an $S$-point ($\mathcal{F}_{T}^{1}$,$\sum_{i \in \mathcal{J}} D_{i})$ of $\mc{B}$, parametrizes trivializations of the $T$-bundle $\mathcal{F}_{T}^{2} := (\mathcal{F}_{T}^{1}(\sum_{i \in \mathcal{J}} \alpha_{i}D_{i})^{w}$ on the formal completion of $D$ inside $X_{S}$. We consider the Cartesian diagram 
\[ 
\begin{tikzcd}
Z_{w}^{\ol{\nu},\ol{\nu}'} \times_{\mc{B}} \mc{P} \arrow[r] \arrow[d] &  Z_{w}^{\ol{\nu},\ol{\nu}'} \arrow[d,"\mf{q}_{1} \times \pi_{w}"]\\
\mc{P} \arrow[r] & \mc{B} 
\end{tikzcd}
\]
and the natural map 
\[  \mc{P} \xrightarrow{\eta} \Bun_{T}^{\ol{\nu}} \times \Div^{(\theta)} =: \mc{B} \xrightarrow{p_{2}} \Div^{(\theta)}\]
defined by composing with the projection. There then exists a commutative diagram
\[ \begin{tikzcd}
&  Z^{\ol{\nu},\ol{\nu}'}_{w} \times_{\mc{B}} \mc{P}  \arrow[d] \arrow[r] & W^{\theta}_{w} \arrow[d,"\pi_{\mc{W}}"]   \\
& \mc{P} \arrow[r]  & \Div^{(\theta)}, 
\end{tikzcd} \]
which is a Cartesian square. 
\end{lemma}
\begin{proof}
Fix an $S$-point $(\mathcal{F}_{T}^{1},(D_{i})_{i \in \mc{J}},\beta)$ of $\mathcal P$, and let $X_{S}^{\circ}$ be the complement of the Cartier divisors. We recall that the factorization of $\mf{q}_2$ in terms of $\mf{q}_{1} \times \pi_{w}$ identifies the second torus bundle $\mathcal{F}_{T}^{2}$ of an $S$-point of $Z_{w}^{\ol{\nu},\ol{\nu}'}$ mapping to $(\mathcal{F}_{T}^1,D)$ via $\mf{q}_{1} \times \pi_{w}$ with $\mathcal{F}_{T}^{2}$ as defined above. 

We recall (see \cite[Section~3.3]{GH}) that the set of $B$-bundles with underlying torus bundle $\mathcal{F}_{T}^{2}$ are given by torsors under the group scheme
\[
U_{\mathcal{F}^{2}_{T}}:= \mathcal{F}^{2}_{T} \times^T U
\]
on $X_S$. 
We recall that an $S$-point of $Z^{\ol{\nu},\ol{\nu}'}_{w}$ defines a section
\[ s: X_{S} \ra \mathcal{F}_{B}^{2} \times^{B} (G/B)_{\leq w} \]
which will factor through the big cell $(G/B)_{w_{0}}$ over $X_{S}^{\circ}$. Moreover, with respect to the fixed representative $w_0$, we have a $T$-equivariant identification
\begin{equation}{\label{eqn: trivialization}}
U\xrightarrow{\sim}BwB/B,
\qquad
u\longmapsto u\dot wB,
\end{equation}
where $T$ acts on $U$ by conjugation. The section $s$ trivializes the $U_{\mathcal F_T^2}$-torsor over $X_S^\circ$; using $\beta$ on the formal completion, transport the torsor there to a $U$-torsor and glue it to the trivial $U$-torsor over $X_S^\circ$. By Beauville--Laszlo gluing, this produces a global $U$-bundle $\mathcal F_U$ on
$X_S$, together with a section
\[
s:X_S\longrightarrow\mathcal F_U\times^U G/B
\]
which gives rise to the desired point $\mathcal{W}_{w_{0}}^{\theta}$. Conversely, an $S$-point of $\mathcal W_{w_0}^{\theta}$ has
a canonical trivialization over $X_S^\circ$, supplied by the section $s: X_{S} \ra \mathcal{F}_{U} \times^{U} (G/B)_{\leq w}$ by gluing with $\beta$ and passing under the isomorphism (\ref{eqn: trivialization}) we obtain a point of $Z^{\ol{\nu},\ol{\nu}'}_{w_{0}}$, which defines an inverse $W_{w_{0}}^{\theta} \times_{\Div^{(\theta)}} \mc{P} \ra Z^{\ol{\nu},\ol{\nu}'}_{w_{0}} \times_{\mc{B}} \mc{P}$. The fact that $\eta$ is a $v$-cover is easy to check since every $T$-bundle on the formal completion of $D$ admits a trivialization \'etale-locally and in particular $v$-locally. 
\end{proof}
Now this lemma implies Proposition \ref{splitmap} in the case that $w = w_{0}$. In particular, this tells us that we have an isomorphism 
\[ \mc{P} \times_{\mc{B}} (\Bun_{T}^{\ol{\nu}} \times \Div^{(\theta)}) \times_{\Div^{(\theta)}} W_{w}^{\theta} \simeq Z_{w}^{\ol{\nu},\ol{\nu}'} \times_{\mc{B}} \mc{P} \]
Proper base-change therefore tells us that $\eta^{*}(\mf{q}_{1} \times \pi_{w})_{!}(\Lambda)$ is pulled back from $\Div^{(\theta)}$. We would now like to descend along $\eta$ to see that the same is true for the cohomology sheaves of $(\mf{q}_{1} \times \pi_{w})_{!}(\Lambda)$ implying the claim using the Postnikov filtration. In order to see this, we note that the map $\eta: \mc{P} \ra \mc{B}$ is a torsor under the positive loop group $L^{+}T$ over $\Div^{(\theta)}$. In particular, we can factor this torsor using the congruence subgroups to obtain a filtration 
\[ \eta: \mc{P} \xrightarrow{\eta_{1}} [\mc{P}/L^{+}T^{\geq 1}] \xrightarrow{\eta_{2}} \mc{B}\]
where $\eta_{1}$ is a torsor under $L^{+}T^{\geq 1}$ and $\eta_{2}$ is a torsor under $L^{+}T/L^{+}T^{\geq 1}$. By \cite[Proposition~VI.4.1]{FS}, pullback under $\eta_{1}$ is fully faithful. Moreover, using the exact same line of reasoning as \cite[Proposition~VI.4.1]{FS} together with the fact that $*$-pullback is $t$-exact, we can deduce that the natural truncated adjunction map
\[ A \ra R^{0}\eta_{2*}\eta_{2}^{*}A \]
is an isomorphism for any sheaf $A \in \D(\mc{B})$ concentrated in degree $0$ in the ordinary $t$-structure, by virtue of the fact that $L^{+}T/L^{+}T^{\geq 1}$ has connected fibers over $\Div^{(\theta)}$ (cf. \cite[Propositions~VI.1.11,VI.1.12]{FS}). Now, set
\[
K=R(q_1\times\pi_{w_0})_!\Lambda,
\qquad
C=R\pi_{\mathcal W!}\Lambda.
\]
Arguing as above, tells us that the pullback along either projection
\[
u_0,u_1:\mathcal P\times_{\mathcal B}\mathcal P
\rightrightarrows\mathcal P
\]
is fully faithful on any sheaf concentrated in degree $0$ in the ordinary $t$-structure. For every $j$, this tells us that the induced isomorphism
\[
\eta^*\mathcal H^j(K)\simeq \eta^{*}p_{2}^*\mathcal H^j(C)
\]
on cohomology sheaves respects the descent data along $\eta$ on either side. Indeed, after
transporting the descent datum through this
isomorphism, its difference from the canonical
descent datum is an automorphism of
$u_0^*\eta^{*}p_{2}^*\mathcal H^j(C)$ whose restriction to the
diagonal $\mc{P} \ra \mc{P} \times_{\mc{B}} \mc{P}$ is the identity. Fully faithfulness of
$u_0^*$ implies that this automorphism is the identity.
Thus, we have
\[
\mathcal H^j(K)\simeq p_{2}^*\mathcal H^j(C).
\]
The Postnikov filtration now expresses $K$ as a successive extensions of pullbacks from $\Div^{(\theta)}$, proving Proposition~\ref{splitmap} for $w=w_0$.

For general $w$, we abusively use the same notation
$\eta:\mathcal P\to\mathcal B
=\operatorname{Bun}_T^{\bar\nu}\times\operatorname{Div}^{(\theta)}$
for the torsor of trivializations, along the formal completion of $D$,
of $\mathcal F_T^2 = (\mathcal F_T^1(\sum_{i\in J}\alpha_iD_i))^{w^{-1}}$. Now we would like to apply a similar argument using the spaces $W_{w}^{\theta}$ in the case that $w$ is a general element. However, we run into a problem that, in general, all we get is a map
\[\mc{P} \times_{\mc{B}} (\Bun_{T}^{\ol{\nu}} \times \Div^{(\theta)}) \times_{\Div^{(\theta)}} W_{w}^{\theta} \ra  Z_{w}^{\ol{\nu},\ol{\nu}'} \times_{\mc{B}} \mc{P}   \]
using the isomorphism $U' \xrightarrow{\simeq} BwB/B$, as in the proof of Proposition \ref{ZastavaCartesian} to obtain a $B'$-bundle $\mathcal{F}_{B'}^{2}$ and then inducing to get a $B$-bundle $\mathcal{F}_{B}^{2}$. However, there might be multiple ways of reducing $\mathcal{F}_{B}^{2}$ to get a $B'$-bundle $\mathcal{F}_{B'}^{2}$, so we cannot define an inverse. We show how to manage the difficulties first in an example.
\begin{Example}{\label{Ex: KeyExample}}
Suppose that $G = \GL_{3}$ and let $\ol{\nu}$ correspond to a fixed tuple of integers $(-e,-f,-g) \in \mathbb{Z}^{3} \simeq B(T)$ via the Kottwitz invariant. We suppose $w$ corresponds to the simple reflection exchanging the first and second basis vectors. After rigidifying the $T$-bundle $\mathcal{F}_{T}$ to be isomorphic to $(\mathcal{O}(e),\mathcal{O}(f),\mathcal{O}(g))$, we can view $\Bun_{B}^{\ol{\nu}}$ as the moduli space of torsors under the twisted $U$-torsor
\[ \begin{pmatrix} 1 &   \mathcal{O}(e) \otimes \mathcal{O}(f)^{\vee} & \mathcal{O}(e) \otimes \mathcal{O}(g)^{\vee}  \\ 0 & 1 & \mathcal{O}(f) \otimes \mathcal{O}(g)^{\vee}  \\ 0 & 0 & 1 \end{pmatrix} \]
over $X$, where we view this as a group scheme over $X$ (see \cite[Section~3.3]{GH} as well as \cite[Example~4.1]{Ham1} for a more complete discussion of what this means and how to understand $\Bun_{B}$ this way). Here the automorphisms (up to rigidification) of a point in $\Bun_{B}^{\ol{\nu}}$ are given by considering the $\mathcal{H}^{0}$ Banach-Colmez spaces attached to these bundles. Similarly, after rigidification, we can view $\Bun_{B'}^{\ol{\nu}}$ as the moduli space of torsors under the twisted $U'$-torsor
\[ \begin{pmatrix} 1 &  \mathcal{O}(e) \otimes \mathcal{O}(f)^{\vee}  & 0 \\ 0 & 1 & 0  \\ 0 & 0 & 1 \end{pmatrix} \]
over $X$, and the map $f_{\ol{\nu}}$ is given by taking direct sums of the extension of $\mathcal{O}(f)$ by $\mathcal{O}(e)$ defined by the point in $\Bun_{B'}^{\ol{\nu}}$ with $\mathcal{O}(g)$. In particular, if $e \geq f \geq g$ (so that all the extensions split) we can see that the fibers of the map $f_{\ol{\nu}}: \Bun_{B'}^{\ol{\nu}} \ra \Bun_{B}^{\ol{\nu}}$ are an iterated fibration in the Banach-Colmez spaces $\mathcal{H}^{0}(\mathcal{O}(e) \otimes \mathcal{O}(g)^{\vee})$ and $\mathcal{H}^{0}(\mathcal{O}(f) \otimes \mathcal{O}(g)^{\vee})$. If we assume that these Banach-Colmez spaces are positive it follows from the proof of \cite[Proposition~V.2.1]{FS} that the adjunction
\[ f_{\ol{\nu}!}f_{\ol{\nu}}^{!} \ra \mathrm{id} \]
is an equivalence in this case. In particular, combining this with the above discussion would give us the proof of Proposition \ref{splitmap} in this case. We now consider $d \in \mathbb{N}_{> 0}$, a general tuple $(e,f,g)$, and fix a closed point $\infty \in X$ in the Fargues-Fontaine curve over an algebraically closed complete non-archimedean field $F$ in characteristic $p$. We look at the short-exact sequence of $\mathcal{O}_{X}$-modules
\[ 
 0 \ra \mathcal{O}_{X}(-d) \ra \mathcal{O}_{X} \ra \mathcal{O}_{X,\infty}/t_{\infty}^{d} \ra 0, \]
where $\mathcal{O}_{X,\infty}$ is the completed local ring and $t_{\infty}$ is the uniformizing parameter corresponding to an untilt $C$ of $F$. Tensoring by $\mathcal{O}(g)$, we get a short exact sequence: 
\begin{equation}{\label{eqn: SESsequence}}
 0 \ra \mathcal{O}_{X}(g - d) \ra \mathcal{O}_{X}(g) \ra \mathcal{O}_{X,\infty}/t_{\infty}^{d} \ra 0 
\end{equation}
Let $\ol{\nu}_{d}$ correspond to the tuple of integers $(-e,-f,d - g) \in \mathbb{Z}^{3}$. Then we consider the natural map $f_{\ol{\nu}_{d}}: \Bun_{B'}^{\ol{\nu}_{d}} \ra \Bun_{B}^{\ol{\nu}_{d}}$. If we choose $d$ sufficiently large such that the spaces $\mathcal{H}^{0}(\mathcal{O}(e) \otimes \mathcal{O}_{X}(g - d)^{\vee})$ and $\mathcal{H}^{0}(\mathcal{O}(f) \otimes \mathcal{O}_{X}(g - d)^{\vee})$ are positive Banach-Colmez spaces then the fibers of $f_{\ol{\nu}_{d}}$ will be an iterated fibration in these positive Banach-Colmez spaces, and therefore we can again conclude that the adjunction 
\[ f_{\ol{\nu}_{d}!}f_{\ol{\nu}_{d}}^{!} \ra \mathrm{id} \]
is an isomorphism. Now we claim that we have a map:
\[ \Bun_{B}^{\ol{\nu}} \ra \Bun_{B}^{\ol{\nu}_{d}} \]
Explicitly, given a point $\Bun_{B}^{\ol{\nu}}$, we have an exact sequence
\[0 \ra \mathcal{E}_{2} \ra \mathcal{E} \ra \mathcal{O}(g) \ra 0  \]
of bundles, where $\mathcal{E}_{2}$ is an extension of $\mathcal{O}(e)$ and $\mathcal{O}(f)$. We can then consider the pullback of this exact sequence with respect to the map $\mathcal{O}(g - d) \ra \mathcal{O}(g)$ given by the short exact sequence (\ref{eqn: SESsequence}), which will give us a point in $\Bun_{B}^{\ol{\nu}_{d}}$. We write $f_{\ol{\nu},d\infty}: \Bun_{B}^{\ol{\nu},d\infty} := \Bun_{B}^{\ol{\nu}} \times_{\Bun_{B}^{\ol{\nu}_{d}}} \Bun_{B'}^{\ol{\nu}_{d}} \ra \Bun_{B}^{\ol{\nu}}$ for the pullback of $f_{\ol{\nu}_{d}}$ along this map. We again conclude that the adjunction 
\[ f_{\ol{\nu},d\infty!}f_{\ol{\nu},d\infty}^{!} \ra \mathrm{id} \]
is an isomorphism. If we could use this map instead of $f_{\ol{\nu}}$, we could prove the claim by arguing as above. Indeed, consider $\mf{q}_{1} \times \pi_{w}: Z^{\ol{\nu},\ol{\nu}'}_{w} \ra \Bun_{T}^{\ol{\nu}} \times \Div^{(\theta)}$. If we base-change all the above spaces to $F$, and let $(\Div \setminus \infty)^{(\theta)}$ be the partially symmetrized power defined by $\Div^{1} \setminus \infty$, the open complement of the closed point $\infty \ra \Div^{1}$ defined by the fixed untilt. 

When applying this construction to $Z_w^{\bar\nu,\bar\nu'}$,
we replace $\bar\nu$ by $\bar\nu'$ and use the projection to the
second $B$-bundle. Thus
\[
\widetilde Z
:=
Z_w^{\bar\nu,\bar\nu'}
\times_{\operatorname{Bun}_B^{\bar\nu'}}
\operatorname{Bun}_B^{\bar\nu',d\infty}.
\]
We again write $f_{\bar\nu',d\infty}$ for the resulting map
$\widetilde Z\to Z_w^{\bar\nu,\bar\nu'}$. In this application, the section $s$ is modified at $\infty$
together with the $B$-bundle $\mathcal{F}_{B}^{2}$, using its $H_w$-reduction
on $X_S\setminus D$. Thus the modification preserves the open
$w$-cell at $\infty$ and leaves the defect divisor $D$ unchanged.

There is a forgetful map
\[ r: \tilde{Z} \times_{\Bun_{T}^{\ol{\nu}} \times (\Div \setminus \infty)^{(\theta)}} \mc{P} \ra \mc{P} \times_{(\Div \setminus \infty)^{(\theta)}} W_{w}^{\theta} \]
by the same argument as in Lemma \ref{ZastavaCartesian}. However, this map is not quite an isomorphism. Indeed, the source records the additional data of data of the extension
\[ 0 \ra \mathcal{E}_{2} \ra \mathcal{E} \ra \mathcal{O}(g) \ra 0 \]
corresponding to a $H_{w} := TU'' := B \cap wBw^{-1}$-torsor, where $U''$ is generated by the roots $\alpha > 0$ such that $w^{-1}(\alpha) >0$ together with a $v$-local on $S$ splitting of this (for $d$ sufficiently large as above) after pullback along (\ref{eqn: SESsequence}). For $S \in \Perf$, the groupoid of extensions together with a splitting of their pullback
is the Picard groupoid (as in \cite{Ham1}) homotopy fiber over zero of the map of complexes
\[
  R\Gamma(X_{S},\mathcal V_{S})[1]
     \longrightarrow R\Gamma(X_{S},\mathcal V_{S}(d\infty))[1],
\]
where $\mathcal{V}_{S} := \mathcal{E}_{2,X_{S}} \otimes \mathcal{O}_{X_{S}}(g)^{\vee}$ is as above.

The exact sequence
\[ 0\longrightarrow \mc{V}_S\longrightarrow \mc{V}_S(d\infty)
  \longrightarrow \mc{Q}_S\longrightarrow0,
\qquad \mc{Q}_S:=\mc{V}_S(d\infty)/\mc{V}_S \]
identifies this with the functor sending $S \mapsto H^{0}(X_{S},\mc{Q}_{S})$. This is $v$-locally isomorphic to $(\mathcal{O}_{X_{S},\infty}/t^{d})^{\oplus 2}$. In particular, by filtering by powers of $t$ it is an iterated fibration in copies of the affine line. This implies that the adjunction map
\[ Rr_{!}r^{!} \ra \mathrm{id} \]
is an isomorphism. By combining this with the above descent argument, this allows us to deduce Proposition \ref{splitmap} for the restriction of
\[
(\mf{q}_1\times\pi_w)_!
f_{\bar\nu',d\infty!}f_{\bar\nu',d\infty}^!(\Lambda)
\simeq(\mf{q}_1\times\pi_w)_!(\Lambda).
\] 
to the open stratum $(\Div \setminus \infty)^{(\theta)} \subset \Div^{(\theta)}$. 

This allows us to reduce Proposition \ref{splitmap} to studying this restriction by varying the untilt $\infty$ through a straight-forward diagram chase and performing the above descent argument. More precisely, to see why the character twist by $\delta_{B'} \otimes \delta_{B}^{-1}$ appears in Proposition \ref{splitmap}, we note that $f_{\ol{\nu}_{d}}: \Bun_{B'}^{\ol{\nu}_{d}} \ra \Bun_{B}^{\ol{\nu}_{d}}$ is $\ell$-cohomologically smooth. In particular, one can show, by the exact same argument as \cite[Proposition~3.9]{GH}, that the dualizing complex of $\Bun_{B'}$ is isomorphic to the sheaf $\Delta_{B'}$ on $\Bun_{T}$ given by the modulus character of $\delta_{B'}$ pulled back to $\Bun_{B'}$ and given appropriate shifts, just as in Theorem \ref{dualob}. One can in turn compute the relative dualizing complex of $f_{\ol{\nu}_{d}}$ in terms of the pullback of $\Delta_{B'} \otimes \Delta_{B}^{-1}$. We then see that $f_{\ol{\nu},d\infty}$ is also $\ell$-cohomologically smooth as a base-change of $f_{\ol{\nu}_{d}}$ \cite[Proposition~23.16]{Ecod} with dualizing object $f_{\ol{\nu},d\infty}^{!}(\Lambda)$ computed in terms of $\Delta_{B'} \otimes \Delta_{B}^{-1}$ pulled back to the fiber product $\Bun_{B}^{\ol{\nu},d\infty}$. Moreover, this dualizing complex can be identified with $r^{!}(\mathcal{L})$ up to a shift, where $\mathcal{L}$ is defined by pulling back $\Delta_{B'} \otimes \Delta_{B}^{-1}$ on $\Bun_{T}$ (indeed note that the scaling action of $\bb{Q}_{p}^{*}$ on the dualizing complex of affine space is trivial (cf. the proof of \cite[Lemma~2.6]{GH})).
\end{Example}

With this motivating example, all that remains is to formalize the above argument. In particular, first off note, by \cite[Corollary~V.2.3]{FS}, that for the proof of Proposition \ref{splitmap} it suffices to consider the base-change of all the above spaces to the base $\ast = \Spd(F)$ for $F$ an algebraically closed perfectoid field in characteristic $p$. We consider such a field with fixed characteristic $0$ untilt $C$, and let $\infty \ra \Div^{1}$ be the closed $F$-point defined by this untilt. We consider the open complement $\Div^{1} \setminus \infty$, and the partially symmetrized powers $(\Div^{1} \setminus \infty)^{(\theta)}$ defined by the open subset. Since the statement is analytically local on $\Div^{(\theta)}$, by varying the point $\infty$ it suffices to verify Proposition \ref{splitmap} over this open subset. We abuse notation and write $Z^{\ol{\nu},\ol{\nu}'}_{w}$ and $W_{w}^{\theta}$ for the base-change to this open subspace for the rest of the section. Now let's consider the spaces defined by the locally pro-finite sets $\ul{B(\mathbb{Q}_{p})}$ (resp. $\ul{B'(\mathbb{Q}_{p})}$), and the natural map
\[ f_{0}: \Bun_{B'}^{0} \simeq [\ast/\ul{B'(\mathbb{Q}_{p})}] \ra \Bun_{B}^{0} \simeq [\ast/\ul{B(\mathbb{Q}_{p})}] \]
of $v$-stacks. The fibers of this map are an iterated fibration in $\mathcal{H}^{0}(\mathcal{O}_{X}) = \ul{\mathbb{Q}_{p}}$ indexed by the positive roots $\hat{\alpha} > 0$ such that $w^{-1}(\hat{\alpha}) > 0$.

We apply the modification to the \emph{second} $B$-bundle,
whose torus degree is $\bar\nu'$. Choose a modification
$\mathcal F_{\bar\nu'}\dashrightarrow\mathcal F_{\bar\nu_\infty}$
at $\infty$, arising from a sufficiently large multiple of a
regular Galois-invariant cocharacter, and write
$\bar\nu_\infty\in B(T)$ for the resulting degree.
Choose its direction so that, for each positive root $\hat\alpha$,
the map of root bundles is a regular inclusion with order
$n_{\hat\alpha}\geq0$ and the modified root bundles for
$\operatorname{Lie}(N_w)$ have strictly positive slopes, where
$N_w$ is the unipotent radical of $H_w:=B\cap wBw^{-1}$.
Since the orders arise from a single cocharacter, these
modifications define a map
\[
\operatorname{Bun}_B^{\bar\nu'}
\longrightarrow\operatorname{Bun}_B^{\bar\nu_\infty}.
\]
The fibers of
\[
f_{\bar\nu_\infty}:
\operatorname{Bun}_{B'}^{\bar\nu_\infty}
\longrightarrow\operatorname{Bun}_B^{\bar\nu_\infty}
\]
are iterated positive Banach--Colmez fibration.  We define
\[
\operatorname{Bun}_{B'}^{\bar\nu',\infty}:=
\operatorname{Bun}_{B'}^{\bar\nu_\infty}
\times_{\operatorname{Bun}_B^{\bar\nu_\infty}}
\operatorname{Bun}_B^{\bar\nu'},
\qquad
f_{\bar\nu',\infty}:
\operatorname{Bun}_{B'}^{\bar\nu',\infty}
\longrightarrow\operatorname{Bun}_B^{\bar\nu'}.
\]
By the proof of \cite[Proposition~V.2.1]{FS}, we then have $f_{\bar\nu',\infty!}f_{\bar\nu',\infty}^{!}
\simeq\operatorname{id}$. Finally, put
\[
\widetilde Z_w^{\bar\nu,\bar\nu'}:=
Z_w^{\bar\nu,\bar\nu'}
\times_{\operatorname{Bun}_B^{\bar\nu'}}
\operatorname{Bun}_{B'}^{\bar\nu',\infty}.
\]
By the previous adjunction, it suffices to show the analogue of Proposition \ref{splitmap} for the composition 
\[ \tilde{Z}_{w}^{\ol{\nu},\ol{\nu}'} \ra Z_{w}^{\ol{\nu},\ol{\nu}'} \xrightarrow{\mf{q}_{1} \times \pi_{w}} \Bun_{T}^{\ol{\nu}} \times (\Div \setminus \infty)^{(\theta)} \]
and the shriek pullback of the constant sheaf along the first map.  However, by now arguing exactly as in the proof of Lemma \ref{ZastavaCartesian}, with $B'$ and its unipotent radical $U'$ replacing $B$ and $U$, we can deduce the existence of a map
\[
r:
\widetilde Z_w^{\bar\nu,\bar\nu'}
\times_{\mathcal B}\mathcal P
\longrightarrow
\mathcal P
\times_{(\operatorname{Div}\setminus\infty)^{(\theta)}}W_w^\theta.
\]
If we set $H_{w} := B \cap wBw^{-1}$ then, as in Example \ref{Ex: KeyExample}, and let $N_{w}$ denote its unipotent radical. This will $v$-locally have fibers that are an iterated fibration indexed by the root spaces $\hat{\alpha}$ of $\Lie(N_{w})$. For $t$ a uniformizing parameter of $\infty$ and $n_{\hat{\alpha}}$ sufficiently large relative to the modifications, the successive pieces of this fibration will be $v$-locally given by $t^{-n_{\hat\alpha}}\bb{B}_{\mathrm{dR},F}^{+}/\bb{B}_{\mathrm{dR},F}^{+}
\simeq \bb{B}_{\mathrm{dR},F}^{+}/t^{n_{\hat\alpha}}$. Filtering by powers of $t$, this will be an iterated fibration in the affine spaces $\bb{A}^{1,\Diamond}_{F}$ and hence we can deduce that the adjunction map
\[ r_{!}r^{!} \ra \mathrm{id} \]
is an isomorphism. Now, as in Example \ref{Ex: KeyExample}, this can be used to prove Proposition \ref{splitmap} after pullback along $\eta: \mc{P} \ra \mc{B}$, where the twist by $\delta_{B'} \otimes \delta_{B}^{-1}$ comes from the relative dualizing complex of $f_{\ol{\nu}_{\infty}}$ and in turn $f_{\ol{\nu},\infty}$ just as in the previous Example. Moreover, this dualizing complex can be identified with $r^{!}(\mathcal{L})$ up to a shift, where $\mathcal{L}$ is defined by pulling back $\Delta_{B'} \otimes \Delta_{B}^{-1}$ on $\Bun_{T}$ (indeed note that the scaling action of $\bb{Q}_{p}^{*}$ on the dualizing complex of affine space is trivial (cf. the proof of \cite[Lemma~2.6]{GH})). This will then prove the claim after pullback along $\eta$. By performing the above descent argument, this implies the claim before pullback along $\eta$.
\subsection{Commutation with Verdier Duality}
In this section, we will establish the following.
\begin{theorem}{\label{commverddual}}
Let $\phi_{T}$ be a toral parameter with generic regular mod $\varpi$-reduction and that assumption \ref{compatibility} holds, we have that 
\[ \mathbb{D}_{\Bun_{G}}\nmEis(\mathcal{S}_{\phi_{T}}) \simeq \nmEis\mathbb{D}_{\Bun_{T}}(\mathcal{S}_{\phi_{T}}), \]
where $\mathbb{D}_{\Bun_{T}}(\mathcal{S}_{\phi_{T}}) \simeq \mathcal{S}_{\phi_{T}^{\vee}}$. 
\end{theorem}
For $b \in B(G)_{\mathrm{un}}$ and $w \in W_{b}$, we set $\rho^{\chi}_{b,w} := i_{B_{b}}^{J_{b}}(\chi^{w^{-1}}) \otimes \delta_{P_{b}}^{-1/2}$ and suppress the superscript by $\chi$ when it is clear from context.
Using Corollary \ref{normstalksdescription}, we reduce Theorem \ref{commverddual} to the following claim.
\begin{theorem}{\label{thm: inertsheavestorsion}}
For all $b \in B(G)_{\mathrm{un}}$, $w \in W_{b}$, and $\chi$ associated to a toral parameter $\phi_{T}$ with generic regular mod $\varpi$-reduction, assuming \ref{compatibility} the natural map 
\[ j_{b!}(\rho_{b,w}) \ra Rj_{b*}(\rho_{b,w}) \]
is an isomorphism.
\end{theorem}
\begin{proof}
We note that the sheaf $Rj_{b*}(\rho_{b,w})$ is ULA over the point since it is the Verdier dual of a complex which is ULA over the point using \cite[Theorem~V.7.1]{FS} and \cite[Corollary~IV.2.25]{FS}. If we write $\pi: \Bun_{G} \ra \ast$ and $\pi^{b}: \Bun_{G}^{b} \ra \ast$ and $A := \rho_{b,w}$ and let $B \in \D(\Lambda) = \Detale(\ast,\Lambda)$ be a $\Lambda$-module then we have
\[
\begin{aligned}
Rj_{b*}A \otimes_{\Lambda} \pi^{*}(B)
&\simeq
R\mathcal{H}om(\bb{D}_{\Bun_{G}}(Rj_{b*}A), R\pi^!B) \\
&\simeq
R\mathcal{H}om(j_{b!}\bb{D}_{\Bun_{G}^{b}}(A), R\pi^!B) \\
&\simeq
Rj_{b*}R\mathcal{H}om(\bb{D}_{\Bun_{G}^{b}}(A), j_{b}^{!}R\pi^!B) \\
&\simeq
Rj_{b*}R\mathcal{H}om(\bb{D}_{\Bun_{G}^{b}}(A), R\pi_{b}^!B) \\
&\simeq
Rj_{b*}(A\otimes_{\Lambda}\pi_{b}^*B).
\end{aligned}
\]
where the first and last isomorphisms use the ULAness of the sheaves and \cite[Proposition~IV.2.19]{FS}; the second uses Verdier duality for $j_b$; the third uses adjunction for internal Hom \cite[Proposition~23.3]{Ecod}; and the fourth uses $j_b^!\pi^!\simeq\pi_b^!$. Now taking $B = \ol{\bb{F}}_{\ell}$ and applying derived Nakayama, we may assume without loss of generality that $\Lambda = \ol{\bb{F}}_{\ell}$, as in the proof of Proposition \ref{nonsplitcontribution}.

By \cite[Section~IX.7.1]{FS}, any irreducible constituent of $Rj_{b*}(\rho_{b,w})|_{\Bun_{G}^{b'}} \in \D(J_{b'}(\mathbb{Q}_{p}),\Lambda)$ must have Fargues-Scholze parameter equal to $\phi$. Therefore, by the argument for Corollary \ref{noredvanishing}, the ULAness of $Rj_{b*}(\rho_{b,w})$ and Assumption \ref{compatibility}, the restriction $Rj_{b*}(\rho_{b,w})|_{\Bun_{G}^{b'}}$ will only be non-zero if $b' \in B(G)_{\mathrm{un}}$, and any irreducible constituent must be an irreducible constituent of $\rho_{b',w'}$, for some $b' \in B(G)_{\mathrm{un}}$ and $w' \in W_{b'}$.

Moreover, since the sheaf $Rj_{b*}(\rho_{b,w})$ is ULA over the point. This allows us to deduce that $Rj_{b*}(\rho_{b,w})|_{\Bun_{G}^{b'}}$ is a bounded complex with finite length cohomology, as in the proof of Proposition \ref{nonsplitcontribution}. This reduces us to showing that the complex 
\[ \RHom_{J_{b'}(\mathbb{Q}_{p})}(j_{b*}(\rho_{b,w})|_{\Bun_{G}^{b'}},\rho_{b',w'}) \simeq \RHom(j_{b*}(\rho_{b,w}), j_{b'*}(\rho_{b',w'})) \]
is trivial for all $b' \in B(G)_{\mathrm{un}}$ and $w' \in W_{b'}$ with $b' \neq b$. We set $\ol{\nu} = w(b_{T})$. By using Theorem \ref{dualob}, Corollary \ref{normstalksdescription}, and the fact that generic regularity is a condition stable under switching $B$ and $B^{-}$ and $\chi$ and $\chi^{-1}$, we obtain that it suffices to show that 
\[ \RHom(\nmEis_{*}^{\ol{\nu}}(\mathcal{S}_{\phi_{T}}),\nmEis^{\ol{\nu}'}_{-*}(\mathcal{S}_{\phi_{T}}))  \]
for all $\ol{\nu}' \in B(T)$ such that its image $b' \in B(G)$ satisfies that $b \neq b'$. We can now rewrite this as
\[ \RHom(\nmCT^{\nu'}_{-!}\nmEis_{*}^{\nu}(\mathcal{S}_{\phi_{T}}),\chi) \simeq \RHom(\nmCT^{\nu'}_{*}\nmEis_{*}^{\nu}(\mathcal{S}_{\phi_{T}}),\chi), \]
using Theorem \ref{thm: geometricsecondadjointness}. Using projection formula, Proposition \ref{prop: qsmooth}, and Theorem \ref{dualob}, we can further rewrite this as 
\[ \RHom(\CT^{\nu'}_{*}\nmEis_{*}^{\nu}(\mathcal{S}_{\phi_{T}}),\chi \otimes \delta_{B}^{1/2}). \] 
The claim now follows by Proposition \ref{bruhatstrata}, as in the proof of Proposition \ref{nonsplitcontribution}.
\end{proof}

\section{The Hecke Eigensheaf Property}{\label{consteigsheaf}}
\subsection{Tilting Eigensheaves}{\label{s: TiltingEigensheaves}}
We would now like to combine our work in the previous sections and use it to construct eigensheaves. We would like to do this in a uniform way for coefficient systems $\Lambda$ appearing in \ref{assumption: coefficients}, where $\Lambda$ has the $\ell$-adic topology. One of the issues is that the representation theory of $\phantom{}^{L}G/\Lambda$ is substantially different if $\Lambda$ is rational versus if it is torsion. With rational coefficients, the category is semisimple with simple objects given by $V_{(\mu^{\Gamma},\tau)}$ for $\mu^{\Gamma} \in \domgamorb$ a $\Gamma$-orbit of a geometric dominant cocharacter $\mu$ and $\tau$ an irreducible representation of the finite group $Q_{\mu} := W_{E_{\mu}}/W_{E}$. If $\Lambda$ is torsion then this is no longer so straightforward; the representations $V_{(\mu^{\Gamma},\tau)}$ can fail to be irreducible. To develop a good theory of algebraic representations with these coefficients, we will need to invoke our assumption that $\ell$ is very decent with respect to $G$ and use the theory of tilting modules (in particular, the assumption that $\ell \nmid |Q| = |W_{\bb{Q}_{p}}/W_{E}|$, where $E$ is the splitting field). We first discuss the general notion of a tilting module. Fix $H$ a split connected reductive group over any coefficient system $\Lambda/\bb{Z}_{\ell}$, with a split maximal torus $T_H\subset H$ and a choice of Borel containing it. We consider the involution
\[ \mathbb{D}: \Rep_{\Lambda}(H) \ra \Rep_{\Lambda}(H) \]
\[ V \mapsto (V^{*})^{\sigma} \]
where $V^{*}$ is the dual representation and $\sigma$ is the Chevalley involution. For $\lambda \in \mathbb{X}^{*}(T_H)^{+}$ a dominant character, we let $V^{\lambda}$ denote the highest weight representation attached to $\lambda$ by Borel-Weil-Bott, and we write $V_{\lambda} := \mathbb{D}(V^{\lambda})$. We now come to our key definition.
\begin{definition}{\label{defn: TiltingModules}}
Given $V \in \Rep_{\Lambda}(H)$, we say that $V$ has a Weyl (resp. good) filtration if it admits a filtration whose successive quotients are isomorphic to $V^{\lambda}$ (resp. $V_{\lambda}$). We say $V$ is tilting if it admits both a good and a Weyl filtration. We write $\Tilt_{\Lambda}(H) \subset \Rep_{\Lambda}(H)$ for the full sub-category of tilting modules. 
\end{definition} 
The category of tilting modules is additive, but usually not abelian. If $\Lambda$ is rational (i.e as in \ref{assumption: coefficients} (4)) the highest weight modules are simple, and $\Tilt_{\Lambda}(H) = \Rep_{\Lambda}(H)$. Therefore, this is only an interesting notion if $\Lambda$ is torsion or is integral. The key point of moving to this sub-category is that we have the following generalization of usual highest weight theory due Ringel and Donkin \cite{Rin}, \cite{Don}. 
\begin{theorem}{\label{highestweightilting}}{\label{thm: highestweighttiltingmodules}}
For each $\lambda \in \mathbb{X}^{*}(T_H)^{+}$, there exists a unique indecomposable tilting module $\mathcal{T}_{\lambda} \in \Tilt_{\Lambda}(H)$ with highest weight $\lambda$. We have that $\dim(\mathcal{T}_{\lambda}(\lambda)) = 1$, and, for varying $\lambda$, this parametrizes all indecomposable tilting modules.
\end{theorem}
We also get the usual classification of all tilting modules in terms of highest weight tilting modules.
\begin{proposition}{\cite[Section~E.22]{JanRepresentationsofAlgebraicGroups},\cite[Lemma~7.3]{Mat}}{\label{directsumtilting}}
For all $V \in \Tilt_{\Lambda}(H)$, there exists unique integers $n(\lambda) \in \mathbb{N}_{\geq 0}$ for all $\lambda \in \mathbb{X}^{*}(T_H)^{+}$ and an isomorphism
\[ V \simeq \bigoplus_{\lambda \in \mathbb{X}^{*}(T_H)^{+}} (\mathcal{T}_{\lambda})^{n(\lambda)} \]
of tilting modules. 
\end{proposition}
Now we come to a difficult result which was proven by \cite{Wang} for groups of type $A_{n}$, \cite{Don} for almost all groups, and \cite{Mat} in general. 
\begin{theorem}{\label{tensortilting}}
If we have two tilting modules $V,V' \in \Tilt_{\Lambda}(H)$ then the tensor product $V \otimes V'$ is tilting. 
\end{theorem}
We can now extend this to the $L$-group using our assumption that $\ell$ is very decent with respect to $G$; in particular, our assumption that $\ell \nmid |Q|$. By Theorem \ref{highestweightilting} applied to $H = \hat{G}$, we deduce that we have a well-defined category $\Tilt_{\Lambda}(\hat{G})$ of tilting modules, where each object can be written as a direct sum of highest weight tilting modules $\mathcal{T}_{\mu}$ for $\mu \in \mathbb{X}_{*}(T_{\ol{\mathbb{Q}}_{p}})^{+}$ using Proposition \ref{directsumtilting}. Now, given such a $\mu$, we consider the reflex field $E_{\mu}$, and extend this to a representation of $W_{E_{\mu}}/W_{E} \ltimes \hat{G}$, as in \cite[Lemma~2.1.2]{Kott1}. We set $Q_{\mu} := W_{E_{\mu}}/W_{E}$, and let $\Rep_{\Lambda}(Q_{\mu})$ be the category of $\Lambda$-valued representations of the finite group $Q_{\mu}$. We let $\Irr_{\Lambda}^{\proj}(Q_{\mu})$, denote a set of representatives for the indecomposable finite-projective $\Lambda[Q_\mu]$-modules. For $\tau \in \Irr^{\proj}_{\Lambda}(Q_{\mu})$, we define the tilting module $\mathcal{T}_{\mu^{\Gamma},\tau}$ as the induction of the tensor product  $\mathcal{T}_{\mu} \otimes \tau$ of $W_{E_{\mu}}/W_{E} \ltimes \hat{G}$-representations (where we regard $\tau$ as a representation of $W_{E_{\mu}}/W_{E} \ltimes \hat{G}$ via the natural map $W_{E_{\mu}}/W_{E} \ltimes \hat{G} \ra Q_{\mu}$) from $W_{E_{\mu}}/W_{E} \ltimes \hat{G}$ to $W_{\bb{Q}_{p}}/W_{E} \ltimes \hat{G} = \phantom{}^{L}G$, and let $\Tilt_{\Lambda}(\phantom{}^{L}G) \subset \Rep_{\Lambda}(\phantom{}^{L}G)$ be the full sub-category given by the additive Karoubi envelope of such objects, where we note that $\mathcal{T}_{\mu^{\Gamma},\tau}$ only depends on the $\Gamma$-orbit $\mu^{\Gamma} \in \domgamorb$ attached to $\mu$. Now, since $\ell$ is very decent, it follows that $W_{\mathbb{Q}_{p}}$ acts on $\hat{G}$ via a quotient $Q$ that is prime to $\ell$ by definition, and therefore $W_{\mathbb{Q}_{p}}/W_{E_{\mu}}$ is also of order prime to $\ell$. It follows that  restriction, induction, tensor products, and taking direct summands preserve finite-projective representations of the relevant subgroups of the varying subgroups $Q_{\mu}$. Combing this observation, Frobenius Reciprocity/Mackey theory, and Theorem \ref{tensortilting}, we conclude that $\mathcal{T}_{\mu^{\Gamma},\tau}$ is indeed indecomposable, and that $\Tilt(\phantom{}^{L}G)$ is preserved under tensor products. Moreover, the indecomposable objects are given by $\mathcal{T}_{\mu^{\Gamma},\tau}$ for varying $\mu \in \domgamorb$ and $\tau \in \Irr^{\proj}_{\Lambda}(Q_{\mu})$. If $\tau$ is trivial then we omit it from the notation. This allows us to define the following. 
\begin{definition}{\label{deftilteigsheaf}}
Given a $L$-parameter $\phi: W_{\mathbb{Q}_{p}} \ra \phantom{}^{L}G(\Lambda)$ continuous with respect to the $\ell$-adic topology, we say a sheaf $\mathcal{S}_{\phi} \in \Dlis(\Bun_{G},\Lambda)$ is a tilting eigensheaf with eigenvalue $\phi$ if, for all finite index sets $I$ and $V \in \Tilt_{\Lambda}(\phantom{}^{L}G^{I})$ with associated map $r_{V}: \phantom{}^{L}G^{I} \ra \GL(V)$, there exist isomorphisms
\[ \eta_{V,I}: T_{V}(\mathcal{S}_{\phi}) \simeq \mathcal{S}_{\phi} \boxtimes r_{V} \circ \phi^{\boxtimes I} \]
of sheaves in $\Dlis(\Bun_{G},\Lambda)^{BW_{\mathbb{Q}_{p}}^{I}}$, which are natural with respect to morphisms of finite index sets $I$ and the fusion product, natural with respect to maps of $\phantom{}^{L}G^{I}$-representations in $V$, and compatible with compositions/exterior tensor products. If $\Lambda = F$ for $F/\bb{Q}_{\ell}$ a discretely valued algebraic extension this recovers Definition \ref{defeigsheaf}. We similarly say $\mathcal{S}_{\phi}$ is a weak tilting eigensheaf if only the isomorphisms $\eta_{V,I}$ exist. 
\end{definition}
\begin{remark}{\label{rem: modularhighestweight}}
Our discussion of highest weight theory in \S \ref{sec: unramifiedelments} for $\hat{G}^{\Gamma}$ also extends to the tilting modules $\mathcal{T}_{\mu}$ under our assumption that $\ell$ is very decent. In particular, using Theorem \ref{highestweightilting}, we can understand the possible weights occurring in $\mathcal{T}_{\mu}|_{\hat{G}^{\Gamma}}$ in terms of $\ol{\nu} \in \coinv$, which lie in the convex hull of the $W_{G}$ orbit of $\mu_{\Gamma} \in \coinvdom$, and the $\ol{\nu}$ weight space of $\mathcal{T}_{\mu}|_{\hat{G}^{\Gamma}}$ will be a direct sum over the weight spaces $\mathcal{T}_{\mu}(\nu)$ for $\nu \in \cochar$ mapping to $\ol{\nu}$, as in Lemma \ref{lemma: coinv versus orbs}. To see this, we note, by \cite[Proposition~VIII.5.11]{FS}\footnote{Here we are citing version 2 on the ArXiv} and our assumption that $\ell$ is very decent, that we have the following:
\begin{enumerate}
    \item $\hat{G}^{\Gamma}$ is a smooth linear algebraic group, and $\hat{G}^{\Gamma,\circ}$ is reductive.   
    \item $\hat{G}^{\Gamma}/\hat{G}^{\Gamma,\circ} \simeq \hat{T}^{\Gamma}/\hat{T}^{\Gamma,\circ}$ is of order prime to $\ell$, where $\hat{G}^{\Gamma,\circ}$ (resp. $\hat{T}^{\Gamma,\circ}$) denotes the neutral component  of $\hat{G}^{\Gamma}$ (resp. $\hat{T}^{\Gamma}$), and the isomorphism follows as in \cite[Lemma~4.6]{Zhu}.
\end{enumerate}
These observations allows us to see the highest weight theory of the (possibly disconnected) group $\hat{G}^{\Gamma}$ behaves as expected with modular coefficients by mimicking the proof of Lemma \ref{dualgrouphighestweight}.
\end{remark}

\subsection{The Construction of the Eigensheaf}
We fix a toral parameter $\phi_{T}: W_{\mathbb{Q}_{p}} \ra  \phantom{}^{L}T(\Lambda)$, with induced parameter $\phi: W_{\mathbb{Q}_{p}} \xrightarrow{\phi_{T}} \phantom{}^{L}T(\Lambda) \ra \phantom{}^{L}G(\Lambda)$. Our goal is to construct a candidate tilting eigensheaf with respect to the parameter $\phi$. If $\Lambda$ is torsion then we have already carried this out. It is simply the sheaf $\nmEis(\mathcal{S}_{\phi_{T}}) \in \D(\Bun_{G})$ viewed as a sheaf in $\Dlis(\Bun_{G},\Lambda)$ via the identification $\D(\Bun_{G}) \simeq \Dlis(\Bun_{G},\Lambda)$ \cite[Proposition~VII.6.6]{FS}, obtained by embedding both categories into $\D_{\blacksquare}(\Bun_{G},\Lambda)$. Here we regard $\Lambda/\bb{Z}_{\ell}$ as a condensed $\bb{Z}_{\ell}$-algebra which is relatively discrete over $\bb{Z}_{\ell}$ when forming these categories of solid sheaves. To move beyond this case, we need to invoke the following Lemma.
\begin{lemma}{\label{EisULA}}
Assuming \ref{compatibility}, we have, for $\Lambda$ as in Setup~\ref{assumption: coefficients}(1) and $\phi_{T}$ with generic regular mod $\varpi$-reduction, that the sheaf $\nmEis(\mathcal{S}_{\phi_{T}}) \in \D(\Bun_{G})$ is ULA with respect to the structure map $\Bun_{G} \ra \ast$.
\end{lemma}
\begin{proof}
Given $A \in \D(\Bun_{G})$, we recall that $A$ being ULA with respect to the map $\Bun_{G} \ra \ast$ is equivalent to saying that its stalks $A|_{\Bun_{G}^{b}}$ are valued in a complex of smooth representations such that $A^{K}|_{\Bun_{G}^{b}}$ is a perfect complex of $\Lambda$-modules for all open pro-$p$ subgroups $K \subset J_{b}(\mathbb{Q}_{p})$ \cite[Theorem~V.7.1]{FS}. In particular, the result follows from Corollary \ref{normstalksdescription}.
\end{proof}
Now let $\Lambda$ be as in Setup \ref{assumption: coefficients} (2), so $\Lambda = \mathcal{O}_{F}$ for $F/\bb{Q}_{\ell}$ a finite extension.  Then, by taking inverse limits with respect to the mod $\ell^{n}$ reductions of $\phi_{T}$, and considering the systems of sheaves given by applying the Eisenstein functor to the eigensheaf attached to these reductions, we obtain a sheaf 
\[ \widehat{\nmEis(\mathcal{S}_{\phi_{T}})} \in \D_{\text{\'et}}^{\mathrm{ULA}}(\Bun_{G},\mathcal{O}_{F}) \subset \Detale(\Bun_{G},\mathcal{O}_{F}) \]
in the category of ULA $\ell$-adically complete \'etale sheaves, as considered in \cite[Section~26]{Ecod}. For all $V \in \Rep_{\mathcal{O}_{F}}(\phantom{}^{L}G^{I})$, this $\ell$-complete category comes equipped with a Hecke operator 
\[ \hat{T}_{V}: \D(\Bun_{G},\mathcal{O}_{F}) \ra \D(\Bun_{G},\mathcal{O}_{F})^{BW_{\bb{Q}_{p}}^{I}} \]
defined by taking inverse limits of the Hecke operators with coefficients in $\mathcal{O}_{F}/\ell^{n}$ for $n \geq 1$. 

We want to use this to construct a sheaf on $\Dlis(\Bun_{G},\mathcal{O}_{F})$. In order to do this, we note that there is a natural $\ell$-adic completion functor
\[ 
\mc{D}_{\mathrm{lis}}(\Bun_{G},\mathcal{O}_{F}) \ra \varprojlim_{n} \mc{D}_{\mathrm{lis}}(\Bun_{G},\mathcal{O}_{F}/\varpi^{n}) \simeq \varprojlim_{n} \mc{D}_{\text{\'et}}(\Bun_{G},\mathcal{O}_{F}/\ell^{n}) =: \mc{D}_{\text{\'et}}(\Bun_{G},\mathcal{O}_{F})
\] 
\[ A \mapsto \varprojlim_{n} A \otimes_{\mathcal{O}_{F}} \mathcal{O}_{F}/\varpi^{n},\]
where we identify the lisse and \'etale categories for torsion coefficients and use the notation $\mc{D}$ to denote the $\infty$-categorical enhancement. We write 
\begin{equation}{\label{eqn: lcompletinglissesheaves}}
\gamma_{\Bun_{G}}: \mathrm{D}_{\mathrm{lis}}(\Bun_{G},\mathcal{O}_{F}) \ra \D_{\text{\'et}}(\Bun_{G},\mathcal{O}_{F})
\end{equation}
for the resulting functor on the homotopy category. This was studied in \cite[Section~7.1.4]{caraiani2026intersectioncohomologyigusastacks}, and we have the following basic result on its structure.
\begin{lemma}{\label{lemma: FromLCompletetoSmooth}}{\cite[Lemma~7.4]{caraiani2026intersectioncohomologyigusastacks}}
For $F/\bb{Q}_{\ell}$ a finite algebraic extension with ring of integers $\mathcal{O}_{F}$, the following is true. 
\begin{enumerate}
\item The functor $\gamma_{\Bun_{G}}$ upgrades to a natural transformation of semi-orthogonal decompositions with respect to the Harder-Narasimhan stratification and defined by $*$-pullback and $!$-pushforward  (see the reference for the precise statement).
\item The functor 
\[ \gamma_{b}:  \D_{\mathrm{lis}}(\Bun_{G}^{b},\mathcal{O}_{F}) \simeq \D(J_{b}(\bb{Q}_{p}),\mathcal{O}_{F}) \ra \hat{\D}(J_{b}(\bb{Q}_{p}),\mathcal{O}_{F}) \simeq \Detale(\Bun_{G}^{b},\mathcal{O}_{F})  \]
induced by (1) and the equivalences of \cite[Proposition~2.4]{Han1} and \cite[Proposition~VII.7.1]{FS} is given by $\ell$-adic completion. Here $\hat{\D}(J_{b}(\bb{Q}_{p}),\mathcal{O}_{F})$ is the category of $\ell$-complete smooth representations, as defined in \cite[Section~2.2]{Han1}.
\item We consider the subcategory $\D_{\PerfAdm}(J_{b}(\bb{Q}_{p}),\mathcal{O}_{F}) \subset \D(J_{b}(\bb{Q}_{p}),\mathcal{O}_{F})$ (resp. $\hat{\D}_{\PerfAdm}(J_{b}(\bb{Q}_{p}),\mathcal{O}_{F}) \subset \hat{\D}(J_{b}(\bb{Q}_{p}),\mathcal{O}_{F})$) of $A \in \D(J_{b}(\bb{Q}_{p}),\mathcal{O}_{F})$ (resp. $A \in \hat{\D}(J_{b}(\bb{Q}_{p}),\mathcal{O}_{F})$) such that the invariants of $A$ (resp. $A/\ell$)  under every pro-$p$ open $K \subset J_{b}(\bb{Q}_{p})$ is a perfect complex of $\mathcal{O}_{F}$-modules (resp. $\mathcal{O}_{F}/\ell$-modules). The functor $\gamma_{b}$ of (2) induces an equivalence 
\[
 \gamma_b^{\mathrm{PfAdm}}:
 \D_{\mathrm{PfAdm}}(J_b(\bb Q_p),\mc O_F)
 \xrightarrow{\simeq}
 \hat\D_{\mathrm{PfAdm}}(J_b(\bb Q_p),\mc O_F)
\]
with inverse given by taking smooth vectors.
\item  The functor $\gamma_{\Bun_{G}}$ restricts to an equivalence
\begin{equation}{\label{eqn: EquivalenceonULASheaves}} 
\gamma_{\Bun_{G}}^{\ULA}: \Dlis^{\ULA}(\Bun_{G},\mathcal{O}_{F}) \xrightarrow{\simeq} \Detale^{\ULA}(\Bun_{G},\mathcal{O}_{F})
\end{equation}
on the full subcategories of ULA objects, where $\Detale^{\ULA}(\Bun_{G},\mathcal{O}_{F})$ is the set of sheaves whose mod $\varpi$-reduction is ULA. Moreover, it intertwines Verdier duality\footnote{On $\Dlis(\Bun_{G},\Lambda)$, we simply define this to be internal Hom towards the constant sheaf, where we recall that the dualizing sheaf on $\Bun_{G}$ with torsion coefficients is the constant sheaf \cite[Proposition~1.1]{GH}.} on the category of lisse-\'etale ULA sheaves with Verdier duality on the category of $\ell$-complete \'etale ULA sheaves.
\item The equivalence $\gamma^{\ULA}_{\Bun_{G}}$ is Hecke-equivariant. In other words, for all $V \in \Rep_{\mathcal{O}_{F}}(\phantom{}^{L}G^{I})$, we have a natural equivalence 
\[ \hat{T}_{V}\gamma^{\ULA}_{\Bun_{G}} \simeq \gamma^{\ULA}_{\Bun_{G}}T_{V} \]
of functors $\Dlis^{\ULA}(\Bun_{G},\mathcal{O}_{F}) \ra \Detale^{\ULA}(\Bun_{G},\mathcal{O}_{F})^{BW_{\bb{Q}_{p}}^{I}}$, where we recall that Hecke operators preserve the subcategory of ULA objects by \cite[Theorem~I.7.2]{FS}.
\end{enumerate}
\end{lemma}
\begin{proof}
(1), (4), and (5) follow directly from the reference, where we note that (5) only treats the case of highest weight representations, but the proof of the general case is exactly the same. Similarly, (2) and (3) follow from the surrounding discussion. We note that (4) essentially reduces to (1) and (3) using the characterization of $\Dlis^{\ULA}(\Bun_{G},\mathcal{O}_{F})$ (resp. $\Detale^{\ULA}(\Bun_{G},\mathcal{O}_{F})$) as objects $A \in \Dlis(\Bun_{G},\mathcal{O}_{F})$ (resp. $A \in \Detale(\Bun_{G},\mathcal{O}_{F})$) such that $j_{b}^{*}A \in \D(J_{b}(\bb{Q}_{p}),\mathcal{O}_{F})$ (resp. $j_{b}^{*}A \in \hat{\D}(J_{b}(\bb{Q}_{p}),\mathcal{O}_{F})$) lies in the subcategory $j_{b}^{*}A \in \D_{\PerfAdm}(J_{b}(\bb{Q}_{p}),\mathcal{O}_{F})$ (resp. $j_{b}^{*}A \in \hat{\D}_{\PerfAdm}(J_{b}(\bb{Q}_{p}),\mathcal{O}_{F})$).
\end{proof}
We can use the isomorphism of (\ref{eqn: EquivalenceonULASheaves}) to obtain an object in $\Dlis(\Bun_{G},\mathcal{O}_{F})$, which we denote by $\nmEis(\mathcal{S}_{\phi_{T}})$. This will be our candidate sheaf when the coefficient system $\Lambda$ is $\mathcal{O}_{F}$ for $F/\bb{Q}_{\ell}$ a finite extension. Before proceeding to rational coefficients, we consider the case where $\Lambda = \mathcal{O}_{F}$ for $F/\bb{Q}_{\ell}$ a (possibly infinite) algebraic extension (i.e $\Lambda$ is as in Setup \ref{assumption: coefficients} (3)). In this case, let
$\phi_T:W_{\bb Q_p}\ra{}^LT(\mathcal O_F)$
be a continuous parameter, where $\mathcal O_F$ has
the $\ell$-adic topology. After restricting to $W_{L}$
for a finite extension $L/\bb Q_p$ splitting $T$,
the parameter is a tuple of characters
$\psi_j:W_L\ra\mathcal O_F^\times$.
Their wild inertia images are finite since $p\ne\ell$.
After enlarging $L$ to kill these images, the relationship $\sigma\tau\sigma^{-1}=\tau^{q_{L}}$ for arithmetic Frobenius $\sigma$ and a topological
generator $\tau$ of tame inertia implies
$\psi_j(\tau)^{q_L-1}=1$, where $q_L$ is the residue
cardinality of $L$. Thus $\phi_T(I_L)$ is finite.
Since $I_L$ has finite index in $I_{\bb Q_p}$,
$\phi_T(I_{\bb Q_p})$ is finite as well.
The finitely many inertia values together with the
value of a Frobenius element are defined over a finite
subextension $F'/\bb Q_\ell$ of $F$. Therefore, there exists a finite subextension $F/F'/\bb{Q}_{\ell}$ such that $\phi_{T} := \phi_{T,F'} \otimes_{\mathcal{O}_{F'}} \mathcal{O}_{F}$. 
We then define 
\begin{equation}{\label{eqn: ExtendingtheCoefficients}}
 \nmEis(\mathcal{S}_{\phi_{T}}) := \nmEis(\mathcal{S}_{\phi_{T,F'}}) \otimes^{\blacksquare}_{\mathcal{O}_{F'}} \mathcal{O}_{F}, 
\end{equation}
where $\otimes^{\blacksquare}$ is the solid tensor product and $\nmEis(\mathcal{S}_{\phi_{T,F'}}) \in \Dlis^{\ULA}(\Bun_{G},\mathcal{O}_{F'})$ is defined as above. This is a well-defined object of $\Dlis(\Bun_{G},\mathcal{O}_{F})$ by the same logic as that preceding \cite[Lemma~7.5]{caraiani2026intersectioncohomologyigusastacks} and is ULA by the proof of \cite[Lemma~7.5 (2)]{caraiani2026intersectioncohomologyigusastacks}\footnote{Strictly speaking, this only discusses the required reasoning for the tensor product $\otimes^{\blacksquare}_{\mathcal{O}_{F}} F$, where $F/\bb{Q}_{\ell}$ is a finite extension. However, the exact same arguments apply to this solid tensor product and we will implicitly use this when citing \cite[Lemma~7.5]{caraiani2026intersectioncohomologyigusastacks} below.}. One easily checks that it doesn't depend on the choice of $F'$ and $\phi_{T,F'}$. 

Now it remains to describe the desired sheaf when $\Lambda = F$ is as in Setup \ref{assumption: coefficients} (4). To treat this case, we will need to assume that $\phi_{T}$ is of the form $\phi_{T} := \ol{\phi}_{T} \otimes_{\mathcal{O}_{F}} F$ for a toral $L$-parameter $\ol{\phi}_{T}: W_{\bb{Q}_{p}} \ra \phantom{}^{L}T(\mathcal{O}_{F})$. In this case, just as in (\ref{eqn: ExtendingtheCoefficients}), we may define
\begin{equation}{\label{eqn: RationalizingtheCoefficients}}
 \nmEis(\mathcal{S}_{\phi_{T}}) := \nmEis(\mathcal{S}_{\ol{\phi}_{T}}) \otimes^{\blacksquare}_{\mathcal{O}_{F}} F \in \Dlis^{\ULA}(\Bun_{G},F). 
\end{equation}
Now with the candidate eigensheaf defined, we begin the proof of the eigensheaf property. To capture the necessary integrality conditions, we define the following. 
\begin{definition}{\label{defn: IntegralParameter}}
For $\Lambda$ as in Setup \ref{assumption: coefficients}, we say a toral $L$-parameter $\phi_{T}: W_{\bb{Q}_{p}} \ra \phantom{}^{L}T(\Lambda)$ is integral if it admits a mod $\ell$-reduction. In particular, if $\Lambda = F$ for $F/\bb{Q}_{\ell}$ an algebraic extension, we assume it is of the form $\ol{\phi}_{T} \otimes_{\mathcal{O}_{F}} F$ for some parameter $\ol{\phi}_{T}: W_{\mathbb{Q}_{p}} \ra \phantom{}^{L}T(\mathcal{O}_{F})$. We refer to $\ol{\phi}_{T} \otimes_{\mathcal{O}_{F}} \mathcal{O}_{F}/\mf{m}_{F}$, where $\mf{m}_{F} \subset \mathcal{O}_{F}$ is the maximal ideal, as the mod $\ell$-reduction of such a parameter.
\end{definition}
\subsection{The Hecke Eigensheaf Property}
In order to verify that $\nmEis(\mathcal{S}_{\phi_{T}})$ is a Hecke eigensheaf, we will need to introduce the following additional condition on our parameter $\phi_{T}$.
\begin{definition}{\label{def: strongmureg}}
For a toral parameter $\phi_{T}: W_{\mathbb{Q}_{p}} \ra \phantom{}^{L}T(\Lambda)$ and a geometric dominant cocharacter $\mu$ and $\tau \in \Irr^{\proj}_{\Lambda}(Q_{\mu})$, we say $\phi_{T}$ is strongly $(\mu,\tau)$-regular if the Galois cohomology complexes
\[ R\Gamma(W_{\mathbb{Q}_{p}}, (\mathcal{T}_{\mu^{\Gamma},\tau}^{\vee} \otimes \mathcal{T}_{\mu^{\Gamma},\tau})|_{\phantom{}^{L}T}^{\nt}\circ \phi_{T}) \]
is trivial. Here $\mathcal{T}_{\mu^{\Gamma},\tau} \in \Rep_{\Lambda}(\phantom{}^{L}G)$ is the highest weight tilting module constructed in \S \ref{s: TiltingEigensheaves} and the superscript $(-)^{\nt}$ indicates that we pass to the summand with non-trivial weights. If $\tau$ is trivial then we omit from the terminology (i.e we say $\phi_{T}$ is strongly $\mu$-regular in this case). More generally, for a finite index set $I$ and $V \in \Tilt_{\Lambda}(\phantom{}^{L}G^{I})$, we say that $\phi_{T}$ is $V$-strongly regular if the Galois cohomology 
\[ R\Gamma(W_{\bb{Q}_{p}}^{I}, (V^{\vee} \otimes V)|_{\phantom{}^{L}T^{I}}^{\nt} \circ \phi_{T}^{\boxtimes I}) \]
is trivial.
\end{definition}
To give some flavor for this condition, we note the following simple lemma.
\begin{lemma}{\label{lemma: stronguregexamples}}
Suppose that $G = \GL_{n}$ and let $\mu = (1,0,\ldots,0,0)$ denote the standard cocharacter then strong $\mu$-regularity is implied by genericity. 
\end{lemma}
\begin{proof}
We note that since $\mu$ is minuscule the standard representation $V_{\mu}$ has weights given by Weyl group orbits of the highest weight and this agrees with the highest weight tilting module by Lemma \ref{mintilting}. From here, it easily follows that the difference of distinct weights of $V_{\mu}$ in $\hat{T}$ are given by coroots of $G$ and the claim follows.
\end{proof}
With this notion introduced, we can finally state our main Theorem.
\begin{theorem}{\label{Eiseigsheaf}}
Assume \ref{compatibility}. For $\Lambda$ as in Setup \ref{assumption: coefficients}, we consider $\phi_{T}: W_{\mathbb{Q}_{p}} \ra \phantom{}^{L}T(\Lambda)$ an integral parameter with generic regular mod $\ell$ reduction with induced $L$-parameter $\phi$ of $G$. There then exists a perverse sheaf $\nmEis(\mathcal{S}_{\phi_{T}}) \in \Dlis(\Bun_{G},\Lambda)$ which is a filtered eigensheaf with eigenvalue $\phi$ as in Corollary \ref{appliedfilteigsheaf}. If $V \in \Tilt_{\Lambda}((\phantom{}^{L}G)^I)$
is a direct sum of
$\boxtimes_{i \in I} \mathcal{T}_{(\mu_{i}^{\Gamma},\tau_i)}$ for geometric dominant cocharacters $\mu_{i}$ and $\tau_{i} \in \Irr_{\Lambda}^{\proj}(Q_{\mu_{i}})$, and $\phi_{T}$ is strongly $(\mu_{i},\tau_{i})$-regular (Definition \ref{def: strongmureg}), the filtration on $T_{V}(\nmEis(\mathcal{S}_{\phi_{T}}))$ splits (by taking direct sums of the unique splitting on the direct summands), and we have an isomorphism
\[ T_{V}(\nmEis(\mathcal{S}_{\phi_{T}})) \simeq \nmEis(\mathcal{S}_{\phi_{T}}) \boxtimes r_{V} \circ \phi^{\boxtimes I} \]
of sheaves in $\Dlis(\Bun_{G})^{BW_{\mathbb{Q}_{p}}^{I}}$. In particular, if $\phi_{T}$ is strongly $V$-regular for all $V \in \Tilt_{\Lambda}(\phantom{}^{L}G^{I})$ then $\nmEis(\mathcal{S}_{\phi_{T}})$ is a tilting eigensheaf. For $b \in B(G)$, the stalk $\nmEis(\mathcal{S}_{\phi_{T}})|_{\Bun_{G}^{b}} \in \D(\Bun_{G}^{b}) \simeq \D(J_{b}(\mathbb{Q}_{p}),\Lambda)$ is isomorphic to
\begin{enumerate}
\item  $\bigoplus_{w \in W_{b}} i_{B_{b}}^{J_{b}}(\chi^{w^{-1}}) \otimes \delta_{P_{b}}^{-1/2}[-\langle 2\hat{\rho},\nu_{b} \rangle]$ if $b \in B(G)_{\mathrm{un}}$,
\item $0$ if $b \notin B(G)_{\mathrm{un}}$.  
\end{enumerate}
Moreover, if $\mathbb{D}_{\Bun_{G}}$ denotes Verdier duality on $\Bun_{G}$, we have an isomorphism
\[ \mathbb{D}_{\Bun_{G}}(\nmEis(\mathcal{S}_{\phi_{T}})) \simeq \nmEis(\mathcal{S}_{\phi_{T}^{\vee}}) \]
of sheaves in $\Dlis(\Bun_{G},\Lambda)$.
\end{theorem}
\begin{proof}
The construction of $\nmEis(\mathcal{S}_{\phi_{T}})$ was discussed in the previous section. The transfer of the filtered Hecke eigensheaf property when $\Lambda = \mathcal{O}_{F}$ to this sheaf follows from the torsion case (Corollary \ref{appliedfilteigsheaf}), Lemma \ref{lemma: FromLCompletetoSmooth} (5) for $\Lambda = \mathcal{O}_{F}$ for $F/\bb{Q}_{\ell}$ a finite extension. For $\Lambda = \mathcal{O}_{F},F$ for $F/\bb{Q}_{\ell}$ an infinite extension, this follows from \cite[Lemma~7.5 (5)] {caraiani2026intersectioncohomologyigusastacks} and its proof.

The claim on Verdier duality follows from the torsion case (Theorem \ref{commverddual}), Lemma \ref{lemma: FromLCompletetoSmooth} (4), and the proof of \cite[Lemma~7.5 (5)]{caraiani2026intersectioncohomologyigusastacks}. Here we note that if we have an $L$-parameter $\phi_{T}: W_{\bb{Q}_{p}} \ra \phantom{}^{L}T(\mathcal{O}_{F})$ which is the base-change of $\phi_{T,F'}: W_{\bb{Q}_{p}} \ra \phantom{}^{L}T(\mathcal{O}_{F'})$ generic regularity of the mod $\ell$-reduction of $\phi_{T,F'}$ is equivalent to generic regularity of the mod $\ell$-reduction of $\phi_{T,F}$, and generic regularity of the mod $\ell$-reduction of $\phi_{T,F'}$ implies generic regularity of $\phi_{T,F'} \otimes_{\mathcal{O}_{F'}} \mathcal{O}_{F'}/\varpi^{n}$ for all $n \geq 1$ and a uniformizing element $\varpi \in \mathcal{O}_{F'}$. 

Similarly, the description of the stalks follows from the torsion case (Corollary \ref{normstalksdescription}), the construction, and Lemma \ref{lemma: FromLCompletetoSmooth} (3).

It remains to show that the filtration on $T_{V}(\nmEis(\mathcal{S}_{\phi_{T}}))$ splits uniquely for $\boxtimes_{i \in I} \mathcal{T}_{\mu_{i}^{\Gamma},\tau_{i}} = V \in \Tilt_{\Lambda}(\phantom{}^{L}G^{I})$ if $\phi_{T}$ is strongly $(\mu_{i},\tau_{i})$-regular for all $i \in I$. To do this, we decompose 
\[ (\mathcal{T}^{\vee}_{\mu_{i}^{\Gamma},\tau_{i}} \otimes \mathcal{T}_{\mu_{i}^{\Gamma},\tau_{i}})|_{\phantom{}^{L}T} \simeq D_{i} \oplus N_{i} \]
for all $i \in I$, where $D_{i}$ (resp. $N_{i}$) is the weight zero summand (resp. non-zero weight summand). Put $C:=R\operatorname{Hom}_{\Dlis(\Bun_G,\Lambda)}(\nmEis(\mathcal S_{\phi_T}),\nmEis(\mathcal S_{\phi_T}))$, with trivial Weil action. The extension and morphism groups between distinct graded pieces are direct summands of the $H^1$ and $H^0$, respectively, of
\[ C\otimes_\Lambda^{\mathbb L} R\Gamma(W_{\bb{Q}_{p}}^{I},(\boxtimes_{i \in S} N_{i}\boxtimes  \boxtimes_{i \notin S} D_{i}) \circ \phi_{T}^{\boxtimes I})\]
for a non-empty subset $S \subset I$. However, this complex is isomorphic to
\[ C\otimes_\Lambda^{\mathbb L} \bigotimes^{\bb{L}}_{i \in S} R\Gamma(W_{\bb{Q}_{p}},N_{i} \circ \phi_{T}) \otimes^{\bb{L}} \bigotimes^{\bb{L}}_{i \notin S} R\Gamma(W_{\bb{Q}_{p}},D_{i} \circ \phi_{T}). \]

In particular, if $\phi_T$ is strongly $(\mu_i,\tau_i)$-regular for all $i\in I$, this complex vanishes, since $R\Gamma(W_{\bb Q_p},N_i\circ\phi_T)\simeq0$ for all $i\in I$. Thus both $H^1$ and $H^0$ vanish, so the filtration splits uniquely. In particular, if $\phi_T$ is $V$-strongly regular for every finite set $I$
and every $V\in\operatorname{Tilt}_{\Lambda}({}^LG^I)$,
the same argument gives a unique splitting for every $V$.
Applying this also to direct sums makes the splittings natural
in morphisms of representations. Compatibility with tensor
products and fusion follows from uniqueness and the corresponding
compatibility of the filtration.
\end{proof}
We now fix an integral parameter $\phi_{T}$ with generic regular
mod $\ell$-reduction and consider the sheaf
$\nmEis(\mathcal{S}_{\phi_{T}}) \in \Dlis(\Bun_{G},\Lambda)$
supplied by the previous theorem. We come to the following definition.
\begin{definition}{\label{def: muregularity}}
Given a tuple of geometric dominant cocharacters $(\mu_{i})_{i \in I} \in (\domcochar)^{I}$ and indecomposable projective representations $\tau_{i} \in \Irr_{\Lambda}^{\proj}(Q_{\mu_{i}})$ for some finite index set $I$, we say that $\phi_{T}: W_{\mathbb{Q}_{p}} \ra \phantom{}^{L}T(\Lambda)$ is $(\mu_{i},\tau_{i})_{i \in I}$-regular if the filtration on 
\[ T_{V}(\nmEis(\mathcal{S}_{\phi_{T}})) \]
splits (but not necessarily uniquely) for the tilting module $V = \boxtimes_{i \in I} \mathcal{T}_{\mu_{i}^{\Gamma},\tau_{i}} \in \Tilt_{\Lambda}(\phantom{}^{L}G^{I})$. If the $\tau_{i}$ are trivial we omit it from the terminology (i.e we simply say $(\mu_{i})_{i \in I}$-regular).
\end{definition}
As seen in the proof of Theorem \ref{Eiseigsheaf}, if $\phi_{T}$ is strongly $(\mu_{i},\tau_{i})$-regular for all $i \in I$ then it follows that $\phi_{T}$ is $(\mu_{i},\tau_{i})_{i \in I}$-regular. For certain $(\mu_{i},\tau_{i})$, strong $(\mu_{i},\tau_{i})$-regularity will be implied by genericity, and the following Proposition allows us to deduce that the filtration splits in more cases making $\mu$-regularity a strictly more general notion then strong $\mu$-regularity, as the terminology suggests.
\begin{proposition}{\label{prop: muregdirectsum}}
Suppose that $(\mu_{1i},\tau_{i1})_{i \in I}, (\mu_{2i},\tau_{i2})_{i \in I} \in (\domcochar)^{I}$ are tuples such that $\phi_{T}$ is $(\mu_{1i},\tau_{i1})_{i \in I}$ and $(\mu_{2i},\tau_{i2})_{i \in I}$-regular. Then if $V \simeq \boxtimes_{i \in I} \mathcal{T}_{\mu_{i}^{\Gamma},\tau_{i}} \in \Tilt_{\Lambda}(\phantom{}^{L}G^{I})$ occurs as a direct summand of the tensor product $\boxtimes_{i \in I} \mathcal{T}_{(\mu_{1i}^{\Gamma},\tau_{i1})} \otimes \mathcal{T}_{(\mu_{2i}^{\Gamma},\tau_{i2)}} \in \Tilt(\phantom{}^{L}G^{I})$ it follows that $\phi_{T}$ is $(\mu_{i},\tau_{i})_{i \in I}$-regular. 
\end{proposition}
\begin{proof}
We omit the $\tau$s from the notation for simplicity. It suffices to show that, given $(\mu_{1i})_{i \in I},(\mu_{2i})_{i \in I} \in (\mathbb{X}_{*}(T_{\ol{\mathbb{Q}}_{p}})^{+})^{I}$ with $V_{1} := \boxtimes_{i \in I} \mathcal{T}_{\mu_{1i}^{\Gamma}}$ and $V_{2} := \boxtimes_{i \in I} \mathcal{T}_{\mu_{2i}^{\Gamma}}$ such that we know the filtration on $T_{V_{1}}(\nmEis(\mathcal{S}_{\phi_{T}}))$ and $T_{V_{2}}(\nmEis(\mathcal{S}_{\phi_{T}}))$ splits, the same is true for the filtration on $T_{V}(\nmEis(\mathcal{S}_{\phi_{T}}))$, where $V \in \Tilt_{\Lambda}(\phantom{}^{L}G^{I})$ is a direct summand of $V_{1} \otimes V_{2}$. To do this, we can use the isomorphism 
\[ T_{V_{1}}T_{V_{2}}(\nmEis(\mathcal{S}_{\phi_{T}}))|_{\triangle} \simeq T_{V_{1} \otimes V_{2}}(\nmEis(\mathcal{S}_{\phi_{T}})) \]
coming from the fusion product, where $\triangle: \Dlis(\Bun_{G},\Lambda)^{BW_{\mathbb{Q}_{p}}^{I \sqcup I}} \ra \Dlis(\Bun_{G},\Lambda)^{BW_{\mathbb{Q}_{p}}^{I}}$ is the natural map given by diagonal restriction. By assumption, the diagonal restriction of the filtration on the LHS splits. Moreover, by the compatibilities of the filtration, we know that the filtration on the LHS refines the filtration on the RHS. In particular, we deduce that the filtration on $T_{V_{1} \otimes V_{2}}(\nmEis(\mathcal{S}_{\phi_{T}}))$ has a splitting, and since $V$ is a direct summand of $V_{1} \otimes V_{2}$, we know that the filtration also splits on $T_{V}(\nmEis(\mathcal{S}_{\phi_{T}}))$.  
\end{proof}
In particular, by considerations of highest weight, the tilting module $ \boxtimes_{i \in I} \mathcal{T}_{(\mu_{1i} + \mu_{2i})^{\Gamma}}$ is always a direct summand of $\boxtimes_{i \in I} \mathcal{T}_{\mu_{1i}^{\Gamma}} \otimes \mathcal{T}_{\mu_{2i}^{\Gamma}}$. Using this, we deduce the following Corollary.
\begin{corollary}{\label{cor: muregadd}}
If $\phi_{T}$ is $(\mu_{1i})_{i \in I}$ and $(\mu_{2i})_{i \in I}$-regular then it is also $(\mu_{1i} + \mu_{2i})_{i \in I}$-regular. 
\end{corollary}
This motivates the following definition. 
\begin{definition}{\label{def: normreg}}
We say a parameter $\phi_{T}$ is normalized regular if it is integral with generic regular mod $\ell$ reduction and if there exists a set of elements $\mu_{k}$ for $k = 1,\ldots,n$ such that:
\begin{enumerate}
    \item The $\mu_{k}$ generate $\domcochar$ in the sense that any $\mu = \sum_{k = 1}^{n} n_{k}\mu_{k}$ as elements in $\domcochar$ for some $n_{k} \in \mathbb{Z}_{\geq 0}$. 
    \item $\phi_{T}$ is strongly $\mu_{k}$-regular, for all $k = 1,\ldots,n$.
\end{enumerate}
\end{definition}
This allows us to state the following corollary.
\begin{corollary}{\label{cor: muregopencond}}
If $\phi_{T}$ is normalized regular then it is $(\mu_{i})_{i \in I}$-regular for all finite index sets $I$ and $\mu_{i} \in \domcochar$. In particular, if $G$ is split (so that the $\tau_{i}$ are always trivial) the sheaf $\nmEis(\mathcal{S}_{\phi_{T}})$ is a (weak) tilting eigensheaf with eigenvalue $\phi$ given by the composite of $\phi_{T}$ with the natural map $\phantom{}^{L}T \ra \phantom{}^{L}G$. 
\end{corollary}
\begin{proof}
We recall that $\nmEis(\mathcal{S}_{\phi_{T}})$ being a (weak) tilting eigensheaf means that, for all $V \in \Tilt_{\Lambda}(\phantom{}^{L}G^{I})$, we know that we have an isomorphism 
\[ T_{V}(\nmEis(\mathcal{S}_{\phi_{T}})) \simeq \nmEis(\mathcal{S}_{\phi_{T}}) \boxtimes r_{V} \circ \phi^{\boxtimes I} \]
but we do not know that these isomorphisms satisfy the desired compatibilities. To show this, it suffices to show the filtration on $T_{V}(\nmEis(\mathcal{S}_{\phi_{T}}))$ just splits (not necessarily uniquely) for all highest weight tilting modules $V = \boxtimes_{i \in I} \mathcal{T}_{\mu_{i}^{\Gamma}}$ corresponding to the $\Gamma$-orbits of varying $(\mu_{i})_{i \in I} \in (\domcochar)^{I}$. Since $\phi_{T}$ is strongly $\mu_{k}$-regular for the $k =1,\ldots,n$ appearing in the definition of normalized regularity, it follows, as in the proof of Theorem \ref{Eiseigsheaf}, that the filtration on $T_{V}(\nmEis(\mathcal{S}_{\phi_{T}}))$ splits for any tuple of the generating cocharacters $(\mu_{i})_{i \in I}$ such that $\mu_{i} = \mu_{k}$ for all $i \in I$ and varying $k = 1,\ldots,n$ appearing in Definition \ref{def: normreg}. However, since every tuple of cocharacters $\mu_{i}$, can be written as a linear combination of $\mu_{k}$ for $k = 1,\ldots,n$, this follows from Corollary \ref{cor: muregadd}.  
\end{proof}
By choosing $\mu_{k}$ for $k = 1,\ldots,n$ to be, for example a set of minuscule/quasi-minuscule generators of $\domcochar$, this allows to see that $\mu$-regularity for all $\mu$ is actually an open condition on the variety of unramified twists of a fixed $\phi_{T}$. This will be sufficient for most applications of the Eisenstein series to describing the cohomology of local shtukas and Shimura varieties in the Grothendieck group. However, for more refined applications of our results it becomes important to spell out the precise assumptions needed to deduce $\mu$-regularity for all $\mu$. Ideally, it should be implied by generic regularity for all $\mu$, as is suggested by Conjecture \ref{conj: idealworld}. For $G = \GL_{n}$, we show how to verify this from the results proven above.
\begin{corollary}{\label{cor: GLnisideal}}
For $G = \GL_{n}$, if $\phi_{T}$ is an integral $L$-parameter with generic mod $\ell$-reduction and $(\ell,n!) = 1$ when $\Lambda$ is integral and torsion then the sheaf $\nmEis(\mathcal{S}_{\phi_{T}})$  is a (weak) tilting eigensheaf.
\end{corollary}
\begin{proof}
We give the proof for $\Lambda$ torsion, as in Setup \ref{assumption: coefficients} (1), with the other cases being strictly easier. It is easy to see that generic implies regular, so that if $\phi_{T}$ is generic then it is automatically generic regular. It suffices to show for all tuples $(\mu_{i})_{i \in I} \in (\domcochar)^{I}$ that $\phi_{T}$ is $(\mu_{i})_{i \in I}$ regular. By Lemma \ref{lemma: stronguregexamples}, we know that if $\phi_{T}$ is generic then it is strongly $\mu$-regular for the cocharacter $\mu = (1,0,\ldots,0)$ corresponding to the standard representation $V_{\mathrm{std}}$ of $\hat{G} \simeq \GL_{n}$. By Corollary \ref{cor: muregadd}, it suffices to show that if $\phi_{T}$ is generic then $\mu$-regularity holds for the fundamental coweights $\mu = \omega_{i} = (1^{i},0^{n - i})$ for $i = 1,\ldots,n$ and their dominant inverses (with the proof for the inverses being the same). We recall that the $\omega_{i}$ are minuscule in this case and therefore it follows by Lemma \ref{mintilting} that the representation $\Lambda^{i}(V_{\mathrm{std}}) \simeq V_{\omega_{i}}$ is irreducible and thereby equal to the tilting module $\mathcal{T}_{\omega_{i}}$. However, now the result follows from Proposition \ref{prop: muregdirectsum} and the fact that the idempotent
\[ e_{\mathrm{alt}} := \frac{1}{i!} \sum_{\sigma \in S_{i}} \mathrm{sgn}(\sigma)\sigma \in \mathrm{End}(V^{\otimes_{i}}_{\mathrm{std}})\]
well-defined by our assumption that $(\ell,n!) = 1$ splits $\Lambda^{i}(V_{\mathrm{std}})$ of as a direct summand of $V_{\mathrm{std}}^{\otimes i}$. The fact that assumption \ref{compatibility} is satisfied follows from the compatibility of Fargues-Scholze with Harris-Taylor \cite[Theorem~I.9.6 (ix)]{FS} and Jacquet Langlands \cite[Theorem~I.0.3]{KW} together with other standard properties of the Fargues-Scholze and Harris-Taylor correspondence.
\end{proof}
With our main results in hand, we can move into applications. 
\section{Applications}{\label{s: Applications}}
Now we will deduce some applications to the cohomology of local shtuka spaces. In \S 10.1, we will use our eigensheaf to derive an analogue of an averaging formula of Shin for the cohomology of local shtuka spaces. In \S 10.2, we will discuss a refined version of this averaging formula, and use it to derive a very explicit formula for the isotypic parts of local shtuka spaces with respect to parabolic inductions of characters coming from generic regular parameters. By combining the refined averaging formula with a
shtuka analogue of Boyer's trick, we will give a
geometric realization of the intertwining operators described in \ref{prop:opposite-borel-comparison}. In particular, the
$\mu$-ordinary cohomology realizes the different
Weyl translates of principal series representations,
while the filtered eigensheaf property relates them
to a single principal series. This also leads to a
generic-fiber analogue of the results of Xiao and
Zhu~\cite{XZ}.
\\\\
Let's first recall the key definitions.  We say a local shtuka datum is a triple $(G,b,\mu)$ for $\mu$ a geometric dominant cocharacter of $G$ and $b \in B(G,\mu)$ an element of the $\mu$-admissible locus of the Kottwitz set of $G$ (Definition \ref{bgmudefinition}). We let $E$ be the reflex field of $\mu$. The triple $(G,b,\mu)$ defines a diamond
\[ \Sht(G,b,\mu)_{\infty} \rightarrow \Spd(\Breve{E}) \]
parameterizing modifications $\mathcal{F}_{b} \rightarrow \mathcal{F}_{G}^{0}$ with meromorphy bounded by $\mu$ on $X$ at the Cartier divisor defined by the morphism to $C$\footnote{Note that this is the space denoted $\Sht(G,b,\mu^{-1})$ in \cite{SW}. We find that our convention simplifies certain formulae.}. It carries an action of $G(\mathbb{Q}_{p}) \times J_{b}(\mathbb{Q}_{p})$ and a (non-effective) descent datum from $\Breve{E}$ down to $E$. This allows us to consider the tower of quotients
\[ \Sht(G,b,\mu)_{\infty}/\underline{K} =: \Sht(G,b,\mu)_{K} \]
for varying open compact subgroups $K \subset G(\mathbb{Q}_{p})$. We write $\mathcal{S}_{\mu}$ for the  $\ell$-adic \'etale sheaf (in the sense of \cite[Section~26]{Ecod}) $\bb{Z}_{\ell}[\sqrt{p}]$ attached to the highest weight tilting module $\mathcal{T}_{\mu}$ of $Q_{\mu} \ltimes \hat{G}$, as defined in  \S \ref{tiltingcocharacters}. This is given by pulling back the sheaf on $\Hck_{G,E}$ defined by $\mathcal{T}_{\mu}$ and Theorem \ref{geomsatake} (extended to $\bb{Z}_{\ell}[\sqrt{p}]$-coefficients and adic sheaves, as discussed in \cite{FS}) along the natural map $\Sht(G,b,\mu) \ra \Hck_{G,E}$. In particular, the sheaf $\mathcal{S}_{\mu}$ is equivariant with respect to the actions of $G(\mathbb{Q}_{p})$ and $J_{b}(\mathbb{Q}_{p})$ by construction. Letting $\Sht(G,b,\mu)_{K,\mathbb{C}_{p}}$ denote the base-change of these spaces to $\mathbb{C}_{p}$, we can now define the complex
\[ R\Gamma_{c}(G,b,\mu) := \colim_{K \rightarrow \{1\}} R\Gamma_{c}(\Sht(G,b,\mu)_{K,\mathbb{C}_{p}},\mathcal{S}_{\mu}) \]
of $G(\mathbb{Q}_{p}) \times J_{b}(\mathbb{Q}_{p}) \times W_{E}$-modules, where $R\Gamma_{c}(-)$ is computed as a colimit of the compactly supported cohomology of the restriction of $\mathcal{S}_{\mu}$ to quasi-compact opens $U$ of $\Sht(G,b,\mu)_{K,\bb{C}_{p}}$, and on the quasi-compact opens the compactly supported cohomology is computed in the category $\bb{Z}_{\ell}[\sqrt{p}]$-adic \'etale sheaves. The action of $G(\bb{Q}_{p})$ is smooth by construction. Moreover, the action of $J_{b}(\bb{Q}_{p})$ is smooth by \cite[Proposition~IX.3.2]{FS} and \cite[Proposition~VII.5.2]{FS} applied to the quasi-compact opens $U$ together with the \cite[Definition/Proposition VI.7.10]{FS} in order to compare the $\natural$-pushforward appearing in \cite[Proposition~IX.3.2]{FS} with the colimit of compactly supported cohomology in the category of adic \'etale sheaves on the quasi-compact opens (See the proof of \cite[Theorem~IX.3.1]{FS}). The Weil group action is induced by the (non-effective) descent datum from $\Breve{E}$ down to $E$ (See \cite[Section~IX.3]{FS} for details).

We now want to disentangle the $G(\mathbb{Q}_{p})$ and $J_{b}(\mathbb{Q}_{p})$ action. To do this, for $\pi$ (resp. $\rho$) a smooth irreducible representation of $G(\mathbb{Q}_{p})$ (resp. $J_{b}(\mathbb{Q}_{p})$) on $\Lambda$-modules for $\Lambda$ as in Setup \ref{assumption: coefficients}, we define the $\pi$ (resp. $\rho$)-isotypic part. I.e the complexes
\[ R\Gamma_{c}(G,b,\mu)[\pi] := R\Gamma_{c}(G,b,\mu) \otimes_{\mathcal{H}(G)} \pi \]
and
\[ R\Gamma_{c}(G,b,\mu)[\rho] := R\Gamma_{c}(G,b,\mu) \otimes_{\mathcal{H}(J_{b})} \rho \]
where $\mathcal{H}(G) := C^{\infty}_{c}(G(\mathbb{Q}_{p}),\Lambda)$ (resp. $\mathcal{H}(J_{b})$) is the usual smooth Hecke algebra of $G$ (resp. $J_{b}$). We recall that $\Sht(G,b,\mu)_{\infty}$ has dimension equal to $\langle 2\hat{\rho}, \mu \rangle =: d$. The complexes $R\Gamma_{c}(G,b,\mu)[\pi]$ (resp. $R\Gamma_{c}(G,b,\mu)[\rho]$) are concentrated in degrees $-d \leq i \leq d$ and are valued in admissible $J_{b}(\mathbb{Q}_{p})$ (resp. $G(\mathbb{Q}_{p})$) representations, which are moreover of finite length with $\ol{\mathbb{Q}}_{\ell}$-coefficients \cite[Page~317]{FS}. Similarly, we define the complexes 
\[ R\Gamma_{c}^{\flat}(G,b,\mu)[\pi] := R\mathcal{H}om(R\Gamma_{c}(G,b,\mu),\pi) \]
and 
\[ R\Gamma_{c}^{\flat}(G,b,\mu)[\rho] := R\mathcal{H}om(R\Gamma_{c}(G,b,\mu),\rho). \]
We will make regular use of the relationship between these complexes and Hecke operators. For simplicity, we will mostly stick to the isotypic parts $R\Gamma^{\flat}_{c}(G,b,\mu)[\rho]$, as it removes various annoying twists and shifts\footnote{We thank Naoki Imai for drawing our attention to this.}; however, one can easily translate to the isotypic parts $R\Gamma_{c}(G,b,\mu)[\rho]$ in the cases we consider, using Corollary \ref{rhoduality}. We write 
\[ T_{\mu}: \Dlis(\Bun_{G},\Lambda) \ra \Dlis(\Bun_{G},\Lambda)^{BW_{E}} \]
for the Hecke operator defined by the representation $\mathcal{T}_{\mu}$. For $b \in B(G)$, consider the natural map 
\[ f: \mathcal{J}_{b}^{> 0} \ra \ast, \]
where $\Aut(\mathcal{E}_{b}) \simeq J_{b}(\mathbb{Q}_{p}) \ltimes \mathcal{J}_{b}^{> 0}$, and the semidirect product structures is given by allowing $\Aut(\mathcal{E}_{b})$ to act on the canonical reduction of $\mathcal{E}_{b}$.  The semi-direct product structure on $\Aut(\mathcal{E}_{b})$ allows us to pass to the $v$-stack quotient
\[ \ol{f}: [\mathcal{J}_{b}^{> 0}/\ul{J_{b}(\bb{Q}_{p})}] \ra [\ast/\ul{J_{b}(\bb{Q}_{p})}], \]
and, since $f$ is an iterated fibration of Banach-Colmez spaces, we have an isomorphism: 
\[ \ol{f}_{!}(\Lambda) \simeq \kappa[-2\langle 2\hat{\rho}_{G},\nu_{b} \rangle] \]
for some character $\kappa: J_{b}(\mathbb{Q}_{p}) \ra \Lambda^{*}$ (cf. \cite[Lemma~4.18]{IG}). In other words, there is a $J_{b}(\bb{Q}_{p})$-equivariant isomorphism
\[ f_{!}(\Lambda) \simeq \kappa[-2\langle 2\hat{\rho}_{G},\nu_{b} \rangle]. \]
This character appears in the following relationship between isotypic parts of shtuka spaces and Hecke operators. 
\begin{lemma}{\label{shimhecke}}{\cite[Section~IX.3]{FS}}
For any $\Lambda/\bb{Z}_{\ell}[\sqrt{p}]$, given a local shtuka datum $(G,b,\mu)$ as above and $\pi$ (resp. $\rho$) a smooth admissible representation of $G(\mathbb{Q}_{p})$ (resp. $J_{b}(\mathbb{Q}_{p})$) on $\Lambda$-modules, we can consider the associated sheaves $\rho \in \D(J_{b}(\mathbb{Q}_{p}),\Lambda) \simeq \Dlis(\Bun_{G}^{b},\Lambda)$ and $\pi \in \D(G(\mathbb{Q}_{p}),\Lambda) \simeq \Dlis(\Bun_{G}^{\mathbf{1}},\Lambda)$ on the HN-stratum $j_{b}: \Bun_{G}^{b} \hookrightarrow \Bun_{G}$ and $j_{\mathbf{1}}: \Bun_{G}^{\mathbf{1}} \hookrightarrow \Bun_{G}$, respectively. There then exists an isomorphism
\[ R\Gamma_{c}(G,b,\mu)[\rho \otimes \kappa^{-1}][2\langle 2\hat{\rho}_{G}, \nu_{b} \rangle]  \simeq  j_{\mathbf{1}}^{*}T_{\mu}j_{b!}(\rho)\]
of complexes of $G(\mathbb{Q}_{p}) \times W_{E}$-modules and an isomorphism
\[ R\Gamma_{c}(G,b,\mu)[\pi] \simeq j_{b}^{*}T_{\mu^{-1}}j_{\mathbf{1}!}(\pi) \]
of complexes of $J_{b}(\mathbb{Q}_{p}) \times W_{E}$-modules, where $\mu^{-1} := -w_{0}(\mu)$ is a dominant inverse of $\mu$. Similarly, we have an isomorphism
\[ R\Gamma_{c}^{\flat}(G,b,\mu)[\rho] \simeq j_{\mathbf{1}}^{*}T_{\mu}j_{b*}(\rho)  \]
of complexes of $G(\mathbb{Q}_{p}) \times W_{E}$-modules, and an isomorphism
\[ R\Gamma_{c}^{\flat}(G,b,\mu)[\pi][-2\langle 2\rho_{G}, \nu_{b} \rangle]  \simeq  j_{b}^{!}T_{\mu^{-1}}j_{\mathbf{1}*}(\pi) \otimes \kappa^{-1} \]
of complexes of $J_{b}(\mathbb{Q}_{p}) \times W_{E}$-modules.
\end{lemma}
\begin{proof}{(Sketch)}
The third and fourth relationship can be obtained from the first and second by acting via Verdier duality and using Hom-Tensor duality together with the assumed admissibility. 
For the first and second relationship, this is a very long diagram chase. To see the idea, for the first relationship we note that the case of minuscule $\mu$ and torsion $\Lambda$  was established in the process of proving \cite[Theorem 8.5.7]{MingjiaPolI}. The idea for the first and second relationship with torsion coefficients and general $\mu$ is exactly the same. For the case of rational and integral coefficients, the argument is exactly the same however one needs to compare the $\natural$-pushforward in the category of solid sheaves appearing in the definition of $T_{\mu}$ given in \S \ref{s: HeckeEigensheavesforTori} with the definition of $R\Gamma_{c}(G,b,\mu)$ given above. This can be accomplished by using \cite[Proposition~VII.5.2]{FS} and \cite[Definition/Proposition VI.7.10]{FS}  together with the fact that the natural pushforward commutes with colimits (since it is a left adjoint), as in the proof of \cite[Theorem~IX.3.1]{FS}.

We note that the character $\kappa$ and the shift appear since, after several applications of base-change, the difference between the LHS and RHS of the first relationship will be controlled by the $!$-pushforward of the constant sheaf along the map
\[ \Sht(G,b,\mu)_{\infty} \ra \Gr_{G,\mu^{-1}}^{b} \]
from the Shtuka space to the adic Newton stratum of the $B_{dR}^{+}$ affine Grassmannian. This is in particular a $\mathcal{J}_{b}$-torsor so the cohomology of $\mathcal{J}_{b}^{> 0}$ contributes (cf. \cite[Corollary~8.5.6]{MingjiaPolI}) .
\end{proof}
We can be more explicit about this character. In particular, we have the following, which follows from \cite[Corollary~1.8]{GH}.
\begin{proposition}{\label{prop: unramifiedcomparison}}
We have an isomorphism $\kappa \simeq \delta_{P_{b}}^{-1}$, where $\delta_{P_{b}}$ is the modulus character of the standard parabolic $P_{b}$ with Levi factor $M_{b}$ transferred to $J_{b}$ along the inner twisting. In particular, we have, by Lemma \ref{shimhecke}, isomorphisms 
\[ R\Gamma_{c}(G,b,\mu)[\rho \otimes \delta_{P_{b}}][2\langle 2\hat{\rho}_{G}, \nu_{b} \rangle]  \simeq  j_{\mathbf{1}}^{*}T_{\mu}j_{b!}(\rho)\]
and 
\[ R\Gamma_{c}^{\flat}(G,b,\mu)[\rho] \simeq j_{\mathbf{1}}^{*}T_{\mu}j_{b*}(\rho)  \]
of $G(\mathbb{Q}_{p}) \times W_{E}$-modules. 
\end{proposition}
We note that, since we used the tilting module $\mathcal{T}_{\mu}$ in the definition of $R\Gamma_{c}(G,b,\mu)$ and the Hecke operator $T_{\mu}$, the complex $R\Gamma_{c}(G,b,\mu)$ will in general be different from the usual complex defined with respect to the highest weight module $V_{\mu}$. However, they will agree when we impose the following condition on $\mu$ with respect to our coefficient system $\Lambda$.
\begin{definition}{\label{defmutilting}}
For $\Lambda$ as in assumption \ref{assumption: coefficients}, we will say $\mu$ is tilting if the representation $V_{\mu} \in \Rep_{\Lambda}(\hat{G})$ lies in the full subcategory $\Tilt_{\Lambda}(\hat{G})$ of tilting modules.
\end{definition}
\begin{remark}
If $\Lambda = F$ for $F/\bb{Q}_{\ell}$ a (possibly infinite) algebraic extension then this condition always holds. Moreover, by Lemma \ref{mintilting}, this always holds if $\mu$ is minuscule. In Appendix \ref{tiltingcocharacters}, we give more insight into this notion. 
\end{remark}
From now on, $\Lambda$ will be any coefficient system as in Setup \ref{assumption: coefficients} unless otherwise stated. We consider a toral parameter $\phi_{T}: W_{\mathbb{Q}_{p}} \ra \phantom{}^{L}T(\Lambda)$ with associated smooth character $\chi: T(\mathbb{Q}_{p}) \ra \Lambda^{*}$. Unless otherwise stated, we will assume that $\phi_{T}$ is integral with generic regular mod $\ell$-reduction, as defined in \ref{defn: IntegralParameter}. Given such a $\phi_{T}$, we set $\phi$ to be the $L$-parameter of $G$ induced via the natural embedding $\phantom{}^{L}T(\Lambda) \ra \phantom{}^{L}G(\Lambda)$ and consider the sheaf $\nmEis(\mathcal{S}_{\phi_{T}}) \in \Dlis(\Bun_{G},\Lambda)$ given by Theorem \ref{Eiseigsheaf}.
We begin our analysis by relating our eigensheaves to an averaging formula of Shin. From now on, we will also assume \ref{compatibility}.
\subsection{The Averaging Formula}{\label{s: avgformula}}
For $\mu$ a geometric dominant cocharacter with reflex field $E$, we write $r_{\mu}: W_{E} \ltimes \hat{G} \ra \mathrm{GL}(\mathcal{T}_{\mu})$ for the map defined by $\mathcal{T}_{\mu}$. Since the Hecke operator $T_{\mu^{\Gamma}}$ attached to $\mathcal{T}_{\mu^{\Gamma}} \in \Rep_{\Lambda}(\phantom{}^{L}G)$ factors through the Hecke operator $T_{\mu}$ attached to $\mathcal{T}_{\mu}$ (cf. \cite[Pages~313-315]{FS}) if $\phi_{T}$ is $\mu$-regular then, by definition, we have an isomorphism
\[ T_{\mu}\nmEis(\mathcal{S}_{\phi_{T}}) \simeq r_{\mu} \circ \phi|_{W_{E}} \boxtimes \nmEis(\mathcal{S}_{\phi_{T}}) \]
of objects in $\Dlis(\Bun_{G},\Lambda)^{BW_{E}}$. Now let's apply the restriction functor $j_{\mathbf{1}}^{*}(-)$ to both sides of this isomorphism. By the description of the stalks, we know that the RHS is equal to $r_{\mu} \circ \phi|_{W_{E}} \boxtimes \pi$, where $\pi := i_{B}^{G}(\chi)$ is the normalized parabolic induction of the character $\chi$ attached to $\phi_{T}$ by class field theory. We can also simplify the LHS. In particular, first off note that, since any $G$-bundle $\mathcal{F}_{G}$ on $X$ that occurs as a modification $\mathcal{F}_{G} \dashrightarrow \mathcal{F}_{G}^{0}$ of type $\mu$ lies in the set $B(G,\mu)$ by \cite[Proposition~A.9]{R}, we have an isomorphism
\[ j_{\mathbf{1}}^{*}(T_{\mu}(\nmEis(\mathcal{S}_{\phi_{T}}))) \simeq j_{\mathbf{1}}^{*}T_{\mu}(\nmEis(\mathcal{S}_{\phi_{T}})|_{B(G,\mu)}) \]
where here we view $B(G,\mu)$ as the open subset of $\Bun_{G}$ defined by the identification $B(G) \simeq |\Bun_{G}|$ describing the underlying topological space $|\Bun_{G}|$ of $\Bun_{G}$ \cite[Theorem~1.1]{Vi}, where $B(G)$ has the natural topology given by the partial ordering. We can further refine this by applying excision with respect to the locally closed stratification by the Harder-Narasimhan strata $\Bun_{G}^{b} \subset \Bun_{G}$ for $b \in B(G,\mu)$,  using \cite[Proposition~VII.7.3]{FS}. The excision spectral sequence then tells us that the LHS has a filtration whose graded pieces are isomorphic to:
\[ j_{\mathbf{1}}^{*}(T_{\mu}(j_{b!}j_{b}^{*}(\nmEis(\mathcal{S}_{\phi_{T}})))), \]
which, using Proposition \ref{inertsheaves}, and the description of the stalks of the Eigensheaf one can show is isomorphic to  $j_{\mathbf{1}}^{*}(T_{\mu}(j_{b*}j_{b}^{*}(\nmEis(\mathcal{S}_{\phi_{T}}))))$. In particular, in $K_{0}(G(\mathbb{Q}_{p}) \times W_{E},\Lambda)$, the Grothendieck group of $\Lambda$-valued smooth admissible $G(\mathbb{Q}_{p})$-representations with a continuous action of $W_{E}$, this tells us that we have an equality: 
\begin{equation}{\label{eqn: averagingformula}} \sum_{b \in B(G,\mu)} [j_{\mathbf{1}}^{*}(T_{\mu}(j_{b*}j_{b}^{*}(\nmEis(\mathcal{S}_{\phi_{T}})))] = [r_{\mu} \circ \phi|_{W_{E}} \boxtimes \pi] 
\end{equation} 
Now, using the description of the stalks of $\nmEis(\mathcal{S}_{\phi_{T}})$ and Lemma \ref{shimhecke}, we can spell out the LHS more clearly. In particular, we define the following. 
\begin{definition}{\label{defn: Red}}
For $\phi_{T}$ an arbitrary toral parameter with induced parameter $\phi$ and $b \in B(G)$, we define the complex of smooth admissible $J_{b}(\mathbb{Q}_{p})$-representations $\Red_{b,\phi}$ as follows. If $b \notin B(G)_{\mathrm{un}}$, we set $\Red_{b,\phi}$ to be equal to $0$, and, if $b \in B(G)_{\mathrm{un}}$, we set $\Red_{b,\phi}$ to be equal to
\[ \bigoplus_{w \in W_{b}} i_{B_{b}}^{J_{b}}(\chi^{w^{-1}}) \otimes \delta_{P_{b}}^{-1/2}[-\langle 2\hat{\rho}, \nu_{b} \rangle] \in \D(J_{b}(\mathbb{Q}_{p}),\Lambda) \]
where $\delta_{P_{b}}$ is the modulus character of $J_{b}$ defined by the standard parabolic $P_{b} \subset G$ with Levi factor $M_{b} \simeq J_{b}$, as before. 
\end{definition}
This allows us to deduce the following from equation (\ref{eqn: averagingformula}). 
\begin{theorem}{\label{avgingformula}} 
For $\phi_{T}: W_{\mathbb{Q}_{p}} \ra \phantom{}^{L}T(\Lambda)$ an integral parameter with generic regular mod $\ell$ reduction, if $\pi := i_{B}^{G}(\chi)$ is the normalized parabolic induction of the smooth character $\chi$ attached to $\phi_{T}$ then, for any geometric dominant cocharacter $\mu$ such that $\phi_{T}$ is $\mu$-regular, we have an equality
\[ \sum_{b \in B(G,\mu)} [R\Gamma^{\flat}_{c}(G,b,\mu)[\Red_{b,\phi}]] = [r_{\mu} \circ \phi|_{W_{E}} \boxtimes \pi] \]
in $K_{0}(G(\mathbb{Q}_{p}) \times W_{E},\Lambda)$. 
\end{theorem}
\begin{remark}
We note that, in the above analysis, we didn't necessarily have to restrict to the HN-stratum $\Bun_{G}^{\mathbf{1}}$ or even a single cocharacter. In particular, by considering Hecke operators defined by representations in $\Tilt_{\Lambda}(\phantom{}^{L}G^{I})$ for a finite index set $I$, we could have deduced an analogous formula for shtuka spaces with $I$ legs for an arbitrary finite index set $I$. We could have also restricted to any HN-stratum; however, if the HN-stratum is not defined by a basic element, the answer is not as clean as above. In particular, given a $G$-bundle $\mathcal{F}_{b}$ on $X$ corresponding to a general element $b \in B(G)$, it is to the best of our knowledge completely unknown exactly which $G$-bundles $\mathcal{F}_{G}$ occur as modifications $\mathcal{F}_{G} \dashrightarrow \mathcal{F}_{b}$ of type $\mu$. It would be interesting to understand this question better. We leave it to the reader to work out the precise statements of these more general implications.
\end{remark}
We can use this claim to deduce the averaging formula for an arbitrary toral parameter $\phi_{T}$ when $\Lambda = \ol{\mathbb{Q}}_{\ell}$, by viewing both sides as trace forms on $K_{0}(T(\mathbb{Q}_{p}),\ol{\mathbb{Q}}_{\ell})$ and using a continuity argument. We note that, in this case, the cocharacter $\mu$ is always tilting so we have that $R\Gamma_{c}(G,b,\mu)$ and $R\Gamma^{\flat}_{c}(G,b,\mu)$ are just the usual complexes. We recall that $f: K_{0}(G(\mathbb{Q}_{p}),\ol{\mathbb{Q}}_{\ell}) \ra \ol{\mathbb{Q}}_{\ell}$ is a trace form if it can be written as $\tr(\delta|-)$ for $\delta \in C^{\infty}_{c}(G(\mathbb{Q}_{p}),\ol{\mathbb{Q}}_{\ell})$. We now fix a $\delta \in C^{\infty}_{c}(G(\mathbb{Q}_{p}),\ol{\mathbb{Q}}_{\ell})$ and $\gamma \in W_{E}$, we define the following functions attached to this datum:
\[ f^{\delta,\gamma}_{L}: K_{0}(T(\mathbb{Q}_{p}),\ol{\mathbb{Q}}_{\ell}) \ra \ol{\mathbb{Q}}_{\ell} \]
\[ \chi \mapsto \tr(\delta \times \gamma|i_{B}^{G}(\chi) \boxtimes r_{\mu} \circ \iota(\chi)|_{W_{E}}) \]
\[ f^{\delta,\gamma}_{R}: K_{0}(T(\mathbb{Q}_{p}),\ol{\mathbb{Q}}_{\ell}) \ra \ol{\mathbb{Q}}_{\ell} \]
\[ \chi \mapsto \sum_{b \in B(G,\mu)_{\mathrm{un}}} \sum_{w \in W_{b}} \tr(\delta \times \gamma| R\Gamma_{c}^{\flat}(G,b,\mu)[i_{B_{b}}^{J_{b}}(\chi^{w^{-1}}) \otimes \delta_{P_{b}}^{-1/2}])(-1)^{\langle 2\hat{\rho}_{G}, \nu_{b} \rangle},  \]
where $\iota(\chi) \simeq \phi_{T}$ is given by local class field theory. We have the following lemma.
\begin{lemma}
The functions $f^{\delta,\gamma}_{L}$ and $f^{\delta,\gamma}_{R}$ define trace forms on $K_{0}(T(\mathbb{Q}_{p}),\ol{\mathbb{Q}}_{\ell})$.  
\end{lemma}
\begin{proof}
This follows from the fact that normalized parabolic induction takes trace forms to trace forms as can be checked from the characterization of trace forms in the trace Paley-Wiener theorem \cite{BDK}, and the fact that $\tr(\delta \times \gamma|R\Gamma_{c}(G,b,\mu)[-])$ defines a trace form on $K_{0}(J_{b}(\mathbb{Q}_{p}),\ol{\mathbb{Q}}_{\ell})$ by \cite[Theorem~6.5.4]{KW}. 
\end{proof}
This gives the following. 
\begin{theorem}{\label{qellavgingformula}} 
For $\phi_{T}: W_{\mathbb{Q}_{p}} \ra \phantom{}^{L}T(\ol{\mathbb{Q}}_{\ell})$ an arbitrary toral parameter with associated character $\chi$, and $i_{B}^{G}(\chi) =: \pi$, we have an equality
\[ \sum_{b \in B(G,\mu)} [R\Gamma^{\flat}_{c}(G,b,\mu)[\Red_{b,\phi}]] = [r_{\mu} \circ \phi|_{W_{E}} \boxtimes \pi] \]
in $K_{0}(G(\mathbb{Q}_{p}) \times W_{E},\ol{\mathbb{Q}}_{\ell})$, assuming $G$ is split. 
\end{theorem}
\begin{proof}
It suffices to show that the trace forms $f^{\delta,\gamma}_{L}(\chi)$ and  $f^{\delta,\gamma}_{R}(\chi)$ agree for varying $\delta$ and $\gamma$ and all $\chi \in K_{0}(T(\mathbb{Q}_{p}),\ol{\mathbb{Q}}_{\ell})$. We define the difference $\Delta_{\delta,\gamma} := f^{\delta,\gamma}_{L}(\chi) - f^{\delta,\gamma}_{R}(\chi)$. We say that a subset $S \in K_{0}(T(\mathbb{Q}_{p}),\ol{\mathbb{Q}}_{\ell})$ is dense if $\Delta_{\delta,\gamma}(x) = 0$ for all $x \in S$ implies that $\Delta_{\delta,\gamma} = 0$. Using Theorem \ref{avgingformula} and Corollary \ref{cor: muregopencond}, we can reduce to showing that the subset $S$ of all characters $\chi$ which are normalized regular and admit a $\ol{\mathbb{Z}}_{\ell}$-lattice is dense. This is relatively easy to show. In particular, if we view $\Delta_{\delta,\gamma}$ as a regular function on the variety of unramified twists of a fixed character $\chi$ then the set of characters admitting a $\ol{\mathbb{Z}}_{\ell}$ lattice is Zariski-dense. Moreover, the locus where $\chi$ is normalized regular is also clearly Zariski dense if $G$ is split, since it is implied by insisting that $\chi$ precomposed with sums of coroots is not the norm or trivial character for the sums of coroots appearing in differences of distinct weights of $V_{\mu_{k}}$, for a finite list of generating cocharacters $\mu_{k}$ with $k = 1,\ldots,n$ appearing in the definition of $\mu$-regularity, and it is also clear for the condition of generic regularity. Therefore, the claim follows, using the previous Lemma. 
\end{proof}
This theorem is compatible with existing results. We recall that Shin \cite{Shin} and Bertoloni-Meli \cite{BM}, have described similar averaging formulas. In particular, given a refined endoscopic datum $\mathfrak{e} = (H,\mathcal{H},s,\eta)$ (Definition \ref{def: refinedendoscopy}), Shin and Bertoloni-Meli define maps
\[ \Red_{b}^{\mathfrak{e}}(-): K_{0}^{st}(H(\mathbb{Q}_{p}),\ol{\mathbb{Q}}_{\ell}) \ra K_{0}(J_{b}(\mathbb{Q}_{p}),\ol{\mathbb{Q}}_{\ell}) \]
where $K_{0}^{st}(H(\mathbb{Q}_{p}),\ol{\mathbb{Q}}_{\ell})$ denotes the Grothendieck group of stable virtual $\ol{\mathbb{Q}}_{\ell}$-representations of $H(\mathbb{Q}_{p})$. If we are given an $L$-parameter $\phi$ which factors as $\mathcal{L}_{\mathbb{Q}_{p}} \xrightarrow{\phi^{H}} \mathcal{H} \xrightarrow{\phantom{}^{L}\eta} \phantom{}^{L}G$ then, using the local Langlands correspondence for $H$, we are able to attach a stable distribution $S\Theta_{\phi} \in K_{0}^{st}(H(\mathbb{Q}_{p}),\ol{\mathbb{Q}}_{\ell})$ which should satisfy the endoscopic character identities as in \cite[Conjecture~D]{Kal}. The averaging formula (Conjecture \ref{conj: classicavgformula}) is a conjectural formula for
\[
\sum_{b\in B(G,\mu)}
(-1)^{\langle2\hat\rho_G,\nu_b\rangle}
[
R\Gamma_c^\flat(G,b,\mu)
\bigl[\operatorname{Red}_b^e(S\Theta_{\phi_H,H})\bigr]
].
\]
in $K_{0}(G(\mathbb{Q}_{p}) \times W_{E},\ol{\mathbb{Q}}_{\ell})$. Our averaging formula is related to the case when $\mf{c}_{\triv} = (G,\phantom{}^{L}G,1,\mathrm{id})$ is the trivial endoscopic datum. In particular, if $\phi_{T}$ is a generic toral parameter then, by Lemma \ref{regmonodromy}, $\phi$ should define an actual $L$-parameter with trivial monodromy, and we can consider the $L$-packet $\Pi_{\phi}(G)$ under the local Langlands correspondence for $G$ appearing in Assumption \ref{compatibility}. By Assumption \ref{compatibility} (3), the members of the $L$-packet are among the irreducible constituents of $i_{B}^{G}(\chi)$.

We assume that the associated stable character satisfies the expected
multiplicity formula for the constituents of principal series (this is only relevant if $\chi$ is not regular),
so that $S\Theta_\phi=[i_B^G(\chi)]$ under the local Langlands in assumption \ref{compatibility}.

Therefore, we have that $S\Theta_{\phi} = [\pi]$ in $K_{0}^{st}(G(\mathbb{Q}_{p}),\ol{\mathbb{Q}}_{\ell})$, and in the appendix we verify that the following is true. 
\begin{proposition}{\label{classicalavgrel}} 
Assume $G$ is split. Let
$\chi:T(\mathbb Q_p)\to\overline{\mathbb Q}_\ell^\times$
be a smooth generic character, and put $\pi=i_B^G(\chi)$.
Assume that the stable distribution attached to the associated
$L$-parameter $\phi$ under the local Langlands correspondence
appearing in Assumption~\ref{compatibility} satisfies
\[
S\Theta_\phi=[\pi]
\]
in $K_0(G(\mathbb Q_p),\overline{\mathbb Q}_\ell)^{\mathrm{st}}$.
Then we have
\[
[\operatorname{Red}_{b,\phi}]
=
\operatorname{Red}_b^{e_{\mathrm{triv}}}([\pi])
(-1)^{\langle2\hat\rho_G,\nu_b\rangle}.
\]
in the Grothendieck group $K_{0}(J_{b}(\mathbb{Q}_{p}),\ol{\mathbb{Q}}_{\ell})$, and Conjecture \ref{conj: classicavgformula} holds true for the $L$-parameter $\phi$ attached to $\chi$. 
\end{proposition} 
We would now like to refine our averaging formula further. In particular, using Theorem \ref{commverddual}, we can upgrade this equality in the Grothendieck group to a genuine isomorphism of complexes. 
\subsection{The Refined Averaging Formula and Intertwining Operators}
Consider an element $b \in B(G)_{\mathrm{un}}$ with dominant reduction $b_{T} \in B(T)$. For $w \in W_{b}$, we set $\rho_{b,w} := i_{B_{b}}^{J_{b}}(\chi^{w^{-1}}) \otimes \delta_{P_{b}}^{-1/2}$ to be the twisted  normalized parabolic induction, and consider $j_{b}: \Bun_{G}^{b} \hookrightarrow \Bun_{G}$, the inclusion of the locally closed HN-stratum corresponding to $b$. We let $\phi_{T}$ be an integral parameter with coefficients in $\Lambda$ as in Setup \ref{assumption: coefficients} with regular generic mod $\ell$-reduction and consider the sheaf $\nmEis(\mathcal{S}_{\phi_{T}})$ furnished by Theorem \ref{Eiseigsheaf}. We have an isomorphism
\begin{equation}{\label{eqn: stalkcalculation}}
j_{b!}(\rho_{b,w})[-\langle 2\hat{\rho},\nu_{b} \rangle] \simeq \nmEis^{w(b_{T})}(\mathcal{S}_{\phi_{T}}).
\end{equation}
We can act by $\mathbb{D}_{\Bun_{G}}(-)$ on both sides of (\ref{eqn: stalkcalculation}). By the commutation of Eisenstein series with Verdier duality, the RHS of (\ref{eqn: stalkcalculation}) becomes
\[ \nmEis^{w(b_{T})}(\mathcal{S}_{\phi_{T}^{\vee}}) \simeq j_{b!}(\rho^{*}_{b,w} \otimes \delta_{P_{b}}^{-1})[-\langle 2\hat{\rho},\nu_{b} \rangle] \]
where $\rho^{*}_{b,w} = (i_{B_{b}}^{J_{b}}(\chi^{w^{-1}}) \otimes \delta_{P_{b}}^{-1/2})^{*} = i_{B_{b}}^{J_{b}}((\chi^{w^{-1}})^{-1}) \otimes \delta_{P_{b}}^{1/2}$ denotes the contragredient. 
On the other hand, the LHS of (\ref{eqn: stalkcalculation}) becomes
\[ j_{b*}(\mathbb{D}_{\Bun_{G}^{b}}(\rho_{b,w}))[\langle 2\hat{\rho},\nu_{b}\rangle] \]
We now need to be a bit careful. In particular, we recall that $\Bun_{G}^{b} \simeq [\ast/\mathcal{J}_{b}]$, where $\mathcal{J}_{b}$ is the group diamond parameterizing automorphisms of $\mathcal{F}_{b}$, and we are implicitly using the identification $\Dlis(\Bun_{G}^{b},\Lambda) \simeq \D(J_{b}(\mathbb{Q}_{p}),\Lambda)$ given by pullback along the map $p: [\ast/\mathcal{J}_{b}] \ra [\ast/\ul{J_{b}(\mathbb{Q}_{p})}]$, as in \cite[Proposition~V.2.2,VII.7.1]{FS}. Therefore, we need to account for the shifts and twists given by $p^{!}$. We can use that the natural section $s$ of $p$ is an iterated fibration of positive Banach-Colmez space and Proposition \ref{prop: unramifiedcomparison}, to show that $p^{!}(-) \simeq p^{*}(- \otimes \delta_{P_{b}}^{-1})[-2\langle 2\hat{\rho},\nu_{b} \rangle]$, and therefore the LHS of (\ref{eqn: stalkcalculation}) becomes $j_{b*}(\rho_{b,w}^{*} \otimes \delta_{P_{b}}^{-1})[\langle 2\hat{\rho},\nu_{b}\rangle - 2\langle 2\hat{\rho},\nu_{b}\rangle] = j_{b*}(\rho_{b,w}^{*} \otimes \delta_{P_{b}}^{-1})[-\langle 2\hat{\rho},\nu_{b} \rangle]$. In conclusion, we have an isomorphism: 
\[ j_{b*}(\rho_{b,w}^{*} \otimes \delta_{P_{b}}^{-1})[-\langle 2\hat{\rho}, \nu_{b} \rangle] \simeq j_{b!}(\rho_{b,w}^{*} \otimes \delta_{P_{b}}^{-1})[-\langle 2\hat{\rho},\nu_{b} \rangle] \]
Relaxing the contragredients and cancelling the shifts, we deduce an isomorphism:
\[ j_{b*}(\rho_{b,w}) \simeq j_{b!}(\rho_{b,w}). \]
In conclusion, we deduce the following, which was already mentioned in Theorem \ref{thm: inertsheavestorsion}, but we have now also handled the case of more general coefficient systems using the discussion in \S 9. 
\begin{proposition}{\label{inertsheaves}}
For $b \in B(G)_{\mathrm{un}}$, $w \in W_{b}$, $\phi_{T}$ an integral toral parameter with generic regular mod $\ell$-reduction, and $\rho_{b,w} = i_{B_{b}}^{J_{b}}(\chi^{w^{-1}}) \otimes \delta_{P_{b}}^{-1/2}$, we have an isomorphism
\[ j_{b!}(\rho_{b,w}) \simeq j_{b*}(\rho_{b,w}) \]
of objects in $\Dlis(\Bun_{G},\Lambda)$. 
\end{proposition}
\begin{remark}
If $b$ is basic then this precisely says that the sheaf defined by $\rho_{b,w}$ is inert in the sense of \cite[Definition~2.19]{Han1}. In particular, this Proposition, in conjunction with \cite[Theorem~2.22]{Han1} and Lemma \ref{regmonodromy}, seems to suggest that inert sheaves should correspond precisely to the representations whose semi-simplified $L$-parameter does not come from the semi-simplification of a parameter with non-trivial monodromy. For example, if one takes the constant sheaf on $\Bun_{G}^{\mathbf{1}}$ and considers $j_{\mathbf{1}!}(\Lambda)$ then we have that $\mathbb{D}_{\Bun_{G}}(j_{\mathbf{1}!}(\Lambda)) \simeq j_{\mathbf{1}*}(\Lambda)$ which one can check is not isomorphic to $j_{!}(\Lambda)$ (see the proof of \cite[Proposition~3.16]{GH}). Similarly, we see that the $L$-parameter attached to the trivial representation comes from the semi-simplification of a parameter with non-trivial monodromy (the Steinberg parameter).
\end{remark}
Now consider $\mu$ a geometric dominant cocharacter of $G$ with reflex field $E$ and an element $b \in B(G,\mu)$. Applying $j_{\mathbf{1}}^{*}T_{\mu}(-)$ to both sides of the previous isomorphism, we conclude, using Proposition \ref{prop: unramifiedcomparison}, an isomorphism
\[ R\Gamma_{c}(G,b,\mu)[\rho_{b,w} \otimes \delta_{P_{b}}] \simeq R\Gamma^{\flat}_{c}(G,b,\mu)[\rho_{b,w}][-2\langle 2\rho_{G},\nu_{b} \rangle]. \] 
\begin{corollary}{\label{rhoduality}}
Let $(G,b,\mu)$ be a local shtuka datum. For $b \in B(G)_{\mathrm{un}}$, $w \in W_{b}$, and $\phi_{T}$ an integral toral parameter such that $\phi_{T}$ has generic regular mod $\ell$ reduction, there is an isomorphism
\[ R\Gamma_{c}(G,b,\mu)[\rho_{b,w} \otimes \delta_{P_{b}}] \simeq R\Gamma^{\flat}_{c}(G,b,\mu)[\rho_{b,w}][-2\langle 2\rho_{G},\nu_{b} \rangle] \]
of complexes of $G(\mathbb{Q}_{p}) \times W_{E}$-modules. 
\end{corollary}
We now claim that the cohomology of $R\Gamma_{c}(G,b,\mu)[\rho_{b,w} \otimes \delta_{P_{b}}]$ should be concentrated in degree $\langle 2\hat{\rho}, \nu_{b} \rangle$, for $\rho_{b,w}$ as above. To do this, let's put ourselves back in the position of an integral $\phi_{T}$ with generic regular mod $\ell$ reduction. We saw that in the previous section that the excision spectral sequence applied to $\nmEis(\mathcal{S}_{\phi_{T}})|_{B(G,\mu)}$ gives rise to a filtration whose graded pieces are isomorphic to $j_{b!}j_{b}^{*}(\nmEis(\mathcal{S}_{\phi_{T}}))$, but Proposition \ref{inertsheaves} implies that these graded pieces are also isomorphic to $j_{b*}j_{b}^{*}(\nmEis(\mathcal{S}_{\phi_{T}}))$.
In particular, this allows us to deduce that the edge maps in the excision spectral sequence split, and therefore the sequence degenerates, giving an isomorphism:
\[ \bigoplus_{b \in B(G,\mu)} j_{b*}j_{b}^{*}\nmEis(\mathcal{S}_{\phi_{T}}) \simeq \nmEis(\mathcal{S}_{\phi_{T}})|_{B(G,\mu)}. \]
We now would like to apply the eigensheaf property. So fix a geometric dominant cocharacter, and assume that $\phi_{T}$ is $\mu$-regular. If $\pi = i_{B}^{G}(\chi)$ is the normalized parabolic induction of $\chi$ as above then, using our description of the stalks, we deduce the following "refined averaging formula".
\begin{theorem}{\label{refinedaveragingformula}}
For $\phi_{T}: W_{\mathbb{Q}_{p}} \ra \phantom{}^{L}T(\Lambda)$ an integral toral parameter with generic regular mod $\ell$-reduction and $\mu$ a geometric dominant cocharacter such that $\phi_{T}$ is $\mu$-regular, we have an isomorphism 
\[ \bigoplus_{b \in B(G,\mu)_{\mathrm{un}}} \bigoplus_{w \in W_{b}} R\Gamma^{\flat}_{c}(G,b,\mu)[\rho_{b,w}][-\langle 2\hat{\rho}_{G},\nu_{b} \rangle] \simeq (i_{B}^{G}(\chi) \boxtimes r_{\mu} \circ \phi|_{W_{E}}) \]
of complexes of $G(\mathbb{Q}_{p}) \times W_{E}$-modules.
\end{theorem}
Unless otherwise stated, we will from now on assume that $\phi_{T}$ is integral with generic regular mod $\ell$ reduction. Using the previous formula and the assumed regularity of $\chi$ so that $\mathrm{End}_{G(\bb{Q}_{p})}(i_{B}^{G}(\chi)) \simeq \Lambda$ by the geometric lemma, we can give a very explicit description of the complexes $R\Gamma^{\flat}_{c}(G,b,\mu)[\rho_{b,w}]$, for $b \in B(G,\mu)_{\mathrm{un}}$ and $w \in W_{b}$.
\begin{corollary}{\label{inductisotypic}}
For $\mu$ a geometric dominant cocharacter with reflex field $E$ such that $\phi_{T}$ is $\mu$-regular, fixed $b \in B(G,\mu)_{\mathrm{un}}$, and varying $w \in W_{b}$, the complex $R\Gamma^{\flat}_{c}(G,b,\mu)[\rho_{b,w}]$ is isomorphic to $\phi_{b,w}^{\mu} \boxtimes \sigma[\langle 2\hat{\rho},\nu_{b} \rangle]$, for $\phi_{b,w}^{\mu}$ a representation of $W_{E}$ and $\sigma$ a subrepresentation of $i_{B}^{G}(\chi)$. Moreover, we have an isomorphism 
\[ \bigoplus_{b \in B(G,\mu)_{\mathrm{un}}} \bigoplus_{w \in W_{b}} \phi_{b,w}^{\mu} \simeq r_{\mu} \circ \phi|_{W_{E}} \]
of $W_{E}$-representations.
\end{corollary}
This leads to a natural question. How can we describe the $W_{E}$-representations $\phi_{b,w}^{\mu}$ in terms of the weights appearing in $r_{\mu} \circ \phi|_{W_{E}}$. We recall, by Corollary \ref{bgmuweights}, that the orbit of $b_{T}$ under the Weyl group $W_{G}$ can be described as $w(b_{T})$ for $w \in W_{b}$ varying; moreover, using Corollary \ref{bgmuweights} and Remark \ref{rem: modularhighestweight}, we see that we have a correspondence between $B(G,\mu)_{\mathrm{un}}$ and the set of Weyl orbits of weights which can occur in the representation $\mathcal{T}_{\mu}|_{\hat{G}^{\Gamma}}$. Choose a finite extension $E'/E$ splitting $G$. In particular, given $\ol{\nu} \in \mathbb{X}_{*}(T_{\ol{\mathbb{Q}}_{p}})_{\Gamma}$, we consider the subspace
\[ \bigoplus_{\substack{\nu \in \mathbb{X}_{*}(T_{\ol{\mathbb{Q}}_{p}}) \\
\nu_{\Gamma} = \ol{\nu}}} \nu \circ \phi_{T}|_{W_{E'}} \otimes \mathcal{T}_{\mu}(\nu) \]
of $(r_{\mu} \circ \phi)|_{W_{E'}}$, where we note that if we forget the Galois action then this identifies with the $\ol{\nu}$ weight space of $\mathcal{T}_{\mu}|_{\hat{G}^{\Gamma}}$ by Lemma \ref{lemma: coinv versus orbs}. Now the refined averaging formula suggests the following.
\begin{conjecture}{\label{weilgroupaction}}
For all geometric dominant cocharacters $\mu$ such that $\phi_{T}$ is $\mu$-regular, an unramified element $b \in B(G,\mu)_{\mathrm{un}}$, and a Weyl group element $w \in W_{b}$, we have an isomorphism 
\[ \bigoplus_{\substack{\wt{w(b_{T})} \in \mathbb{X}_{*}(T_{\ol{\mathbb{Q}}_{p}}) \\
\wt{w(b_{T})}_{\Gamma} = w(b_{T})}} \wt{w(b_{T})} \circ \phi_{T}|_{W_{E'}} \otimes \mathcal{T}_{\mu}(\wt{w(b_{T})}) \simeq \phi_{b,w}^{\mu}|_{W_{E'}} \] 
of $W_{E'}$-representations, where $b_{T}$ is a dominant reduction of $b$ and $E'/E$ is a choice of splitting field for $G$
\end{conjecture}
For the rest of this section, let us look at some cases where this can be shown explicitly, using a shtuka analogue of Boyer's trick. To illustrate the idea, we begin with a particularly nice example, where Theorem \ref{refinedaveragingformula} and Conjecture \ref{weilgroupaction} can be checked by hand.
\begin{Example}{\label{GL2ex}}
Let $G = \GL_{2}$ and $\mu = (1,0)$. Write $\phi_{T} = \phi_{1} \oplus \phi_{2}$, and consider the set $B(G,\mu)$. It consists of two elements: the $\mu$-ordinary element and the basic element. Only the $\mu$-ordinary element lies in $B(G,\mu)_{\mathrm{un}}$; therefore, only this element contributes to the expression in Theorem \ref{refinedaveragingformula}. Namely, if $b_{\mu}$ denotes the $\mu$-ordinary element, we note that $\langle 2\hat{\rho}, \nu_{b_{\mu}} \rangle = \langle 2\hat{\rho}, \mu \rangle = 1$. We conclude that Theorem \ref{refinedaveragingformula} is an isomorphism
\[ R\Gamma^{\flat}_{c}(G,b_{\mu},\mu)[\chi \otimes \delta_{B}^{-1/2}] \oplus R\Gamma^{\flat}_{c}(G,b_{\mu},\mu)[\chi^{w_{0}} \otimes \delta_{B}^{-1/2}] \simeq i_{B}^{GL_{2}}(\chi) \boxtimes \phi[1] \]
of $G(\mathbb{Q}_{p}) \times W_{\mathbb{Q}_{p}}$-representations. This can be seen through direct computation. In particular, we have an isomorphism $J_{b_{\mu}} \simeq  T$, and, since $\mu$ is minuscule, we have that $\mathcal{S}_{\mu} \simeq \Lambda[1](\frac{1}{2})$. The space $\Sht(G,b_{\mu},\mu)_{\infty,\mathbb{C}_{p}}$ is the moduli space parameterizing modifications $\mathcal{O}(-1) \oplus \mathcal{O} \dashrightarrow \mathcal{O}^{2}$ of type $(1,0)$. Every such modification is determined by an injection $\mathcal{O}(-1) \hookrightarrow \mathcal{O}$ of line bundles. Formally, this implies that the space $\Sht(G,b,\mu)_{\infty,\mathbb{C}_{p}}$ as a space with $\GL_{2}(\mathbb{Q}_{p})$-action is parabolically induced from the space parameterizing such injections as a space with $T(\mathbb{Q}_{p})$ action. Here $T(\mathbb{Q}_{p})$ acts on the space of injections $\mathcal{O}(-1) \hookrightarrow \mathcal{O}$ via the scaling action precomposed with projection to the first factor of $T(\mathbb{Q}_{p})$. This is a manifestation of the fact that $\Sht(G,b_{\mu},\mu)_{\infty,\mathbb{C}_{p}}$ is a $\mathcal{J}_{b}$-torsor over the flag variety $\ul{(G/B)(\mathbb{Q}_{p})} \simeq \mathbb{P}^{1}(\mathbb{Q}_{p}) \subset \mathbb{P}^{1}_{\mathbb{C}_{p}} \simeq \Gr_{G,\leq \mu^{-1},\mathbb{C}_{p}}$, where the last isomorphism is the Bialynicki-Birula map. In particular, note that the compactly supported cohomology of $\ul{G/B(\mathbb{Q}_{p})}$ is precisely the space of compactly supported functions on $(G/B)(\mathbb{Q}_{p})$. All in all, this allows us to conclude isomorphisms 
\[ R\Gamma_{c}^{\flat}(G,b_{\mu},\mu)[\chi \otimes \delta_{B}^{-1/2}] = \Ind_{B^{-}}^{G}(\chi \otimes \delta_{B}^{-1/2}) \boxtimes \phi_{1}[1] = i_{B}^{G}(\chi^{w_{0}}) \boxtimes \phi_{1}[1] \]
\[ R\Gamma_{c}^{\flat}(G,b_{\mu},\mu)[\chi^{w_{0}} \otimes \delta_{B}^{-1/2}] = \Ind_{B^{-}}^{G}(\chi^{w_{0}} \otimes \delta_{B}^{-1/2}) \boxtimes \phi_{2}[1] = i_{B}^{G}(\chi) \boxtimes \phi_{2}[1],  \]
where there is a cancellation of the $\frac{1}{2}$ Tate twist in $\mathcal{S}_{\mu}$ and the Tate twist coming from $\delta_{B}^{-1/2}$, as in \S \ref{GCFT}. Now, if $\chi$ is attached to a generic parameter $\phi_{T}$, this implies that $i_{B}^{G}(\chi)$ is irreducible as in Example \ref{Gl2irred}, and it follows that we have an isomorphism $i_{B}^{G}(\chi) \simeq i_{B}^{G}(\chi^{w_{0}})$, which allows us to conclude the result.
\end{Example}
Now let's generalize this example. In particular, recall that $B(G,\mu)$ has a distinguished $\mu$-ordinary element, denoted $b_{\mu}$, which is the maximal element with respect to the partial ordering on $B(G,\mu)$, and has the property that $\tilde{\mu} = \nu_{b_{\mu}}$, where $\tilde{\mu}$ is the weighted average over the Galois orbit of $\mu$ as in \S \ref{sec: unramifiedelments}. Let $\mu_T$ denote $\mu$ viewed as a $G$-dominant geometric
cocharacter of $T$, and let $b_{\mu_T}$ be the unique element
of $B(T,\mu_T)$. Then $b_{\mu_T}$ is the dominant 
reduction of $b_\mu$. In other words, the element $b_{\mu}$ always lies in $B(G,\mu)_{\mathrm{un}} := B(G,\mu) \cap B(G)_{\mathrm{un}}$. Conjecture \ref{weilgroupaction} suggests to us that this should give rise to the contribution given by the highest weight of $\mathcal{T}_{\mu}|_{\hat{G}^{\Gamma}}$, which will have multiplicity one. We now prove the following result using a shtuka analogue of Boyer's trick \cite{Boy} proven by Gaisin-Imai \cite{IG}. 
\begin{proposition}{\label{muordinary}}
For $\mu$ any geometric dominant cocharacter with reflex field $E$, $b_{\mu} \in B(G,\mu)_{\mathrm{un}}$ the $\mu$-ordinary element with dominant reduction $b_{\mu_{T}}$, $w \in W_{b}$ varying, and $\phi_{T}$ any toral parameter, we have an isomorphism
\[ R\Gamma^{\flat}_{c}(G,b_{\mu},\mu)[\rho_{b_{\mu},w}] \simeq  w(\mu_{T}) \circ \phi_{T}|_{W_{E}} \boxtimes i_{B}^{G}(\chi^{rw^{-1}})[\langle 2\hat{\rho},  \nu_{b_{\mu}} \rangle] \] 
of $W_{E} \times G(\mathbb{Q}_{p})$-representations, where $w \in W_{b}$ are representatives of minimal length and $r = w_{0}^{G}w_{0}^{M_{b}}$ is a product of longest length elements.
\end{proposition}
\begin{proof}
We note that the element $b_{\mu} \in B(G,\mu)$ is Hodge-Newton reducible in the sense of \cite[Definition~4.5]{RV}. In particular, $b_{\mu}$ is induced from the unique element $b_{\mu_{T}} \in B(T,\mu_{T})$ via the natural map $B(T) \ra B(G)$. In particular, the bundle $\mathcal{F}_{b_{\mu}}$ is isomorphic to $\mathcal{F}_{w(b_{\mu_{T}})} \times^{T} G$. By arguing as in  \cite[Lemma~4.11,Proposition~4.12]{IG} , this tells us that the space $\Sht(G,b_{\mu},\mu)_{\infty,\mathbb{C}_{p}}$ parameterizing modifications of the form
\[ \mathcal{F}_{b_{\mu}} \dashrightarrow \mathcal{F}_{G}^{0} \]
will be determined by the spaces $\Sht(T,w(b_{\mu_{T}}),w(\mu_{T}))_{\infty,\mathbb{C}_{p}}$ parametrizing modifications of the form 
\[ \mathcal{F}_{w(b_{\mu_{T}})} \dashrightarrow \mathcal{F}_{T}^{0}\]
with meromorphy equal to $w(\mu_{T})$ for varying $w \in W_{b_{\mu}}$. In the sense that the above modifications of $G$-bundles may be obtained from the modification of $T$-bundles by applying $\times^{T} G$ and acting via automorphisms of $\mathcal{F}_{G}^{0}$ and $\mathcal{F}_{b_{\mu}}$ of $G$-bundles.

In particular, this tells us that the moduli space $\Sht(G,b_{\mu},\mu)_{\infty,\mathbb{C}_{p}}$ parameterizing modifications of meromorphy $\leq \mu$ is actually equal to the open subspace $\Sht(G,b_{\mu},\mu)_{\infty}^{\mu}$ parameterizing modifications of meromorphy equal to $\mu$. This is because any modification induced from a modification 
\[ \mathcal{F}_{w(b_{\mu_{T}})} \dashrightarrow \mathcal{F}_{T}^{0}\]
of type $w(\mu_{T})$ will be of type $\mu$, which implies that we have an isomorphism $\mathcal{S}_{\mu} \simeq \Lambda[d](\frac{d}{2})$ on the Satake sheaf pulled back to $\Sht(G,b_{\mu},\mu)_{\infty,\mathbb{C}_{p}} = \Sht(G,b_{\mu},\mu)^{\mu}_{\infty,\bb{C}_{p}}$, where $d = \langle 2\hat{\rho},\nu_{b_{\mu}} \rangle = \langle 2\hat{\rho},\mu \rangle$ using \cite[Proposition~VI.7.5]{FS}. Here we need to be a bit careful since $\mathcal{S}_{\mu}$ is the pullback of the sheaf associated to the tilting module $\mathcal{T}_{\mu}$ not $V_{\mu}$ as per usual. However, we note that the above discussion tells us that the Newton stratum in the Schubert cell/variety $\Gr_{G,\leq \mu^{-1},\mathbb{C}_{p}}^{b_{\mu}} = \Gr_{G,\mu^{-1},\mathbb{C}_{p}}^{b_{\mu}}$ has only non-empty intersection with the semi-infinite cells $\mathrm{S}_{G,w(\mu_{T}),\mathbb{C}_{p}}$ indexed by the Weyl group orbits of the highest weight, using Remark \ref{rem: semiinfinitecellinterpretation}. Since both $\mathcal{T}_{\mu}$ and $V_{\mu}$ have highest weight with multiplicity one, the discrepancy doesn't matter via Corollary \ref{highweightcohom}. It remains to describe the complex $R\Gamma^{\flat}_{c}(G,b_{\mu},\mu)[\rho_{b_{\mu},w}]$. Using the above observations, \cite[Theorem~4.26]{IG} (as in the proof \cite[Theorem~4.21]{GH}), it follows that we have an isomorphism
\[ R\Gamma^{\flat}_{c}(G,b_{\mu},\mu)[\rho_{b_{\mu},w}] \simeq \Ind_{P_{b_{\mu}}^{-}}^{G}(R\Gamma_{c}^{\flat}(M,b_{\mu_{M}},\mu_{M})[i_{B_{b}}^{J_{b}}(\chi^{w^{-1}}) \otimes \delta_{P_{b_{\mu}}}^{-1/2}][\langle 2\hat{\rho}, \mu \rangle](\langle \hat{\rho},\mu \rangle)) \]
of complexes of $G(\mathbb{Q}_{p}) \times W_{E}$-representation, where $M = M_{b_{\mu}}$ is the centralizer of the slope homomorphism, $\mu_{M}$ is the $G$-dominant choice of $\mu$ viewed as a cocharacter of $M$, and $b_{\mu_{M}} \in B(M,\mu_{M})$ is the $\mu$-ordinary element. We note that, since $b_{\mu} \in B(G)_{\mathrm{un}}$, it follows that $B(M,\mu_{M})$ is a singleton and that $\mu_{M}$ is central with respect to $M$. In particular, $\Sht(M,b_{\mu_{M}},\mu_{M})_{\infty,\bb{C}_{p}}$ is $0$-dimensional and identifies with the locally profinite set $\ul{M(\mathbb{Q}_{p})}$ (recall that $M \simeq J_{b_{\mu}} \simeq J_{b_{\mu_{M}}}$ in this case). This allows us to identify 
\[ R\Gamma^{\flat}_{c}(M_{b},b_{\mu_{M}},\mu_{M})[i_{B_{b}}^{J_{b}}(\chi^{w^{-1}}) \otimes \delta_{P_{b_{\mu}}}^{-1/2}] \simeq i_{B \cap M_{b}}^{M_{b}}(\chi^{w^{-1}}) \otimes \delta_{P_{b_{\mu}}}^{-1/2} \boxtimes w(\mu_{T}) \circ \phi_{T}(-\langle \hat{\rho}, \mu \rangle)|_{W_{E}}, \] 
where we can identify the $1$-dimensional Weil group action through excursion algebra considerations. Therefore, we get an isomorphism
\begin{equation*}
\begin{split}
 R\Gamma^{\flat}_{c}(G,b_{\mu},\mu)[\rho_{b_{\mu},w}] \simeq  \Ind_{P_{b_{\mu}}^{-}}^{G}(i_{B \cap M_{b}}^{M_{b}}(\chi^{w^{-1}}) \otimes \delta_{P_{b_{\mu}}}^{-1/2}) \boxtimes w(\mu_{T}) \circ \phi_{T}|_{W_{E}}[\langle 2\hat{\rho}, \mu \rangle](\langle -\hat{\rho}, \mu \rangle + \langle \hat{\rho}, \mu \rangle) & \simeq \\ i_{B}^{G}(\chi^{rw^{-1}}) \boxtimes w(\mu_{T}) \circ \phi_{T}|_{W_{E}}[\langle 2\hat{\rho}, \mu \rangle] 
\end{split} 
\end{equation*}
of complexes of $G(\mathbb{Q}_{p}) \times W_{E}$-representations which gives the desired result by transitivity of parabolic induction.
\end{proof}
\begin{remark}
We note that in the proof we did not use any of our results on geometric Eisenstein series. It would be interesting to generalize some of these computations to some non-principal situations. In particular, if one works with a general parabolic $P$ with Levi factor $M$ and a supercuspidal parameter $\phi_{M}: W_{\mathbb{Q}_{p}} \ra \phantom{}^{L}M$ and assumes that Fargues' conjecture holds on the Levi subgroup $M$, then the results of \cite{IG} guarantee that one has similar formulas relating the cohomology of the shtuka spaces of $G$ to basic local shtuka spaces of $M$ for Hodge-Newton reducible $b \in B(G,\mu)$ admitting a reduction to $M$. The description of the eigensheaf in Fargues' conjecture, as described for odd unramified unitary groups in \cite{BMNH} and partially for $\mathrm{GSp}_{4}$ in \cite{Ham}, gives one a very explicit description of these basic local shtuka spaces, which would in turn give a computational approach to understanding and generalizing these formulas beyond the principal case. We fully anticipate that some of the methods we use in the principal case generalize to the non-principal case; however, the results seem much more technical, and these computations would give a nice foothold into the problem.  
\end{remark}
This explicit calculation has some interesting consequences. In particular, we already saw in Example \ref{GL2ex} that relating Proposition \ref{muordinary} to Theorem \ref{refinedaveragingformula} required using the existence of an isomorphism $i_{B}^{G}(\chi) \simeq i_{B}^{G}(\chi^{w})$; i.e. intertwining operators.  The significance of Proposition~\ref{muordinary} is that it
realizes the representations $i_{B}^{G}(\chi)$
inside the cohomology of local shtuka spaces.

To make this precise, choose a strictly dominant
Galois-invariant cocharacter $\mu$ and put
\[
C_\mu :=
j_1^*T_\mu\bigl(
\operatorname{nEis}(\mathcal S_{\phi_T})
\bigr).
\]
After forgetting the Weil-group action, the filtered
eigensheaf property gives $C_\mu$ a finite filtration
whose graded pieces are direct sums of copies of
$\pi=i_B^G(\chi)$.

On the other hand, Proposition~\ref{muordinary}, together
with the description of the Eisenstein stalks and
the clean-extension property, identifies direct
summands of $C_\mu$ with the Weyl translates
\[
\pi_w:=i_B^G(\chi^w),
\qquad w\in W_G.
\]
The comparison between these two descriptions
of $C_\mu$ will recover the intertwining isomorphisms $\pi_{w} \simeq \pi$ through geometry.
Although the existence of such isomorphisms was already established
in Section~\ref{s: SecondAdjointness}, this gives a geometric explanation
for their occurrence. In particular, using this line of reasoning we can deduce the following.
\begin{corollary}\label{intertwiningoperators}
Assume Assumption \ref{compatibility}, and let $\Lambda$ be as in
Setup \ref{assumption: coefficients}. Let
$\chi: T(\mathbb{Q}_{p}) \ra \Lambda^{*}$ be the character attached
to an integral toral parameter $\phi_{T}$ whose mod $\ell$-reduction
is generic regular.
\begin{enumerate}
\item For every $w \in W_{G}$, there is an isomorphism
\[
i_{B}^{G}(\chi) \simeq i_{B}^{G}(\chi^{w}).
\]
\item If $\Lambda$ is a field, then $i_{B}^{G}(\chi)$ and
$\rho_{b,w}$, for all $b \in B(G)_{\mathrm{un}}$ and $w \in W_{b}$,
are absolutely irreducible.
\\\\
If $\Lambda$ is not a field, these representations become
absolutely irreducible after reduction modulo the maximal ideal
of $\Lambda$, and also after extension of scalars to $F$ when
$\Lambda = \mathcal{O}_{F}$.
\end{enumerate}
\end{corollary}
\begin{proof}
Write $\pi_{w} := i_{B}^{G}(\chi^{w})$ for $w \in W_{G}$,
and put $\pi := \pi_{1}$. Let $r_{B}^{G}$ denote the normalized
Jacquet module.

As already seen in the proof of Proposition \ref{prop:opposite-borel-comparison}, regularity of the mod $\ell$-reduction implies that the filtration from the geometric
lemma \cite[Section~2.8]{Dat} therefore splits giving
\[
r_{B}^{G}(\pi_{w})
\simeq \bigoplus_{w' \in W_{G}} \chi^{w'},
\]
and by Frobenius reciprocity, it follows that
\[
\operatorname{Hom}_{G(\mathbb{Q}_{p})}(\pi_{w},\pi_{w'})
\simeq \Lambda
\]
for all $w,w' \in W_{G}$ and that this identification is compatible with base-change in $\Lambda$.

If $\bar{k}$ is an algebraic closure of the
residue field of $\Lambda$, the natural map
\[
\operatorname{Hom}_{G(\mathbb{Q}_{p})}(\pi_{w},\pi_{w'})
\otimes_{\Lambda}\bar{k}
\longrightarrow
\operatorname{Hom}_{G(\mathbb{Q}_{p})}
(\pi_{w}\otimes_{\Lambda}\bar{k},
      \pi_{w'}\otimes_{\Lambda}\bar{k})
\]
is an isomorphism. Indeed, only the $\chi^{w'}$ summand
of $r_{B}^{G}(\pi_{w})$ contributes.
The same calculation holds after extending scalars to
$\bar{k}$, since the reductions of the Weyl conjugates
of $\chi$ remain pairwise distinct.

We first prove the assertions when $\Lambda$ is an
algebraically closed field. Choose a $\Gamma$-invariant
strictly dominant geometric cocharacter $\mu$. Then
$M_{b_{\mu}} = J_{b_{\mu}} = T$, $W_{b_{\mu}} = W_{G}$,
and the element $r$ in Proposition \ref{muordinary}
is $w_{0}^{G}$. Put
\[
C :=
j_{\mathbf{1}}^{*}T_{\mu}
(\nmEis(\mathcal{S}_{\phi_{T}})),
\]
and forget its Weil group action. By the filtered
eigensheaf property in Theorem \ref{Eiseigsheaf} and the
description of the stalk at $\mathbf{1}$, the complex
$C$ is concentrated in degree zero and admits a finite
filtration whose successive quotients are copies of $\pi$.

For every $w \in W_{G}$, the sheaf
$\nmEis^{w(b_{\mu_{T}})}(\mathcal{S}_{\phi_{T}})$ is a
direct summand of $\nmEis(\mathcal{S}_{\phi_{T}})$.
Combining \eqref{eqn: stalkcalculation},
Proposition \ref{inertsheaves}, Lemma \ref{shimhecke},
and Proposition \ref{muordinary}, we obtain, after
forgetting the Weil group action,
\[
\begin{split}
j_{\mathbf{1}}^{*}T_{\mu}
(\nmEis^{w(b_{\mu_{T}})}
      (\mathcal{S}_{\phi_{T}}))
&\simeq
R\Gamma_{c}^{\flat}(G,b_{\mu},\mu)[\rho_{b_{\mu},w}]
[-\langle 2\hat{\rho},\nu_{b_{\mu}}\rangle] \\
&\simeq \pi_{w_{0}^{G}w^{-1}}.
\end{split}
\]
Thus every $\pi_{w}$ is a direct summand of $C$.

We use the elementary observation that, if
$\operatorname{End}(P)$ is a division ring, every
nonzero map from $P$ to an object admitting a finite
filtration with successive quotients $P$ is injective.
Indeed, its composite with the final quotient is either
an isomorphism or zero; in the latter case, the assertion
follows by induction on the filtration length.

Applying this observation to a nonzero map
$\pi \ra \pi_{w} \hookrightarrow C$ shows that
$\pi \ra \pi_{w}$ is injective. The same argument with
$\chi$ replaced by $\chi^{w}$ gives an injection
$\pi_{w} \ra \pi$, since generic regularity is preserved under passing to Weyl group conjugates. Both representations have
finite length, so these injections imply
$\pi \simeq \pi_{w}$. This proves (1) over algebraically
closed fields.

To prove (2), over algebraically closed fields we use (1) and argue exactly as in the proof of Corollary \ref{cor:principal-series-detection} (1) using Frobenius reciprocity and second adjointness.

Now suppose that $\Lambda=\mathcal O_F/\varpi^n$
or $\Lambda=\mathcal O_F$. For $w\in W_G$, choose
generators
\[
f\in\operatorname{Hom}_{G(\mathbb Q_p)}(\pi,\pi_w),
\qquad
g\in\operatorname{Hom}_{G(\mathbb Q_p)}(\pi_w,\pi).
\]
By compatibility with reduction to an algebraic
closure of the residue field, the reductions of
$f$ and $g$ are nonzero, hence are isomorphisms
by the algebraically closed field case.
Consequently, their compositions are
multiplication by units of $\Lambda$, so
$f$ is an isomorphism.

For $\Lambda=F$ of characteristic zero,
choose an integral lattice of $\chi$ as in
Definition~\ref{defn: IntegralParameter}. The desired isomorphisms
follow by extending scalars from the
corresponding result over $\mathcal O_F$. This proves (1) for all the
allowed coefficients.

If $\Lambda$ is a field, the algebraically closed
field case proves absolute irreducibility of $\pi$
and $\rho_{b,w}$. The assertions about reduction
modulo the maximal ideal and extension of scalars
from $\mathcal{O}_{F}$ to $F$ also follow from the algebraically closed field case by the compatibility of normalized
parabolic induction with extension of scalars. This proves (2) for all the allowed coefficients.
\end{proof}
This suggests an interesting relationship between the theory of geometric Eisenstein series over the Fargues-Fontaine curve and the classical theory of intertwining operators and the Langlands quotient, which we hope is explored more in the future. This prospect becomes even more exciting in the $\ell$-modular situation. The theory of intertwining operators in this context has been partially explored by Dat \cite[Sections~6-8]{Dat}, and the Langlands quotient theorem does not naively hold in this context, as the following example illustrates. 
\begin{Example}
Suppose that $\ell \neq 2$, $\Lambda = \ol{\bb{F}}_{\ell}$, $\ell \mid p - 1$, and $G = \GL_{2}$. We let $\chi = \delta_{B}^{-1/2}$ be a square root of the modulus character. We note that, by our assumption that $\ell \mid p - 1$, we have an isomorphism $|\cdot| \simeq \mathbf{1}_{T}$. It follows that we have that $\delta_{B}^{-1/2} \simeq \delta_{B}^{1/2}$ (in particular, $\delta_{B}^{1/2}$ is not regular). We consider the usual short exact sequence
\[ 0 \ra \mathbf{1}_{G} \ra i_{B}^{G}(\delta_{B}^{-1/2}) \ra \mathrm{St}_{G} \ra 0  \]
where $\mathrm{St}_{G}$ denotes the Steinberg representation. Acting by smooth duality actually gives a splitting of the short exact sequence and in turn a chain of isomorphisms
\[ i_{B}^{G}(\delta_{B}^{-1/2}) \simeq \mathrm{St}_{G} \oplus \mathbf{1}_{G} \simeq i_{B}^{G}(\delta_{B}^{1/2})  \]
of smooth $G(\mathbb{Q}_{p})$-representations. In particular, we see that $i_{B}^{G}(\delta_{B}^{1/2})$ does not have a unique irreducible quotient in this case, so that the Langlands quotient theorem cannot naively hold.  
\end{Example}
Let's now explore a case in which Proposition \ref{muordinary} can be used to verify Conjecture \ref{weilgroupaction}. Suppose that $\mu$ is a geometric dominant cocharacter such that the image $\mu_{\Gamma} \in \mathbb{X}_{*}(T_{\ol{\mathbb{Q}}_{p}})_{\Gamma}^{+}$ is quasi-minuscule or minuscule with respect to the pairing with $\mathbb{X}_{*}(\hat{T}^{\Gamma})$. In this case, we recall that the orbit of the highest weight space $\mathcal{T}_{\mu}|_{\hat{G}^{\Gamma}}(b_{\mu})$ forms a closed orbit under the relative Weyl group $W_{G}$, where $b_{\mu} \in B(G,\mu)_{\mathrm{un}}$ is the $\mu$-ordinary element. It then follows that all the weight spaces of $\mathcal{T}_{\mu}|_{\hat{G}^{\Gamma}}$ in this orbit will be given by the $\kappa$-invariants of $w(b_{\mu_{T}}) \in B(T) \simeq \mathbb{X}^{*}(\hat{T}^{\Gamma})$, for $w \in W_{G}$ varying. However, by Proposition \ref{muordinary}, we see that all the weight spaces of this form come from the contribution of the $\mu$-ordinary element to the refined averaging formula. If $\mu_{\Gamma}$ is minuscule with respect to the above pairing this is the only weight space in $\mathcal{T}_{\mu}|_{\hat{G}^{\Gamma}}$, and we see that Conjecture \ref{weilgroupaction} is true. If $\mu_{\Gamma}$ is quasi-minuscule, the only other element in $B(G,\mu)_{\mathrm{un}}$ is the basic element, denoted $\mu^{\flat}$. By Corollary \ref{bgmuweights}, the highest weight representation $\mathcal{T}_{\mu}|_{\hat{G}^{\Gamma}}$ admits a central weight space $\mathcal{T}_{\mu}|_{\hat{G}^{\Gamma}}(\mu_{T}^{\flat})$ in this case, where $\mu_{T}^{\flat} \in B(T) \simeq \mathbb{X}^{*}(\hat{T}^{\Gamma})$ is the (unique) reduction to $T$ of $\mu^{\flat} \in B(G,\mu)_{\mathrm{un}}$. We deduce the following from the refined averaging formula. 
\begin{corollary}{\label{zeroweightcontrib}}
Let $\mu$ be a geometric dominant cocharacter such that $\phi_{T}$ is strongly $\mu$-regular with reflex field $E$ and $\phi_{T}$ is integral with generic regular mod $\ell$-reduction. Assume that $\mu_{\Gamma}$ is quasi-minuscule with respect to the pairing with $\mathbb{X}_{*}(\hat{T}^{\Gamma})$. Let $b := \mu^{\flat} \in B(G,\mu)$ be the unique basic element. It follows by Corollary \ref{bgmuweights} that $\mu^{\flat}$ is unramified in this case. We let $\mu^{\flat}_{T}$ be its unique reduction to $B(T)$. There is an isomorphism
\[R\Gamma_{c}(G,b,\mu)[i_{B_{b}}^{J_{b}}(\chi)]|_{W_{E'}} \simeq R\Gamma^{\flat}_{c}(G,b,\mu)[i_{B_{b}}^{J_{b}}(\chi)]|_{W_{E'}} \simeq  \bigoplus_{\substack{\tilde{\mu}^{\flat}_{T} \in \mathbb{X}_{*}(T_{\ol{\mathbb{Q}}_{p}}) \\
\tilde{\mu}^{\flat}_{T\Gamma} = \mu^{\flat}_{T}}} \tilde{\mu}^{\flat}_{T} \circ \phi_{T}|_{W_{E'}} \otimes \mathcal{T}_{\mu}(\tilde{\mu}^{\flat}_{T}) \boxtimes i_{B}^{G}(\chi) \]
of complexes of $W_{E'} \times G(\mathbb{Q}_{p})$-modules, where $E'|E$ denotes the splitting field of $G$.
\end{corollary}
\begin{proof}
The isomorphism of the $\flat$ isotypic part with the non-$\flat$ isotypic part follows from Corollary \ref{rhoduality}, noting that for $b$ basic the shifts and twists do not occur. The rest follows from combining Theorem \ref{refinedaveragingformula} and Proposition \ref{muordinary}. Noting that, by the strong $\mu$-regularity assumption on $\phi_{T}$, the Weil group action of the contribution of the highest weight to the refined averaging formula must be distinct from the Weil group action on the contribution of the central weight spaces, by the vanishing of the $H^{0}$s for the differences of these weight spaces. 
\end{proof}
\begin{remark}
If $G$ is split then the condition that the central weight space of $\mathcal{T}_{\mu}|_{\hat{G}^{\Gamma}}$ is non-zero cannot occur if $\mu$ is minuscule. However, if $G$ is not split then it can occur that the central weight space is non-trivial even if $\mu$ is minuscule. For example, if one considers an odd quasi-split unitary group $\U_n$ and the cocharacter $(1,0,\ldots,0,0)$ then the $\sigma$-centralizer will be isomorphic to $\U_n$, and therefore the basic element $b \in B(G,\mu)$ lies in $B(G)_{\mathrm{un}}$. A more in depth characterization of when this can occur is given in \cite[Remark~4.2.11]{XZ}. This result is a generic fiber analogue of the results in \cite{XZ}. Here, in the case that $G$ is unramified, Xiao and Zhu relate the irreducible components of affine Deligne-Lusztig varieties to the central weight spaces appearing above, and use these irreducible components to construct cohomological correspondences on the special fibers of certain Shimura varieties using uniformization. These affine Deligne-Lusztig varieties are precisely the special fibers of a natural integral model of $\Sht(G,\mu^{\flat},\mu)_{\infty}/\ul{K}$ for a choice of hyperspecial subgroup $K \subset G(\mathbb{Q}_{p})$. More precisely, one can relate the cohomology of the spaces $\Sht(G,\mu^{\flat},\mu)_{\infty}/\ul{K}$ to the cohomology of these affine Deligne-Lusztig varieties by applying nearby cycles to an integral model of $\Sht(G,\mu^{\flat},\mu)_{\infty}/\ul{K}$ at hyperspecial level, and this should allow one to see a direct comparison between these two results.
\end{remark}
\appendix 
\section{Intertwining Operators and the Irreducibility of Principal Series}{\label{append: irred of principal series}}
We want to show the irreducibility of principal series representations obtained from the characters $\chi$ attached to parameters $\phi_{T}: W_{\mathbb{Q}_{p}} \ra \phantom{}^{L}T(\ol{\mathbb{Q}}_{\ell})$ satisfying the conditions in \ref{def: charnormreg}. We let $\chi: T(\mathbb{Q}_{p}) \ra \mathbb{C}^{*}$ be the character attached to $\phi_{T}$ via local class field theory and a fixed isomorphism $\ol{\mathbb{Q}}_{\ell} \simeq \mathbb{C}$ sending the chosen square root of $p$ in $\ol{\mathbb{Q}}_{\ell}$ to the positive real square root. Our goal will be to show the following two facts (Note that for Proposition \ref{prop: regular implies irred} one may also argue as in the proof of Corollary \ref{cor:principal-series-detection} (1) using second adjointness and Frobenius reciprocity; however, we give a separate proof in this section which will be in line with our proof of the other fact.).
\begin{proposition}{\label{prop: regular implies irred}}
Suppose that $\chi$ is a regular character in the sense that it is not fixed under any nontrivial $w \in W_{G}$ then if we have an isomorphism
\[ i_{B}^{G}(\chi) \simeq i_{B}^{G}(\chi^{w}) \]
for every $w \in W_{G}$ the representation $i_{B}^{G}(\chi)$ is irreducible. 
\end{proposition}
\begin{proposition}{\label{prop: genericinterisom}}
If $\chi: T(\mathbb{Q}_{p}) \ra \mathbb{C}^{*}$ is a generic character then, for all $w \in W_{G}$, we have an isomorphism
\[ i_{B}^{G}(\chi) \simeq i_{B}^{G}(\chi^{w}) \]
of smooth $G(\mathbb{Q}_{p})$-representations.
\end{proposition}
By combining these two Propositions, we deduce the following.
\begin{corollary}{\label{cor: genericreg implies irred}}
For $\chi$ a generic regular character, the induction $i_{B}^{G}(\chi)$ is always irreducible. 
\end{corollary}
The idea behind proving such results is due to Speh and Vogan \cite[Theorem~3.14]{SV} in the archimedean case. They study the reducibility of principal series using the Langlands classification \cite[Section~IV]{BW}. Strictly speaking, their analysis is for the archimedean place, but it is easy to see that it extends to the case of a $p$-adic group using the analogous Langlands classification there \cite[Section~XI.2]{BW}. They (roughly) break the problem of understanding the reducibility points of principal series representations into two parts:
\begin{enumerate}
     \item (\cite[Theorem~3.14(a)]{SV}) Understanding the reducibility points of non-unitary principal series with respect to the parabolic inductions from $T$ to $M_{i}$ for $i \in \mathcal{J}$, where $M_{i}$ is the rank $1$ Levi subgroup of $G$ whose relative Dynkin diagram is given by $\{i\} \subset \mathcal{J}$.
     \item (\cite[Theorem~3.14(b)]{SV}) Understanding the reducibility points of unitary principal series representations with respect to induction from $T$ to a (not necessarily rank $1$) proper Levi subgroup of $G$. 
\end{enumerate}
We will now explain this heuristic in our case. To see analogous analysis worked out more explicitly for specific $p$-adic reductive groups, we point the reader to \cite[Section~7]{Tad}, \cite[Section~3]{GM}, \cite[Section~3]{Mati} for a small sample. For $(2)$, a very definitive answer to such questions can be found in the paper of Keys \cite{Keys}. In particular, by \cite[Corollary~1]{Keys} the number of irreducible components of such a unitary parabolic induction can be computed in terms of the Knapp-Stein R-Group \cite{KS,Kn2,Kn1}, which Keys determines for all split groups. While this is very interesting, we will not address this here. In particular, we have the following.
\begin{lemma}{\label{unitaryirred}}
If $\chi$ is a regular unitary character then the normalized parabolic induction $i_{B}^{G}(\chi)$ is irreducible. 
\end{lemma}
\begin{proof}
This follows from the geometric lemma and Frobenius reciprocity, which shows $\mathrm{dim}_{\bb{C}}(\mathrm{Hom}_{G(\bb{Q}_{p})}(i_{B}^{G}(\chi),i_{B}^{G}(\chi))) = \mathrm{dim}_{\bb{C}}(\mathrm{Hom}_{T(\bb{Q}_{p})}(r_{B}^{G}i_{B}^{G}(\chi),\chi)) = 1$, and the fact that $i_{B}^{G}(\chi)$ is unitary and therefore fully decomposable (see \cite[Theorem~6.6.1]{Cas} or \cite{Bru}).
\end{proof}
Now consider $\mathbb{X}^{*}(T_{\ol{\mathbb{Q}}_{p}})^{\Gamma} \otimes \mathbb{R} \simeq \mathbb{X}^*(A) \otimes \mathbb{R} \simeq \mathbb{R}^{d}$ the set of unramified characters. For $s \in \mathbb{R}^{d}$, we write $\nu^{s}$ for the associated unramified character. We will say that $\nu^{s}$ is positive (resp. strictly positive) if, for all simple (reduced) positive coroots $\alpha_{i,A}$, the precomposition of $\nu^{s}$ with $\alpha_{i,A}$ is positive (resp. strictly positive), or in other words that $s$ lies in the positive Weyl chamber of $\mathbb{X}^*(A) \otimes \mathbb{R}$ defined by the Borel. We write $\chi = \mu_{\chi}\nu^{s_{\chi}}$, where $\mu_{\chi}$ is a unitary character and $\nu^{s_{\chi}}$ is an unramified character of $T$, for some $s_{\chi} \in \mathbb{R}^{d}$. We now consider intertwining operators. Recall that, for $w \in W_{G}$ and $s \in \mathbb{R}^{d}$, we have the intertwining operator
\[ I_{w}(\mu_{\chi},s): i_{B}^{G}(\mu_{\chi}\nu^{s}) \ra i_{B}^{G}((\mu_{\chi}\nu^{s})^{w})  \]
\[ f(g) \mapsto \int_{U_{w}} f(w^{-1}ug)du \]
where $U_{w} := U \cap wU^{-}w^{-1}$ and $U^{-}$ is the unipotent radical of the opposite Borel. This integral will converge if $\nu^{s}$ lies sufficiently deep in the dominant Weyl chamber, and admits a meromorphic continuation as a function of $s$ (where one imposes this constraint on the real part). Away from these poles, it gives rise to an intertwining operator between $i_{B}^{G}(\mu_{\chi}\nu^{s})$ and $i_{B}^{G}((\mu_{\chi}\nu^{s})^{w})$ in the usual representation theory sense. For our purposes, we will be interested in the intertwining operator $I_{w_{0}}(\mu_{\chi},s)$ for the element of longest length $w_{0}$. In this case, one can see that the operator is convergent for all $s$ which are strictly positive, and the image of the operator is the unique irreducible Langlands quotient of $i_{B}^{G}(\mu_{\chi}\nu^{s})$ (see \cite[Sections~XI.2.6,XI.2.7]{BW}). This quotient has multiplicity one and therefore the intertwining space between $i_{B}^{G}(\chi)$ and $i_{B}^{G}(\chi^{w_{0}})$ is one-dimensional. It follows that, if $\nu^{s_{\chi}}$ is strictly positive, it suffices to exhibit an isomorphism $i_{B}^{G}(\chi) \simeq i_{B}^{G}(\chi^{w_{0}})$ to show irreducibility. We will now use this kind of analysis to prove Proposition \ref{prop: regular implies irred}.
\begin{proof}(Proposition \ref{prop: regular implies irred})
First off note that, for all $\chi$ and $w \in W_{G}$, we have an equality
\[ [i_{B}^{G}(\chi)] = [i_{B}^{G}(\chi^{w})] \] 
in $K_{0}(G(\mathbb{Q}_{p}))$ (see for example \cite[Theorem~4]{vanDijk}). This allows us to, without loss of generality, assume that $\nu^{s_{\chi}}$ is positive. Now, consider the set of $i \in \mathcal{J}$ such that the precomposition of $\nu^{s_{\chi}}$ with $\alpha_{i,A}$ is trivial. This defines a parabolic $P_{\chi}$ of $G$ which we decompose as $P_{\chi} = M_{\chi}N_{\chi}$, where $A_{\chi}$ is the maximal split torus in the center of $M_{\chi}$. If $\mathcal{J}_{M_{\chi}}$ denotes the vertices of the relative Dynkin diagram of $M_{\chi}$ then $\mathcal{J}_{M_{\chi}} \subset \mathcal{J}$ corresponds to the set of simple positive coroots where this precomposition is trivial. Since $\nu^{s_\chi}$ is trivial on the coroots of $M_\chi$, it extends uniquely to a positive real-valued unramified character $\nu_2$ of $M_\chi$. Set $\nu_1:=\mathbf1$. We consider the parabolic induction
\[ i_{B \cap M_{\chi}}^{M_{\chi}}(\mu_{\chi} \otimes \nu_{1}) \]
where we now note that $\mu_{\chi} \otimes \nu_{1}$ is unitary by construction. Therefore, $i_{B \cap M_{\chi}}^{M_{\chi}}(\mu_{\chi} \otimes \nu_{1})$ is unitary and thereby fully decomposable. It follows that, since $i_{B \cap M_{\chi}}^{M_{\chi}}(\mu_{\chi} \otimes \nu_{1})$ and $i_{B \cap M_{\chi}}^{M_{\chi}}((\mu_{\chi} \otimes \nu_{1})^{w_{0}^{M_{\chi}}})$ are equal in the Grothendieck group, we have an isomorphism $i_{B \cap M_{\chi}}^{M_{\chi}}(\mu_{\chi} \otimes \nu_{1}) \simeq i_{B \cap M_{\chi}}^{M_{\chi}}((\mu_{\chi} \otimes \nu_{1})^{w_{0}^{M_{\chi}}})$, where $w_{0}^{M_{\chi}}$ is the element of longest length of the Weyl group of $M_{\chi}$. Now, we have an isomorphism:
\[ i_{B}^{G}(\chi) \simeq i_{P_{\chi}}^{G}((i_{B \cap M_{\chi}}^{M_{\chi}}(\mu_{\chi} \otimes \nu_{1})) \otimes \nu_{2})  \]
Since $\nu_2$ is invariant under the Weyl group of $M_\chi$, regularity of $\chi$ implies regularity of $\mu_\chi$ for $M_\chi$. Thus, it follows by Lemma \ref{unitaryirred}, that the unitary induction $i_{B \cap M_{\chi}}^{M_{\chi}}(\mu_{\chi} \otimes \nu_{1})$ is irreducible. Therefore, the RHS is the induction of an  irreducible tempered representation times an unramified character $\nu_{2}$ satisfying the property that $\langle \alpha_{i,A}, \nu_{2} \rangle > 0$ for all $i \in \mathcal{J} \setminus \mathcal{J}_{M_{\chi}}$ by construction. Again applying the Langlands classification \cite[Sections~XI.2.6,XI.2.7]{BW}, the intertwining operator attached to the parabolic $P_{\chi}$ and the element $w_{0}w_{0}^{M_{\chi}}$ converges for $s = s_{\chi}$, and since it maps to a unique quotient of multiplicity one it suffices to exhibit an isomorphism between $i_{P_{\chi}}^{G}((i_{B \cap M_{\chi}}^{M_{\chi}}(\mu_{\chi} \otimes \nu_{1})) \otimes \nu_{2})$ and the induction twisted by $w_{0}w_{0}^{M_{\chi}}$. However, since we just saw that $i_{B \cap M_{\chi}}^{M_{\chi}}((\mu_{\chi} \otimes \nu_{1})^{w_{0}^{M_{\chi}}}) \simeq i_{B \cap M_{\chi}}^{M_{\chi}}(\mu_{\chi} \otimes \nu_{1})$, it suffices to show we have an isomorphism $i_{B}^{G}(\chi) \simeq i_{B}^{G}(\chi^{w_{0}})$. This establishes the claim.
\end{proof}
Now we just need to show Proposition \ref{prop: genericinterisom}. 
\begin{proof}
We claim that this reduces to the analogous question for $G$ a group of rank $1$. In particular, let's consider for $i \in \mathcal{J}$ the simple positive (reduced) coroot $\alpha := \alpha_{i,A}$ and the rank $1$ parabolic $P_{\alpha} = M_{\alpha}N_{\alpha}$ attached to it. For all simple positive coroots $\alpha$, we have an isomorphism
\[ i_B^G(\chi)\simeq i_{P_\alpha}^G(i_{B\cap M_\alpha}^{M_\alpha}(\chi)). \]
However, if $w_{\alpha}$ is the simple reflection corresponding to $\alpha$, we have that 
\[ i_B^G(\chi^{w_\alpha})\simeq i_{P_\alpha}^G(i_{B\cap M_\alpha}^{M_\alpha}(\chi^{w_\alpha})). \]
Therefore, if we can show the existence of an isomorphism: 
\[ i_{B\cap M_\alpha}^{M_\alpha}(\chi)\simeq i_{B\cap M_\alpha}^{M_\alpha}(\chi^{w_\alpha}). \]
It will imply that we have an isomorphism:
\[ i_{B}^{G}(\chi) \simeq i_{B}^{G}(\chi^{w_{\alpha}}) \]
Now we can proceed by induction on the length of Weyl group elements. We just described the base case, but then, by replacing $\chi$ with $\chi^{w_{\alpha}}$ we can proceed by considering another simple reflection attached to another simple positive coroot. We note that, since $\chi$ being generic is a condition for all coroots (not just simple), at each step of the induction we are tasked with showing the following. 
\begin{proposition}{\label{rank1irred}}
Let $G$ be the restriction of scalars from a finite extension $L/\bb{Q}_{p}$ of an absolutely simple, simply connected, quasi-split connected reductive group of split rank $1$. Let $\alpha$ be the unique simple (reduced) positive coroot of $G$ and $w_{\alpha}$ the corresponding simple reflection. Then, for all $\chi$ a generic character of $T(\mathbb{Q}_{p})$, we have an isomorphism 
\[ i_{B}^{G}(\chi) \simeq i_{B}^{G}(\chi^{w_{\alpha}}) \]
of smooth $G(\mathbb{Q}_{p})$-representations. 
\end{proposition}
It remains to justify the absolutely simple simply connected assumption. To do this, note that given a $G$ not satisfying these conditions, we can find a surjective central isogeny $f: \tilde{G} \ra G$, where $\tilde{G}$ is a product of torii and Weil restrictions from finite extensions of absolutely simple simply connected groups. If we let $\tilde{B}$ be the preimage of the Borel $B$ with maximal torus given by $\tilde{T}$ the preimage of $T$ then, since $\Ker(f)$ is contained in the center, we have an isomorphism $\tilde{G}/\tilde{B} \simeq G/B$. This implies that we have an isomorphism:
\[ i_{B}^{G}(\chi)|_{\tilde{G}(\mathbb{Q}_{p})} \simeq i_{\tilde{B}}^{\tilde{G}}(\chi|_{\tilde{T}(\mathbb{Q}_{p})}) \]
Moreover, since $f$ will induce an isomorphism on the root spaces, it follows that if $\chi$ is generic with respect to $G$ then $\chi|_{\tilde{T}(\mathbb{Q}_{p})}$ is generic with respect to $\tilde{G}$. Using that the rank one intertwining operator for $G$ pulls back to $\tilde{G}$, this reduces us to exhibiting the desired isomorphism for $\tilde{G}$. Now, we prove Proposition \ref{rank1irred} through brute force. In particular, we will use Tits' classification theorem \cite{Tits} (see also \cite{Carb} for the classification in rank $1$). We adopt the same notation as in \cite{Tits}. Since we are assuming the group to be quasi-split, there are two cases. We work over $L = \bb{Q}_{p}$ for simplicity with the proof in general being analogous.
\subsubsection{$\phantom{}^{1}A^{1}_{1,1}$}
In this case, we have that $G = \mathrm{SL}_{2}$. We saw in Example \ref{Sl2ex} that genericity guaranteed irreducibility aside from the case where $\chi^{2} \simeq \mathbf{1}$, but since this is a unitary character it still follows that we have the desired isomorphism.
\subsubsection{$\phantom{}^{2}A^{1}_{2,1}$}
In this case, the group cannot be split; in particular, we have that $G = \mathrm{SU}_{3}$ is a quasi-split special unitary group attached to a quadratic extension $E/\mathbb{Q}_{p}$. We saw in Example \ref{U3ex} that $\chi$ being generic guaranteed irreducibility in the regular cases, and hence the desired isomorphism in general (where for nonregular case it is automatic).
\end{proof}
\section{Tilting Cocharacters}{\label{tiltingcocharacters}}
We consider a general quasi-split connected reductive group $G/\mathbb{Q}_{p}$, and a geometric dominant cocharacter $\mu \in \mathbb{X}_{*}(T_{\ol{\mathbb{Q}}_{p}})^{+}$. For $\Lambda$ a coefficient system as in assumption \ref{assumption: coefficients}, we are interested in understanding the condition of $\mu$ being tilting (Definition \ref{defmutilting}). Recall that this means that the representation $V_{\mu} \in \Rep_{\Lambda}(\hat{G})$ attached to $\mu$ lies in the subcategory $\Tilt_{\Lambda}(\hat{G})$. If $\Lambda = F/\bb{Q}_{\ell}$ is an algebraic extension this is always true, and so the discussion is only relevant in the integral and torsion cases. Since this only involves the representation theory $\hat{G}$ we may, without loss of generality, assume $G$ is split in what follows. We also assume that $\Lambda = \ol{\bb{F}}_{\ell}$. In this case, it is simply the question of when the highest weight module $V_{\mu}$ of $\hat{G}$ is irreducible with $\Lambda$-coefficients. This question has been studied extensively (see for example \cite[Pages~283-286]{JanRepresentationsofAlgebraicGroups} for a comprehensive overview). In the first two sections, we discuss some general theory to determine when $\mu$ is tilting in this split case, and then provide a table summarizing when $\mu$ is tilting in the case that $\mu$ is a fundamental coweight. 
\subsection{General Theory}
We assume $G$ is a split connected reductive group throughout this section, with Langlands dual group $\hat{G}$. We will be interested in understanding when the highest weight $G$-module $V_{\mu} \in \Rep_{\Lambda}(\hat{G})$ is tilting or equivalently when it is simple. We begin with the following lemma.
\begin{lemma}{\label{mintilting}}
If $\mu$ is minuscule then it is tilting. 
\end{lemma}
\begin{proof}
In this case, the weights of $V_{\mu}$ form a closed Weyl group orbit. It follows that $V_{\mu}$ is always simple and therefore tilting. 
\end{proof}
We would now like to provide a finer criterion for irreducibility. To do this, we will introduce some notation. Given a coroot $\nu$ and $r \in \mathbb{Z}$, we consider the affine reflection of $\mathbb{X}_{*}(T) \otimes \mathbb{R}$ given as
\[ s_{\nu,r}(\mu) := s_{\nu}(\mu) + r\nu  \]
where $s_{\nu} \in W_{G}$ is the reflection attached to $\nu$. We set $W_{\ell}$ to be the subgroup generated by reflections $s_{\nu,n\ell}$, where $\nu$ is a coroot and $n \in \mathbb{Z}$, and write $\rho$ for the half sum of all positive coroots of $G$. Elements $w \in W_{\ell}$ act on $\mu \in \mathbb{X}_{*}(T) \otimes \mathbb{R}$, via the standard dot action $w\cdot\mu :=  w(\mu + \rho) - \rho$. In other words, we regard $s_{\nu,n\ell}$ as a reflection around the hyperplane:
\[ \{\mu \in \mathbb{X}_{*}(T) \otimes \mathbb{R} \text{ | } \langle \mu + \rho, \nu^{\vee} \rangle = n\ell\} \]
It follows that the standard alcove for this action is given by
\[ C = \{\mu \in \mathbb{X}_{*}(T) \otimes \mathbb{R} \text{ | } 0 < \langle \mu + \rho, \nu^{\vee} \rangle < \ell \text{ for all positive coroots }\nu \} \]
and we denote the closure by $\ol{C}$. We now have the following slightly more general criterion for the irreducibility of $V_{\mu}$. 
\begin{proposition}{\cite[Corollary~5.6]{JanRepresentationsofAlgebraicGroups}}{\label{alcovecriterion}}
Suppose that $\mu \in \ol{C} \cap \mathbb{X}_{*}(T)^{+}$ then $\mu$ is tilting. 
\end{proposition}
For a given $\mu$, this will give us a lower bound on the $\ell$ for which $\mu$ is tilting. However, it is only a sufficient condition and not necessary. In particular, note that we have the following.
\begin{theorem}{\cite[Theorem~2.6]{Mat}}{\label{steintilting}}
If $\mu = (\ell -1)\rho\in\mathbb X_*(T)$, then $\mu$ is tilting.
\end{theorem}
Whenever this weight belongs to $\mathbb X_*(T)$, $V_{(\ell - 1)\rho}$ is simple, but $(\ell - 1)(\rho)$ will not usually lie in  $\ol{C} \cap \mathbb{X}_{*}(T)^{+}$. Moreover, if  we define the Coxeter number $h = \max_{\nu}\{ \langle \rho, \nu^{\vee} \rangle + 1 \}$ ranging over all coroots $\nu$ 
then it is easy to see that $C \cap X^*(T)^+ \neq \emptyset$  is equivalent to $\ell \ge h$, and so, for small $\ell$, Proposition \ref{alcovecriterion} tells us nothing. To tackle these more general cases, we introduce a sum formula for the characters of the representations. Namely, for $V \in \Rep_{\Lambda}(\hat{G})$, we write $\ch(V) := \sum_{\nu \in \mathbb{X}_{*}(T)} \dim(V(\nu))e^{\nu}$ for the character of $V$. For $\mu$ a dominant cocharacter, we write $\chi(\mu) := \ch(V_{\mu})$ for the character of $V_{\mu}$. For an arbitrary weight $\eta$, we extend this notation by setting $\chi(\eta)=0$ if $\eta+\rho$ is singular, and otherwise setting $\chi(\eta)=(-1)^{\ell(w)}\chi(w\cdot\eta)$, where $w\in W_G$ makes $w\cdot\eta$ dominant. Then we have the following.
\begin{proposition}{\cite[Section~8.19]{JanRepresentationsofAlgebraicGroups}}{\label{sumformula}}
For each $\mu \in \mathbb{X}_{*}(T)^{+}$, there is a filtration of $\hat{G}$-modules
\[ \cdots \subset V_{\mu}^{2} \subset V_{\mu}^{1} \subset V_{\mu}^{0} = V_{\mu} \]
such that 
\[ \sum_{i > 0} \ch(V_{\mu}^{i}) = \sum_{\nu} \sum_{0 < m\ell < \langle \mu + \rho,\nu^{\vee} \rangle} \nu_{\ell}(m\ell)\chi(s_{\nu,m\ell} \cdot \mu) \]
where $\nu_{\ell}(-)$ is the $\ell$-adic valuation and $\nu$ ranges over all coroots. Moreover, we have that $V_{\mu}/V_{\mu}^{1}$ is isomorphic to the irreducible head of $V_{\mu}$.  In particular, we see that $\mu$ is tilting if and only if 
\[ \sum_{\nu} \sum_{0 < m\ell < \langle \mu + \rho,\nu^{\vee} \rangle} \nu_{\ell}(m\ell)\chi(s_{\nu,m\ell} \cdot \mu) = 0 \]
\end{proposition}
This generalizes Proposition \ref{alcovecriterion} and gives a computational method for verifying when $\mu$ is tilting. See for example \cite[Section~8.20]{JanRepresentationsofAlgebraicGroups} for this worked out for $G = \mathrm{SL}_{4}$ and $\ell > 3$, \cite{GS} for a table answering this question for $G$ an exceptional group and certain $\mu$, and \cite{NG} for an analogous table for certain exceptional groups and low rank classical groups. In general, a precise classification of when $\mu$ is tilting for all $G$ seems to be quite complicated, and to the best of our knowledge is unknown. However, when $G$ is of type $A_{n - 1}$, there exists a complete classification.
\begin{proposition}{\cite[II.8.21]{JanRepresentationsofAlgebraicGroups}}
For $G$ of type $A_{n - 1}$, $\mu$ is tilting if and only if for each positive coroot $\nu$ of $G$ the following is satisfied. Write $\langle \mu + \rho, \nu^{\vee} \rangle = a\ell^{s} + b\ell^{s + 1}$, with $a,b,s \in \mathbb{N}$ and $0 < a < \ell$. Then, for $b > 0$, there have to be positive coroots $\beta_{0},\beta_{1},\ldots,\beta_{b}$ such that $\langle \mu + \rho, \beta_{i}^{\vee} \rangle = \ell^{s + 1}$ for $1 \leq i \leq b$, $\langle \mu + \rho, \beta_{0}^{\vee} \rangle = a\ell^{s}$, $\nu = \sum_{i = 0}^{b} \beta_{i}$, and $\sum_{i = 1}^{b} \beta_{i}$ is a coroot. 
\end{proposition}
For general types, we will content ourselves with describing the fundamental coweights, where we can give a full description of the tilting condition. 
\subsection{The Tilting Condition for Fundamental Coweights}
We assume that $G$ is a split adjoint group, and let $\hat{\alpha}_{j}$ denote the simple roots, where we use the enumeration as in \cite[Pages~250-275]{Bou}. We choose fundamental coweights characterized by $\langle \varpi_{i}, \hat{\alpha}_{j} \rangle = \delta_{ij}$. We will be interested in the question of when the representation $V_{\varpi_{i}}$ is irreducible. In this case, we have a complete classification \cite[Pages~286-287]{JanRepresentationsofAlgebraicGroups},\cite[Section~4.6]{Jan1}. We note that, if $\varpi_{i}$ is minuscule, this is automatic by Lemma \ref{mintilting}, so in what follows we simply provide a list of $\ell$ for the non-minuscule $\varpi_{i}$ (see \cite[Page~221]{RaLa} for a classification). This namely implies the case of $A_{n}$ is trivial, since all fundamental coweights are minuscule. The fourth column enumerates the set of $\ell$ for which $\varpi_{i}$ is not tilting. A common condition in modular representation theory is that of being very good, so we have also enumerated the condition that $\ell$ is very good for the different types.
\begin{center}
\setlength{\tabcolsep}{4pt}
\begin{tabular}{|c|c|c|c|c|} 
\hline
Type of $G$ & Type of $\hat{G}$ & $\mu$ & $\ell$ & $\ell$ very good \\
\multirow{5}{4em}{} & & & & \\ 
\hline
$B_{n}, n \geq 2$ & $C_{n}$ & $\varpi_{i}$, $1 < i \leq n$ & $\ell \mid \binom{n + 1 - (i + j)/2}{(i - j)/2}$, $0 \leq j < i$, $j \cong i \mod 2$ & $\ell \neq 2$\\ 
\hline 
$C_{n},n \geq 2$ & $B_{n}$ & $\varpi_{i}$, $1 \leq i < n$ & $\ell = 2$ & $\ell \neq 2$ \\ 
\hline 
$D_{n}, n \geq 4$ & $D_{n}$ & $\varpi_{i}$, $1 < i < n - 1$ & $\ell = 2$ & $\ell \neq 2$ \\
\hline
$E_{6}$ & $E_{6}$ & $\varpi_{1},\varpi_{6}$ & $\emptyset$ & $\ell \neq 2,3$ \\
 &  & $\varpi_{3},\varpi_{5}$ & $\ell = 2$ & \\
 &  & $\varpi_{2}$ & $\ell = 3$ &  \\
 &  & $\varpi_{4}$ & $\ell = 2,3$ &  \\
 \hline 
$E_{7}$ & $E_{7}$ & $\varpi_{1}$ & $\ell = 2$ & $\ell \neq 2,3$\\
 &  & $\varpi_{2}$ & $\ell = 3$ & \\
 &  & $\varpi_{3}, \varpi_{5}$ & $\ell = 2,3$ &  \\
 &  & $\varpi_{4}$ & $\ell = 2,3,13$  & \\
 &  & $\varpi_{6}$ & $\ell = 2,7$ & \\
 &  & $\varpi_{7}$ & $\emptyset$ & \\
 \hline
$E_{8}$ & $E_{8}$ & $\varpi_{1}$ & $\ell = 2$ & $\ell \neq 2,3,5$ \\
 &  & $\varpi_{2}$ & $\ell = 2,3,7$ &  \\
 &  & $\varpi_{3}$ & $\ell = 2,3,19$&  \\
 &  & $\varpi_{4}$ & $\ell = 2,3,5,13,19$ &  \\
 &  & $\varpi_{5}$ & $\ell = 2,3,5$ &  \\
 &  & $\varpi_{6}$ & $\ell = 2,3,5,7$ &  \\
 &  &  $\varpi_{7}$ & $\ell = 2,3,5$  &  \\
 &  & $\varpi_{8}$  & $\emptyset$ &  \\
      \hline
$F_{4}$ & $F_{4}$ & $\varpi_{1}$ & $\ell = 3$ & $\ell \neq 2,3$\\
&  & $\varpi_{2},\varpi_{3}$ & $\ell = 2,3$ & \\
&  & $\varpi_{4}$ & $\ell = 2$ & \\
\hline 
$G_{2}$ & $G_{2}$ & $\varpi_{1}$ & $\ell = 3$ & $\ell \neq 2,3$\\
 &  & $\varpi_{2}$ & $\ell = 2$ &  \\
\hline
\end{tabular}
\end{center}
By comparing the fourth and fifth columns, we deduce the following.
\begin{proposition}
For $G$ a split adjoint connected reductive group over $\mathbb{Q}_{p}$, if $\ell$ is very good then $\varpi_{i}$ for any $i \in \mathcal{J}$ is tilting for all $G$ of type $A_{n},C_{n},D_{n},E_{6},F_{4}$, and $G_{2}$.
\end{proposition}
\section{Relationship to the Classical Averaging Formula, by Alexander Bertoloni-Meli}{\label{Classicavg}}
In this appendix, we show that the averaging formula proven in Theorem \ref{qellavgingformula} is compatible with existing formulas and conjectures in the literature.
\subsection{Averaging Formulas}
To begin, we recall the general statement of these averaging formulas. Such formulas first appeared in the book of Harris--Taylor (\cite{HT}) and are classically deduced by studying the geometry of the mod-$p$ fibers of Shimura varieties. These fibers admit a Newton stratification in terms of the set $B(G,\mu)$ and the strata are uniformized by Rapoport--Zink spaces and Igusa varieties. The cohomological consequence of this is the formula of Mantovan (\cite{Mant}) which in the normalization of $R\Gamma_c^\flat$ used in
\S \ref{s: Applications}, is given as
\[
\begin{aligned}
&\sum_i\varinjlim_{K\subset G(\mathbb A_f)}
(-1)^i H_c^i
\bigl(\operatorname{Sh}(G,X)_K,\overline{\mathbb Q}_\ell\bigr)
\\
&\quad =
(-1)^{\langle2\hat\rho_G,\mu\rangle}
(
\sum_{b\in B(G,\mu)}
R\Gamma_c^\flat(G,b,\mu)
[
\sum_j(-1)^j
\varinjlim_{K^p\subset G(\mathbb A_f^p)}
H_c^j
(\operatorname{Ig}_{K^p}^b,\overline{\mathbb Q}_\ell)
]
)
(-\frac{\langle2\hat\rho_G,\mu\rangle}{2}).
\end{aligned}
\]
in $K_0(G(\A_f) \times W_{E_{\mu}})$, where $\Shim(G,X)_K$ (resp. $\Ig^b_{K^p}$) is the Shimura variety (resp. Igusa variety) determined by the associated data. Averaging formulas can then be deduced by studying isotypic pieces of the above formula.

In order to precisely state these averaging formulas, we first recall some facts about stable characters following \cite{Hiraga}. Let $G$ be a connected reductive group with Levi subgroup $M$ and parabolic $P$. Let $K_0(G(\Q_p), \C)^{st} \subset K_0(G(\Q_p), \C)$ denote the subgroup of virtual representations with stable character. Then the normalized Jacquet module and parabolic induction functors induce morphisms
\begin{equation*}
    i^G_P: K_0(M(\Q_p), \C) \to K_0(G(\Q_p), \C), \quad \quad  r^G_P: K_0(G(\Q_p), \C) \to K_0(M(\Q_p), \C).
\end{equation*}
Moreover, one can show these operations preserve stability so that we get homomorphisms
\begin{equation*}
    i^G_P: K_0(M(\Q_p), \C)^{st} \to K_0(G(\Q_p), \C)^{st}, \quad \quad  r^G_P: K_0(G(\Q_p), \C)^{st} \to K_0(M(\Q_p), \C)^{st}.
\end{equation*}

Now let $G^*$ denote the unique quasi-split group that is an inner form of $G$. We assume that $G$ arises as an extended pure inner twist of $G^*$ and fix this extra structure $(G, \varrho, z)$. One can work with more general $G$ using Kaletha's theory of rigid inner twists, but this is not necessary to explore the connections to this paper where $G$ is always quasi-split. 

We also need to introduce endoscopy for $G$. For convenience, we recall:
\begin{definition}{\label{def: refinedendoscopy}}
A \emph{refined endoscopic datum} for a connected reductive group $G$ over a local field $F$ is a tuple $(H, \mc{H}, s,\eta)$ which consists of
\begin{itemize}
    \item a quasi-split group $H$ over $F$,
    \item an extension $\mc{H}$ of $W_F$ by $\widehat{H}$ such that the map $W_F \to \mathrm{Out}(\widehat{H})$ coincides with the map $\rho_H: W_F \to \mathrm{Out}(\widehat{H})$ induced by the action of $W_F$ on $\widehat{H} \subset {}^LH$,
    \item an element $s \in Z(\widehat{H})^{\Gamma}$,
    \item an $L$-homomorphism $\eta: \mc{H} \to {}^LG$,
\end{itemize}
satisfying the condition:
\begin{itemize}
    \item we have $\eta(\widehat{H}) = Z_{\widehat{G}}(s)^{\circ}$.
\end{itemize}
\end{definition}

Now suppose that $(H, \mc{H}, s,\eta)$ is a refined endoscopic datum for $G$. Then after fixing splittings of $G, H, \widehat{G}, \widehat{H}$, there is a canonical endoscopic transfer of distributions inducing a morphism 
\begin{equation*}
    \Trans^G_H: K_0(H(\Q_p), \C)^{st} \to K_0(G(\Q_p), \C).
\end{equation*}

Furthermore, suppose we have a refined endoscopic datum $(H_M, \mc{H}_M, s, \eta_M)$ of $M$ such that $\mc{H}_M$ is a Levi subgroup of $\mc{H}$ and the following diagram commutes:
\begin{equation*}
  \begin{tikzcd}
\mc{H} \arrow[r, "{\eta}"] & {}^LG\\
\mc{H}_M \arrow[u, hook] \arrow[r, "{\eta_M}"] & {}^LM \arrow[u, hook].
\end{tikzcd}  
\end{equation*}
The datum $(H_M, \mc{H}_M, H, M, s, \eta)$ along with these compatibilities is called an \emph{embedded} endoscopic datum in \cite{BM,BMS}. Our fixed choice of splittings of $G$, $H$ and their duals determines from $P$ a parabolic subgroup $P_{H_M}$ of $H$ with Levi subgroup $H_M$. We then have an equality
\begin{equation}{\label{eqn: transindcommute}}
    \Trans^G_H \circ i^H_{P_{H_M}} = i^G_P \circ \Trans^M_{H_M}.
\end{equation}
There is also a compatibility of $\Trans$ and $r$ which we now recall. A refined endoscopic datum $\mf{e}=(H, \mc{H}, s, \eta)$ of $G$ and a Levi subgroup $M \subset G$ can be upgraded to the structure of an embedded endoscopic datum in potentially many non-equivalent ways and these are parametrized by a set $D(M, \mf{e}) \cong W(\widehat{H}) \setminus W(M, H) / W(\widehat{M})$, where $W(M,H)$ is defined to be the subset of the Weyl group $W(\widehat{G})$ of $\widehat{G}$ such that $Z_{{}^LH}((w \circ \eta)^{-1}(Z(\widehat{M})^{\Gamma}))$ surjects onto $W_F$, and where $W(\widehat{H})$ is identified with a subgroup of $W(\widehat{G})$ via $\eta$. Then we have
\begin{equation*}
    r^G_P \circ \Trans^G_H = \sum\limits_{ D(M, \mf{e})} \Trans^{M}_{H_M} \circ r^H_{P_{H_M}}.
\end{equation*}

We now define the map $\Red^{\mf{e}}_b$ which plays a crucial role in the statement of the averaging formula. We define
\begin{equation}
    \Red^{\mf{e}}_b: K_0(H(\Q_p), \C)^{st} \to K_0(J_b(\Q_p), \C)
\end{equation}
by 
\begin{equation*}
   \Red^{\mf{e}}_b=(\sum\limits_{D(M,\mf{e})} \Trans_{H_M}^{J_b} \circ r^H_{P^{op}_H}) \otimes \ov{\delta}^{-\frac{1}{2}}_Pe(J_{b}), 
\end{equation*}
where $\ov{\delta}_{P}$ is the transport of the modulus character for $P$ to $J_b$, and $e(J_{b}) \in \{\pm 1\}$ is the Kottwitz sign. For the averaging formulas below, we assume that $G$ is quasi-split,
as in the main text. For $b\in B(G,\mu)$, we have
\[
e(J_b)
=
(-1)^{\langle2\hat\rho_G,\mu\rangle
-\langle2\hat\rho_G,\nu_b\rangle}.
\]
using \cite[Lemma~A.2.1]{KW}.

Then the (still largely conjectural) averaging formula gives a relation satisfied by $R\Gamma^{\flat}_c(G,b,\mu)$ at Langlands (or Arthur) parameters $\phi$ for which there is an associated stable distribution $S\Theta_{\phi,G}$ on $G$ satisfying endoscopic character identities as in \cite[Conjecture D]{Kal}. For tempered $L$-parameters, we use the stable distributions and endoscopic character identities supplied by an assumed refined
local Langlands correspondence.

To describe the expected formula, fix such a parameter $\phi$ and suppose that $(H, \mc{H}, s, \eta)$ is an endoscopic datum such that $\phi$ factors as $\mc{L}_F \xrightarrow{\phi^H} \mc{H} \to {}^LG$. Then we expect
\begin{conjecture}[Averaging Formula]{\label{conj: classicavgformula}}
We expect the following equality in $K_0(G(\Q_p) \times W_{E_{\mu}})$:
\[
\begin{aligned}
&\sum_{b\in B(G,\mu)}
(-1)^{\langle2\hat\rho_G,\nu_b\rangle}
[
R\Gamma_c^\flat(G,b,\mu)[\operatorname{Red}_b^e(S\Theta_{\phi_H,H})]]
\\
&\qquad =
\operatorname{Trans}_H^G(S\Theta_{\phi_H,H})
\boxtimes
\operatorname{tr}
\bigl(r_\mu\circ\phi|_{W_{E_\mu}}\mid s\bigr).
\end{aligned}
\]
In particular, when $\mf{e}$ is the trivial endoscopic datum given by $\mf{e}_{\triv} = (G^*, {}^LG^*, 1, \id)$ then we expect
\[
\begin{aligned}
&\sum_{b\in B(G,\mu)}
(-1)^{\langle2\hat\rho_G,\nu_b\rangle}
[
R\Gamma_c^\flat(G,b,\mu)
[
\operatorname{Red}_b^{e_{\mathrm{triv}}}
(S\Theta_{\phi,G^*})]]
\\
&\qquad =
S\Theta_{\phi,G}\boxtimes r_\mu\circ\phi|_{W_{E_\mu}}.
\end{aligned}
\]
\end{conjecture}

For $\GL_n$, this is known (up to accounting for differences in normalization) in the trivial endoscopic case for all representations by \cite{Shin}. Because $L$-packets are singletons for $\GL_n$, the trivial endoscopic case implies the endoscopic versions of the formula. In \cite{BMN}, the formula is proven for discrete parameters and elliptic endoscopy of unramified $\GU_n$, for $n$ odd. In \cite{BM}, a strategy is outlined to prove this formula in the elliptic endoscopic cases using the cohomology of Igusa and Shimura varieties. This strategy should (eventually) yield results comparable to \cite{BMN} whenever adequate global results are known about the Langlands correspondence and the cohomology of Shimura and Igusa varieties. 

The averaging formulas imply strong results about $R\Gamma^{\flat}_c(G,b,\mu)$. For instance, in \cite[\S6]{BMN} it is shown that the averaging formula for each elliptic endoscopic group and for $\phi$ a supercuspidal parameter implies the Kottwitz conjecture as in \cite[Conjecture 7.3]{RV}.

\subsection{Proof of Proposition \ref{classicalavgrel}}

The averaging formula in \S \ref{s: avgformula} corresponds to the case of the trivial endoscopic triple $\mf{e}_{\triv} =(G^*,{}^LG^*,1, \id)$. Hence to check that Theorem \ref{qellavgingformula} agrees with \ref{conj: classicavgformula}, it suffices to verify that
\[
[\operatorname{Red}_{b,\phi}]
=
(-1)^{\langle 2\hat\rho_G,\nu_b\rangle}
\operatorname{Red}_b^{e_{\mathrm{triv}}}
(S\Theta_{\phi,G^*}).
\]
for $\phi$ induced from $\phi_{T}$ generic. Since $\phi_{T}$ is generic, by Lemma \ref{regmonodromy} $\phi$ should give rise to a well-defined $L$-parameter with trivial monodromy. By the hypothesis of Proposition~\ref{classicalavgrel}, we have that
$S\Theta_{\phi,G^*}=[\pi]=[i_B^G(\chi)]$, under the $\LLC_{G}$ appearing in Assumption \ref{compatibility}. Suppose first that $b \in B(G)_{\mathrm{un}}$. In this case $J_b\simeq M_b$, so $e(J_b)=1$. Then we have
\begin{equation*}
    [\Red_{b, \phi}] = \sum\limits_{w \in W_G/W_{M_b}} i^{J_b}_{B_b}(\chi^{w^{-1}}) \otimes \delta_{P_{b}}^{-1/2}(-1)^{\langle 2\hat{\rho}_{G}, \nu_{b} \rangle}.
\end{equation*}
in $K_{0}(J_{b}(\mathbb{Q}_{p}))$. The set $D(M, \mf{e})$ is a singleton and corresponds to the trivial embedded datum where $H_M=M$. Note that $r^G_{P^{op}_b}(i^G_B(\chi)) = r^G_{P_b}(i^G_B(\chi))$ and that the latter term can be simplified by the geometric lemma of \cite{BZ}. 
\[
\begin{aligned}
\operatorname{Red}_b^{e_{\mathrm{triv}}}
(S\Theta_{\phi,G^*})
&=
\bigl(
\operatorname{Trans}_{M_b}^{J_b}\circ r_{P_b^{\mathrm{op}}}^{G}
\bigr)([\pi])\otimes\delta_{P_b}^{-1/2}
\\
&=
\sum_{w\in W_G/W_{M_b}}
[
i_{B_b}^{J_b}(\chi^{w^{-1}})
\otimes\delta_{P_b}^{-1/2}
]
\\
&=
(-1)^{\langle2\hat\rho_G,\nu_b\rangle}
[\operatorname{Red}_{b,\phi}],
\end{aligned}
\]
where the last equality is as in Definition~\ref{defn: Red} and we have implicitly used  \eqref{eqn: transindcommute}.

Now consider the case where $b \notin B(G)_{\mathrm{un}}$. We must show that $\Red^{\mf{e}_{\triv}}_b(i^G_B(\chi)) = 0$, for which it suffices to show that $\Trans^{J_b}_{M_b}(i^{M_b}_{B \cap M_b}(\chi^{w^{-1}}))=0$ for each $w \in W_G/W_{M_b}$. This follows from the fact that $T$ does not transfer to $J_b$ by assumption and the character of $i^{M_b}_{B \cap M_b}(\chi^{w^{-1}})$ is supported on the conjugates of $T$ as per \cite[Theorem 3]{vanDijk}.

\printbibliography 
\end{document}